\documentclass[12pt,a4paper,reqno]{amsart}
\usepackage{amsfonts,amsthm,amsmath,amssymb,
	amsbsy,euscript,textcomp} 

\newtheorem{theor}{Theorem}

\newtheorem*{claimNo}{Claim}
\theoremstyle{definition}
\newtheorem*{theorNo}{Theorem}

\newtheorem{proposition}[theor]{Proposition}
\newtheorem{lemma}[theor]{Lemma}
\newtheorem{cor}[theor]{Corollary}
\newtheorem{define}{Definition}

\newtheorem{example}{Example}
\newtheorem{implement}{Implementation}
\newtheorem{method}{Method}
\theoremstyle{remark}
\newtheorem{rem}{Remark}

\def\oldvec{\mathaccent "017E\relax }
\DeclareMathOperator{\Or}{\mathsf{O\oldvec{r}}}

\newcommand{\pinner}{\mathbin{\mathchoice
{\hbox{\vrule width0.6em depth0pt height0.4pt
	\vrule width0.4pt depth0pt height0.8ex}}
{\hbox{\vrule width0.6em depth0pt height0.4pt
	\vrule width0.4pt depth0pt height0.8ex}}
{\hbox{\kern0.14em
	\vrule width0.48em depth0pt height0.4pt
	\vrule width0.4pt depth0pt height0.6ex\kern0.14em}}
{\hbox{\kern0.1em
	\vrule width0.39em depth0pt height0.4pt
	\vrule width0.4pt depth0pt height0.5ex\kern0.1em}}}}

\newcommand{\BBR}{\mathbb{R}}

\newcommand{\BBZ}{\mathbb{Z}}
\newcommand{\BBP}{\mathbb{P}}

\newcommand{\bcP}{{\boldsymbol{\mathcal{P}}}}

\newcommand{\cP}{\mathcal{P}}

\newcommand{\bx}{{\boldsymbol{x}}}

\DeclareMathOperator{\Arg}{Arg}

\DeclareMathOperator{\Jac}{Jac}
\DeclareMathOperator{\Edge}{Edge}
\DeclareMathOperator{\Assoc}{Assoc}

\newcommand{\lshad}{[\![}
\newcommand{\rshad}{]\!]}

\newcommand{\by}[1]{\textit{{#1}}}
\newcommand{\jour}[1]{\textit{{#1}}}
\newcommand{\vol}[1]{\textbf{{#1}}}
\newcommand{\book}[1]{\textrm{{#1}}}

\newcommand*{\vcenteredhbox}[1]{\begingroup
\setbox0=\hbox{#1}\parbox{\wd0}{\box0}\endgroup}

\makeatletter
\DeclareRobustCommand{\square}{\mathbin{\mathpalette\morphic@square\relax}}
\newcommand{\morphic@square}[2]{%
  \sbox\z@{$\m@th#1\rule{4pt}{4pt}$}%
  \vcenter{\box\z@}%
}
\makeatother

\title[Computer\/-\/assisted proof schemes in deformation quantization]{The expansion $\star$ mod $\bar{o}(\hbar^4)$ and\\ computer\/-\/assisted proof schemes in\\ the Kontsevich deformation quantization}
\author[R.~Buring]{R.~Buring${}^{*}$}
\thanks{${}^*$\textit{Address}: 
Institut f\"ur Mathematik, 
Johannes Gutenberg\/--\/Uni\-ver\-si\-t\"at,
Staudingerweg~9, 
\mbox{D-\/55128} Mainz, Germany.
	\quad${}^*$\textit{E-mail}: \texttt{rburing\symbol{"40}uni-mainz.de}}

\author[A.~V.~Kiselev]{A.\,V.\,Kiselev${}^{\S}$}
\thanks{${}^{\S}$\textit{Address}: Johann Ber\-nou\-lli Institute for Mathematics and Computer Science, University of Groningen,
	P.O.~Box 407, 9700~AK Groningen, The Netherlands.
	\quad${}^{\S}$\textit{E-mail}: \texttt{A.V.Kiselev\symbol{"40}rug.nl}
}

\date{14 March 2018; in original form 20 December 2017, in final form 5 October 2019}



\begin{document}
\begin{abstract}
The Kontsevich deformation quantization combines Poisson dynamics, noncommutative geometry, number theory, and calculus of oriented graphs.
To ma\-na\-ge the algebra and differential calculus of series of weighted graphs, we present software modules: these allow generating the 
Kontsevich graphs, expanding the noncommutative $\star$-\/product by using \textit{a priori} 
undetermined coefficients, and 
deriving 
linear relations between the weights of graphs.
Throughout this text we illustrate 
the assembly of the Kontsevich $\star$-\/product up to order~$4$ in the
deformation parameter~$\hbar$.
Already at this stage, the $\star$-\/product involves hundreds of graphs; expressing all their coefficients via~$149$ 
weights of basic graphs (of which $67$~weights are now known exactly), we express the remaining $82$~weights in terms of only $10$~parameters (more spe\-ci\-fi\-cal\-ly, in terms of only~$6$ parameters modulo gauge\/-\/equivalence). 
Finally, we out\-li\-ne a scheme for computer\/-\/assisted proof of the associativity,
modulo~$\bar{o}(\hbar^4)$, for the newly built $\star$-\/product expansion.
\end{abstract}
\maketitle

\vspace{-1.5em}

\tableofcontents

\newpage

\subsection*{Introduction}
\addcontentsline{toc}{section}{Introduction}
On every finite-dimensional affine (i.e.\ piecewise\/-\/linear) manifold~$N^n$, the Kontsevich 
star\/-\/product~$\star$ \cite{MK97} is an associative but not necessarily commutative deformation of the usual product $\times$ in the algebra of functions $C^\infty(N^n)$ towards a given Poisson bracket $\{\cdot,\cdot\}_\cP$ on $N^n$ (see also \cite{Gerstenhaber,BFFLSI,BFFLSII}).
Specifically, whenever $\star = \times + \hbar\,\{\cdot,\cdot\}_\cP + \bar{o}(\hbar)$ is an infinitesimal deformation, it can always be completed to an associative star\/-\/product $\star = \times + \hbar\,\{\cdot,\cdot\}_\cP + \sum_{k\geqslant 2}\hbar^k B_k(\cdot,\cdot)$ in the space of formal power series $C^\infty(N^n)[[\hbar]]$; this was proven in~\cite{MK97}. 
An explicit calculation of the bi-linear bi-differential terms $B_k(\cdot,\cdot)$ at high orders $\hbar^k$ is a computationally hard problem.
In this paper we reach the order $k=4$ in expansion of $\star$ by using software modules for the Kontsevich graph calculus, which we presently discuss.

Convenient in practice
, the idea from~\cite{MK97} (see also~\cite{MK93,MK94,Ascona96}) is to draw every de\-ri\-va\-tion $\partial_i \equiv \partial/\partial x^i$ (with respect to a local coordinate~$x^i$ on a chart in the Poisson ma\-ni\-fold~$N^n$ at hand) as 
decorated edge $\begin{picture}(20.0,5.0)\put(0.0, 2.5){\vector(1,0){20.0}\put(-13,2.0){\tiny$i$}}\end{picture}$, so that large differential expressions become oriented graphs.
For example, the Poisson bracket $\{f,g\}_\cP(\bx) = \sum_{i,j=1}^n (f)\overleftarrow{\partial_i}\bigr|_{\bx}\cdot\cP^{ij}(\bx)\cdot\overrightarrow{\partial_j}\bigr|_{\bx}(g)$ of two functions $f,g \in C^\infty(N^n)$ is depicted by the graph $(f)\xleftarrow{i}\cP^{ij}\xrightarrow{j}(g)$;
here $\cP^{ij}$ is the skew-symmetric matrix of Poisson bracket coefficients and the summation over $i,j$ running from~$1$ to the dimension~$n$ of~$N^n$ is implicit.
In these terms, the known --\,from~\cite{sqs15}\,-- expansion of the Kontsevich star-product\footnote{We recall that the expansion $\star$ mod $\bar{o}(\hbar^2)$ in \cite{MK97} was gauge-equivalent to the genuine one so that the two-cycle graph at $\hbar^2/6$ in the first line of above formula \eqref{EqStarOrd3} was gauged out: see Example~\ref{ExGaugedLoopImplement} on p.~\pageref{ExGaugedLoopImplement} where we explain how this is done.} looks as follows:\footnote{The indication~$L$ and~$R$ for Left $\prec$ Right, respectively, matches the indices --\,which the pairs of edges carry\,-- with the ordering of indices in the coefficients of the Poisson structure contained in the arrowtail vertex. Note that exactly \emph{two} edges are issued from every internal vertex in every graph in formula~\eqref{EqStarOrd3};
not everywhere displayed in~\eqref{EqStarOrd3}, the ordering $L\prec R$ in each term is determined from same object's expansion~\eqref{EqStarOrd3Formula}.}
\vspace{-3mm}
\begin{multline}
\text{\raisebox{-8.5pt}{
\unitlength=0.7mm
\linethickness{0.4pt}
\begin{picture}(12.67,5.67)
\put(2.00,5.00){\circle*{1.33}}
\put(12.00,5.00){\circle*{1.33}}
\put(7.00,5.00){\makebox(0,0)[cc]{$\star$}}
\put(2.00,1.33){\makebox(0,0)[cc]{$f$}}
\put(12.00,1.33){\makebox(0,0)[cc]{$g$}}
\end{picture}
}}
= 
\text{\raisebox{-8.5pt}{
\unitlength=0.7mm
\linethickness{0.4pt}
\begin{picture}(15.00,5.67)
\put(0.00,5.00){\line(1,0){15.00}}
\put(2.00,5.00){\circle*{1.33}}
\put(13.00,5.00){\circle*{1.33}}
\put(2.00,1.33){\makebox(0,0)[cc]{$f$}}
\put(13.00,1.33){\makebox(0,0)[cc]{$g$}}
\end{picture}
}}
{+}\frac{\hbar^1}{1!}
\text{\raisebox{-12pt}{
\unitlength=0.7mm
\linethickness{0.4pt}
\begin{picture}(15.00,16.67)
\put(0.00,5.00){\line(1,0){15.00}}
\put(2.00,5.00){\circle*{1.33}}
\put(13.00,5.00){\circle*{1.33}}
\put(2.00,1.33){\makebox(0,0)[cc]{$f$}}
\put(13.00,1.33){\makebox(0,0)[cc]{$g$}}
\put(7.33,16.00){\circle*{1.33}}
\put(7.33,16.00){\vector(-1,-2){5.00}}
\put(7.33,16.00){\vector(1,-2){5.00}}
\end{picture}
}}
{+}\frac{\hbar^2}{2!}
\text{\raisebox{-12pt}{
\unitlength=0.7mm
\linethickness{0.4pt}
\begin{picture}(15.00,20.67)
\put(0.00,5.00){\line(1,0){15.00}}
\put(2.00,5.00){\circle*{1.33}}
\put(13.00,5.00){\circle*{1.33}}
\put(2.00,1.33){\makebox(0,0)[cc]{$f$}}
\put(13.00,1.33){\makebox(0,0)[cc]{$g$}}
\put(7.67,11.67){\circle*{1.33}}
\put(7.67,20.00){\circle*{1.33}}
\put(7.67,20.00){\vector(-1,-3){4.67}}
\put(7.67,20.00){\vector(1,-3){4.67}}
\put(7.67,11.67){\vector(-3,-4){4.00}}
\put(7.67,11.67){\vector(3,-4){4.33}}
\end{picture}
}}
{+}\frac{\hbar^2}{3}{\Biggl(}
\text{\raisebox{-12pt}{
\unitlength=0.7mm
\linethickness{0.4pt}
\begin{picture}(15.00,17.67)
\put(0.00,5.00){\line(1,0){15.00}}
\put(2.00,5.00){\circle*{1.33}}
\put(13.00,5.00){\circle*{1.33}}
\put(2.00,1.33){\makebox(0,0)[cc]{$f$}}
\put(13.00,1.33){\makebox(0,0)[cc]{$g$}}
\put(7.33,11.33){\circle*{1.33}}
\put(2.00,17.00){\circle*{1.33}}
\put(2.00,17.00){\vector(0,-1){11.33}}
\put(2.00,17.00){\vector(1,-1){5.33}}
\put(7.33,11.33){\vector(1,-1){5.33}}
\put(7.33,11.33){\vector(-1,-1){5.33}}
\end{picture}
}}
{+}
\text{\raisebox{-12pt}{
\unitlength=0.7mm
\linethickness{0.4pt}
\begin{picture}(15.00,18.00)
\put(0.00,5.00){\line(1,0){15.00}}
\put(2.00,5.00){\circle*{1.33}}
\put(13.00,5.00){\circle*{1.33}}
\put(2.00,1.33){\makebox(0,0)[cc]{$f$}}
\put(13.00,1.33){\makebox(0,0)[cc]{$g$}}
\put(7.33,11.33){\circle*{1.33}}
\put(7.33,11.33){\vector(1,-1){5.33}}
\put(7.33,11.33){\vector(-1,-1){5.33}}
\put(13.00,17.33){\circle*{1.33}}
\put(13.00,17.33){\vector(0,-1){11.67}}
\put(13.00,17.33){\vector(-1,-1){5.33}}
\end{picture}
}}
{\Biggr)}+\frac{\hbar^2}{6}
\text{\raisebox{-12pt}{
\unitlength=0.7mm
\linethickness{0.4pt}
\begin{picture}(15.00,20.33)
\put(0.00,5.00){\line(1,0){15.00}}
\put(2.00,5.00){\circle*{1.33}}
\put(13.00,5.00){\circle*{1.33}}
\put(2.00,1.33){\makebox(0,0)[cc]{$f$}}
\put(13.00,1.33){\makebox(0,0)[cc]{$g$}}
\put(2.00,15.00){\circle*{1.33}}
\put(13.00,15.00){\circle*{1.33}}
\put(13.00,15.00){\vector(0,-1){9.33}}
\put(2.00,15.00){\vector(0,-1){9.33}}
\bezier{128}(2.00,15.00)(7.00,9.00)(12.67,15.00)
\bezier{128}(13.00,15.00)(7.00,20.33)(2.67,15.00)
\put(11.67,14.00){\vector(1,1){0.67}}
\put(3.33,16.00){\vector(-1,-1){0.67}}
\end{picture}
}}
{+}\\
{+}
\frac{\hbar^3}6{\Biggl(}
\text{\raisebox{-12pt}{
\unitlength=0.70mm
\linethickness{0.4pt}
\begin{picture}(15.00,29.00)
\put(0.00,5.00){\line(1,0){15.00}}
\put(2.00,5.00){\circle*{1.33}}
\put(13.00,5.00){\circle*{1.33}}
\put(2.00,1.33){\makebox(0,0)[cc]{$f$}}
\put(13.00,1.33){\makebox(0,0)[cc]{$g$}}
\put(7.67,11.67){\circle*{1.33}}
\put(7.67,20.00){\circle*{1.33}}
\put(7.67,20.00){\vector(-1,-3){4.67}}
\put(7.67,20.00){\vector(1,-3){4.67}}
\put(7.67,11.67){\vector(-3,-4){4.00}}
\put(7.67,11.67){\vector(3,-4){4.33}}
\put(7.67,28.33){\circle*{1.33}}
\put(7.67,28.33){\circle*{1.33}}
\put(7.67,28.33){\vector(-1,-4){5.67}}
\put(7.67,28.33){\vector(1,-4){5.67}}
\end{picture}
}}
{+}
\text{\raisebox{-12pt}{
\unitlength=0.70mm
\linethickness{0.4pt}
\begin{picture}(15.00,23.67)
\put(0.00,5.00){\line(1,0){15.00}}
\put(2.00,5.00){\circle*{1.33}}
\put(13.00,5.00){\circle*{1.33}}
\put(2.00,1.33){\makebox(0,0)[cc]{$f$}}
\put(13.00,1.33){\makebox(0,0)[cc]{$g$}}
\put(2.00,15.00){\circle*{1.33}}
\put(13.00,15.00){\circle*{1.33}}
\put(13.00,15.00){\vector(0,-1){9.33}}
\put(2.00,15.00){\vector(0,-1){9.33}}
\bezier{128}(2.00,15.00)(7.00,9.00)(12.67,15.00)
\bezier{128}(13.00,15.00)(7.00,20.33)(2.67,15.00)
\put(11.67,14.00){\vector(1,1){0.67}}
\put(3.33,16.00){\vector(-1,-1){0.67}}
\put(7.33,23.00){\circle*{1.33}}
\put(7.33,23.00){\vector(-2,-3){4.33}}
\put(7.33,23.00){\vector(2,-3){4.67}}
\end{picture}
}}
{+}
\text{\raisebox{-12pt}{
\unitlength=0.70mm
\linethickness{0.4pt}
\begin{picture}(15.00,20.33)
\put(0.00,5.00){\line(1,0){15.00}}
\put(2.00,5.00){\circle*{1.33}}
\put(13.00,5.00){\circle*{1.33}}
\put(2.00,1.33){\makebox(0,0)[cc]{$f$}}
\put(13.00,1.33){\makebox(0,0)[cc]{$g$}}
\put(2.00,15.00){\circle*{1.33}}
\put(13.00,15.00){\circle*{1.33}}
\put(13.00,15.00){\vector(0,-1){9.33}}
\put(2.00,15.00){\vector(0,-1){9.33}}
\bezier{128}(2.00,15.00)(7.00,9.00)(12.67,15.00)
\bezier{128}(13.00,15.00)(7.00,20.33)(2.67,15.00)
\put(11.67,14.00){\vector(1,1){0.67}}
\put(3.33,16.00){\vector(-1,-1){0.67}}
\put(7.00,10.33){\circle*{1.33}}
\put(7.00,10.33){\vector(-1,-1){5.33}}
\put(7.00,10.33){\vector(1,-1){5.33}}
\end{picture}
}}
{+}
\text{\raisebox{-12pt}{
\unitlength=0.70mm
\linethickness{0.4pt}
\begin{picture}(15.00,23.33)
\put(0.00,5.00){\line(1,0){15.00}}
\put(2.00,5.00){\circle*{1.33}}
\put(13.00,5.00){\circle*{1.33}}
\put(2.00,1.33){\makebox(0,0)[cc]{$f$}}
\put(13.00,1.33){\makebox(0,0)[cc]{$g$}}
\put(7.33,11.33){\circle*{1.33}}
\put(2.00,17.00){\circle*{1.33}}
\put(2.00,17.00){\vector(0,-1){11.33}}
\put(2.00,17.00){\vector(1,-1){5.33}}
\put(7.33,11.33){\vector(1,-1){5.33}}
\put(7.33,11.33){\vector(-1,-1){5.33}}
\put(7.33,22.67){\circle*{1.33}}
\put(7.33,22.67){\vector(-1,-1){5.33}}
\put(7.33,22.67){\vector(0,-1){10.67}}
\end{picture}
}}
{+}
\text{\raisebox{-12pt}{
\unitlength=0.70mm
\linethickness{0.4pt}
\begin{picture}(15.00,22.67)
\put(0.00,5.00){\line(1,0){15.00}}
\put(2.00,5.00){\circle*{1.33}}
\put(13.00,5.00){\circle*{1.33}}
\put(2.00,1.33){\makebox(0,0)[cc]{$f$}}
\put(13.00,1.33){\makebox(0,0)[cc]{$g$}}
\put(7.33,11.33){\circle*{1.33}}
\put(7.33,11.33){\vector(1,-1){5.33}}
\put(7.33,11.33){\vector(-1,-1){5.33}}
\put(13.00,17.33){\circle*{1.33}}
\put(13.00,17.33){\vector(0,-1){11.67}}
\put(13.00,17.33){\vector(-1,-1){5.33}}
\put(7.33,22.00){\circle*{1.33}}
\put(7.33,22.00){\vector(1,-1){5.33}}
\put(7.33,22.00){\vector(0,-1){9.67}}
\end{picture}
}}
{+}
\text{\raisebox{-12pt}{
\unitlength=0.70mm
\linethickness{0.4pt}
\begin{picture}(15.00,23.33)
\put(0.00,5.00){\line(1,0){15.00}}
\put(2.00,5.00){\circle*{1.33}}
\put(13.00,5.00){\circle*{1.33}}
\put(2.00,1.33){\makebox(0,0)[cc]{$f$}}
\put(13.00,1.33){\makebox(0,0)[cc]{$g$}}
\put(7.33,11.33){\circle*{1.33}}
\put(2.00,17.00){\circle*{1.33}}
\put(2.00,17.00){\vector(0,-1){11.33}}
\put(2.00,17.00){\vector(1,-1){5.33}}
\put(7.33,11.33){\vector(1,-1){5.33}}
\put(7.33,11.33){\vector(-1,-1){5.33}}
\put(2.00,22.67){\circle*{1.33}}
\bezier{128}(2.00,22.00)(-5.00,13.50)(2.33,4.75)
\put(1,6.33){\vector(1,-1){0.67}}
\put(2.00,22.67){\vector(1,-2){5.00}}
\end{picture}
}}
{+}
\text{\raisebox{-12pt}{
\unitlength=0.70mm
\linethickness{0.4pt}
\begin{picture}(20.00,24.67)
\put(0.00,5.00){\line(1,0){15.00}}
\put(2.00,5.00){\circle*{1.33}}
\put(13.00,5.00){\circle*{1.33}}
\put(2.00,1.33){\makebox(0,0)[cc]{$f$}}
\put(13.00,1.33){\makebox(0,0)[cc]{$g$}}
\put(7.33,11.33){\circle*{1.33}}
\put(7.33,11.33){\vector(1,-1){5.33}}
\put(7.33,11.33){\vector(-1,-1){5.33}}
\put(13.00,17.33){\circle*{1.33}}
\put(13.00,17.33){\vector(0,-1){11.67}}
\put(13.00,17.33){\vector(-1,-1){5.33}}
\put(13.00,24.00){\circle*{1.33}}
\put(13.00,24.00){\vector(-1,-2){6.00}}
\bezier{128}(13.00,24.00)(20.00,17.00)(13.67,6.33)
\put(14.00,7.33){\vector(-1,-3){0.33}}
\end{picture}
}}
{\Biggr)}{+}\\
{+}\frac{\hbar^3}3{\Biggl(}
\text{\raisebox{-12pt}{
\unitlength=0.70mm
\linethickness{0.4pt}
\begin{picture}(15.00,17.66)
\put(0.00,5.00){\line(1,0){15.00}}
\put(2.00,5.00){\circle*{1.33}}
\put(13.00,5.00){\circle*{1.33}}
\put(2.00,1.33){\makebox(0,0)[cc]{$f$}}
\put(13.00,1.33){\makebox(0,0)[cc]{$g$}}
\put(7.33,11.33){\circle*{1.33}}
\put(2.00,17.00){\circle*{1.33}}
\put(2.00,17.00){\vector(0,-1){11.33}}
\put(2.00,17.00){\vector(1,-1){5.33}}
\put(7.33,11.33){\vector(1,-1){5.33}}
\put(7.33,11.33){\vector(-1,-1){5.33}}
\put(7.33,8.00){\circle*{1.33}}
\put(7.33,8.00){\vector(-2,-1){5.33}}
\put(7.33,8.00){\vector(2,-1){5.33}}
\end{picture}
}}
{+}
\text{\raisebox{-12pt}{
\unitlength=0.70mm
\linethickness{0.4pt}
\begin{picture}(15.00,17.99)
\put(0.00,5.00){\line(1,0){15.00}}
\put(2.00,5.00){\circle*{1.33}}
\put(13.00,5.00){\circle*{1.33}}
\put(2.00,1.33){\makebox(0,0)[cc]{$f$}}
\put(13.00,1.33){\makebox(0,0)[cc]{$g$}}
\put(7.33,11.33){\circle*{1.33}}
\put(7.33,11.33){\vector(1,-1){5.33}}
\put(7.33,11.33){\vector(-1,-1){5.33}}
\put(13.00,17.33){\circle*{1.33}}
\put(13.00,17.33){\vector(0,-1){11.67}}
\put(13.00,17.33){\vector(-1,-1){5.33}}
\put(7.33,8.33){\circle*{1.33}}
\put(7.33,8.33){\vector(-2,-1){5.33}}
\put(7.33,8.33){\vector(2,-1){5.33}}
\end{picture}
}}
{\Biggr)}
{+}
\frac{\hbar^3}6{\Biggl(}
\text{\raisebox{-12pt}{
\unitlength=0.70mm
\linethickness{0.4pt}
\begin{picture}(15.00,23.67)
\put(0.00,5.00){\line(1,0){15.00}}
\put(2.00,5.00){\circle*{1.33}}
\put(13.00,5.00){\circle*{1.33}}
\put(2.00,1.33){\makebox(0,0)[cc]{$f$}}
\put(13.00,1.33){\makebox(0,0)[cc]{$g$}}
\put(2.00,15.00){\circle*{1.33}}
\put(13.00,15.00){\circle*{1.33}}
\put(13.00,15.00){\vector(0,-1){9.33}}
  \put(12,7.5){\llap{{\tiny R}}}
\put(2.00,15.00){\vector(0,-1){9.33}}
  \put(3,7.5){{\tiny L}}
\bezier{128}(2.00,15.00)(7.00,9.00)(12.67,15.00)
\bezier{128}(13.00,15.00)(7.00,20.33)(2.67,15.00)
\put(11.67,14.00){\vector(1,1){0.67}}
\put(3.67,15.67){\vector(-1,-1){0.67}}
\put(2.00,23.00){\circle*{1.33}}
\bezier{128}(2.00,23.00)(-5,14.00)(1.33,6.00)
\put(2.00,23.00){\vector(3,-2){10.67}}
\put(1,6.67){\vector(1,-2){0.67}}
\end{picture}
}}
{+}
\text{\raisebox{-12pt}{
\unitlength=0.70mm
\linethickness{0.4pt}
\begin{picture}(19.67,23.00)
\put(0.00,5.00){\line(1,0){15.00}}
\put(2.00,5.00){\circle*{1.33}}
\put(13.00,5.00){\circle*{1.33}}
\put(2.00,1.33){\makebox(0,0)[cc]{$f$}}
\put(13.00,1.33){\makebox(0,0)[cc]{$g$}}
\put(2.00,15.00){\circle*{1.33}}
\put(13.00,15.00){\circle*{1.33}}
\put(13.00,15.00){\vector(0,-1){9.33}}
  \put(12,7.5){\llap{{\tiny R}}}
\put(2.00,15.00){\vector(0,-1){9.33}}
  \put(3,7.5){{\tiny L}}
\bezier{128}(2.00,15.00)(7.00,9.00)(12.67,15.00)
\bezier{128}(13.00,15.00)(7.00,20.33)(2.67,15.00)
\put(11.67,14.00){\vector(1,1){0.67}}
\put(3.33,16.00){\vector(-1,-1){0.67}}
\put(13.00,22.33){\circle*{1.33}}
\put(13.00,22.33){\vector(-3,-2){10.33}}
\bezier{128}(13.00,22.33)(19.67,14.33)(13.67,5.67)
\put(14.33,6.67){\vector(-1,-3){0.33}}
\end{picture}
}}{+}
\text{\raisebox{-12pt}{
\unitlength=0.70mm
\linethickness{0.4pt}
\begin{picture}(15.00,21.33)
\put(0.00,5.00){\line(1,0){15.00}}
\put(2.00,5.00){\circle*{1.33}}
\put(13.00,5.00){\circle*{1.33}}
\put(2.00,1.33){\makebox(0,0)[cc]{$f$}}
\put(13.00,1.33){\makebox(0,0)[cc]{$g$}}
\put(7.33,11.33){\circle*{1.33}}
\put(7.33,11.33){\vector(1,-1){5.33}}
\put(7.33,11.33){\vector(-1,-1){5.33}}
\put(13.00,17.33){\circle*{1.33}}
\put(13.00,17.33){\vector(0,-1){11.67}}
\put(13.00,17.33){\vector(-1,-1){5.33}}
\put(2.00,22.67){\circle*{1.33}}
\put(2.00,22.67){\vector(0,-1){16.67}}
\put(2.00,22.67){\vector(2,-1){10.33}}
\end{picture}
}}
{+}
\text{\raisebox{-12pt}{
\unitlength=0.70mm
\linethickness{0.4pt}
\begin{picture}(15.00,23.33)
\put(0.00,5.00){\line(1,0){15.00}}
\put(2.00,5.00){\circle*{1.33}}
\put(13.00,5.00){\circle*{1.33}}
\put(2.00,1.33){\makebox(0,0)[cc]{$f$}}
\put(13.00,1.33){\makebox(0,0)[cc]{$g$}}
\put(7.33,11.33){\circle*{1.33}}
\put(2.00,17.00){\circle*{1.33}}
\put(2.00,17.00){\vector(0,-1){11.33}}
\put(2.00,17.00){\vector(1,-1){5.33}}
\put(7.33,11.33){\vector(1,-1){5.33}}
\put(7.33,11.33){\vector(-1,-1){5.33}}
\put(13.00,22.67){\circle*{1.33}}
\put(13.00,22.67){\vector(0,-1){16.67}}
\put(13.00,22.67){\vector(-2,-1){10.33}}
\end{picture}
}}
{\Biggr)}{+}\overline{o}(\hbar^3).
\label{EqStarOrd3}
\end{multline}
By construction, every oriented edge carries its own index and every {\em internal} vertex (not containing the arguments~$f$ or~$g$) is inhabited by a copy of the coefficient matrix $\cP = (\cP^{ij})$ of the Poisson bracket $\{\cdot,\cdot\}_\cP$.
This means that expansion~\eqref{EqStarOrd3} encodes the analytic formula
\begin{multline*}
f \star g = f \times g 
+\hbar
\cP^{ij} \partial_{i} f \partial_{j} g 
+\hbar^{2}\big(
\tfrac{1}{2} \cP^{ij} \cP^{k\ell} \partial_{k} \partial_{i} f \partial_{\ell} \partial_{j} g 
+\tfrac{1}{3} \partial_{\ell} \cP^{ij} \cP^{k\ell} \partial_{k} \partial_{i} f \partial_{j} g \\
-\tfrac{1}{3} \partial_{\ell} \cP^{ij} \cP^{k\ell} \partial_{i} f \partial_{k} \partial_{j} g 
-\tfrac{1}{6} \partial_{\ell} \cP^{ij} \partial_{j} \cP^{k\ell} \partial_{i} f \partial_{k} g 
\big)
+\hbar^{3}\big(
\tfrac{1}{6} \cP^{ij} \cP^{k\ell} \cP^{mn} \partial_{m} \partial_{k} \partial_{i} f \partial_{n} \partial_{\ell} \partial_{j} g 
\end{multline*}
\begin{multline}
-\tfrac{1}{6} \partial_{m} \partial_{\ell} \cP^{ij} \partial_{n} \partial_{j} \cP^{k\ell} \cP^{mn} \partial_{i} f \partial_{k} g 
-\tfrac{1}{6} \cP^{ij} \partial_{n} \cP^{k\ell} \partial_{\ell} \cP^{mn} \partial_{k} \partial_{i} f \partial_{m} \partial_{j} g \\ 
-\tfrac{1}{6} \partial_{m} \partial_{\ell} \cP^{ij} \partial_{n} \cP^{k\ell} \cP^{mn} \partial_{k} \partial_{i} f \partial_{j} g  
-\tfrac{1}{6} \partial_{m} \partial_{\ell} \cP^{ij} \partial_{n} \cP^{k\ell} \cP^{mn} \partial_{i} f \partial_{k} \partial_{j} g \\ 
+\tfrac{1}{6} \partial_{n} \partial_{\ell} \cP^{ij} \cP^{k\ell} \cP^{mn} \partial_{m} \partial_{k} \partial_{i} f \partial_{j} g  
+\tfrac{1}{6} \partial_{n} \partial_{\ell} \cP^{ij} \cP^{k\ell} \cP^{mn} \partial_{i} f \partial_{m} \partial_{k} \partial_{j} g \\ 
+\tfrac{1}{3} \partial_{n} \cP^{ij} \cP^{k\ell} \cP^{mn} \partial_{m} \partial_{k} \partial_{i} f \partial_{\ell} \partial_{j} g  
-\tfrac{1}{3} \partial_{n} \cP^{ij} \cP^{k\ell} \cP^{mn} \partial_{k} \partial_{i} f \partial_{m} \partial_{\ell} \partial_{j} g \\ 
-\tfrac{1}{6} \partial_{\ell} \cP^{ij} \partial_{n} \partial_{j} \cP^{k\ell} \cP^{mn} \partial_{m} \partial_{i} f \partial_{k} g  
+\tfrac{1}{6} \partial_{n} \partial_{\ell} \cP^{ij} \partial_{j} \cP^{k\ell} \cP^{mn} \partial_{i} f \partial_{m} \partial_{k} g \\ 
-\tfrac{1}{6} \partial_{n} \cP^{ij} \cP^{k\ell} \partial_{\ell} \cP^{mn} \partial_{k} \partial_{i} f \partial_{m} \partial_{j} g  
-\tfrac{1}{6} \partial_{\ell} \cP^{ij} \partial_{n} \cP^{k\ell} \cP^{mn} \partial_{k} \partial_{i} f \partial_{m} \partial_{j} g  
\big) + \bar{o}(\hbar^3).\label{EqStarOrd3Formula}
\end{multline}
We now see that the language of Kontsevich graphs is more intuitive and easier to percept than writing formulae.
The calculation of the associator
$\Assoc_\star(f,g,h) = (f \star g)\star h - f\star(g\star h)$ 
can also be done in a pictorial way (see section \ref{SecGraphsOnGraphs} on p. \pageref{SecGraphsOnGraphs}).
The coefficients of graphs at~$\hbar^k$ in a star-product expansion are given by the Kontsevich integrals over the configuration spaces of $k$~distinct points in the Lobachevsky plane~$\mathbb{H}$, see~\cite{MK97} and~\cite{CF2000}. 
Although proven to exist, such weights of graphs are very hard 
to obtain in practice.\footnote{In fact, 
there are many other admissible graphs, not shown in~\eqref{EqStarOrd3}, in which every internal vertex is a tail for two oriented edges, but the weights of those graphs are found 
to be zero.}
Much research has been done on deriving helpful relations between the weights in order to facilitate their calculation \cite{WillwacherFelderIrrationality, Polyak, Kathotia, FelderShoikhet, BenAmar}. 
In Example~\ref{ExOrder3} on p.~\pageref{ExOrder3} we explain how expansion \eqref{EqStarOrd3} modulo $\bar{o}(\hbar^3)$ was obtained in \cite{sqs15}. 
The techniques which were then sufficient are no longer enough
to build the Kontsevich $\star$-\/product beyond the order~$\hbar^3$; clearly, extra 
mathematical concepts and computational tools must be developed.
In this paper we present the software in which several known relations between the Kontsevich graph weights are taken into account; 
we express the weights of all 
graphs at~$\hbar^4$ in terms of 
$10$~master-parameters.
(To be more precise, the ten master-parameters are reduced to just $6$ by taking the quotient over certain four degrees of gauge freedom in the associative star-product expansions mod $\bar{o}(\hbar^4)$.)
This paper is aimed to provide much more than a reference to computer programs: it also contains a synopsis of the proofs for the ideas in the construction, as well as an explanation of the parts which require computer implementation.
Now, the values of Kontsevich graph weights and, with more input from the work in progress \cite{BanksPanzerPym, PymPrivateComm}, \emph{all} the values which specify $\star$ mod $\bar{o}(\hbar^4)$ are the main result of this paper.
These weights (as well as the ones of higher-order expansion terms) are subject to conjectures and open problems (see \cite{WillwacherFelderIrrationality, BanksPanzerPym}).

This paper contains four 
chapters.
In chapter~\ref{SecGraphs} we introduce the software to encode and generate the Kontsevich graphs 
and operate with series of such graphs. In particular, the coefficients of graphs in series can 
be undetermined variables. 
The series are then reduced modulo the skew\/-\/symmetry of graphs 
(under the swapping of Left~$\rightleftarrows$~Right in their construction).
Thirdly, a series can be evaluated at a given Poisson structure: that is,
a copy of the bracket is placed at every internal vertex.

Chapter~\ref{SecStar} is devoted to the construction of Kontsevich's $\star$-\/product: containing a given Poisson structure in its leading deformation term, this bi\/-\/linear operation is not necessarily commutative but it is required to be associative; hence the 
coefficients of a power series for~$\star$ must be specified.
For example, at order~$k=4$ of the deformation parameter~$\hbar$ there are 
$149$~parameters to be found.
(The actual number of graphs at~$\hbar^4$ is much greater; we here count the ``basic'' graphs only.)
We review a number of methods to obtain the weights of Kontsevich graphs; the spectrum of techniques employed ranges from complex analysis and direct numeric integration~\cite{Decin} to finding linear relations between such weights by using abstract geometric 
reasonings.
The associativity of Kontsevich's $\star$-\/product is a major source
of relations between the graph weights;
at~$\hbar^4$ such relations are \emph{linear} because everything is known about the weights up to order three.
We obtain these relations at order four in chapter~\ref{SecAssoc} and we solve that system of 
linear algebraic equations for $149$~unknowns. 
The solution is expressed in terms of only~$10$ mas\-ter\/-\/parameters,
see formula~\eqref{EqStar4} on pp.~\pageref{SecConcl}--\pageref{EqStar4}.%
\footnote{The values of all these ten master-parameters have recently been claimed by Panzer and Pym \cite{PymPrivateComm} as a result of implementation of another technique to calculate the Kontsevich weights: see Table~\ref{Tab10Pym} on p.~\pageref{Tab10Pym} in Appendix~\ref{App10Pym}.
In particular, the values which we conjecture in Table~\ref{TableConjectured} fully agree with the exact values suggested in \cite{PymPrivateComm}.
Based on this external input, the expansion of the Kontsevich $\star$-product becomes \eqref{EqStarWith10Pym} on pp.~\pageref{EqStarWith10PymStart}--\pageref{EqStarWith10PymEnd}.}
It is readily seen that the final formula \eqref{EqStarWith10Pym} on pp.~\pageref{EqStarWith10PymStart}--\pageref{EqStarWith10PymEnd}, in which the ten parameters are assigned specific real values so that \emph{all} the coefficients in $\star$ mod $\bar{o}(\hbar^4)$ are the values of Kontsevich's integrals, is the genuine formula of the Kontsevich $\star$-product.
Indeed, our formula $\star$ mod $\bar{o}(\hbar^4)$ is obtained according to the Proof scheme for Theorem~\ref{ThmBig} on p.~\pageref{ThmBig} below.
We compute a big system of equations which is satisfied by Kontsevich's formula: it is constrained by Lemmas~\ref{LemmaPermute}--\ref{LemmaMult} (basic identities), Proposition~\ref{PropCyclic} (cyclic weight relations), Methods 1--3 (associativity), and vanishing of some integrands (cf. Appendix~\ref{AppNumericalWeights}).
Having solved this system, we incorporate external input \cite{PymPrivateComm,BanksPanzerPym} in the form of direct calculation of Kontsevich's integrals.

The algebraic system constructed in section~\ref{SecRestrictAssoc} 
was obtained by restricting the associativity for~$\star$ to (a class of) specific Poisson structures.
We want however to prove that for the newly found collection of graph weights, the $\star$-\/product is associative for \emph{every} Poisson structure on \emph{all} finite\/-\/dimensional affine manifolds.
For that, in section~\ref{SecFactor} we design a computer\/-\/assisted proof scheme that is independent of the bracket (and of a manifold at hand). 
Specifically, in Theorem~\ref{ThMainAssocOrd4} on p.~\pageref{ThMainAssocOrd4} we reveal how the associator for Kontsevich's $\star$-\/product, taken modulo~$\bar{o}(\hbar^4)$, is factorised via the Jacobiator~$\Jac(\cP)$ or via its differential consequences that all vanish identically for Poisson structures~$\cP$ on the manifolds~$N^n$. 
We discover in particular that such factorisation,
\[
\Assoc_\star(f,g,h)=\Diamond\bigl(\cP,\Jac(\cP),\Jac(\cP)\bigr)\mod \bar{o}(\hbar^4),
\]
is quadratic and has differential order two with respect to the Jacobiator.
For all Poisson brackets~$\{\cdot,\cdot\}_{\cP}$ on finite\/-\/dimensional affine manifolds~$N^n$ our ten\/-\/parameter
expression of the $\star$-\/product does agree up to~$\bar{o}(\hbar^4)$ with 
previously known results about the values of Kontsevich graph weights at some fixed values of the ten master\/-\/parameters and about the linear relations between those weights at all values of the master\/-\/parameters.\footnote{%
From Theorem~\ref{ThMainAssocOrd4} we also assert that the associativity of Kontsevich's $\star$-\/product does not carry on 
but it can leak at orders~$\hbar^{\geqslant 4}$ of the deformation parameter, should one enlarge the construction of~$\star$ to an affine bundle set\/-\/up of $N^n$-\/valued fields over a given affine manifold~$M^m$ and of variational Poisson brackets~$\{\cdot,\cdot\}_{\bcP}$ for local functionals $F,G,H\colon C^\infty(M^m\to N^n)\to\Bbbk$, see~\cite{gvbv,sqs13,dq17,prg15} and~\cite{cycle16}.} 
In an extensive Discussion on pp.~\pageref{SecDiscussion}--\pageref{SecDiscussionEnd}, we compare (and, again, verify) our result with other work, namely by Gutt et al \cite{AmmarChloupGuttUniversal3}, Ben Amar \cite{BenAmar}, Kathotia \cite{Kathotia}, Willwacher \cite{WillwacherObstruction}, and Penkava--Vanhaecke \cite{PenkavaVanhaecke}.
Further discussion of our result is contained in section 6.2 on p.~61 in the most recent preprint \cite{BanksPanzerPym}.
The following list of insights is gained as a byproduct of our approach:
\begin{itemize}
\item Relations between the Kontsevich graph weights can be obtained by viewing 
the $\star$-product associator $\operatorname{Assoc}_\star(\mathcal{P})(f,g,h) = 0$
for a Poisson structure 
$\mathcal{P} = \mathcal{P}(\boldsymbol\psi)$
as a polydifferential operator on $f,g,h,\boldsymbol\psi$ (see \S\ref{SecRestrictAssoc}, Method 3).
This new technique (effective by virtue of computer implementation) yields many new relations. In particular:

\item All the weights of graphs at order $3$ in the $\star$-product are uniquely determined (see Example~\ref{ExOrder3} on p.~\pageref{ExOrder3}) by the associativity equation up to order $4$ for Poisson structure \eqref{Eq3DPoisson} on $\mathbb{R}^3$, the elementary Lemmas \ref{LemmaPermute}--\ref{LemmaMult}, and the cyclic weight relations up to order 4.
This is one instance of:

\item Linear relations between \emph{only} weights of graphs at order $n$ can be obtained (in an effective, predictable way) from the associativity equation at orders greater than $n$ (see Remark~\ref{RemCanUseHigherOrders} on p.~\pageref{RemCanUseHigherOrders}).
This is explained using the decomposition of a polydifferential operator into homogeneous components and the notion of ``prime'' Kontsevich graphs.

\item The proof of associativity of the $\star$-product at order $4$ must involve a \emph{second}-order differential consequence of the Jacobi identity (see the second part of Theorem~\ref{ThMainAssocOrd4} on p.~\pageref{ThMainAssocOrd4}).
In particular, a naive jet space extension of Kontsevich's star product, where derivatives are replaced by variations, is in general not associative at order $4$ (see Corollary~\ref{CorVariational} on p.~\pageref{CorVariational}).

\item The mechanism of vanishing via differential consequences of the Jacobi identity may start working for the $\star$-product expansion itself (see Theorem~\ref{ThNull} on p.~\pageref{ThNull}).
In fact, the order $4$ is the first where this may happen.
(It could have happened at order $3$, if the weights of graphs were different.)

\item So far, from the work of Willwacher (see \cite{WillwacherObstruction}) it was known that graphs with two-cycles, or loops, cannot be eliminated all at once from the star-product by using gauge transformations.
At $\hbar^2$, the only such graph can be removed indeed (see Example~\ref{ExGaugedLoop}); at $\hbar^3$ there are four loopful graphs out of $13$ graphs with nonzero coefficients.
We discovered a totally unexpected fact (see p.~\pageref{LoopsDominant} below): at $\hbar^4$, the graphs with loops are dominant: $138$ out of $247$.
\end{itemize}
\centerline{\rule{1in}{0.7pt}}

\noindent
The software implementation \cite{kgs} consists of a~\texttt{C++} library and a set of command-line programs.
The latter are specified in what follows; a full list of new~\texttt{C++} subroutines and their call syntaxis is contained in Appendix~\ref{AppCPP}.
Whenever a command-line program refers to just one particular function in~\texttt{C++}, we indicate that in the text.
The current text refers to version {\tt 0.66} of the software.
This and future versions are available from \[ \verb"https://github.com/rburing/kontsevich_graph_series-cpp" \]

\noindent
All data files constructed and referred to in this article (in plain text format, which can be appreciated independently of the software) are available in the \verb"data" subdirectory: \[ \verb"https://github.com/rburing/kontsevich_graph_series-cpp/tree/master/data" \]

\noindent$\boldsymbol{\copyright}$\qquad
The copyright for all newly designed software modules which are 
presented in this paper is retained by R.\,Buring; 
provisions of the~MIT free software license apply.

\newpage

\section{Weighted graphs}\label{SecGraphs}
\noindent%
In this section we introduce the software to operate with series of oriented graphs.
	
\subsection{Normal forms of graphs and their machine-readable format}\label{SecNormalForm}
As it was explained in the introduction, we consider graphs whose vertices contain Poisson structures and whose edges represent derivatives. To be precise, the class of graphs to deal with is as follows.

\begin{define}\label{DefKontsevichGraph}
Let us consider a class of oriented graphs on $m+n$ vertices labelled $0$,\ $\ldots$,\ $m+n-1$ such that the consecutively ordered vertices $0$,\ $\ldots$,\ $m-1$ are sinks, and each of the internal vertices $m$,\ $\ldots$,\ $m+n-1$ is a source for two edges. 
For every internal vertex, the two outgoing edges are ordered using $L \prec R$: the preceding edge is labeled $L$ (Left) and the other is $R$ (Right).
An oriented graph on $m$ sinks and $n$ internal vertices 
is a \emph{Kontsevich graph} of type $(m,n)$.
We denote by $G_{m,n}$ the set of all Kontsevich graphs of type $(m,n)$, and by $\tilde{G}_{m,n}$ the subset of $G_{m,n}$ consisting of all those graphs having neither double edges nor tadpoles.
\end{define}

\begin{example}
The star-product expansion \eqref{EqStarOrd3} contains graphs in $\tilde{G}_{2,k}$ for $0 \leqslant k \leqslant 3$.
\end{example}

\begin{rem}
The class of graphs which we consider is not the most general type 
considered by Kontsevich in~\cite{MK97}.
In the construction of the Formality morphism there also appear graphs with sources for more or fewer (than two) arrows.
However, in our approach to the problem at hand, which is the construction of a $\star$-product expansion that would be associative modulo~$\hbar^k$ for some~$k\gg0$, we shall only meet graphs from the class of Definition~\ref{DefKontsevichGraph}.
Actually, to be more accurate, the Leibniz graphs in Definition~\ref{DefLeibnizGraph} on p.~\pageref{DefLeibnizGraph} are Kontsevich graphs where some vertices have \emph{three} outgoing edges; these are expanded into ordinary Kontsevich graphs (built of wedges) by inserting the Jacobiator at the tri-valent vertex; see \cite{Kiev18} for more details.
\end{rem}

\begin{rem}
There can be tadpoles 
or cycles in a graph~$\Gamma \in G_{m,n}$,
see Fig.~\ref{FigTadpoleEye}.
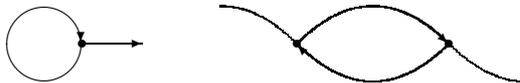
\begin{figure}[htb]
\unitlength=1mm
\special{em:linewidth 0.4pt}
\linethickness{0.4pt}
\begin{picture}(25.00,15.00)
\put(10.00,10.00){\circle{10.00}}
\put(15.00,10.00){\circle*{1}}
\put(14.85,11.15){\vector(0,-1){1.00}}
\put(15.00,10.00){\vector(1,0){8.00}}
\end{picture}
\quad
\unitlength=1mm
\special{em:linewidth 0.4pt}
\linethickness{0.4pt}
\begin{picture}(40.00,20.00)
\put(10.00,10.00){\circle*{1}}
\put(30.00,10.00){\circle*{1}}
\bezier{256}(10.00,10.00)(20.00,20.00)(30.00,10.00)
\bezier{256}(10.00,10.00)(20.00,0.00)(30.00,10.00)
\bezier{48}(0.00,15.00)(5.00,15.00)(10.00,10.00)
\bezier{48}(30.00,10.00)(35.00,5.00)(40.00,5.00)
\put(28.9,11.1){\vector(1,-1){1.00}}
\put(11.1,8.9){\vector(-1,1){1.00}}
\end{picture}
\caption{A tadpole and an ``eye''.}
\label{FigTadpoleEye}
\end{figure}
\end{rem}


A Kontsevich graph $\Gamma \in G_{m,n}$ is uniquely determined by the numbers~$n$ and~$m$ together with the list of ordered pairs of targets for the internal vertices.
For reasons which will become clear immediately below, we now consider a Kontsevich graph $\Gamma$ together with a {\em sign} $s \in \{0, \pm 1\}$, denoted by concatenation of the symbols: $s\Gamma$.

\begin{implement}[encoding]\label{DefEncoding}
The format to store a signed graph $s\Gamma$ with $\Gamma \in G_{m,n}$ is the integer number $m>0$, the integer $n \geqslant 0$, the sign $s$, followed by the (possibly empty, when $n=0$) list of $n$~ordered pairs of targets for edges issued from the internal vertices $m$,\ $\ldots$,\ $m+n-1$, respectively.
The full format is then ($m$,\ $n$, $s$; list of ordered pairs); in plain text we also write \verb"m n s   <list of ordered pairs>".
In the software, the class \texttt{KontsevichGraph} represents these signed Kontsevich graphs.
\end{implement}

\begin{example}
The graph \raisebox{-9pt}{
\unitlength=0.7mm
\linethickness{0.4pt}
\begin{picture}(15.00,22.67)
\put(2.00,5.00){\circle*{1.33}}
\put(13.00,5.00){\circle*{1.33}}
\put(7.33,11.33){\circle*{1.33}}
\put(2.00,17.00){\circle*{1.33}}
\put(2.00,17.00){\vector(0,-1){11.33}}
\put(-2, 10){\tiny $L$}
\put(2.00,17.00){\vector(1,-1){5.33}}
\put(4,16){\tiny $R$}
\put(1,18.75){\tiny 3}
\put(7.33,11.33){\vector(1,-1){5.33}}
\put(5,5.5){\tiny $L$}
\put(7.33,11.33){\vector(-1,-1){5.33}}
\put(10,9.5){\tiny $R$}
\put(8,12.5){\tiny 2}
\end{picture}
} has encoding \verb"2 2 1   0 1  0 2".
\end{example}

We recall that to every Kontsevich graph one associates a polydifferential operator by placing a copy of the Poisson bracket at each vertex.
To a signed graph one associates the polydifferential operator of the graph multiplied by the sign.
The skew\/-\/symmetry of the Poisson bracket implies that the same polydifferential operator may be represented by several different signed graphs, all having different encodings.

\begin{example}\label{ExEncodings}
Taken with the signs in the first row, the graphs in the second row all represent the same polydifferential operator:

\vskip .3em
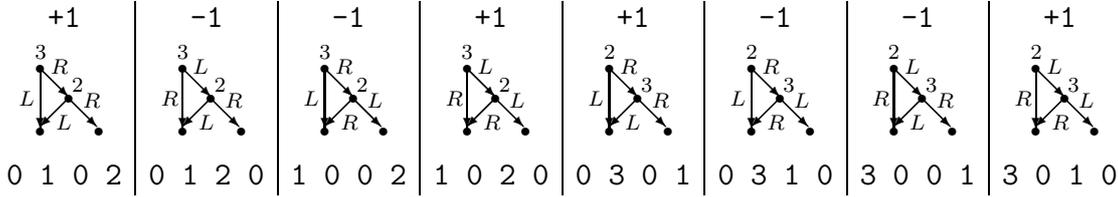
\begin{tabular}{c | c | c | c | c | c | c | c}
{\tt +1} & {\tt -1} & {\tt -1} & {\tt +1} & {\tt +1} & {\tt -1} & {\tt -1} & {\tt +1} \\
\raisebox{-12pt}{
\unitlength=0.7mm
\linethickness{0.4pt}
\begin{picture}(15.00,22.67)
\put(2.00,5.00){\circle*{1.33}}
\put(13.00,5.00){\circle*{1.33}}
\put(7.33,11.33){\circle*{1.33}}
\put(2.00,17.00){\circle*{1.33}}
\put(2.00,17.00){\vector(0,-1){11.33}}
\put(-2, 10){\tiny $L$}
\put(2.00,17.00){\vector(1,-1){5.33}}
\put(4,16){\tiny $R$}
\put(1,18.75){\tiny 3}
\put(7.33,11.33){\vector(1,-1){5.33}}
\put(5,5.5){\tiny $L$}
\put(7.33,11.33){\vector(-1,-1){5.33}}
\put(10,9.5){\tiny $R$}
\put(8,12.5){\tiny 2}
\end{picture}
}
&
\raisebox{-12pt}{
\unitlength=0.7mm
\linethickness{0.4pt}
\begin{picture}(15.00,22.67)
\put(2.00,5.00){\circle*{1.33}}
\put(13.00,5.00){\circle*{1.33}}
\put(7.33,11.33){\circle*{1.33}}
\put(2.00,17.00){\circle*{1.33}}
\put(2.00,17.00){\vector(0,-1){11.33}}
\put(-2, 10){\tiny $R$}
\put(2.00,17.00){\vector(1,-1){5.33}}
\put(4,16){\tiny $L$}
\put(1,18.75){\tiny 3}
\put(7.33,11.33){\vector(1,-1){5.33}}
\put(5,5.5){\tiny $L$}
\put(7.33,11.33){\vector(-1,-1){5.33}}
\put(10,9.5){\tiny $R$}
\put(8,12.5){\tiny 2}
\end{picture}
}
&
\raisebox{-12pt}{
\unitlength=0.7mm
\linethickness{0.4pt}
\begin{picture}(15.00,22.67)
\put(2.00,5.00){\circle*{1.33}}
\put(13.00,5.00){\circle*{1.33}}
\put(7.33,11.33){\circle*{1.33}}
\put(2.00,17.00){\circle*{1.33}}
\put(2.00,17.00){\vector(0,-1){11.33}}
\put(-2, 10){\tiny $L$}
\put(2.00,17.00){\vector(1,-1){5.33}}
\put(4,16){\tiny $R$}
\put(1,18.75){\tiny 3}
\put(7.33,11.33){\vector(1,-1){5.33}}
\put(5,5.5){\tiny $R$}
\put(7.33,11.33){\vector(-1,-1){5.33}}
\put(10,9.5){\tiny $L$}
\put(8,12.5){\tiny 2}
\end{picture}
}
&
\raisebox{-12pt}{
\unitlength=0.7mm
\linethickness{0.4pt}
\begin{picture}(15.00,22.67)
\put(2.00,5.00){\circle*{1.33}}
\put(13.00,5.00){\circle*{1.33}}
\put(7.33,11.33){\circle*{1.33}}
\put(2.00,17.00){\circle*{1.33}}
\put(2.00,17.00){\vector(0,-1){11.33}}
\put(-2, 10){\tiny $R$}
\put(2.00,17.00){\vector(1,-1){5.33}}
\put(4,16){\tiny $L$}
\put(1,18.75){\tiny 3}
\put(7.33,11.33){\vector(1,-1){5.33}}
\put(5,5.5){\tiny $R$}
\put(7.33,11.33){\vector(-1,-1){5.33}}
\put(10,9.5){\tiny $L$}
\put(8,12.5){\tiny 2}
\end{picture}
}
&
\raisebox{-12pt}{
\unitlength=0.7mm
\linethickness{0.4pt}
\begin{picture}(15.00,22.67)
\put(2.00,5.00){\circle*{1.33}}
\put(13.00,5.00){\circle*{1.33}}
\put(7.33,11.33){\circle*{1.33}}
\put(2.00,17.00){\circle*{1.33}}
\put(2.00,17.00){\vector(0,-1){11.33}}
\put(-2, 10){\tiny $L$}
\put(2.00,17.00){\vector(1,-1){5.33}}
\put(4,16){\tiny $R$}
\put(1,18.75){\tiny 2}
\put(7.33,11.33){\vector(1,-1){5.33}}
\put(5,5.5){\tiny $L$}
\put(7.33,11.33){\vector(-1,-1){5.33}}
\put(10,9.5){\tiny $R$}
\put(8,12.5){\tiny 3}
\end{picture}
}
&
\raisebox{-12pt}{
\unitlength=0.7mm
\linethickness{0.4pt}
\begin{picture}(15.00,22.67)
\put(2.00,5.00){\circle*{1.33}}
\put(13.00,5.00){\circle*{1.33}}
\put(7.33,11.33){\circle*{1.33}}
\put(2.00,17.00){\circle*{1.33}}
\put(2.00,17.00){\vector(0,-1){11.33}}
\put(-2, 10){\tiny $L$}
\put(2.00,17.00){\vector(1,-1){5.33}}
\put(4,16){\tiny $R$}
\put(1,18.75){\tiny 2}
\put(7.33,11.33){\vector(1,-1){5.33}}
\put(5,5.5){\tiny $R$}
\put(7.33,11.33){\vector(-1,-1){5.33}}
\put(10,9.5){\tiny $L$}
\put(8,12.5){\tiny 3}
\end{picture}
}
&
\raisebox{-12pt}{
\unitlength=0.7mm
\linethickness{0.4pt}
\begin{picture}(15.00,22.67)
\put(2.00,5.00){\circle*{1.33}}
\put(13.00,5.00){\circle*{1.33}}
\put(7.33,11.33){\circle*{1.33}}
\put(2.00,17.00){\circle*{1.33}}
\put(2.00,17.00){\vector(0,-1){11.33}}
\put(-2, 10){\tiny $R$}
\put(2.00,17.00){\vector(1,-1){5.33}}
\put(4,16){\tiny $L$}
\put(1,18.75){\tiny 2}
\put(7.33,11.33){\vector(1,-1){5.33}}
\put(5,5.5){\tiny $L$}
\put(7.33,11.33){\vector(-1,-1){5.33}}
\put(10,9.5){\tiny $R$}
\put(8,12.5){\tiny 3}
\end{picture}
}
&
\raisebox{-12pt}{
\unitlength=0.7mm
\linethickness{0.4pt}
\begin{picture}(15.00,22.67)
\put(2.00,5.00){\circle*{1.33}}
\put(13.00,5.00){\circle*{1.33}}
\put(7.33,11.33){\circle*{1.33}}
\put(2.00,17.00){\circle*{1.33}}
\put(2.00,17.00){\vector(0,-1){11.33}}
\put(-2, 10){\tiny $R$}
\put(2.00,17.00){\vector(1,-1){5.33}}
\put(4,16){\tiny $L$}
\put(1,18.75){\tiny 2}
\put(7.33,11.33){\vector(1,-1){5.33}}
\put(5,5.5){\tiny $R$}
\put(7.33,11.33){\vector(-1,-1){5.33}}
\put(10,9.5){\tiny $L$}
\put(8,12.5){\tiny 3}
\end{picture}
}
\\
{\tt 0 1 0 2} &
{\tt 0 1 2 0} &
{\tt 1 0 0 2} &
{\tt 1 0 2 0} &
{\tt 0 3 0 1} &
{\tt 0 3 1 0} &
{\tt 3 0 0 1} &
{\tt 3 0 1 0}
\end{tabular}

\noindent In the third row the target list (for internal vertices $2$ and $3$, respectively) is written.
\end{example}

We would like to know whether two (encodings of) signed graphs specify the same topological portrait --- up to a permutation of internal vertices and/or a possible swap $L \rightleftarrows R$ for some pair(s) of outgoing edges.
To compare two given encodings of a signed graph, let us define its normal form.
Such normal form is a way to pick the representative modulo the action of group $S_n \times (\BBZ_2)^n$ on the space~$G_{m,n}$.

\begin{define}[normal form]\label{DefNormalForm} 
The list of targets of a graph $\Gamma \in G_{m,n}$ can be considered as a $2n$-digit integer written in base-$(n+m)$ notation.
By running over the entire group $S_n \times (\mathbb{Z}_2)^n$, and by this over all the different re-labelings of $\Gamma$, we obtain many different integers written in base-$(n+m)$.
The {\em absolute value} $|\Gamma|$ of $\Gamma$ is the re-labeling of $\Gamma$ such that its list of targets is {\em minimal} as a nonnegative base-$(n+m)$ integer.
For a signed graph $s\Gamma$, the {\em normal form} is the signed graph $t|\Gamma|$ which represents the same polydifferential operator as $s\Gamma$.
Here we let $t=0$ if the graph is zero (see Remark \ref{RemZeroGraphs} below).
\end{define}

\begin{example}
The minimal base-$4$ number in the third column of Example \ref{ExEncodings} is {\tt 0~1~0~2}.
Hence the absolute value of each of the graphs in Example \ref{ExEncodings} is the first graph.
The normal form of each of the signed graphs in Example \ref{ExEncodings} is the first graph taken with the appropriate sign $\pm 1$; the encodings of the normal forms are then {\tt 2 2 $\pm$1\ \ 0 1 0 2}.
\end{example}

This normal form is implemented in software as the method \texttt{normalize()} of the class \texttt{KontsevichGraph}.
By running over the entire symmetry group, it will be inefficient when the number of vertices is large.
In the future this method could be replaced by a more efficient one, without requiring changes to the rest of the code.

\begin{rem}\label{RemZeroGraphs}
The graphs $\Gamma \in G_{m,n}$ for which the associated polydifferential operator vanishes, by being equal to minus itself, are called {\em zero}.
This property can be detected during the calculation of the normal form of a signed graph.
One starts with the encoding of a signed graph.
Obtain a ``sorted'' encoding (representing the same polydifferential operator) by sorting the outgoing edges in every pair in nondecreasing order: each swap $L \rightleftarrows R$ entails a reversion of the sign.
Now run over the group~$S_n$ of permutations of the internal vertices in the graph at hand, relabeling those vertices.
Should the list of targets in the sorted encoding of a relabeling be equal to the list of targets in the original sorted encoding, but the sign be opposite, then the graph is zero.
%
We will see in Chapter~\ref{SecStar} (specifically, in Lemma~\ref{LemmaSwapLR} on p.~\pageref{LemmaSwapLR}) that the weights of these graphs also vanish, this time by the anticommutativity of certain differentials under the wedge product.
\end{rem}

\begin{example}\label{ExZeroGraph}
Consider the graph
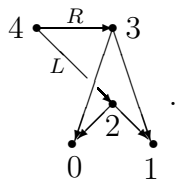
\[
\vcenteredhbox{
\unitlength=1mm
\special{em:linewidth 0.4pt}
\linethickness{0.4pt}
\begin{picture}(31.00,25.67)(0,3)
\put(15.00,25.00){\vector(1,0){10.00}}
\put(25.00,25.00){\vector(-1,-3){5.00}}
\put(25.00,25.00){\vector(1,-3){5.00}}
\put(25.00,15.00){\vector(1,-1){4.67}}
\put(25.00,15.00){\vector(-1,-1){4.67}}
\put(15.00,25.00){\line(1,-1){6.67}}
\put(23,17){\vector(1,-1){2}}
\bezier{16}(23.00,17.00)(24.00,16.00)(25.00,15.00)
\put(15.00,25.00){\circle*{1}}
\put(25.00,25.00){\circle*{1}}
\put(25.00,15.00){\circle*{1}}
\put(30.33,9.67){\circle*{1}}
\put(19.67,9.67){\circle*{1}}
\put(13.33,25.00){\makebox(0,0)[rc]{4}}
\put(26.67,25.00){\makebox(0,0)[lc]{3}}
\put(20.00,25.67){\makebox(0,0)[cb]{\tiny$R$}}
\put(17.67,21.00){\makebox(0,0)[ct]{\tiny$L$}}
\put(25.00,13.33){\makebox(0,0)[ct]{2}}
\put(20.00,8){\makebox(0,0)[ct]{0}}
\put(30.00,7.67){\makebox(0,0)[ct]{1}}
\end{picture}
}.
\]
with the encoding {\tt 2 3 1\ \ 0 1 0 1 2 3}.
For the identity permutation we obtain the initial sorted encoding {\tt 2 3 1\ \ 0 1 0 1 2 3} (it was already sorted).
For the permutation $2 \leftrightarrows 3$ we obtain the encoding {\tt 2 3 1\ \ 0 1 0 1 3 2}; upon sorting the pairs it becomes {\tt 2 3 -1\ \ 0 1 0 1 2 3}.
The list of pairs coincides with the initial sorted encoding but the sign is opposite; hence the graph is zero.
\end{example}

The notion of normal form of graphs allows one to generate lists of graphs with different topological portraits (e.g., Kontsevich graph series, see section~\ref{SecGraphSeries} below) by using the following algorithm.
Initially, the list of generated graphs is empty.
For every possible encoding (according to Implementation~\ref{DefEncoding}) in a run\/-\/through, its normal form with sign~$+1$ or~$0$ is added to the list if it is not contained there (otherwise, the offered encoding is skipped).

\begin{implement}
To generate all the Kontsevich graphs with {\tt m} sinks and {\tt n} internal vertices in $\tilde{G}_{m,n}$ (without tadpoles or double edges), the command is
\begin{verbatim}
    > generate_graphs n m
\end{verbatim}
The procedure lists all such graphs (one per line) in the standard output.
The second argument {\tt m} may be omitted: the default value is {\tt m = 2}.

\noindent
Similarly, to generate only normal forms (with sign~$+1$ or~$0$), the call is 
\begin{verbatim}
    > generate_graphs n m --normal-forms=yes
\end{verbatim}
The optional argument {\tt --with-coefficients=yes} indicates that (numbered) undetermined coefficients should be listed alongside the graphs (the default is {\tt no}); see \S\ref{SecGraphSeries}.

\noindent
(Accordingly, see {\tt KontsevichGraph::graphs} in Appendix~\ref{AppCPP}.)
\end{implement}

\begin{example}
The Kontsevich graphs in $\tilde{G}_{m,n}$ with {\em one} internal vertex
\begin{verbatim}
    > generate_graphs 1
    2 1 1   0 1
    2 1 1   1 0
\end{verbatim}
consist of the wedge with its two different labellings.
We can verify that the number of Kontsevich graphs on $n$ internal vertices and two sinks is $(n(n+1))^n$:
\begin{verbatim}
    > generate_graphs 2 | wc -l
    36
    > generate_graphs 3 | wc -l
    1728
    > generate_graphs 4 | wc -l
    160000
    > generate_graphs 5 | wc -l
    24300000
\end{verbatim}
Here, ``{\tt | wc -l}'' counts the number of lines in the output (\texttt{wc} is from GNU \texttt{coreutils}).
\end{example}

Let us remember that while a list of graphs is generated, more options can be chosen to restrict the graphs: e.g., only \emph{prime} graphs can be taken into account, graphs of which the mirror\/-\/reflection is already on the list can be skipped, and/or only those graphs in which each sink receives at least one arrow can be taken.
The purpose and implementation of these options will be explained in the next chapter 
(see p.~\pageref{SecBasic} below).


\subsection{Series of graphs: file format}\label{SecGraphSeries}
We now specify how formal power series expansions of graphs are implemented in software.
Denote by~$\hbar$ the formal parameter; in machine\/-\/readable format, a power series expansion in~$\hbar$ is a list of coefficients of~$\hbar^k$, $k \geqslant 0$.
The coefficients are formal sums of signed graphs (see {\tt KontsevichGraphSum} in Appendix \ref{AppCPP}) in which the coefficients can be of any type, e.g.,
\begin{itemize}
\item integer or floating point numbers (e.g., $0.333$),
\item rational numbers (e.g., $1/3$),
\item undetermined variables (resp., {\tt OneThird}).
\end{itemize}
To be precise, the library \cite{kgs} contains the class {\tt KontsevichGraphSeries} which depends on a template parameter {\tt T}; it specifies the type of all the coefficients of graphs in the series.
In the command\/-\/line programs, the external type {\tt GiNaC::ex}, which is the expression type of the {\tt GiNaC} library~\cite{GiNaC}, allows all of the above values (and combinations of them).
Hence a series under study 
can contain coefficients of all types at once; the coefficient of a graph itself can be a sum of different types of objects (e.g., {\tt p16}${} + 0.25$).

\begin{implement}[series encoding]
In the file format for formal power series expansions, two kinds of lines are possible: either
\begin{verbatim}
    h^k:
\end{verbatim}
or (separated by whitespace)
\begin{verbatim}
    <encoding of a graph>   <coefficient>
\end{verbatim}
\end{implement}
The precision of the formal power series expansion is indicated by the highest {\tt k} occurring in lines of the form ``{\tt h\symbol{"5E}k:}''.
Hence one can control this bound 
by adding such a line with a high {\tt k} at the end of the file.

\begin{example}
The Kontsevich $\star$-product (see \S \ref{SecStar}) is a graph series given up to the second order in the deformation parameter $\hbar$ in the file {\tt star2w.txt} which reads
\begin{verbatim}
    h^0:
    2 0 1              1
    h^1:
    2 1 1   0 1        1
    h^2:
    2 2 1   0 1 0 1    1/2
    2 2 1   0 1 0 2    w_2_1
    2 2 1   0 1 1 2    w_2_2
    2 2 1   0 3 1 2    w_2_3
\end{verbatim}
\label{ExStarTwoUndetermined}
\end{example}

\begin{implement}
\label{ImplSubstituteRelations}
The substitution of undetermined coefficients by their actual values, as well as re-expression of indeterminates via other such objects, is done by using the program
\begin{verbatim}
    > substitute_relations <graph-series-file> <subsitutions-file>
\end{verbatim}
Its command line arguments are two file names: the first file contains the series and the second file consists of a list of substitutions (one per line), each substitution written in the form
\begin{verbatim}
    <variable>==<what it is set equal to>
\end{verbatim}
The command line program sends the series with all those substitutions to the standard output.
\end{implement}

\begin{example}
\label{ExStarTwo}
The values of the unknowns in Example \ref{ExStarTwoUndetermined} are written in {\tt weights2.txt}:
\begin{verbatim}
    w_2_1==1/3
    w_2_2==-1/3
    w_2_3==-1/6
\end{verbatim}
Whence the star-product is given modulo $\bar{o}(\hbar^2)$ as follows:
\begin{verbatim}
    $ substitute_relations star2w.txt weights2.txt > star2.txt
    $ cat star2.txt
    h^0:
    2 0 1              1
    h^1:
    2 1 1   0 1        1
    h^2:
    2 2 1   0 1 0 1    1/2
    2 2 1   0 1 0 2    1/3
    2 2 1   0 1 1 2    -1/3
    2 2 1   0 3 1 2    -1/6
\end{verbatim}
Here \texttt{cat} from GNU \texttt{coreutils} is used to display the file.
\end{example}

In practice one may encounter graph series containing many graphs and undetermined coefficients.
To split a graph series into parts, the following command is helpful.

\begin{implement}
To extract the part of a graph series proportional to a given expression, use the call
\begin{verbatim}
    > extract_coefficient <graph-series-file> <expression>
\end{verbatim}
In the standard output one obtains a modification of the original graph series: each graph coefficient {\tt c} is now replaced by the coefficient of {\tt <expression> } in {\tt c}.
If the coefficient of {\tt <expression>} in {\tt c} is identically zero, then the graph is skipped.
The special value \verb"<expression> = 1" yields the constant part of the graph series (all the undetermined variables in the input are set to zero).
\end{implement}

\begin{example}
From the file in Example \ref{ExStarTwo}, we extract the part proportional to {\tt w\symbol{"5F}2\symbol{"5F}1}:
\begin{verbatim}
    > extract_coefficient star2w.txt w_2_1
    h^0:
    h^1:
    h^2:
    2 2 1   0 1 0 2    1
\end{verbatim}
It is just 
one graph.
\end{example}

\subsection{Reduction modulo skew\/-\/symmetry}
Let us recall that for every internal vertex in a Kontsevich graph, the pair of out-going edges is ordered by the relation Left $\prec$ Right and by a 
mark\/-\/up of those two edges 
using~$L$ and~$R$.
By construction, the coefficients of a graph series are sums of \emph{signed} graphs; each signed graph is specified by its encoding, see Implementation~\ref{DefEncoding} on p. \pageref{DefEncoding} above.
Starting from the vector space of formal sums of signed graphs with real coefficients, 
we pass to its 
quotient.
Namely, we postulate that graphs which differ only by their internal vertex labeling are equal.
Further, we proclaim that every reversal of the edge order in any pair (from the same internal vertex) entails the reversion of the graph sign.
Lastly, we introduce the relations
\[
\texttt{<coeff>${}\cdot{}$(sign)$\Gamma$raph${}={}$<sign${}\cdot{}$coeff>${}\cdot(+1)\Gamma$raph},
\]
for each signed graph~{\tt (sign)$\Gamma$raph} with any coefficient {\tt <coeff>}.

The combined effect of these relations is that each sum of signed graphs may be reduced to a sum of normal forms (see Definition~\ref{DefNormalForm}) with sign $+1$.
Recall that the ordering mechanism Left $\prec$ Right creates graphs that equal zero because they are equal to minus themselves (see Remark~\ref{RemZeroGraphs} and Example~\ref{ExZeroGraph}).

\begin{rem}
To avoid such comparison of graphs with zero all the time and so, to increase efficiency, every signed graph is brought to its normal form as soon as it is constructed.
It is this moment when zero graphs acquire zero signs.
\end{rem}

The algorithm to reduce a sum of graphs modulo skew\/-\/symmetry runs as follows.
For the starting graph or every next graph in the list, its sign (if nonzero) is set equal to +1 and its coefficient is modified, if necessary, by using the rule
\begin{equation}\label{EqSignRelCoeff}
\texttt{<coeff>${}\cdot{}$<sign> = <sign${}\cdot{}$coeff>${}\cdot{}$(+1)}.
\end{equation}
Every graph with sign~$0$ is removed.
Then the graph at hand (in its normal form, times a coefficient) is compared, disregarding signs, with all the graphs which follow in the list.
A match found, its coefficient is added --\,using relation~\eqref{EqSignRelCoeff}\,--
to the coefficient of the graph we started with; the match itself is removed.
By this reduction procedure for graph sums, all vanishing graphs with zero signs are excluded from the list.

\begin{implement}
To reduce a graph series expansion modulo skew\/-\/symmetry, call
\begin{verbatim}
    > reduce_mod_skew <graph-series-file> [--print-differential-orders]
\end{verbatim}
The resulting graph series is sent to the standard output.
The optional argument {\tt --print-differential-orders} controls whether the differential orders of the graphs (as operators acting on the sinks) are included in the output, with lines such as
\begin{verbatim}
    # 2 1
\end{verbatim}
indicating subsequent graphs have differential order $(2,1)$.
(The corresponding methods are {\tt KontsevichGraphSeries<T>::reduce\_mod\_skew()} and {\tt KontsevichGraphSum<T>} {\tt::reduce\_mod\_skew()} in Appendix \ref{AppCPP}.)
\end{implement}

\begin{example}
We put the zero graph from Example \ref{ExZeroGraph} with the coefficient $+1$ into a file {\tt zerograph3.txt}:
\begin{verbatim}
    h^3:
    2 3 1  0 1 0 1 2 3    1
\end{verbatim}
We confirm that {\tt reduce\_mod\_skew} kills it:
\begin{verbatim}
    > reduce_mod_skew zerograph3.txt
    h^3:
\end{verbatim}
The output is an empty list of graphs.
\end{example}



\begin{rem}
An alternative for the implementation of \texttt{reduce\symbol{"5F}mod\symbol{"5F}skew} is to make use of the plain text file format, in three passes.
In the first pass, put all graphs in normal form with sign $+1$ (updating the coefficients).
Recall that graphs are listed \emph{first} in the file format for graph series, so the problem of collecting terms is the same as sorting the file.
In the second pass, use \texttt{sort} from GNU \texttt{coreutils} to sort the file (this uses the very efficient ``external $R$-way merge'' algorithm).
In the third pass: for each normal form, add up the coefficients of every copy in the list (since the list is sorted, one need not look far).
The implementation of this algorithm is left as an exercise to the reader.
\end{rem}

\begin{rem}
Sums of graphs may also be reduced modulo the (graphical) Jacobi identity and its (pictorial) differential consequences; this is the subject of section~\ref{SecFactor}.
\end{rem}

\subsection{Evaluate a given graph series at a given Poisson structure}\label{SecPoissonEvaluate}
Let us recall that every Kontsevich graph contains at least one sink. 
Every edge (decorated with an index, say~$i$, over which the summation runs from~$1$ to~$n=\dim N^n$) denotes the derivation with respect to a local coordinate~$x^i$ at a given point~$\bx$ of the affine manifold~$N^n$ (hence the edge de\-no\-tes~$\partial/\partial x^i|_\bx$).
Every internal vertex (if any) encodes a copy of a given Poisson structure~$\cP$.
Should the labellings of two outgoing edges be~$\begin{picture}(20.0,5.0)\put(0.0, 2.5){\vector(1,0){20.0}\put(-13,-2.4){$i$}}\end{picture}$ and~$\begin{picture}(20.0,5.0)\put(0.0, 2.5){\vector(1,0){20.0}\put(-13,-2.4){$j$}}\end{picture}$ so that 
the edge with~$i$ precedes that with~$j$, the Poisson structure in that vertex is~$\cP^{ij}(\bx)$ (that is, the ordering $i \prec j$ is preserved; moreover, the reference to a point~$\bx$ is common to all vertices).
Now, every Kontsevich graph (with a coefficient after 
it) represents a (poly)\/dif\-fe\-ren\-tial operator with respect to the content of sink(s); to build that operator, we apply the derivations (at~$\bx \in N^n$) to objects in the arrowhead vertices, multiply the content of all vertices at a fixed set of index values, 
and then sum over all the indices.

\begin{example}[Jacobi identity]
For all Poisson structures~$\cP$ and all triples of arguments from the algebra~$C^\infty(N^n)$ of functions on the Poisson manifold at hand, we have that
\begin{equation}\label{EqJacFig}
\vcenteredhbox{
\raisebox{3.3mm}
[6.5mm][3.5mm]{ 
\unitlength=1mm
\special{em:linewidth 0.4pt}
\linethickness{0.4pt}
\begin{picture}(12,15)
\put(0,-10){
\begin{picture}(12.00,15.00)
\put(0.00,10.00){\framebox(12.00,5.00)[cc]{$\bullet\ \bullet$}}
\put(2.00,10.00){\vector(-1,-3){1.33}}
\put(6.00,10.00){\vector(0,-1){4.00}}
\put(10.00,10.00){\vector(1,-3){1.33}}
\put(0.00,4.00){\makebox(0,0)[cb]{\tiny\it1}}
\put(6.00,4.00){\makebox(0,0)[cb]{\tiny\it2}}
\put(11.67,4.00){\makebox(0,0)[cb]{\tiny\it3}}
\end{picture}
}\end{picture}}}\ \ \ 
\mathrel{{:}{=}}
\text{\raisebox{-12pt}[25pt]{
\unitlength=0.70mm
\linethickness{0.4pt}
\begin{picture}(26.00,16.33)
\put(0.00,5.00){\line(1,0){26.00}}
\put(2.00,5.00){\circle*{1.33}}
\put(13.00,5.00){\circle*{1.33}}
\put(24.00,5.00){\circle*{1.33}}
\put(2.00,1.33){\makebox(0,0)[cc]{\tiny\it1}}
\put(13.00,1.33){\makebox(0,0)[cc]{\tiny\it2}}
\put(24.00,1.33){\makebox(0,0)[cc]{\tiny\it3}}
\put(7.33,11.33){\circle*{1.33}}
\put(7.33,11.33){\vector(1,-1){5.5}}
\put(7.33,11.33){\vector(-1,-1){5.5}}
\put(13,17){\circle*{1.33}}
\put(13,17){\vector(1,-1){11.2}}
\put(13,17){\vector(-1,-1){5.1}}
\put(3.00,10.00){\makebox(0,0)[cc]{\tiny$i$}}
\put(12.00,10.00){\makebox(0,0)[cc]{\tiny$j$}}
\put(24.00,10.00){\makebox(0,0)[cc]{\tiny$k$}}
\end{picture}
}}
{-}
\text{\raisebox{-12pt}[25pt]{
\unitlength=0.70mm
\linethickness{0.4pt}
\begin{picture}(26.00,16.33)
\put(0.00,5.00){\line(1,0){26.00}}
\put(2.00,5.00){\circle*{1.33}}
\put(13.00,5.00){\circle*{1.33}}
\put(24.00,5.00){\circle*{1.33}}
\put(2.00,1.33){\makebox(0,0)[cc]{\tiny\it1}}
\put(13.00,1.33){\makebox(0,0)[cc]{\tiny\it2}}
\put(24.00,1.33){\makebox(0,0)[cc]{\tiny\it3}}
\put(13,11.33){\circle*{1.33}}
\put(13,11.33){\vector(2,-1){10.8}}
\put(13,11.33){\vector(-2,-1){10.8}}
\put(18.5,17){\circle*{1.33}}
\put(18.5,17){\vector(-1,-1){5.2}}
\put(18.5,17){\vector(-1,-2){5.6}}
\put(13,15){\tiny $L$}
\put(17,12){\tiny $R$}
\put(4.00,10.00){\makebox(0,0)[cc]{\tiny$i$}}
\put(11.00,8.00){\makebox(0,0)[cc]{\tiny$j$}}
\put(22.00,10.00){\makebox(0,0)[cc]{\tiny$k$}}
\end{picture}
}}
{-}
\text{\raisebox{-12pt}[25pt]{
\unitlength=0.70mm
\linethickness{0.4pt}
\begin{picture}(26.00,16.33)
\put(0.00,5.00){\line(1,0){26.00}}
\put(2.00,5.00){\circle*{1.33}}
\put(13.00,5.00){\circle*{1.33}}
\put(24.00,5.00){\circle*{1.33}}
\put(2.00,1.33){\makebox(0,0)[cc]{\tiny\it1}}
\put(13.00,1.33){\makebox(0,0)[cc]{\tiny\it2}}
\put(24.00,1.33){\makebox(0,0)[cc]{\tiny\it3}}
\put(18.33,11.33){\circle*{1.33}}
\put(18.33,11.33){\vector(1,-1){5.5}}
\put(18.33,11.33){\vector(-1,-1){5.5}}
\put(13,17){\circle*{1.33}}
\put(13,17){\vector(-1,-1){11.2}}
\put(13,17){\vector(1,-1){5.1}}
\put(3.00,10.00){\makebox(0,0)[cc]{\tiny$i$}}
\put(13.00,10.00){\makebox(0,0)[cc]{\tiny$j$}}
\put(24.00,10.00){\makebox(0,0)[cc]{\tiny$k$}}
\end{picture}
}}
= 0. 
\end{equation}%
\vskip .5em
\noindent
In formulae, by ascribing the index $\ell$ to the unlabeled edge, the identity reads
\[
(
\partial_\ell \cP^{ij} \cP^{\ell k} +
\partial_\ell \cP^{jk} \cP^{\ell i} +
\partial_\ell \cP^{ki} \cP^{\ell j}
)
\partial_i({\tiny\it 1}) \partial_j({\tiny\it 2}) \partial_k({\tiny\it 3}) = 0.
\]
Indeed, the coefficient of $\partial_i \otimes \partial_j \otimes \partial_k$ is the familiar form of the Jacobi identity.
\end{example}

In fact, the graph itself is the most convenient way to transcribe the formulae which one constructs from it, see~\cite[\S2.1]{dq17} for more details.\footnote{\label{FootVariational}%
In the variational set\/-\/up of Poisson field models, the affine manifold~$N^n$ is realised as fibre in an affine bundle~$\pi$ over another affine manifold~$M^m$ equipped with a volume element.
The variational Poisson brackets $\{\cdot,\cdot\}_{\boldsymbol{\mathcal{P}}}$ are then defined for integral functionals that take sections of such bundle~$\pi$ to numbers.
The encoding of variational poly\-dif\-fe\-ren\-tial operators by the Kontsevich graphs now reads as follows.
Decorated by an index~$i$, every edge denotes the variation with respect to the $i$th~coordinate along the fibre.
By construction, the variations act by first differentiating their argument with respect to the fibre variables (or their derivatives along the base~$M^m$); secondly, the integrations by parts over the underlying space~$M^m$ are performed.
Whenever two or more arrows arrive at a graph vertex, its content is first differentiated the corresponding number of times with respect to the jet fibre variables in~$J^\infty(\pi)$ and only then it can be differentiated with respect to local coordinates on the base manifold~$M^m$.
The assumption that both the manifolds~$M^m$ and~$N^n$ be affine makes the construction coordinate\/-\/free, see~\cite{gvbv,cycle16} and~\cite{sqs13,prg15}.
}
The computer implementation is straightforward.
We acknowledge however that it is one of the most needed instruments.

\begin{implement}\label{ImplPoissonEvaluate}
The call is
\begin{verbatim}
    > poisson_evaluate <graph-series-filename> <poisson-structure>
\end{verbatim}
and options for {\tt <poisson-structure>} are\footnote{The current version of the software does not allow specification of an arbitrary Poisson structure at runtime (e.g. input as a matrix of functions); however, in the source file {\tt util/poison\symbol{"5F}structure\_examples.hpp} the list of Poisson structures (as matrices) can be extended to one's heart's desire.}
\begin{itemize}
\item {\tt 2d-polar},
\item {\tt 3d-generic},
\item {\tt 3d-polynomial},
\item {\tt 4d-determinant},
\item {\tt 4d-rank2},
\item {\tt 9d-rank6}.
\end{itemize}
The output is a list of coefficients of the differential operator that the graph series represents, filtered by (\hskip -0.7pt\textit{a}) powers of $\hbar$, (\hskip -0.6pt\textit{b}) the differential order as an operator acting on the sinks, and (\hskip -0.7pt\textit{c}) the actual derivatives falling on the sinks.
\end{implement}

\begin{example}
\label{ExJacEval}
Put the graph sum for the Jacobiator $\Jac(\cP)$ in {\tt jacobiator.txt}:
\begin{verbatim}
    3 2 1   0 1 2 3    -1
    3 2 1   0 2 1 3    1
    3 2 1   0 4 1 2    -1
\end{verbatim}
We evaluate it at a Poisson structure:
\begin{verbatim}
    > poisson_evaluate jacobiator.txt 2d-polar
    Coordinates: r t 
    Poisson structure matrix:
    [[0, r^(-1)]
    [-r^(-1), 0]]

    h^0:
    # 1 1 1 
    # [ r ] [ r ] [ r ]
    0
    # [ r ] [ r ] [ t ]
    0
    # [ r ] [ t ] [ r ]
    0
    # [ r ] [ t ] [ t ]
    0
    # [ t ] [ r ] [ r ]
    0
    # [ t ] [ r ] [ t ]
    0
    # [ t ] [ t ] [ r ]
    0
    # [ t ] [ t ] [ t ]
    0
\end{verbatim}
For example, the pair of lines
\begin{verbatim}
    # [ r ] [ t ] [ r ]
    0
\end{verbatim}
indicates that the coefficient of $\partial_r \otimes \partial_t \otimes \partial_r$ is zero in the polydifferential operator.
\end{example}

Restriction of graph series to Poisson structures will be essential in section~\ref{SecRestrictAssoc} below where systems of linear algebraic equations between the Kontsevich graph weights in~$\star$ will be obtained by restricting the associativity equation $\Assoc_\star(f,g,h) = 0$ to a given Poisson bracket.


\section{The Kontsevich $\star$-product}\label{SecStar}
\noindent%
The star\/-\/product $\star = \times + \hbar\{\cdot,\cdot\}_\cP + \bar{o}(\hbar)$ in $C^\infty(N^n)[[\hbar]]$ is an associative unital noncommutative deformation of the associative unital commutative product~$\times$ in the algebra of functions $C^\infty(N^n)$ on a given affine manifold~$N^n$ of dimension $n < \infty$.
The bi\/-\/linear bi\/-\/differential $\star$-\/product is realized as a formal power series in~$\hbar$ by using the weighted Kontsevich graphs.
In fact, the bi\/-\/differential operator at~$\hbar^k$ is a sum of \emph{all} Kontsevich graphs
$\Gamma \in \tilde{G}_{2,k}$ without tadpoles, with $k$~internal vertices (and two sinks) 
taken with some weights~$w(\Gamma)$.
Let us recall their original definition~\cite{MK97}.

\begin{define}
Every Kontsevich graph $\Gamma \in \tilde{G}_{2,k}$ can be embedded in the closed upper half\/-\/plane $\mathbb{H} \cup \mathbb{R} \subset \mathbb{C}$ by placing the internal vertices at pairwise distinct points in~$\mathbb{H}$ and the external vertices at~$0$ and~$1$; the edges are drawn as geodesics with respect to the hyperbolic metric, i.e.\ as vertical lines and circular segments.
The 
angle~$\varphi(p,q)$ between two distinct points $p, q \in \mathbb{H}$ is the angle between the geodesic from~$p$ to~$q$ and the geodesic from~$p$ to~$\infty$ (measured counterclockwise from the latter): $$\varphi(p,q) = \Arg\bigg(\frac{q-p}{q-\bar{p}}\bigg),$$
and it can be extended to $\mathbb{H} \cup \mathbb{R}$ by continuity.
The \emph{weight} of a Kontsevich graph~$\Gamma \in \tilde{G}_{2,k}$ is given by the integral\footnote{We omit the factor~$1/k!$ that was written in~\cite{MK97}, to make the weight multiplicative (see Lemma~\ref{LemmaMult}).}
\begin{equation}\label{EqWeight}
w(\Gamma) = \frac{1}{(2\pi)^{2k}}\int_{C_k(\mathbb{H})} \bigwedge_{j=1}^k {\rm d}\varphi(p_j,p_{\text{Left}(j)}) \wedge {\rm d}\varphi(p_j,p_{\text{Right}(j)}),
\end{equation}
over the \emph{configuration space} of $k$~points in the upper half\/-\/plane $\mathbb{H} \subset \mathbb{C}$,
\[
C_k(\mathbb{H}) = \{(p_1,\ldots,p_k) \in \mathbb{H}^k : p_i \textrm{ pairwise distinct}\};
\]
the integrand is defined pointwise at $(p_1,\ldots,p_k)$ by considering the embedding of $\Gamma$ in $\mathbb{H}$ that sends the~$j$th internal vertex to~$p_j$; the numbers~$\text{Left}(j)$ 
and~$\text{Right}(j)$ are the left and right targets of $j$th~vertex, respectively.
(If $\text{Left}(j)$ is the first or the second sink, put $p_{\text{Left}(j)} = 0$ or $1$ respectively; the same goes for $p_{\text{Right}(j)}$ if $\text{Right}(j)$ is a sink.)
\end{define}

\begin{theorNo}[Kontsevich \cite{MK97}]
For every Poisson bi\/-\/vector~$\cP$ on~$N^n$ and an infinitesimal deformation $\times \mapsto \times + \hbar\{\cdot,\cdot\}_\cP + \bar{o}(\hbar)$ towards the respective Poisson bracket, the $\hbar$-\/linear star\/-\/product
\begin{equation}\label{EqStar}
\star = \times + \sum_{k\geqslant 1} \frac{\hbar^k}{k!} \sum_{\Gamma \in \tilde{G}_{2,k}} w(\Gamma)\, \Gamma(\cP)(\cdot,\cdot)\colon 
C^\infty(N^n)[[\hbar]] \times C^\infty(N^n)[[\hbar]] \to C^\infty(N^n)[[\hbar]]
\end{equation}
is associative. 
\end{theorNo}

\begin{lemma}\label{LemmaPermute}
Permuting the internal vertex labels of a Kontsevich graph leaves the weight unchanged.
\end{lemma}

\begin{proof}
Such a permutation re\/-\/orders the factors in a wedge product of two-forms.
\end{proof}

\begin{lemma}\label{LemmaSwapLR}
Swapping $L \rightleftarrows R$ at an internal vertex of a Kontsevich graph $\Gamma \in \tilde{G}_{2,k}$ implies the reversal of the sign of its weight.
\end{lemma}

\begin{proof}
Anticommutativity of wedge product of two differentials in formula~\eqref{EqWeight}. 
\end{proof}

\begin{lemma}\label{LemmaMirror}
The weight of a graph $\Gamma \in \tilde{G}_{2,k}$ and its mirror\/-\/reflection~$\bar{\Gamma}$ are related by $w(\bar{\Gamma}) = (-)^kw(\Gamma)$.
\end{lemma}

\begin{proof}
Taking the reflection of a graph (with respect to the vertical line $\Re(z) = 1/2$) is an orientation\/-\/reversing coordinate change on each of the $k$ ``factors'' $\mathbb{H}$ in~$C_k(\mathbb{H})$.
\end{proof}

\begin{lemma}[\cite{Dito}]\label{LemmaDito} 
For a Kontsevich graph such that at least one sink receives no edge(s), its weight is zero.%
\footnote{The fact that the differential order of~$\star$ is positive with respect to either of its arguments should be expected, in view of the required property of the $\star$-\/product to be unital: $f \star 1 = f = 1 \star f$.}
\end{lemma}


\begin{lemma}\label{LemmaMult}
The map $w\colon \sqcup_k \tilde{G}_{2,k} \to \mathbb{R}$ that assigns weights to graphs is multiplicative,
\begin{equation}\label{EqMult}
w(\Gamma_i \bar{\times} \Gamma_j) = w(\Gamma_i) \times w(\Gamma_j),
\end{equation}
with respect to the product~$\bar{\times}$ of graphs, 
\[
\Gamma_i \bar{\times} \Gamma_j = (\Gamma_i \sqcup \Gamma_j)\bigr/\{a^{\text{th}}\text{ sink in }\Gamma_i=a^{\text{th}}\text{ sink in }\Gamma_j,\quad 0 \leqslant a \leqslant 1\},
\]
which identifies the respective sinks.
\end{lemma}

\begin{proof}
The integrals converge absolutely \cite{MK97}; apply Fubini's theorem and linearity.
\end{proof}

\begin{example}\label{ExWeightRel}
Some weight relations obtained from the lemmas above:
\[w\biggl(\!\!\!\!\text{\raisebox{-21pt}{
\unitlength=0.7mm
\linethickness{0.4pt}
\begin{picture}(15.00,20.67)
\put(2.00,5.00){\circle*{1.33}}
\put(13.00,5.00){\circle*{1.33}}
\put(7.67,11.67){\circle*{1.33}}
\put(7.67,20.00){\circle*{1.33}}
\put(7.67,20.00){\vector(-1,-3){4.67}}
\put(7.67,20.00){\vector(1,-3){4.67}}
\put(7.67,11.67){\vector(-3,-4){4.00}}
\put(7.67,11.67){\vector(3,-4){4.33}}
\end{picture}
}}\!\!\biggr) = 
w\biggl(\!\!\!\!
\text{\raisebox{-18pt}{
\unitlength=0.7mm
\linethickness{0.4pt}
\begin{picture}(15.00,16.67)
\put(2.00,5.00){\circle*{1.33}}
\put(13.00,5.00){\circle*{1.33}}
\put(7.33,16.00){\circle*{1.33}}
\put(7.33,16.00){\vector(-1,-2){5.00}}
\put(7.33,16.00){\vector(1,-2){5.00}}
\end{picture}
}}\!\!
\biggr)^2;
\quad
w\biggl(
\!\!\!\!
\text{\raisebox{-18pt}{
\unitlength=0.7mm
\linethickness{0.4pt}
\begin{picture}(15.00,16.67)
\put(2.00,5.00){\circle*{1.33}}
\put(13.00,5.00){\circle*{1.33}}
\put(7.33,16.00){\circle*{1.33}}
\put(7.33,16.00){\vector(-1,-2){5.00}}
\put(0.5,10){\tiny $L$}
\put(7.33,16.00){\vector(1,-2){5.00}}
\put(11,10){\tiny $R$}
\end{picture}
}}\!\!
\biggr) = -w\biggl(
\!\!\!\!
\text{\raisebox{-18pt}{
\unitlength=0.7mm
\linethickness{0.4pt}
\begin{picture}(15.00,16.67)
\put(2.00,5.00){\circle*{1.33}}
\put(13.00,5.00){\circle*{1.33}}
\put(7.33,16.00){\circle*{1.33}}
\put(7.33,16.00){\vector(-1,-2){5.00}}
\put(0.5,10){\tiny $R$}
\put(7.33,16.00){\vector(1,-2){5.00}}
\put(11,10){\tiny $L$}
\end{picture}
}}\!\!
\biggr);
\quad
w\biggl(\!\!\!
\text{\raisebox{-18pt}{
\unitlength=0.7mm
\linethickness{0.4pt}
\begin{picture}(15.00,17.67)
\put(2.00,5.00){\circle*{1.33}}
\put(13.00,5.00){\circle*{1.33}}
\put(7.33,11.33){\circle*{1.33}}
\put(2.00,17.00){\circle*{1.33}}
\put(2.00,17.00){\vector(0,-1){11.33}}
\put(2.00,17.00){\vector(1,-1){5.33}}
\put(7.33,11.33){\vector(1,-1){5.33}}
\put(7.33,11.33){\vector(-1,-1){5.33}}
\end{picture}
}}\!\!\biggr) =
w\biggl(\!\!\!
\text{\raisebox{-18pt}{
\unitlength=0.7mm
\linethickness{0.4pt}
\begin{picture}(15.00,18.00)
\put(2.00,5.00){\circle*{1.33}}
\put(13.00,5.00){\circle*{1.33}}
\put(7.33,11.33){\circle*{1.33}}
\put(7.33,11.33){\vector(1,-1){5.33}}
\put(7.33,11.33){\vector(-1,-1){5.33}}
\put(13.00,17.33){\circle*{1.33}}
\put(13.00,17.33){\vector(0,-1){11.67}}
\put(13.00,17.33){\vector(-1,-1){5.33}}
\end{picture}
}}\!\!
\biggr).
\]
\end{example}

\noindent
Lemma \ref{LemmaMult} motivates the following definition.

\begin{define}
A Kontsevich graph $\Gamma\in \tilde{G}_{2,k}$ is called {\em composite} if $\Gamma$ is equal to the $\bar{\times}$-product of some Kontsevich graphs on two sinks and positive number of internal vertices in both of the co-factors.
Otherwise (if such a realization is not possible), the graph is called {\em prime}.
\end{define}

\noindent
Using Lemma \ref{LemmaMult} and induction, we obtain that the weight of a composite graph $\Gamma = \Gamma_1 \bar{\times} \cdots \bar{\times} \Gamma_t$ is the product of the weights of its factors: $w(\Gamma) = w(\Gamma_1) \times \cdots \times w(\Gamma_t)$.

\subsection{Basic set of graphs}

We identify a set of graphs such that the weights of those graphs would suffice to determine all the other weights.

\begin{define}
A \emph{basic} set of graphs on $k$~internal vertices is a set of pairwise distinct normal forms (the signs of which are discarded) of only those Kon\-tse\-vich graphs~$\Gamma \in \tilde{G}_{2,k}$ which are prime, and in which every sink receives at least one edge.
By definition, the basic set contains the normal form of a graph but not its mirror reflection if it differs from the graph at hand.
To decide whether a graph or its mirror\/-\/reflection~$\bar{\Gamma} \neq \Gamma$ is included into a basic set, we take the graph whose absolute value is~\emph{minimal} as a base-$(k+2)$ number.
Note that a basic set on $k\geqslant 3$ vertices {\em does} contain zero graphs.
\end{define}

\begin{cor}
To build $\star$-\/product~\eqref{EqStar} up to $\bar{o}(\hbar^k)$ for some power~$k\geqslant 1$, knowing the Kontsevich weights~$w(\Gamma_i)$ only for a \emph{basic} set of graphs $\Gamma_i \in \tilde{G}_{2,\ell}$ at all~$\ell\leqslant k$ is enough.
Indeed, the weights of all other graphs with $\ell$~internal vertices are calculated from Lemmas~\ref{LemmaPermute}, \ref{LemmaSwapLR}, \ref{LemmaMirror}, \ref{LemmaDito}, and~\ref{LemmaMult}.
\end{cor}

\begin{example}
Consider the prime graph 
$\text{\!\!\!\!\!\raisebox{-8.60pt}{
\unitlength=0.7mm
\linethickness{0.4pt}
\begin{picture}(15.00,17.67)
\put(2.00,5.00){\circle*{1.33}}
\put(13.00,5.00){\circle*{1.33}}
\put(7.33,11.33){\circle*{1.33}}
\put(2.00,17.00){\circle*{1.33}}
\put(2.00,17.00){\vector(0,-1){11.33}}
\put(2.00,17.00){\vector(1,-1){5.33}}
\put(7.33,11.33){\vector(1,-1){5.33}}
\put(7.33,11.33){\vector(-1,-1){5.33}}
\end{picture}
}\!\!\!}$
and its mirror-reflection
$\text{\!\!\!\!\!\raisebox{-8.60pt}{
\unitlength=0.7mm
\linethickness{0.4pt}
\begin{picture}(15.00,18.00)
\put(2.00,5.00){\circle*{1.33}}
\put(13.00,5.00){\circle*{1.33}}
\put(7.33,11.33){\circle*{1.33}}
\put(7.33,11.33){\vector(1,-1){5.33}}
\put(7.33,11.33){\vector(-1,-1){5.33}}
\put(13.00,17.33){\circle*{1.33}}
\put(13.00,17.33){\vector(0,-1){11.67}}
\put(13.00,17.33){\vector(-1,-1){5.33}}
\end{picture}
}\!\!\!}$.
The encodings of their normal forms are {\tt 2 2 1\ \ 0 1 0 2} and {\tt 2 2 1\ \ 0 1 1 2} respectively.
Since {\tt 0~1~0~2} $<$ {\tt 0~1~1~2} as base-$4$ numbers, only the first graph is included in the basic set.
The fork graph
$\text{\!\!\!\!\raisebox{-18pt}{
\unitlength=0.7mm
\linethickness{0.4pt}
\begin{picture}(15.00,16.67)
\put(2.00,5.00){\circle*{1.33}}
\put(13.00,5.00){\circle*{1.33}}
\put(7.33,16.00){\circle*{1.33}}
\put(7.33,16.00){\vector(-1,-2){5.00}}
\put(7.33,16.00){\vector(1,-2){5.00}}
\end{picture}
}\!\!}$
is mirror-symmetric hence it is included anyway. 
\end{example}

The basic set at order~$3$ is displayed in Figure~\ref{FigBasic3}.
\begin{figure}[htb]
\begin{tabular}{c | c | c | c | c}
\text{\raisebox{-12pt}{
\unitlength=1mm
\special{em:linewidth 0.4pt}
\linethickness{0.4pt}
\begin{picture}(15.00,23.67)(0,3)
\put(0.00,22.00){\vector(1,0){10.00}}
\put(10.00,22.00){\vector(-1,-3){5.00}}
\put(10.00,22.00){\vector(1,-3){5.00}}
\put(10.00,12.00){\vector(1,-1){4.67}}
\put(10.00,12.00){\vector(-1,-1){4.67}}
\put(0.00,22.00){\line(1,-1){6.67}}
\put(8,14){\vector(1,-1){2}}
\bezier{16}(8.00,14.00)(9.00,13.00)(10.00,12.00)
\put(0.00,22.00){\circle*{1}}
\put(10.00,22.00){\circle*{1}}
\put(10.00,12.00){\circle*{1}}
\put(15.33,6.67){\circle*{1}}
\put(4.67,6.67){\circle*{1}}
\end{picture}
}}
&
\text{\raisebox{-12pt}{
\unitlength=0.70mm
\linethickness{0.4pt}
\begin{picture}(15.00,23.33)
\put(2.00,5.00){\circle*{1.33}}
\put(13.00,5.00){\circle*{1.33}}
\put(7.33,11.33){\circle*{1.33}}
\put(2.00,17.00){\circle*{1.33}}
\put(2.00,17.00){\vector(0,-1){11.33}}
\put(2.00,17.00){\vector(1,-1){5.33}}
\put(7.33,11.33){\vector(1,-1){5.33}}
\put(7.33,11.33){\vector(-1,-1){5.33}}
\put(2.00,22.67){\circle*{1.33}}
\bezier{128}(2.00,22.00)(-5.00,13.50)(2.33,4.75)
\put(1,6.33){\vector(1,-1){0.67}}
\put(2.00,22.67){\vector(1,-2){5.00}}
\end{picture}
}}
&
\text{\raisebox{-12pt}{
\unitlength=0.70mm
\linethickness{0.4pt}
\begin{picture}(15.00,23.33)
\put(5.00,5.00){\circle*{1.33}}
\put(16.00,5.00){\circle*{1.33}}
\put(-2.00,17.00){\circle*{1.33}}
\put(-2.00,17.00){\vector(1,0){6.33}}
\put(-2.00,17.00){\vector(1,-2){6.23}}
\put(5.00,17.00){\circle*{1.33}}
\put(5.00,17.00){\vector(0,-1){11.33}}
\put(5.00,17.00){\vector(1,-1){5.33}}
\put(10.33,11.33){\circle*{1.33}}
\put(10.33,11.33){\vector(1,-1){5.33}}
\put(10.33,11.33){\vector(-1,-1){5.33}}
\end{picture}
}}
&
\text{\raisebox{-12pt}{
\unitlength=0.70mm
\linethickness{0.4pt}
\begin{picture}(15.00,23.33)
\put(2.00,5.00){\circle*{1.33}}
\put(13.00,5.00){\circle*{1.33}}
\put(2.00,17.00){\circle*{1.33}}
\put(2.00,17.00){\vector(0,-1){11.33}}
\put(2.00,17.00){\vector(1,-1){5.33}}
\put(13.00,17.00){\circle*{1.33}}
\put(13.00,17.00){\vector(0,-1){11.33}}
\put(13.00,17.00){\vector(-1,-1){5.33}}
\put(7.33,11.33){\circle*{1.33}}
\put(7.33,11.33){\vector(1,-1){5.33}}
\put(7.33,11.33){\vector(-1,-1){5.33}}
\end{picture}
}}
&
\text{\raisebox{-12pt}{
\unitlength=0.70mm
\linethickness{0.4pt}
\begin{picture}(15.00,23.33)
\put(2.00,5.00){\circle*{1.33}}
\put(13.00,5.00){\circle*{1.33}}
\put(7.33,11.33){\circle*{1.33}}
\put(2.00,17.00){\circle*{1.33}}
\put(2.00,17.00){\vector(0,-1){11.33}}
\put(2.00,17.00){\vector(1,-1){5.33}}
\put(7.33,11.33){\vector(1,-1){5.33}}
\put(7.33,11.33){\vector(-1,-1){5.33}}
\put(13.00,22.67){\circle*{1.33}}
\put(13.00,22.67){\vector(0,-1){16.67}}
\put(13.00,22.67){\vector(-2,-1){10.33}}
\end{picture}
}}
\\
{\tt 0} &
{\tt w\symbol{"5F}3\symbol{"5F}1} &
{\tt w\symbol{"5F}3\symbol{"5F}2} &
{\tt w\symbol{"5F}3\symbol{"5F}3} &
{\tt w\symbol{"5F}3\symbol{"5F}4} \\
\hline
& & & & \\ 
\raisebox{-12pt}{
\unitlength=0.70mm
\linethickness{0.4pt}
\begin{picture}(15.00,23.33)
\put(2.00,5.00){\circle*{1.33}}
\put(13.00,5.00){\circle*{1.33}}
\put(7.33,11.33){\circle*{1.33}}
\put(2.00,17.00){\circle*{1.33}}
\put(2.00,17.00){\vector(0,-1){11.33}}
\put(2.00,17.00){\vector(1,-1){5.33}}
\put(7.33,11.33){\vector(1,-1){5.33}}
\put(7.33,11.33){\vector(-1,-1){5.33}}
\put(7.33,22.67){\circle*{1.33}}
\put(7.33,22.67){\vector(-1,-1){5.33}}
\put(7.33,22.67){\vector(0,-1){10.67}}
\end{picture}
}
&
\text{\raisebox{-12pt}{
\unitlength=0.70mm
\linethickness{0.4pt}
\begin{picture}(15.00,23.33)
\put(5.00,5.00){\circle*{1.33}}
\put(16.00,5.00){\circle*{1.33}}
\put(5.00,17.00){\circle*{1.33}}
\put(5.00,17.00){\vector(1,-1){5.33}}
\put(-0.90,10.90){\circle*{1.33}}
\put(-0.90,10.90){\vector(1,-1){5.33}}
\bezier{56}(5.00,17.00)(-0.5,18)(-0.90,10.90)
\bezier{56}(5.00,17.00)(4,11)(-0.90,10.90)
\put(10.33,11.33){\circle*{1.33}}
\put(10.33,11.33){\vector(1,-1){5.33}}
\put(10.33,11.33){\vector(-1,-1){5.33}}
\end{picture}
}}
&
\text{\raisebox{-12pt}{
\unitlength=0.70mm
\linethickness{0.4pt}
\begin{picture}(15.00,23.67)
\put(2.00,5.00){\circle*{1.33}}
\put(13.00,5.00){\circle*{1.33}}
\put(2.00,20.00){\circle*{1.33}}
\put(13.00,20.00){\circle*{1.33}}
\put(13.00,20.00){\vector(-2,-3){5.33}}
\put(2.00,20.00){\vector(2,-3){5.33}}
\bezier{64}(2.00,20.00)(7.00,14.00)(12.67,20.00)
\bezier{64}(13.00,20.00)(7.00,25.33)(2.67,20.00)
\put(11.67,19.00){\vector(1,1){0.67}}
\put(3.33,21.00){\vector(-1,-1){0.67}}
\put(7.33,11.33){\circle*{1.33}}
\put(7.33,11.33){\vector(1,-1){5.33}}
\put(7.33,11.33){\vector(-1,-1){5.33}}
\end{picture}
}}
&
\text{\raisebox{-12pt}{
\unitlength=0.70mm
\linethickness{0.4pt}
\begin{picture}(15.00,23.67)
\put(2.00,5.00){\circle*{1.33}}
\put(13.00,5.00){\circle*{1.33}}
\put(-3.33,13.00){\circle*{1.33}}
\put(7.67,13.00){\circle*{1.33}}
\put(7.67,13.00){\vector(-2,-3){5.33}}
\put(-3.33,13.00){\vector(2,-3){5.33}}
\bezier{64}(-3.33,13.00)(1.67,7.00)(7.50,13.00)
\bezier{64}(7.67,13.00)(1.67,18.33)(-2.50,13.00)
\put(6.5,12.00){\vector(1,1){0.67}}
\put(-2.00,14.00){\vector(-1,-1){0.67}}
\put(13.00,17.00){\circle*{1.33}}
\put(13.00,17.00){\vector(0,-1){12.33}}
\put(13.00,17.00){\vector(-3,-2){5.33}}
\end{picture}
}}
&
\raisebox{-12pt}{
\unitlength=0.7mm
\linethickness{0.4pt}
\begin{picture}(18.00,23.67)
\put(2.00,5.00){\circle*{1.33}}
\put(16.00,5.00){\circle*{1.33}}
\put(2.00,14.00){\circle*{1.33}}
\put(2.00,14.00){\vector(0,-1){8.33}}
\put(2.00,14.00){\vector(2,-1){5.33}}
\put(7.33,11.33){\circle*{1.33}}
\put(7.33,11.33){\vector(-1,-1){5.33}}
\put(7.33,11.33){\vector(0,1){6}}
\put(7.33,17.33){\circle*{1.33}}
\put(7.33,17.33){\vector(-2,-1){6}}
\put(7.33,17.33){\vector(2,-3){8.12}}
\end{picture}
}
\\
{\tt w\symbol{"5F}3\symbol{"5F}5} &
{\tt w\symbol{"5F}3\symbol{"5F}6} &
{\tt w\symbol{"5F}3\symbol{"5F}7} &
{\tt w\symbol{"5F}3\symbol{"5F}8} &
{\tt w\symbol{"5F}3\symbol{"5F}9} \\
\hline
& & & & \\
\text{\raisebox{-12pt}{
\unitlength=0.70mm
\linethickness{0.4pt}
\begin{picture}(15.00,23.67)
\put(2.00,5.00){\circle*{1.33}}
\put(13.00,5.00){\circle*{1.33}}
\put(2.00,15.00){\circle*{1.33}}
\put(13.00,15.00){\circle*{1.33}}
\put(13.00,15.00){\vector(0,-1){9.33}}
\put(2.00,15.00){\vector(0,-1){9.33}}
\bezier{64}(2.00,15.00)(7.00,9.00)(12.67,15.00)
\bezier{64}(13.00,15.00)(7.00,20.33)(2.67,15.00)
\put(11.67,14.00){\vector(1,1){0.67}}
\put(3.33,16.00){\vector(-1,-1){0.67}}
\put(-3.00,10.00){\circle*{1.33}}
\put(-3.00,10.00){\vector(1,1){5.03}}
\put(-3.00,10.00){\vector(1,-1){5.03}}
\end{picture}
}}
&
\text{\raisebox{-12pt}{
\unitlength=0.70mm
\linethickness{0.4pt}
\begin{picture}(15.00,23.67)
\put(2.00,5.00){\circle*{1.33}}
\put(13.00,5.00){\circle*{1.33}}
\put(2.00,15.00){\circle*{1.33}}
\put(13.00,15.00){\circle*{1.33}}
\put(13.00,15.00){\vector(0,-1){9.33}}
\put(2.00,15.00){\vector(0,-1){9.33}}
\bezier{64}(2.00,15.00)(7.00,9.00)(12.67,15.00)
\bezier{60}(13.00,15.00)(7.00,20.33)(2.67,15.00)
\put(11.67,14.00){\vector(1,1){0.67}}
\put(3.67,15.67){\vector(-1,-1){0.67}}
\put(2.00,23.00){\circle*{1.33}}
\bezier{128}(2.00,23.00)(-5,14.00)(1.33,6.00)
\put(2.00,23.00){\vector(3,-2){10.67}}
\put(1,6.67){\vector(1,-2){0.67}}
\end{picture}
}}
&
\text{\raisebox{-12pt}{
\unitlength=0.70mm
\linethickness{0.4pt}
\begin{picture}(15.00,23.67)
\put(2.00,5.00){\circle*{1.33}}
\put(13.00,5.00){\circle*{1.33}}
\put(2.00,15.00){\circle*{1.33}}
\put(13.00,15.00){\circle*{1.33}}
\put(13.00,15.00){\vector(0,-1){9.33}}
\put(2.00,15.00){\vector(0,-1){9.33}}
\bezier{64}(2.00,15.00)(7.00,9.00)(12.67,15.00)
\bezier{64}(13.00,15.00)(7.00,20.33)(2.67,15.00)
\put(11.67,14.00){\vector(1,1){0.67}}
\put(3.33,16.00){\vector(-1,-1){0.67}}
\put(7.33,23.00){\circle*{1.33}}
\put(7.33,23.00){\vector(-2,-3){4.33}}
\put(7.33,23.00){\vector(2,-3){4.67}}
\end{picture}
}}
&
\raisebox{-12pt}{
\unitlength=0.7mm
\linethickness{0.4pt}
\begin{picture}(15.00,23.67)
\put(1.66,5.00){\circle*{1.33}}
\put(13.33,5.00){\circle*{1.33}}
\put(3.66,11.33){\circle*{1.33}}
\put(3.66,11.33){\vector(-1,-3){2.0}}
\put(3.66,11.33){\vector(1,0){7.5}}
\put(11.33,11.33){\circle*{1.33}}
\put(11.33,11.33){\vector(1,-3){2.0}}
\put(7.33,17.00){\circle*{1.33}}
\put(7.33,17.00){\vector(-2,-3){3.5}}
\bezier{56}(7.33,17.00)(11,16)(11.33,11.33)
\bezier{56}(7.33,17.00)(8,13)(11.33,11.33)
\end{picture}
}
&
\raisebox{-12pt}{
\unitlength=0.7mm
\linethickness{0.4pt}
\begin{picture}(15.00,23.67)
\put(2.00,5.00){\circle*{1.33}}
\put(13.00,5.00){\circle*{1.33}}
\put(2.00,17.00){\circle*{1.33}}
\put(2.00,17.00){\vector(0,-1){11.33}}
\bezier{56}(2.00,17.00)(4.50,20.00)(7.33,17.00)
\bezier{56}(2.00,17.00)(4.50,14.00)(7.33,17.00)
\bezier{56}(7.33,17.00)(10.50,20.00)(13.00,17.00)
\bezier{56}(7.33,17.00)(10.50,14.00)(13.00,17.00)
\put(7.33,17.00){\circle*{1.33}}
\put(13.00,17.00){\circle*{1.33}}
\put(13.00,17.00){\vector(0,-1){11.33}}
\end{picture}
}
\\
{\tt w\symbol{"5F}3\symbol{"5F}10} &
{\tt w\symbol{"5F}3\symbol{"5F}11} &
{\tt w\symbol{"5F}3\symbol{"5F}12} &
{\tt w\symbol{"5F}3\symbol{"5F}13} &
{\tt w\symbol{"5F}3\symbol{"5F}14} \\
\end{tabular}
\caption{Basic set at order~$3$, with undetermined weights
for nonzero graphs.
(The weights are determined in Example~\ref{ExOrder3} on p.~\pageref{ExOrder3} below.)}\label{FigBasic3}
\end{figure}
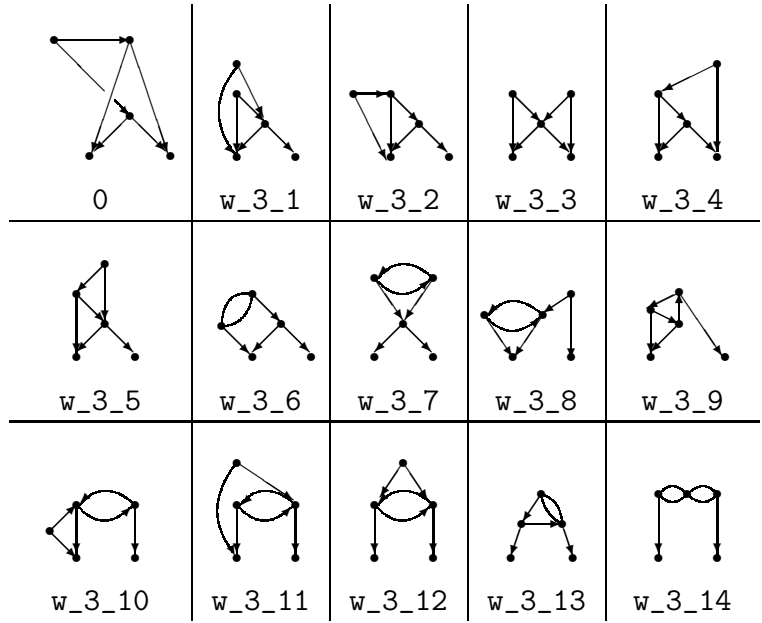

\subsection{``All'' graphs in $\star$ mod $\bar{o}(\hbar^4)$}\label{SecBasic}
In Table \ref{TableBasic} 
we list the number of basic
graphs at every order~$k \leqslant 6$ in the Kontsevich $\star$-\/product.
\begin{table}[ht]
\caption{How many basic graphs there are at low orders~$k$.}\label{TableBasic}
\begin{tabular}{l | l | l | l | l | l | l | l}
Order $=k$ & 0 & 1 & 2 & 3 & 4 & 5 & 6\\
\hline
\#(Basic set) & 0 & 1 & 2 & 15 & 156 & 2307 & 43231 \\
\#(Nonzero in basic set) & 0 & 1 & 2 & 14 & 149 & 2218 & 42050
\end{tabular}
\end{table}
The actual number of graphs 
with respect to which the sums in formula~\eqref{EqStar} expand is of course much greater.

\begin{implement}
To obtain the list of normal forms for graphs from a basic set at order $k$, the following command is available:
\begin{verbatim}
    > generate_graphs k --basic=yes
\end{verbatim}
The list of normal forms is then sent to the standard output.
This command is equivalent to
\begin{verbatim}
    > generate_graphs k --prime=yes --normal-forms=yes \
         --postive-differential-order=yes --modulo-mirror-images=yes
\end{verbatim}
\end{implement}

\begin{example}
\label{ExBasic3}
The list of basic graphs with $\leqslant 3$ internal vertices -- with undetermined coefficients at orders $1, 2, 3$ -- is constructed using the following commands:
\begin{verbatim}
    $ cat > basic3w.txt
    h^0:
    2 0 1     1
    h^1:
    ^D (press Ctrl+D)
    $ generate_graphs 1 --basic=yes --with-coefficients=yes \ 
          >> basic3w.txt
    $ echo 'h^2:' >> basic3w.txt
    $ generate_graphs 2 --basic=yes --with-coefficients=yes \
          >> basic3w.txt
    $ echo 'h^3:' >> basic3w.txt
    $ generate_graphs 3 --basic=yes --with-coefficients=yes \
          >> basic3w.txt
\end{verbatim}
The file \texttt{basic3w.txt} now contains the basic set.
\end{example}

Starting from a basic set, the $\star$-\/product is built up to a certain order $k\geqslant0$ in~$\hbar$.

\begin{implement}
The program
\begin{verbatim}
    > star_product <basic-set-filename>
\end{verbatim}
takes as its input a graph series with a basic set of graphs at each order; the graphs go with coefficients of any nature (i.e.\ number or indeterminate).
The program's output is an 
expansion of the $\star$-\/product up to the order that was specified by the input.
In other words, all the graphs which are produced from the ones contained in a given basic set are generated and their coefficients are (re)calculated from the ones in the input (using Lemmas~\ref{LemmaSwapLR}, \ref{LemmaMirror}, and~\ref{LemmaMult}).
\end{implement}

\begin{example}
\label{ExStar3}
To generate the star-product up to order $3$ with all weights of nonzero basic graphs undetermined (from Example \ref{ExBasic3}), one proceeds as follows:
\begin{verbatim}
    $ star_product basic3w.txt > star3w_unreduced.txt
    $ reduce_mod_skew --print-differential-orders star3w_unreduced.txt \
          > star3w.txt
\end{verbatim}
The file {\tt star3w.txt} now contains the desired star-product.
\end{example}

\subsection{Methods to obtain the weights of basic graphs} 
We deduce 
that to build the $\star$-\/product modulo~$\bar{o}(\hbar^4)$ as many as $149$~weights of nonzero basic graphs $\Gamma_i \in \tilde{G}_{2,4}$ at~$k=4$ must be found (or at least expressed in terms of as few master\/-\/parameters as possible).
In fact, direct calculation of all of the~$149$ Kontsevich integrals is not needed to solve the problem in full
because there exist more algebraic relations between the weights of basic graphs.
In the following proposition we recall a class of such relations.\footnote{%
A convenient approach to 
calculation of Kontsevich weights~\eqref{EqWeight} at order~$3$ by using direct integration (and for that, using methods of complex analysis such as the Cauchy residue theorem) was developed in~\cite{Decin}, see Appendix~\ref{AppNumericalWeights} on p.~\pageref{AppNumericalWeights} below.
However, we note that most successful at~$k=3$, this method is no longer effective for all graphs at~$k \geqslant 4$. 
More progress is badly needed to allow~$k \geqslant 5$.}

\begin{proposition}[cyclic weight relations \cite{WillwacherFelderIrrationality}]\label{PropCyclic}
Let $\Gamma$~be a Kontsevich graph on~$m=2$ ground vertices.
Let $E \subset \Edge(\Gamma)$ be a subset of edges in~$\Gamma$ such that for every~$e \in E$, $\operatorname{target}(e) \neq 0$.
(That is, every edge from the subset~$E$ lands on the sink~$1$ or an internal vertex.)
For every such subset~$E$, define the graph~$\Gamma_E$ as follows: let its vertices be the same as in~$\Gamma$ and
for every edge~$e \in \Edge(\Gamma)$, preserve it in~$\Gamma_E$ if~$e \neq E$, but if~$e \in E$ replace that edge 
by a new edge in~$\Gamma_E$ going from $\operatorname{source}(e)$ to the sink~$0$.
By definition, the ordering $L \prec R$ of outgoing edges is inherited in~$\Gamma_E$ from~$E$ even if the targets of any of those edges are new.
Thirdly, denote by~$N_0(\Gamma_E)$ the number of edges in~$\Gamma_E$ such that their target is the sink~$0$.
Then the Kontsevich weight of a graph~$\Gamma$ is related to the weights of all such graphs~$\Gamma_E$ obtained from~$\Gamma$ by the formula
\begin{equation}\label{EqCyclic}
w(\Gamma) = (-)^n \sum_{\substack{E \subset \Edge(\Gamma) \\ \forall e\in E, \operatorname{target}(e) \neq 0}} (-)^{N_0(\Gamma_E)} w(\Gamma_E).
\end{equation}
Note that this relation is linear in the weights of all graphs.
\end{proposition}

If the graph~$\Gamma$ or, in practice, some of the new graphs~$\Gamma_E$ in~\eqref{EqCyclic} is composite, Lemma~\ref{LemmaMult} provides a further, nonlinear reduction of~$w(\Gamma)$ by using graphs with fewer internal vertices.

\begin{example}
Consider the graph $\Gamma_{3,8}$ in Figure \ref{FigBasic3} with weight $w(\Gamma_{3,8})={}${\tt w\symbol{"5F}3\symbol{"5F}8}.
For every non-empty subset $E$ (with $\text{target}(e) \neq 0$ for every $e \in E$) the graph $(\Gamma_{3,8})_E$ is a zero-weight graph by virtue of one of the Lemmas at the beginning of this chapter.
Hence the only term in the sum on the right\/-\/hand side in \eqref{EqCyclic} is the weight of the graph corresponding to the empty set: $w((\Gamma_{3,8})_\varnothing) = w(\Gamma_{3,8})$.
Since $n=3$ and $N_0(\Gamma_{3,8}) = 2$ we get the cyclic relation $w(\Gamma_{3,8}) = -w(\Gamma_{3,8})$; whence $w(\Gamma_{3,8}) = 0$.
\end{example}

\begin{rem}
It is readily seen that only \emph{prime}, that is, non\/-\/composite graphs~$\Gamma$ need be used to generate \emph{all} 
relations~\eqref{EqCyclic}.
Indeed, every subset~$E$ of edges for a composite graph $\Gamma = \Gamma^1 \bar{\times} \Gamma^2$ splits to a disjoint union $E^1\sqcup E^2$ of such subsets for the graphs~$\Gamma^1$ and~$\Gamma^2$ separately.
Therefore the re\/-\/direction of edges in a composite graph would inevitably yield the composite graph $\Gamma^1_{E^1} \bar{\times} \Gamma^2_{E^2}$.
Now, the multiplicativity of Kontsevich weights and the additivity of the count $N_0(\Gamma_E) = N_0(\Gamma^1_{E^1}) + N_0(\Gamma^2_{E^2})$ can be used to conclude that the relations obtained from composite graphs are redundant. 
\end{rem}

\begin{implement}
The command
\begin{verbatim}
    > cyclic_weight_relations <star-product-file>
\end{verbatim}
treats the input $\star$-product as a clothesline for graphs and their weights.
For each graph $\Gamma$ in the $\star$-product, it outputs the relation \eqref{EqCyclic} between the weights of the respective graphs in the form {\tt LHS - RHS == 0}.
\end{implement}

\begin{example}
At the order three with the $\star$-product from Example \ref{ExStar3}:
\begin{verbatim}
    > cyclic_weight_relations star3w.txt
    ...
    1/3*w_3_6+1/6*w_2_3==0
    1/3*w_3_8==0
    ...
\end{verbatim}
\end{example}


\begin{rem}
For some (basic) graphs it happens that the weight integrand in \eqref{EqWeight}, as a differential $2k$-form, vanishes identically, even if the graph is not zero due to skew\/-\/symmetry.
This is the case for $21$ out of $149$ nonzero basic graphs at~$k=4$; see also Appendix \ref{AppNumericalWeights}.
\end{rem}


For calculations of particular weight integrals we refer to the literature in section \ref{SecDiscussion}.

\begin{rem}[rationality]
Willwacher and Felder~\cite{WillwacherFelderIrrationality} express the weight of a graph in $\tilde{G}_{2,7}$ as $p\cdot\zeta(3)^2/\pi^6 + q$ where $p$ and $q$ are rational numbers and $\zeta$ is the Riemann $\zeta$-\/function.
Whether $\zeta(3)^2/\pi^6$ is rational or not is an open problem.
The software which we presently discuss supports -- through {\tt GiNaC} \cite{GiNaC} -- the input of $\zeta$-values as coefficients, e.g. the expression $\zeta(3)^2/\pi^6$ can be input as {\tt zeta(3)\symbol{"5E}2/Pi\symbol{"5E}6}.
This can be used e.g. to express other weights in terms of such values.
According to Banks--Panzer--Pym in \cite{BanksPanzerPym}, $\mathbb{Q}[(2\pi)^{-1}]$-linear combinations of multiple zeta values actually start to appear in the harmonic weight coefficients in the $\star$-product, but at order $\hbar^6$.
They do not yet appear at order $4$ (or rather, they are all seen to be rational at order $4$).

Is any of the weights transcendental?
This has not been proved, so that it remains unknown whether any of them is or none of them are.
\end{rem}

All the above being said about methods to obtain the values~$w(\Gamma)$ for Kontsevich graph weights and about the schemes to generate linear relations between these numbers, 
we observe 
that the requirement of associativity for the $\star$-\/product modulo~$\bar{o}(\hbar^k)$, whenever that structure is completely known at all orders up to~$\hbar^k$, is an 
ample 
source of relations of that kind.
This will be used intensively in chapter~\ref{SecAssoc} from p.~\pageref{SecAssoc} 
onwards. 
In particular, we mention here that the values of weights of graphs at order $\ell$ may be restricted by the associativity requirement at orders $>\ell$, by restriction to fixed differential orders $(i,j,k)$ (see Lemma \ref{Lemma} on p. \pageref{Lemma}).


\subsection{How graphs act on graphs}
\label{SecGraphsOnGraphs}
Let us have a closer look at the equation of associativity for the sought\/-\/for $\star$-\/product:
\[
\Assoc_\star(f,g,h) = (f \star g) \star h - f \star (g \star h) = 0.
\]
We see that the graph series~$f \star g$ and~$g \star h$ serve as the left- and right co\/-\/multiples of~$h$ and~$f$, respectively, in yet another copy of the star\/-\/product.
To realize the associator by using the Kontsevich graphs, we now explain how graphs act on graphs (here, in every composition~$\star \circ \star$ the graph series acts on a graph series by linearity).

We postulate that the action of graph series on graph series is $\Bbbk[[\hbar]]$-\/linear and $\Bbbk[G_{*,*}]$-\/linear with respect to both the graphs that act and that become the arguments.

Recall that every Poisson bracket is a derivation in each of its arguments.
In consequence, every derivation falling on a sink --\,in a graph~$\Gamma_1$ that acts on a given graph~$\Gamma_2$ taken as the new content of that sink\,-- acts on the sink's content via the Leibniz rule; all the Leibniz rules for the derivations in\/-\/coming to that sink work independently from each other.
Recall that the vertices of a graph represent factors in an expression.

\begin{example}
\label{ExWedgeOnBullets}
Consider the action of a wedge graph~$\Lambda$ on the two\/-\/sinks graph $(\bullet \bullet) \in G_{2,0}$, taken as its second argument.
We have that
\[
\text{\raisebox{-18pt}{
\unitlength=0.7mm
\linethickness{0.4pt}
\begin{picture}(15.00,16.67)
\put(2.00,5.00){\circle*{1.33}}
\put(10.00,5.00){\circle*{1.33}}
\put(14.00,5.00){\circle*{1.33}}
\put(12,5){\oval(8,5)}
\put(7.33,16.00){\circle*{1.33}}
\put(7.33,16.00){\vector(-1,-2){5.00}}
\put(7.33,16.00){\vector(1,-2){4.00}}
\end{picture}
}}
=
\text{\raisebox{-18pt}{
\unitlength=0.7mm
\linethickness{0.4pt}
\begin{picture}(15.00,16.67)
\put(2.00,5.00){\circle*{1.33}}
\put(13.00,5.00){\circle*{1.33}}
\put(17.00,5.00){\circle*{1.33}}
\put(7.33,16.00){\circle*{1.33}}
\put(7.33,16.00){\vector(-1,-2){5.00}}
\put(7.33,16.00){\vector(1,-2){5.00}}
\end{picture}
}}
+
\text{\raisebox{-18pt}{
\unitlength=0.7mm
\linethickness{0.4pt}
\begin{picture}(15.00,16.67)
\put(2.00,5.00){\circle*{1.33}}
\put(13.00,5.00){\circle*{1.33}}
\put(9.00,5.00){\circle*{1.33}}
\put(7.33,16.00){\circle*{1.33}}
\put(7.33,16.00){\vector(-1,-2){5.00}}
\put(7.33,16.00){\vector(1,-2){5.00}}
\end{picture}
}}
.
\]
The result is a sum of Kontsevich graphs of type~$(3,1)$.
Let us remember that the sinks are distinguished by their ordering; in particular the two Kontsevich graphs on the right-hand side are not equal.
\end{example}

\begin{example}
\label{ExWedgeOnWedge}
Now let the wedge graph act on a wedge graph (again, as the former's second argument):
\[
\text{\raisebox{-18pt}{
\unitlength=0.7mm
\linethickness{0.4pt}
\begin{picture}(15.00,16.67)
\put(2.00,5.00){\circle*{1.33}}
\put(10.00,5.00){\circle*{1.33}}
\put(14.00,5.00){\circle*{1.33}}
\put(7.33,16.00){\circle*{1.33}}
\put(7.33,16.00){\vector(-1,-2){5.00}}
\put(7.33,16.00){\vector(1,-2){2.60}}
\put(12,7){\oval(8,8)}
\put(12.00,10.00){\circle*{1.33}}
\put(12.00,10.00){\vector(1,-2){2.60}}
\put(12.00,10.00){\vector(-1,-2){2.60}}
\end{picture}
}}
=
\text{\raisebox{-18pt}{
\unitlength=0.7mm
\linethickness{0.4pt}
\begin{picture}(15.00,16.67)
\put(2.00,5.00){\circle*{1.33}}
\put(8.00,5.00){\circle*{1.33}}
\put(12.00,5.00){\circle*{1.33}}
\put(7.33,16.00){\circle*{1.33}}
\put(7.33,16.00){\vector(-1,-2){5.00}}
\put(7.33,16.00){\vector(1,-2){2.6}}
\put(10.00,10.00){\circle*{1.33}}
\put(10.00,10.00){\vector(1,-2){2.60}}
\put(10.00,10.00){\vector(-1,-2){2.60}}
\end{picture}
}}
+
\text{\raisebox{-18pt}{
\unitlength=0.7mm
\linethickness{0.4pt}
\begin{picture}(15.00,16.67)
\put(2.00,5.00){\circle*{1.33}}
\put(13.00,5.00){\circle*{1.33}}
\put(17.00,5.00){\circle*{1.33}}
\put(7.33,16.00){\circle*{1.33}}
\put(7.33,16.00){\vector(-1,-2){5.00}}
\put(7.33,16.00){\vector(1,-2){5.00}}
\put(15.00,10.00){\circle*{1.33}}
\put(15.00,10.00){\vector(1,-2){2.60}}
\put(15.00,10.00){\vector(-1,-2){2.60}}
\end{picture}
}}
+
\text{\raisebox{-18pt}{
\unitlength=0.7mm
\linethickness{0.4pt}
\begin{picture}(15.00,16.67)
\put(2.00,5.00){\circle*{1.33}}
\put(13.00,5.00){\circle*{1.33}}
\put(9.00,5.00){\circle*{1.33}}
\put(7.33,16.00){\circle*{1.33}}
\put(7.33,16.00){\vector(-1,-2){5.00}}
\bezier{64}(7.33,16.00)(15.00,13.00)(13.00,5.00)
\put(11.00,10.00){\circle*{1.33}}
\put(11.00,10.00){\vector(1,-2){2.60}}
\put(11.00,10.00){\vector(-1,-2){2.60}}
\end{picture}
}}
.
\]
\end{example}

\begin{example}
\label{ExOrder2OnBullets}
Finally, consider a graph in which two arrows fall on the first sink and let its content be $(\bullet \bullet) \in G_{2,0}$:
\[
\text{\raisebox{-18pt}{
\unitlength=0.7mm
\linethickness{0.4pt}
\begin{picture}(15.00,17.67)
\put(1.00,5.00){\circle*{1.33}}
\put(4.00,5.00){\circle*{1.33}}
\put(2.5,5.00){\oval(6,4)}
\put(13.00,5.00){\circle*{1.33}}
\put(7.33,11.33){\circle*{1.33}}
\put(2.00,17.00){\circle*{1.33}}
\put(2.00,17.00){\vector(0,-1){9.9}}
\put(2.00,17.00){\vector(1,-1){5.33}}
\put(7.33,11.33){\vector(1,-1){5.33}}
\put(7.33,11.33){\vector(-1,-1){4.33}}
\bezier{64}(7.33,11.33)(6.33,10.33)(4.33,8.33)
\end{picture}
}}
= 
\text{\raisebox{-18pt}{
\unitlength=0.7mm
\linethickness{0.4pt}
\begin{picture}(15.00,17.67)
\put(5.00,5.00){\circle*{1.33}}
\put(2.00,5.00){\circle*{1.33}}
\put(13.00,5.00){\circle*{1.33}}
\put(7.33,11.33){\circle*{1.33}}
\put(2.00,17.00){\circle*{1.33}}
\put(2.00,17.00){\vector(0,-1){11.33}}
\put(2.00,17.00){\vector(1,-1){5.33}}
\put(7.33,11.33){\vector(1,-1){5.33}}
\put(7.33,11.33){\vector(-1,-1){5.33}}
\end{picture}
}}
+
\text{\raisebox{-18pt}{
\unitlength=0.7mm
\linethickness{0.4pt}
\begin{picture}(15.00,17.67)
\put(-1.00,5.00){\circle*{1.33}}
\put(2.00,5.00){\circle*{1.33}}
\put(13.00,5.00){\circle*{1.33}}
\put(7.33,11.33){\circle*{1.33}}
\put(2.00,17.00){\circle*{1.33}}
\put(2.00,17.00){\vector(0,-1){11.33}}
\put(2.00,17.00){\vector(1,-1){5.33}}
\put(7.33,11.33){\vector(1,-1){5.33}}
\put(7.33,11.33){\vector(-1,-1){5.33}}
\end{picture}
}}
+
\text{\raisebox{-18pt}{
\unitlength=0.7mm
\linethickness{0.4pt}
\begin{picture}(15.00,17.67)
\put(-1.00,5.00){\circle*{1.33}}
\put(2.00,5.00){\circle*{1.33}}
\put(13.00,5.00){\circle*{1.33}}
\put(7.33,11.33){\circle*{1.33}}
\put(2.00,17.00){\circle*{1.33}}
\put(2.00,17.00){\vector(-1,-4){2.8}}
\put(2.00,17.00){\vector(1,-1){5.33}}
\put(7.33,11.33){\vector(1,-1){5.33}}
\put(7.33,11.33){\vector(-1,-1){5.33}}
\end{picture}
}}
+
\text{\raisebox{-18pt}{
\unitlength=0.7mm
\linethickness{0.4pt}
\begin{picture}(15.00,17.67)
\put(-1.00,5.00){\circle*{1.33}}
\put(2.00,5.00){\circle*{1.33}}
\put(13.00,5.00){\circle*{1.33}}
\put(7.33,11.33){\circle*{1.33}}
\put(2.00,17.00){\circle*{1.33}}
\put(2.00,17.00){\vector(0,-1){11.33}}
\put(2.00,17.00){\vector(1,-1){5.33}}
\put(7.33,11.33){\vector(1,-1){5.33}}
\put(7.33,11.33){\vector(-4,-3){8.33}}
\end{picture}
}}
.
\]
\end{example}

These three examples basically cover all the situations; we shall refer to them again,
namely, from the next chapter 
where the restrictions by using the total differential orders are discussed.

So far, we have focused on 
graphs; under the action of a graph on a graph, their coefficients are multiplied.
(This is why the associativity of the $\star$-\/product is an infinite system of quadratic equations for the coefficients of all the graphs).

\begin{implement}
In the class {\tt KontsevichGraphSeries<T>}, the method
\begin{verbatim}
    KontsevichGraphSeries<T>::operator()
\end{verbatim}
allows function\/-\/call syntax for the insertions described above.
As its argument it takes a {\tt std::vector} (that is, a list) of the Kontsevich graph series in $\hbar$; these are the $m$ respective arguments for a Kontsevich graph series.
It returns a {\tt KontsevichGraphSeries<T>}.
The method is called for the object of the class, that is, for the graph series which is to be evaluated at the $m$ specified arguments.
\end{implement}

For example, this allows the realization of Examples~\ref{ExWedgeOnBullets} and~\ref{ExWedgeOnWedge}  in {\tt C++} expressions as {\tt wedge(\{ dot, twodots \})} and {\tt wedge(\{ dot, wedge \})} respectively.

\begin{implement}
\label{ImplAssoc}
To calculate the associator $\Assoc_\star(f,g,h)$ for a given $\star$-\/product and ordered objects $f,g,h$, the call is
\begin{verbatim}
    > star_product_associator <star-product-filename>
\end{verbatim}
where the input file {\tt <star-product-filename>} contains the (truncated) power series for the $\star$-\/product.
In the standard output one obtains a (truncated at the same order in~$\hbar$ as in the input) power series containing, at each power~$\hbar^k$, the sums of graphs from $G_{3,k}$ with coefficients (their admissible types were introduced in~\S\ref{SecGraphSeries} above).
\end{implement}

\begin{example}
The associator for the $\star$-product up to order $2$ (from Example \ref{ExStarTwo}):
\begin{verbatim}
    $ star_product_associator star2.txt
    h^0:
    h^1:
    h^2:
    # 1 1 1 
    3 2 1   0 1 2 3    -2/3
    3 2 1   0 2 1 3    2/3
    3 2 1   0 4 1 2    -2/3
\end{verbatim}
It is $\tfrac{2}{3}\hbar^2$ times the Jacobiator \eqref{EqJacFig}, whose encoding we saw before in Example \ref{ExJacEval}.
\end{example}

\subsection{Gauge transformations}
At first glance, the concept of gauge transformations for (graphs in the) $\star$-products is an extreme opposite of plugging a list of graph series as arguments of a given graph series.
Namely, the idea of a gauge transformation is that a graph series (possibly of finite length) is towered over a single vertex~$\bullet \in G_{1,0}$.
By definition, a gauge transformation of a vertex~$\bullet$ is a map of the form $\bullet \mapsto [\bullet] = \bullet + \hbar\cdot(...)$ taking~$G_{1,0} \to \Bbbk[G_{1,*}][[\hbar]]$.

\begin{example}\label{ExGaugeLoop}
The map $\bullet \mapsto \bullet + \frac{\hbar^2}{12}
\!\!{\text{\raisebox{-7pt}{
\unitlength=0.4mm 
\linethickness{0.4pt}
\begin{picture}(11.67,19.67)
\put(6.00,5.00){\circle*{2.33}}
\put(1.00,15.00){\circle*{1.33}}
\put(11.00,15.00){\circle*{1.33}}
\put(11.00,15.00){\vector(-1,-2){4.67}}
\put(1.00,15.00){\vector(1,-2){4.67}}
\bezier{48}(1.00,15.00)(6.00,11.00)(10.00,14.67)
\bezier{52}(11.00,15.00)(6.00,19.67)(1.33,15.33)
\put(9.33,14.00){\vector(3,2){1.00}}
\put(2.00,16.00){\vector(-3,-2){1.00}}
\end{picture}
}}}$ is a gauge transformation of~$\bullet \in G_{1,0}$.
This graph series is encoded in the following file \texttt{gaugeloop.txt}:
\begin{verbatim}
    h^0:
    1 0 1            1
    h^2:
    1 2 1  0 2 1 0   1/12
\end{verbatim}
\end{example}

The construction of gauge transformations is extended from~$G_{1,0}$ by~$\Bbbk[[\hbar]]$- and $\Bbbk[G_{*,*}]$-\/linearity. This effectively means that in the course of action by a gauge transformation~$\mathsf{t}$ on a graph series $f \in \Bbbk[G_{*,*}][[\hbar]]$, all the arrows work over the vertices in every graph in~$f$ via the Leibniz rule (as it has been explained in the previous section).
This is how one expands $[f]\star[g]$, that is, the Kontsevich $\star$-\/product~\eqref{EqStar} of two gauged arguments~$[f]$ and~$[g]$.
Let us recall further that the shape $[\bullet] = \bullet + \hbar\cdot(\ldots)$, where the gauge tail of~$\bullet$ is given by some graphs from $\Bbbk[G_{1,*}][[\hbar]]$, guarantees the existence of a formal left inverse~${\sf t}^{-1}$ to the original transformation~${\sf t}$, so that $({\sf t}^{-1}\circ{\sf t})(\bullet)=\bullet$.

\begin{lemma}
If $\bullet \mapsto \square = {\sf t}(\bullet) = \bullet + \hbar\Gamma_1(\bullet) + \ldots + \hbar^\ell\Gamma_\ell(\cdot) + \bar{o}(\hbar^\ell)$ is a gauge transformation, let 
\[
{\sf t}^{-1}(\square) = \square + \hbar \gamma_1(\square) + \ldots + \hbar^\ell \gamma_\ell(\square) + \bar{o}(\hbar^\ell)
\]
by setting 
\[
\gamma_0 = \operatorname{id}, \qquad \gamma_m(\square) \mathrel{{:}{=}} -\sum_{k=0}^{m-1} \gamma_k(\Gamma_{m-k}(\square)).
\]
Then ${\sf t}^{-1}({\sf t}(\bullet)) = \bullet$, that is, the transformation ${\sf t}^{-1}\colon \Bbbk[G_{1,*}][[\hbar]] \to \Bbbk[G_{1,*}][[\hbar]]$ is the left inverse 
of~${\sf t}$ up to~$\bar{o}(\hbar)$.
\end{lemma}

It is readily seen that the assembly of the entire~${\sf t}^{-1}$ can require infinitely many operations even if the direct transformation~${\sf t}$ took only finitely many of them, e.g., as in Example~\ref{ExGaugeLoop}.

In these terms, for the Kontsevich $\star$-\/product~\eqref{EqStar} we obtain, by operating with gauge transformations and their formal inverses, a class of 
star products~$\star'$ which are defined by the relation
\begin{equation}\label{EqDefGaugeStar}
{\sf t}(f \star' g) = {\sf t}(f) \star {\sf t}(g), \qquad f,g \in C^\infty(N^n)[[\hbar]].
\end{equation}
Clearly, all these gauged star\/-\/products~$\star'$ remain associative (because $\star$~was) but the coefficients of graphs at an order~$k \geqslant 2$ in~$\hbar$ are no longer necessarily equal to the respective values in~\eqref{EqStar}.
The use of gauge transformations for products allows to gauge out \emph{some} graphs, often at a certain order~$\hbar^k$ in the star\/-\/product expansion.

\begin{example}\label{ExGaugedLoop}
The graph $\!\!\text{\raisebox{-7pt}{
\unitlength=0.4mm 
\linethickness{0.4pt}
\begin{picture}(14.00,20.33)
\put(2.00,5.00){\circle*{1.33}}
\put(13.00,5.00){\circle*{1.33}}
\put(2.00,15.00){\circle*{1.33}}
\put(13.00,15.00){\circle*{1.33}}
\put(13.00,15.00){\vector(0,-1){9.33}}
\put(2.00,15.00){\vector(0,-1){9.33}}
\bezier{64}(2.00,15.00)(7.00,9.00)(12.67,15.00)
\bezier{60}(13.00,15.00)(7.00,20.33)(2.67,15.00)
\put(11.67,14.00){\vector(1,1){0.67}}
\put(3.33,16.00){\vector(-1,-1){0.67}}
\end{picture}
}}$ with a loop is gauged out from the Kontsevich $\star$-\/pro\-duct~\eqref{EqStar} 
by using the gauge transformation ${\sf t}\colon \bullet \mapsto \bullet + \frac{\hbar^2}{12}
\!\!{\text{\raisebox{-7pt}{
\unitlength=0.4mm 
\linethickness{0.4pt}
\begin{picture}(11.67,19.67)
\put(6.00,5.00){\circle*{2.33}}
\put(1.00,15.00){\circle*{1.33}}
\put(11.00,15.00){\circle*{1.33}}
\put(11.00,15.00){\vector(-1,-2){4.67}}
\put(1.00,15.00){\vector(1,-2){4.67}}
\bezier{48}(1.00,15.00)(6.00,11.00)(10.00,14.67)
\bezier{52}(11.00,15.00)(6.00,19.67)(1.33,15.33)
\put(9.33,14.00){\vector(3,2){1.00}}
\put(2.00,16.00){\vector(-3,-2){1.00}}
\end{picture}
}}}$, see Example~\ref{ExGaugeLoop}.
Note that taking the formal inverse~${\sf t}^{-1}$ does create 
loop\/-\/containing graphs at higher orders~$\hbar^{\geqslant 3}$ in the gauged star\/-\/product~$\star'$ which is specified by~\eqref{EqDefGaugeStar}.
\end{example}

\begin{rem}
Not every graph taken in the Kontsevich star\/-\/product~$\star$ at a particular 
order~$\hbar^k$ can be gauged out.
For example, such are the graphs~$\Gamma \in \tilde{G}_{2,*}$ containing an internal vertex~$v$ with edges running from it to both the ground vertices.
\end{rem}

\begin{implement}
The command for gauge transformation is
\begin{verbatim}
    > gauge <star-product-filename> <gauge-transformation-filename>
\end{verbatim}
where 
\begin{itemize}
\item the file {\tt <star-product-filename>} contains a machine\/-\/format graph encoding of star\/-\/product~$\star$ truncated modulo~$\bar{o}(\hbar^k)$ for some $k \geqslant 0$; 
\item the content of {\tt <gauge-transformation-filename>} is a gauge transformation~${\sf t}(\bullet)$, that is, a truncated modulo~$\bar{o}(\hbar^{\ell \geqslant 0}$) series in~$\hbar$ consisting of the Kontsevich graphs built over one sink vertex~$\bullet$.
\end{itemize}
In the standard output one obtains the truncation, modulo~$\bar{o}(\hbar^{\min(k,\ell)})$, of the graph series for the gauged star\/-\/product~$\star'$ defined by $f \star' g = {\sf t}^{-1}({\sf t}(f)\star{\sf t}(g))$.

(The corresponding method is {\tt KontsevichGraphSeries<T>::gauge\symbol{"5F}transform()} in Appendix \ref{AppCPP}.)
\end{implement}

\begin{example}\label{ExGaugedLoopImplement}
Let the gauge transformation from Example \ref{ExGaugedLoop} be stored in the file {\tt gaugeloop.txt}, and recall the $\star$-product up to order two from Example \ref{ExStarTwo} in the file {\tt star2.txt}.
The gauge transformation kills the loop graph:
\begin{verbatim}
    $ gauge star2.txt gaugeloop.txt > star2gauged_unreduced.txt
    $ reduce_mod_skew star2gauged_unreduced.txt > star2gauged.txt
    $ cat star2gauged.txt
    h^0:
    2 0 1             1
    h^1:
    2 1 1  0 1        1
    h^2:
    2 2 1  0 1 0 1    1/2
    2 2 1  0 1 0 2    1/3
    2 2 1  0 1 1 2    -1/3
\end{verbatim}
Indeed, we see that the line
\begin{verbatim}
    2 2 1  0 3 1 2    -1/6
\end{verbatim}
containing the loop graph has disappeared.
\end{example}

Let us note at once that every gauge transformation ${\sf t}$ given by a Kontsevich graph polynomial in~$\hbar$ of degree~$\ell$ can clearly be viewed formally as a polynomial transformation of any degree greater or equal than~$\ell$.
This is why by using the same software we can actually obtain the gauged star\/-\/product~$\star'$ modulo~$\bar{o}(\hbar^4)$ starting with the Kontsevich star\/-\/product~$\star$ modulo~$\bar{o}(\hbar^4)$ and applying the gauge transformation of nominal degree $\ell=2$ from Example~\ref{ExGaugeLoop}.
In other words, the precision in~$\star'$ with respect to~$\hbar$ is the same as in~$\star$ even though the degree of the polynomial gauge transformation~${\sf t}$ is smaller.
In practice, this is achieved by adding an empty list of graphs at the power~$\hbar^k$ to a given gauge transformation of degree~$\ell < k$.

\section{Associativity of the Kontsevich $\star$-product}\label{SecAssoc}
\noindent%
In the final section of this paper we explore two complementary matters.
On the one hand, we analyse how the associativity postulate for the Kontsevich $\star$-\/product contributes to finding the values of weights~$w(\Gamma)$ for graphs~$\Gamma$ in~$\star$.
On the other hand, a point is soon reached when no new information can be obtained about the values of~$w(\Gamma)$: specifically, neither from the fact of associativity of the $\star$-\/product nor from any proven properties of the Kontsevich weights.
We outline a computer\/-\/assisted scheme of reasoning that, working uniformly over the set of all Poisson structures under study, reveals the associativity of $\star$-\/product on the basis of our actual knowledge about the weights~$w(\Gamma)$ of graphs~$\Gamma$ in~it.

In~\cite{sqs15} we reported an exhaustive description of the Kontsevich $\star$-\/product up to~$\bar{o}(\hbar^3)$.
At the next expansion order~$\bar{o}(\hbar^4)$ in~$\star$, we now express the weights of all the $160\,000 = (5 \cdot 4)^4$
graphs~$\Gamma \in \tilde{G}_{2,4}$ (of which up to $10\,000 = (5\cdot 4/2)^4$ are different modulo signs) in terms of only~$10$ parameters; those ten master\/-\/parameters themselves are the (still unknown) Kontsevich weights of the four internal vertex graphs portrayed in Fig.~\ref{10Graphs}.
By following the second strategy we prove that for any values of those ten parameters the $\star$-\/product expansion modulo~$\bar{o}(\hbar^4)$ is 
associative, also up to~$\bar{o}(\hbar^4)$.

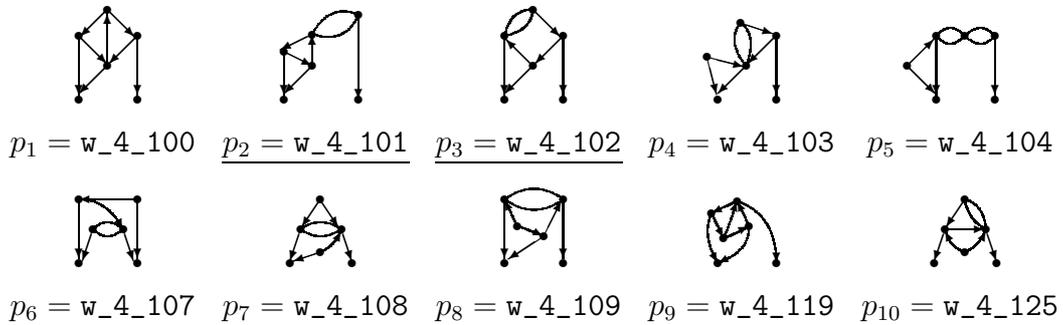
\begin{figure}[ht]
\begin{tabular}{c c c c c}
\raisebox{-8.60pt}{
\unitlength=0.7mm
\linethickness{0.4pt}
\begin{picture}(15.00,23.67)
\put(2.00,5.00){\circle*{1.33}}
\put(13.00,5.00){\circle*{1.33}}
\put(2.00,17.00){\circle*{1.33}}
\put(2.00,17.00){\vector(0,-1){11.33}}
\put(2.00,17.00){\vector(1,-1){5.33}}
\put(7.33,11.33){\circle*{1.33}}
\put(7.33,11.33){\vector(-1,-1){5.33}}
\put(7.33,11.33){\vector(0,1){10}}
\put(7.33,22.00){\circle*{1.33}}
\put(7.33,22.00){\vector(-1,-1){5.1}}
\put(7.33,22.00){\vector(1,-1){5.1}}
\put(13.00,17.00){\circle*{1.33}}
\put(13.00,17.00){\vector(0,-1){11.33}}
\put(13.00,17.00){\vector(-1,-1){5.33}}
\end{picture}
}
&
\raisebox{-8.60pt}{
\unitlength=0.7mm
\linethickness{0.4pt}
\begin{picture}(18.00,23.67)
\put(2.00,5.00){\circle*{1.33}}
\put(16.00,5.00){\circle*{1.33}}
\put(2.00,14.00){\circle*{1.33}}
\put(2.00,14.00){\vector(0,-1){8.33}}
\put(2.00,14.00){\vector(2,-1){5.33}}
\put(7.33,11.33){\circle*{1.33}}
\put(7.33,11.33){\vector(-1,-1){5.33}}
\put(7.33,11.33){\vector(0,1){6}}
\put(7.33,17.33){\circle*{1.33}}
\put(7.33,17.33){\vector(-2,-1){6}}
\bezier{56}(7.33,17.33)(13.67,15.33)(16.00,21.00)
\bezier{56}(7.33,17.33)(10.67,23.33)(16.00,21.00)
\put(16.00,21.00){\circle*{1.33}}
\put(16.00,21.00){\vector(0,-1){16}}
\end{picture}
}
&
\raisebox{-8.60pt}{
\unitlength=0.7mm
\linethickness{0.4pt}
\begin{picture}(15.00,23.67)
\put(2.00,5.00){\circle*{1.33}}
\put(13.00,5.00){\circle*{1.33}}
\put(2.00,17.00){\circle*{1.33}}
\put(2.00,17.00){\vector(0,-1){11.33}}
\put(7.33,11.33){\circle*{1.33}}
\put(7.33,11.33){\vector(-1,-1){5.33}}
\put(7.33,11.33){\vector(-1,1){5.1}}
\bezier{56}(2.00,17.00)(3.00,22.00)(7.33,22.00)
\bezier{56}(2.00,17.00)(7.00,18.00)(7.33,22.00)
\put(7.33,22.00){\circle*{1.33}}
\put(7.33,22.00){\vector(1,-1){5.1}}
\put(13.00,17.00){\circle*{1.33}}
\put(13.00,17.00){\vector(0,-1){11.33}}
\put(13.00,17.00){\vector(-1,-1){5.33}}
\end{picture}
}
&
\raisebox{-8.60pt}{
\unitlength=0.7mm
\linethickness{0.4pt}
\begin{picture}(15.00,23.67)
\put(2.00,5.00){\circle*{1.33}}
\put(13.00,5.00){\circle*{1.33}}
\put(0.00,13.00){\circle*{1.33}}
\put(0.00,13.00){\vector(1,-4){1.8}}
\put(0.00,13.00){\vector(4,-1){7.0}}
\put(7.33,11.33){\circle*{1.33}}
\put(7.33,11.33){\vector(-1,-1){5.33}}
\put(6.33,19.50){\circle*{1.33}}
\put(6.33,19.50){\vector(3,-1){6.1}}
\bezier{56}(6.33,19.50)(9.50,17.00)(7.33,11.33)
\bezier{56}(6.33,19.50)(3.50,16.00)(7.33,11.33)
\put(13.00,17.00){\circle*{1.33}}
\put(13.00,17.00){\vector(0,-1){11.33}}
\put(13.00,17.00){\vector(-1,-1){5.33}}
\end{picture}
}
&
\raisebox{-8.60pt}{
\unitlength=0.7mm
\linethickness{0.4pt}
\begin{picture}(15.00,23.67)
\put(2.00,5.00){\circle*{1.33}}
\put(13.00,5.00){\circle*{1.33}}
\put(2.00,17.00){\circle*{1.33}}
\put(2.00,17.00){\vector(0,-1){11.33}}
\put(-3.50,11.33){\circle*{1.33}}
\put(-3.50,11.33){\vector(1,-1){5.33}}
\put(-3.50,11.33){\vector(1,1){5.1}}
\bezier{56}(2.00,17.00)(4.50,20.00)(7.33,17.00)
\bezier{56}(2.00,17.00)(4.50,14.00)(7.33,17.00)
\bezier{56}(7.33,17.00)(10.50,20.00)(13.00,17.00)
\bezier{56}(7.33,17.00)(10.50,14.00)(13.00,17.00)
\put(7.33,17.00){\circle*{1.33}}
\put(13.00,17.00){\circle*{1.33}}
\put(13.00,17.00){\vector(0,-1){11.33}}
\end{picture}
}
\\
$p_1={}${\tt w\symbol{"5F}4\symbol{"5F}100} &
\underline{$p_2={}$\tt w\symbol{"5F}4\symbol{"5F}101} &
\underline{$p_3={}$\tt w\symbol{"5F}4\symbol{"5F}102} &
$p_4={}${\tt w\symbol{"5F}4\symbol{"5F}103} &
$p_5={}${\tt w\symbol{"5F}4\symbol{"5F}104} \\
\raisebox{-8.60pt}{
\unitlength=0.7mm
\linethickness{0.4pt}
\begin{picture}(15.00,23.67)
\put(2.00,5.00){\circle*{1.33}}
\put(13.00,5.00){\circle*{1.33}}
\put(2.00,17.00){\circle*{1.33}}
\put(2.00,17.00){\vector(0,-1){11.33}}
\bezier{56}(2.00,17.00)(7.33,17.00)(10.33,11.33)
\put(10.33,11.33){\vector(1,-2){0}}
\put(13.00,17.00){\circle*{1.33}}
\put(13.00,17.00){\vector(0,-1){11.33}}
\put(13.00,17.00){\vector(-1,0){11.0}}
\put(4.66,11.33){\circle*{1.33}}
\put(4.66,11.33){\vector(-1,-3){2.0}}
\put(10.33,11.33){\circle*{1.33}}
\put(10.33,11.33){\vector(1,-3){2.0}}
\bezier{56}(4.66,11.33)(7.33,14.00)(10.33,11.33)
\bezier{56}(4.66,11.33)(7.33,8.66)(10.33,11.33)
\end{picture}
}
&
\raisebox{-8.60pt}{
\unitlength=0.7mm
\linethickness{0.4pt}
\begin{picture}(15.00,23.67)
\put(1.66,5.00){\circle*{1.33}}
\put(13.33,5.00){\circle*{1.33}}
\put(3.66,11.33){\circle*{1.33}}
\put(3.66,11.33){\vector(-1,-3){2.0}}
\put(11.33,11.33){\circle*{1.33}}
\put(11.33,11.33){\vector(1,-3){2.0}}
\bezier{56}(3.66,11.33)(7.33,14.00)(11.33,11.33)
\bezier{56}(3.66,11.33)(7.33,8.66)(11.33,11.33)
\put(7.33,17.00){\circle*{1.33}}
\put(7.33,17.00){\vector(-2,-3){3.5}}
\put(7.33,17.00){\vector(2,-3){3.5}}
\put(7.33,7.00){\circle*{1.33}}
\put(7.33,7.00){\vector(-3,-1){5.1}}
\bezier{56}(7.33,7.00)(10,8)(11.33,11.33)
\put(11.33,11.00){\vector(1,2){0}}
\end{picture}
}
&
\raisebox{-8.60pt}{
\unitlength=0.7mm
\linethickness{0.4pt}
\begin{picture}(15.00,23.67)
\put(2.00,5.00){\circle*{1.33}}
\put(13.00,5.00){\circle*{1.33}}
\put(2.00,17.00){\circle*{1.33}}
\put(2.00,17.00){\vector(0,-1){11.33}}
\bezier{56}(2.00,17.00)(7.50,21.00)(13.00,17.00)
\bezier{56}(2.00,17.00)(7.50,13.00)(13.00,17.00)
\put(9.33,10.00){\circle*{1.33}}
\put(9.33,10.00){\vector(1,2){3.33}}
\put(9.33,10.00){\vector(-3,-2){6.66}}
\put(4.33,12.00){\circle*{1.33}}
\bezier{56}(4.33,12.00)(7,11)(9.33,10.00)
\put(9.33,10.00){\vector(2,-1){0}}
\bezier{56}(4.33,12.00)(3.33,14)(2.00,17.00)
\put(2.00,17.00){\vector(-1,2){0}}
\put(13.00,17.00){\circle*{1.33}}
\put(13.00,17.00){\vector(0,-1){11.33}}
\end{picture}
}
&
\raisebox{-8.60pt}{
\unitlength=0.7mm
\linethickness{0.4pt}
\begin{picture}(15.00,23.67)
\put(2.00,5.00){\circle*{1.33}}
\put(13.00,5.00){\circle*{1.33}}
\put(5.66,16.66){\circle*{1.33}}
\bezier{56}(5.66,16.66)(3,15.5)(0.84,14.38)
\put(0.84,14.38){\vector(-2,-1){0}}
\bezier{80}(5.66,16.66)(13,14)(13.00,5.00)
\put(13.00,5.00){\vector(0,-1){0}}
\put(0.84,14.38){\circle*{1.33}}
\bezier{56}(0.84,14.38)(-1,10)(2.00,5.00)
\put(2.00,5.00){\vector(1,-2){0}}
\bezier{56}(0.84,14.38)(2,12)(3.12,9.6)
\put(3.12,9.6){\vector(1,-2){0}}
\put(3.12,9.6){\circle*{1.33}}
\bezier{56}(3.12,9.6)(5,10.33)(7.9,11.88)
\put(7.9,11.88){\vector(2,1){0}}
\bezier{56}(3.12,9.6)(4.5,13)(5.66,16.66)
\put(5.66,16.66){\vector(1,3){0}}
\put(7.9,11.88){\circle*{1.33}}
\bezier{56}(7.9,11.88)(8,7)(2.00,5.00)
\put(2.00,5.00){\vector(-2,-1){0}}
\bezier{56}(7.9,11.88)(7,14)(5.66,16.66)
\put(5.66,16.66){\vector(-1,2){0}}
\end{picture}
}
&
\raisebox{-8.60pt}{
\unitlength=0.7mm
\linethickness{0.4pt}
\begin{picture}(15.00,23.67)
\put(1.66,5.00){\circle*{1.33}}
\put(13.33,5.00){\circle*{1.33}}
\put(3.66,11.33){\circle*{1.33}}
\put(3.66,11.33){\vector(-1,-3){2.0}}
\put(3.66,11.33){\vector(1,0){7.5}}
\put(11.33,11.33){\circle*{1.33}}
\put(11.33,11.33){\vector(1,-3){2.0}}
\put(7.33,17.00){\circle*{1.33}}
\put(7.33,17.00){\vector(-2,-3){3.5}}
\bezier{56}(7.33,17.00)(11,16)(11.33,11.33)
\bezier{56}(7.33,17.00)(8,13)(11.33,11.33)
\put(7.33,7.00){\circle*{1.33}}
\bezier{56}(7.33,7.00)(10,8)(11.33,11.33)
\put(11.33,11.00){\vector(1,2){0}}
\bezier{56}(7.33,7.00)(4.33,8)(3.66,11.33)
\put(3.66,11.00){\vector(-1,2){0}}
\end{picture}
}
\\
$p_6={}${\tt w\symbol{"5F}4\symbol{"5F}107} &
$p_7={}${\tt w\symbol{"5F}4\symbol{"5F}108} &
$p_8={}${\tt w\symbol{"5F}4\symbol{"5F}109} &
\underline{$p_9={}$\tt w\symbol{"5F}4\symbol{"5F}119} &
\underline{$p_{10}={}$\tt w\symbol{"5F}4\symbol{"5F}125} \\
\end{tabular}
\caption[The ten graphs whose unknown weights are taken as the master\/-\/parameters]{The ten graphs whose unknown weights\footnotemark\ are taken as the
master\/-\/parameters $p_i$; in fact, the four graphs whose weights are underlined can be gauged out from $\star$ so that there remain only $6$ parameters that determine it modulo $\bar{o}(\hbar^4)$.}\label{10Graphs}
\end{figure}

\footnotetext{Numerical approximations of two of these weights are listed in Table \ref{TableConjectured} in Appendix \ref{AppNumericalWeights}.} 

%

\subsection{Restriction of the $\star$-\/product associativity equation $\Assoc_\star(f,g,h)=0$ to a Poisson structure~$\cP$}\label{SecRestrictAssoc}
We now view the postulate of associativity for the Kontsevich $\star$-\/product as an equation for coefficients in the graph expansion of~$\star$.
Whenever an expansion modulo~$\bar{o}(\hbar^\ell)$ is known for the $\star$-\/product, one passes to the next order~$\bar{o}(\hbar^{\ell+1})$ by taking all the graphs~$\Gamma \in \tilde{G}_{2,\ell+1}$ with undetermined coefficients, and then expands (with respect to graphs) the associator $\Assoc_\star(f,g,h)$ up to the order~$\bar{o}(\hbar^{\ell+1})$.
This expansion now runs over all the graphs with at most $\ell+1$ internal vertices.
It is readily seen that by construction this associativity equation $\Assoc_\star(f,g,h) = \bar{o}(\hbar^{\ell+1})$ is always \emph{linear}\footnote{%
Should a graph~$\Gamma \in \tilde{G}_{2,\ell+1}$ be composite so that its Kontsevich weight is factorized using formula~\eqref{EqMult}, the resulting nonlinearity with respect to the weights would actually involve only the graphs with at most $\ell$~internal vertices.}
with respect to the coefficients of graphs from~$\tilde{G}_{2,\ell+1}$.

\begin{rem}\label{RemCanUseHigherOrders}
One can still 
get \emph{linear} relations between the weights~$w(\Gamma)$ of graphs $\Gamma \in \tilde{G}_{2,\ell+1}$ at order~$\hbar^{\ell+1}$ in~$\star$ by inspecting the associativity of~$\star$ at \emph{higher} orders --\,ranging from~$\ell+2$ till~$2\ell+1$\,-- in~$\hbar$.
Indeed, a linear relation containing the unknown weights (and the already known lower\/-\/order part of~$\star$ as coefficients) but not the weights of graphs with~$\geqslant \ell+2$ internal vertices can appear whenever a properly chosen homogeneous component of the tri\/-\/differential operator $\Assoc_\star(f,g,h)$ does not contain any weights from higher orders.
For instance, this is the component at homogeneity orders $(i,j,k)$ such that prime graphs $\Gamma \in \tilde{G}_{2,\geqslant \ell + 2}$ of homogeneity orders~$(i+j, k)$ and~$(i,j+k)$ (when viewed as bi\/-\/differential operators) do not exist or if the weights of all such graphs are known in advance.
\end{rem}

\subsubsection{}
Let us also note that in the graph equation $\Assoc_\star(f,g,h) = 0$ that holds by virtue of the Jacobi identity $\Jac(\cP)=0$, not every coefficient of every graph in the expansion should be expected to vanish.
Indeed, the Jacobiator is a vanishing sum of three graphs 
that evaluates to zero at every Poisson structure~$\cP$ which we put into every internal vertex.
This is why the restriction of associativity equation to a given Poisson structure (or to a class of Poisson structures) is a practical way to proceed in solution of the problem of finding the coefficients of graphs in~$\star$.
More specifically, after the restriction of associator $\Assoc_\star(f,g,h)$ 
to a structure~$\cP$ which is known to be Poisson so that all the instances and all derivatives of the Jacobiator $\Jac(\cP)$ are automatically trivialized, the left\/-\/hand side of the associativity equation $\Assoc_\star(f,g,h){\bigr|}_\cP = 0 \mod \bar{o}(\hbar^{\ell+1})$ becomes an analytic expression (\emph{linear} with respect to the unknowns~$w(\Gamma)$ for $\Gamma \in \tilde{G}_{2,\ell+1})$.
At this point one can proceed in several ways. 

We now outline three methods to obtain systems of linear equations upon the unknown weights~$w(\Gamma)$ of basic graphs $\Gamma \in \tilde{G}_{2,\ell+1}$.
Working in local coordinates, we ensure that the unknowns' coefficients in the equations which we derive are \emph{real} numbers.\footnote{From the factorization of associator for~$\star$ via differential consequences of the Jacobi identity for a Poisson structure~$\cP$, which will be revealed in section~\ref{SecFactor} below, it will be seen in hindsight that the construction of linear relations between the graph weights is overall insensitive to a choice of local coordinates in a chart within a given Poisson manifold. Indeed, the factorization will have been achieved simultaneously for all Poisson structures on all the manifolds at once, irrespective of any local coordinates.}

\begin{method}
Let the associator's arguments be given functions $f,g,h \in C^\infty(N^n)$.
Restrict the analytic expression $\Assoc_\star(f,g,h){\bigr|}_\cP$ to a point~$\bx$ of the manifold~$N^n$ equipped with a Poisson structure~$\cP$.
For every choice of $f,g,h \in C^\infty(N^n)$ and of a point~$\bx \in N^n$, the restriction $\Assoc_\star(f,g,h){\bigr|}_\cP(\bx) = 0 \mod \bar{o}(\hbar^{\ell+1})$ yields \emph{one} linear relation between the weights of graphs at order~$\hbar^{\ell+1}$.
Taking the restriction at several points $\bx_1$,\ $\ldots$,\ $\bx_k \in N^n$, one obtains a system of such equations, the rank of which does not exceed the number~$k$ of such points in~$N^n$.
Bounded by the number of unknowns~$w(\Gamma)$, the rank would always stabilize as~$k \to \infty$. 
\end{method}

Examples of Poisson structures~$\cP$ --\,for instance, on the manifolds~$\mathbb{R}^n$\,-- are available from~
\cite{Perelomov} (here $n \geqslant 3$) and~\cite{Van}; 
from Proposition 2.1 on p.~{74} 
in the latter one obtains a class of Poisson (in fact, symplectic) structures with polynomial coefficients on even\/-\/dimensional affine spaces~$\mathbb{R}^{2k}$. 
Besides, there is a regular construction (by using the R-\/matrix formalism, see \cite[p.~{287}]{VanhaeckePoisson}) 
of Poisson brackets on the vector space of square matrices $\operatorname{Mat}(\mathbb{R}, k \times k) \cong \mathbb{R}^{k^2}$ (e.g., in this way one has a rank\/-\/six Poisson structure on~$\mathbb{R}^9$).

Method~1 is the least computationally expensive, so it can be used effectively at the initial stage, e.g., to detect the zero values of certain graph weights: once found, such trivial values allow to decrease the number of unknowns in the further reasoning. 

\begin{method}
Now let $f,g,h \in \Bbbk[x^1,\ldots,x^n]$ be polynomials referred to local 
coordinates $x^1$,\ $\ldots$,\ $x^n$ on~$N^n$.
On that coordinate chart $U_\alpha 
\subset N^n$, take a Poisson structure the coefficients~$\cP^{ij}(\bx)$ of which would also be polynomial.
In consequence, the left\/-\/hand side of the equation $\Assoc_\star(f,g,h){\bigr|}_\cP = 0 \mod \bar{o}(\hbar^{\ell+1})$ then becomes polynomial as well.
Linear in the unknowns~$w(\Gamma)$, all the coefficients of this polynomial equation vanish (independently from each other).
Again, this yields a system of linear algebraic equations for the unknown weights~$w(\Gamma)$ of the Kontsevich graphs $\Gamma \in \tilde{G}_{2,\ell+1}$ in the $\star$-\/product.
\end{method}

We observe that the linear equations obtained by using Method~2 better constrain the set of unknowns~$w(\Gamma)$, that is, the rank of this system is typically higher than in Method~1.
Intuitively, this is because the polynomials at hand are not collapsed to their values at points~$\bx \in N$.

\begin{method}
Keep the associator's arguments $f,g,h$ unspecified and consider a class of Poisson structures $\cP[\psi_1,\ldots,\psi_m]$ depending in a differential polynomial way on functional parameters~$\psi_\alpha$, that is, on arbitrary functions, whenever~$\cP$ is referred to local coordinates.
(For example, let $n=3$ and on~$\mathbb{R}^3$ with Cartesian coordinates $x,y,z$ introduce the class of Poisson brackets using the Jacobian determinants,
\begin{equation}\label{Eq3DPoisson}
\{u,v\}_\cP = p \cdot \det\bigl(\partial(q,u,v)/\partial(x,y,z)\bigr), \qquad q \in C^\infty(\mathbb{R}^3),
\end{equation}
supposing that the density~$p(x,y,z)$ is also smooth on~$\mathbb{R}^3$.)
Now view the associator $\Assoc_\star(f,g,h){\bigr|}_{\cP[\psi_1,\ldots,\psi_m]}$ as a
polydifferential operator in the parameters $f,g,h$ (with respect to which it is linear) and in $\psi_1$,\ $\ldots$,\ $\psi_m$ from~$\cP$.
By splitting the associator, which is postulated to vanish modulo~$\bar{o}(\hbar^{\ell+1})$, into homogeneous differential\/-\/polynomial components, we obtain a system of linear algebraic equations upon the graph weights.
\end{method}

It is readily seen that, whenever the parameters $\psi_1$,\ $\ldots$,\ $\psi_m$ are chosen to be polynomials (here let us suppose for definition that the resulting Poisson structure~$\cP(\bx)$ itself is polynomial), the rank of the algebraic system obtained by Method~3 can be greater than the rank of an analogous system from Method~2.
This is because the analytic expression $\Assoc_\star(f,g,h){\bigr|}_{\cP[\psi_1,\ldots,\psi_m]}$ keeps track of all the parameters, whereas in Method~2 they are merged to a single polynomial.

We finally note that the linear algebraic systems which are produced by each method should be merged.
Indeed, the goal is to maximize the rank and by this, reduce the number of free parameters in the solution.\footnote{%
If the rank of the resulting linear algebraic system is equal to the number of unknowns --\,and if all the coefficients coming from lower orders~$\leqslant \ell$ within the $\star$-\/product expansion with respect to~$\hbar$ are also rational -- then all the solution components are rational numbers as well, cf.~\cite{WillwacherFelderIrrationality}.}


It has been seen in \S\ref{SecGraphsOnGraphs}, Implementation \ref{ImplAssoc} how the associator is calculated in terms of graphs.
The next step --\,namely, 
restriction of the associator to a given Poisson structure\,--
can be performed by using a call {\tt poisson\symbol{"5F}evaluate} as 
it has been explained in~\S\ref{SecPoissonEvaluate}. 
However, the further restriction as described in the Methods has been implemented in a separate program (similar to {\tt poisson\symbol{"5F}evaluate}) which directly outputs the desired relations, as follows.

\begin{implement}
The command
\begin{verbatim}
    > poisson_make_vanish <graph-series-file> <poisson-structure>
\end{verbatim}
sends to the standard output relations such as 
\begin{verbatim}
    -1/24+w_3_1+4*w_3_2==0
\end{verbatim}
between the undetermined coefficients in the input, which must hold if the input graph series is to vanish as a consequence of the Jacobi identity for the specified Poisson structure.
The implementation is described in the Methods above.
The choice of Poisson structure is made in the same way as in Implementation \ref{ImplPoissonEvaluate}.
If the optional extra argument {\tt --linear-solve} is specified, the program will assume that the relations which will be obtained are linear, and attempt to solve the linear system.
\end{implement}

\begin{example}\label{ExOrder3}
To obtain all the weights of basic graphs $\Gamma \in \tilde{G}_{2,3}$ at~$\hbar^3$ in the Kontsevich star\/-\/product~$\star$, it was enough to build the linear system of algebraic equations that combined (\textit{i}) cyclic relations~\eqref{EqCyclic}, (\textit{ii}) the relations which Method~3 produces for generic Poisson structure~\eqref{Eq3DPoisson}, and (\textit{iii}) those linear relations between the weights of $\Gamma \in \tilde{G}_{2,3}$ which --\,in view of Remark~\ref{RemCanUseHigherOrders} on p.~\pageref{RemCanUseHigherOrders}\,-- still do appear at the next power~$\hbar^4$ in $\Assoc_\star(f,g,h) = 0$, by using the same generic Poisson structure 
\eqref{Eq3DPoisson}.
The resulting expansion of~$\star$-product modulo~$\bar{o}(\hbar^3)$ is shown in formula~\eqref{EqStarOrd3} on~
p.~\pageref{EqStarOrd3}.
This result is achieved by using the software as follows.
Starting from the sets of basic graphs up to the order $2$ (with known weights) in the file {\tt basic2.txt}, generate lists of basic graphs (with undetermined weights) up to the order four:

\begin{verbatim}
    $ cp basic2.txt basic3+4w.txt
    $ echo 'h^3:' >> basic3+4w.txt
    $ generate_graphs 3 --basic=yes --with-coefficients=yes \
           >> basic3+4w.txt
    $ echo 'h^4:' >> basic3+4w.txt
    $ generate_graphs 4 --basic=yes --with-coefficients=yes \
           >> basic3+4w.txt
\end{verbatim}
Build the $\star$-product expansion up to the order $4$ from these basic sets:
\begin{verbatim}
    $ star_product basic3+4w.txt > star3+4w_unreduced.txt
    $ reduce_mod_skew star3+4w_unreduced.txt > star3+4w.txt
\end{verbatim}
Generate cyclic weight relations:
\begin{verbatim}
    $ cyclic_weight_relations star3+4w_unreduced.txt \
           > weight_relations_3+4w-cyclic.txt
\end{verbatim}
Build the associator expansion up to the order $4$ from the $\star$-product expansion:
\begin{verbatim}
    $ star_product_associator star3+4w.txt > assoc3+4w.txt
\end{verbatim} 
Obtain relations from the requirement of associativity for the Poisson structure \eqref{Eq3DPoisson}:
\begin{verbatim}
    $ poisson_make_vanish assoc3+4w.txt 3d-generic \ 
           > weight_relations_3+4w-3d.txt
\end{verbatim}
Merge the systems of linear relations:
\begin{verbatim}
    $ cat weight_relations_3+4w-* > weight_relations_3+4w_all.txt
\end{verbatim}
Solving the linear system in {\tt weight\symbol{"5F}relations\symbol{"5F}3+4w\symbol{"5F}all.txt} yields the solution 
\begin{verbatim}
    w_3_1=1/24,   w_3_2=0,      w_3_3=0,  w_3_4=-1/48, w_3_5=-1/48
    w_3_6=0,      w_3_7=0,      w_3_8=0,  w_3_9=0,     w_3_10=0
    w_3_11=-1/48, w_3_12=-1/48, w_3_13=0, w_3_14=0.
\end{verbatim}
Store the set of basic graphs at $\hbar^3$ with their true weights in the file \texttt{basic3.txt} (not removing graphs with zero weights); store the true Kontsevich $\star$-product up to $\hbar^3$ in the file \texttt{star3.txt} and its associator in the file \texttt{assoc3.txt}.

Instead of evaluating the associator in full, we could also have selected (e.g. by reading the file {\tt assoc3+4w.txt}, which also contains lines of the form ``{\tt \# i j k}'') those differential orders $(i,j,k)$ at $\hbar^4$ at which only weights from order $3$ appear, in view of Remark \ref{RemCanUseHigherOrders}: such are $(1,3,2)$, $(2,3,1)$, $(2,1,3)$, $(3,2,1)$, $(3,1,2)$, $(1,2,3)$ and $(2,2,2)$.
\end{example}

\begin{rem} 
A substitution of the values of certain graph weights expressed via other weights is tempting but not always effective.
Namely, we do not advise repeated running of any of the three methods with such expressions taken into account in the input.
Usually, the gain is disproportional to the time consumed; for instead of a coefficient to-express the program now has to handle what typically is a linear combination of several coefficients.
This shows that the only types of substitutions which are effective are either setting the coefficients to fixed numeric values (e.g., to zero) or the shortest possible assignments of a weight value via a single other weight value (like $w(\Gamma_1) = -w(\Gamma_2)$ for some graphs~$\Gamma_1$ and~$\Gamma_2$).
\end{rem}

\subsubsection{The $\star$-\/product expansion at order four}
At order four in the expansion of the Kon\-tse\-vich $\star$-\/product with respect to~$\hbar$, there are $149$~basic graphs $\Gamma \in \tilde{G}_{2,4}$.
The knowledge of their coefficients would completely determine the $\star$-\/product modulo~$\bar{o}(\hbar^4)$.
By using Methods~1--3 from~\S\ref{SecRestrictAssoc}, we found the exact values of $67$~basic graphs and we expressed the remaining $82$~weights in terms of the $10$~master\/-\/parameters (themselves the weights of certain graphs from~$\tilde{G}_{2,4}$; the other $72$~weights are linear functions of these ten).

\begin{theor}\label{ThmBig}
The weights of basic Kontsevich graphs at order~$4$ are subdivided as follows.
The weights of $27$ basic 
graphs 
are equal to zero.
Of these $27$\textup{,} the integrands of $21$~weights are identically zero\textup{,} and the other $6$~weight values were found to be equal to zero.
The remaining $122$ weights of basic graphs $\Gamma \in \tilde{G}_{2,4}$ are arranged as follows\textup{:}
\begin{itemize}
\item[$\cdot$] $40$ nonzero weights are known explicitly\textup{;}
\item[$\cdot$] the values of the remaining $82$~weights are expressed linearly in terms of 
the 
weights of those ten graphs which are shown in Fig.~\textup{\ref{10Graphs}}.
\end{itemize}
$\bullet$ The encoding of entire $\star$-\/product modulo~$\bar{o}(\hbar^4)$\textup{,} that is\textup{,} its part up to~$\bar{o}(\hbar^3)$ known from formula~\eqref{EqStarOrd3} plus $\hbar^4$~times the sum of all the prime and composite weighted graphs with four internal vertices\textup{,} is given in Appendix~\textup{\ref{AppStarEncoding}}.
\textup{(}In that table the weights of composite graphs are numbers\textup{;} for they are expressed via the known coefficients of graphs from~$\tilde{G}_{2,\leqslant 3}$.\textup{)}
The weights of basic graphs at~$\hbar^4$ are expressed in Table~\textup{\ref{Table149via10}} in terms of the ten master\/-\/parameters, see p.~\textup{\pageref{Table149via10}} in Appendix~\textup{\ref{AppStarEncoding}}.
\end{theor}

\noindent
Moreover (as stated in Theorem~\ref{ThMainAssocOrd4} on p.~\pageref{ThMainAssocOrd4} below),
the associativity $\Assoc_\star(f,g,h)=0\mod\bar{o}(\hbar^4)$ is established (up to order four) for the star product $\star\mod\bar{o}(\hbar^4)$ at all values of the ten master\/-\/parameters.

\begin{proof}[Proof scheme \textup{(}for Theorem \textup{\ref{ThmBig})}]
\label{Impl149via10}
We run the software as follows.
First one generates the sets of basic graphs up to order $4$, with undetermined weights at order $4$ (the weights at order $2$ and $3$ are known from e.g. Example \ref{ExStarTwo} and Example \ref{ExOrder3}):
\begin{verbatim}
    $ cp basic3.txt basic4w.txt
    $ echo 'h^4:' >> basic4w.txt
    $ generate_graphs 4 --basic=yes --with-coefficients=yes \
          >> basic4w.txt
\end{verbatim}
(The output is listed in Table~\ref{TblBasic} of Appendix~\ref{AppStarEncoding}.)

\noindent
Build the $\star$-product expansion up to order $4$:
\begin{verbatim}
    $ star_product basic4w.txt > star4w_unreduced.txt
    $ reduce_mod_skew --print-differential-orders star4w_unreduced.txt \
          > star4w.txt
\end{verbatim}
(The output is listed in Table~\ref{TableStar} of Appendix~\ref{AppStarEncoding}.)

\noindent
Generate the linear cyclic weight relations at order $4$:
\begin{verbatim}
    $ cyclic_weight_relations star4w_unreduced.txt \
          > weight_relations_4w-cyclic.txt
\end{verbatim}
Find $21$ relations of the form {\tt w\symbol{"5F}4\symbol{"5F}xxx==0} which hold by virtue of the weight integrand vanishing in formula~\eqref{EqWeight}, by using Implementation \ref{ImplWeightIntegrands} in Appendix \ref{AppNumericalWeights}, and place these relations in the file {\tt weight\symbol{"5F}relations\symbol{"5F}4w-integrandvanishes.txt}.

\noindent
Build the expansion of the associator for the $\star$-product up to the order $4$:
\begin{verbatim}
    $ star_product_associator star4w.txt > assoc4w.txt
\end{verbatim}
(The output is listed in Table~\ref{TableAssoc4} of Appendix~\ref{AppStarEncoding}.)

\noindent
Obtain relations from the requirement of associativity for the Poisson structure \eqref{Eq3DPoisson}:
\begin{verbatim}
    $ poisson_make_vanish assoc4w.txt 3d-generic \
          > weight_relations_4w-3d.txt
\end{verbatim}
Merge the systems of linear equations:
\begin{verbatim}
    $ cat weight_relations_4w-* > weight_relations_4w_total.txt
\end{verbatim}
Solve the resulting system (contained in {\tt weight\symbol{"5F}relations\symbol{"5F}4w\symbol{"5F}total.txt}) by using any relevant software.
One obtains the relations listed in Table \ref{Table149via10} in Appendix \ref{AppStarEncoding}, e.g. in the file {\tt weight\symbol{"5F}relations\symbol{"5F}4w\symbol{"5F}intermsof10.txt}.
To express the star\/-\/product (respectively, the associator for the $\star$-product) in terms of the $10$ parameters, run
\begin{verbatim}
    $ substitute_relations star4w.txt \ 
          weight_relations_4w_intermsof10.txt \
          > star4_intermsof10_unreduced.txt
    $ reduce_mod_skew star4_intermsof10_unreduced.txt \
          > star4_intermsof10.txt
\end{verbatim}
(respectively, substitute into {\tt assoc4w.txt} to obtain \texttt{assoc4\symbol{"5F}intermsof10.txt}); see Implementation \ref{ImplSubstituteRelations}.
\end{proof}

\begin{rem}
Numerical approximations of weights are listed in Tables \ref{TableVerified} and \ref{TableConjectured} in Appendix~\ref{AppNumericalWeights}.
In particular, we have the approximate values of the master\/-\/parameters $p_4={}${\tt w\symbol{"5F}4\symbol{"5F}103}${}\approx -1/11520$ and $p_5={}${\tt w\symbol{"5F}4\symbol{"5F}104}${}\approx 1/2880$.%
\footnote{The values of ten master-parameters have been suggested by Pym and Panzer \cite{PymPrivateComm}, see Table~\ref{Tab10Pym} on p.~\pageref{Tab10Pym} in Appendix~\ref{App10Pym} below.
Their prediction completely agrees with our numeric data.
}
\end{rem}

\begin{rem}
\label{RemVanishInAssoc}
Out of the~$149$ weights of basic graphs in the Kontsevich $\star$-\/product, as many as $28$~weights do not appear in the equation $\Assoc_\star(f,g,h) = 0$ at $\hbar^4$.
A mechanism which works towards such disappearance is that some graphs $\Gamma\in \tilde{G}_{2,4}$ which do not show up are bi\/-\/derivations with respect to the sinks. 
Combined at order four in the associator with only the original undeformed product~$\times$, every such graph is cancelled out from $(f \star g) \star h - f \star (g \star h)$ according to the mechanism which we illustrate here:
\[
\biggl[
\text{\!\!\!\raisebox{-18pt}{
\unitlength=0.7mm
\linethickness{0.4pt}
\begin{picture}(15.00,16.67)
\put(2.00,5.00){\circle*{1.33}}
\put(13.00,5.00){\circle*{1.33}}
\put(7.33,16.00){\makebox(0,0)[cc]{$\square$}}
\put(7.33,16.00){\vector(-1,-2){5.00}}
\put(7.33,16.00){\vector(1,-2){5.00}}
\end{picture}
}}
,\ \bullet\ \bullet\biggr] =
\text{\raisebox{-18pt}{
\unitlength=0.7mm
\linethickness{0.4pt}
\begin{picture}(15.00,16.67)
\put(1.00,5.00){\circle*{1.33}}
\put(4.00,5.00){\circle*{1.33}}
\put(2.5,5){\oval(8,5)}
\put(12.00,5.00){\circle*{1.33}}
\put(7.33,16.00){\makebox(0,0)[cc]{$\square$}}
\put(7.33,16.00){\vector(-1,-2){4.00}}
\put(7.33,16.00){\vector(1,-2){5.00}}
\end{picture}
}}
+
\text{\raisebox{-18pt}{
\unitlength=0.7mm
\linethickness{0.4pt}
\begin{picture}(15.00,16.67)
\put(2.00,5.00){\circle*{1.33}}
\put(13.00,5.00){\circle*{1.33}}
\put(17.00,5.00){\circle*{1.33}}
\put(7.33,16.00){\makebox(0,0)[cc]{$\square$}}
\put(7.33,16.00){\vector(-1,-2){5.00}}
\put(7.33,16.00){\vector(1,-2){5.00}}
\end{picture}
}}
-
\text{\raisebox{-18pt}{
\unitlength=0.7mm
\linethickness{0.4pt}
\begin{picture}(15.00,16.67)
\put(2.00,5.00){\circle*{1.33}}
\put(10.00,5.00){\circle*{1.33}}
\put(14.00,5.00){\circle*{1.33}}
\put(12,5){\oval(8,5)}
\put(7.33,16.00){\makebox(0,0)[cc]{$\square$}}
\put(7.33,16.00){\vector(-1,-2){5.00}}
\put(7.33,16.00){\vector(1,-2){4.00}}
\end{picture}
}}
-
\text{\raisebox{-18pt}{
\unitlength=0.7mm
\linethickness{0.4pt}
\begin{picture}(15.00,16.67)
\put(2.00,5.00){\circle*{1.33}}
\put(13.00,5.00){\circle*{1.33}}
\put(-1.00,5.00){\circle*{1.33}}
\put(7.33,16.00){\makebox(0,0)[cc]{$\square$}}
\put(7.33,16.00){\vector(-1,-2){5.00}}
\put(7.33,16.00){\vector(1,-2){5.00}}
\end{picture}
}}
= 0.
\]
In this way the ten master\/-\/parameters are split into the six which do show up in the associativity equation and the four weights which do not show up in $\Assoc_\star(f,g,h) = 0$ at $\hbar^4$ but which do appear through the cyclic weight relations (see formula~\eqref{EqCyclic} on p.~\pageref{EqCyclic}).
\end{rem}

\subsection{Computer\/-\/assisted proof scheme for associativity of~$\star$ for all~$\{\cdot,\cdot\}_\cP$}\label{SecFactor}
In practice, the methods from \S\ref{SecRestrictAssoc} stop producing linear relations that would be new with respect to the already known constraints for the graph weights.
As soon as such ``saturation'' is achieved, 
the number of master\/-\/parameters in $\star$-\/product expansion may in effect be minimal.
That is, the $\star$-\/product, known so far up to a certain order~$\bar{o}(\hbar^k)$, may in fact be always
associative --\,modulo~$\bar{o}(\hbar^k)$\,-- irrespective of a choice of the Poisson struc\-tu\-re(s)~$\cP$.

In this section we outline a scheme of computer\/-\/assisted reasoning that allows to reveal the factorization $\Assoc_\star(f,g,h) = \Diamond(\cP,\Jac(\cP))(f,g,h)$ of associator for~$\star$ via the Jacobiator $\Jac(\cP)$ that vanishes by definition for every Poisson structure~$\cP$.
At order $k=2$ the factorization $\Diamond(\Jac(\cP))$ is readily seen; the factorizing operator $\Diamond(\Jac(\cP)) = \tfrac{2}{3}\hbar^2\Jac(\cP) + \bar{o}(\hbar^2)$ is a differential operator of order zero, acting on its argument $\Jac(\cP)$ by multiplication.
Involving the Jacobi identity and only seven differential consequences from it at the next expansion order $k=3$, the factorization $\Assoc_\star(f,g,h) = \Diamond(\cP,\Jac(\cP))(f,g,h)$ was established by hand in~\cite{sqs15}.
For higher orders $k\geqslant 4$ the use of software allows to extend this line of reasoning; the scheme which we now provide works uniformly at all orders~$\geqslant 2$.

Let us first inspect how sums of graphs can vanish by virtue of differential consequences of the Jacobi identity $\Jac(\cP) = 0$ for Poisson structures $\cP$ on finite-dimensional affine real manifolds $N^n$.

\begin{lemma}[\cite{sqs15}]\label{Lemma}
A tri\/-\/differential operator $C = \sum_{|I|,|J|,|K|\geqslant 0} c^{IJK}\ \partial_I \otimes \partial_J \otimes \partial_K$ 
with coefficients $c^{IJK} \in C^\infty(N^n)$
vanishes identically if and only if 
all its homogeneous components 
$C_{ijk} = \sum_{|I|=i,|J|=j,|K|=k} c^{IJK}\ \partial_I \otimes \partial_J \otimes \partial_K$ 
vanish for all differential orders $(i,j,k)$ of the respective 
multi\/-\/indices $(I,J,K)$; here $\partial_L = \partial_1^{\alpha_1} \circ \cdots \circ \partial_n^{\alpha_n}$ for a multi\/-\/index $L = (\alpha_1, \ldots, \alpha_n)$. 
\end{lemma}

Lemma~\ref{Lemma} states in practice that for every arrow falling on the Jacobiator (for which, in turn, a triple of arguments is specified), the expansion of the Leibniz rule yields four fragments which vanish separately.
Namely, there is the fragment such that the derivation acts on the content $\cP$ of the Jacobiator's two internal vertices, and there are three fragments such that the arrow falls on the first, second, or third argument of the Jacobiator.
It is readily seen that the action of a derivative on an argument of the Jacobiator effectively amounts to an appropriate redefinition of its respective argument (cf. Examples~\ref{ExWedgeOnBullets}--\ref{ExOrder2OnBullets} on p.~\pageref{ExWedgeOnBullets}).
Therefore, a restriction to the order $(1,1,1)$ is enough in the run\/-\/through over all the graphs which contain Jacobiator~\eqref{EqJacFig} and which stand on the three arguments $f,g,h$ of the operator~$\Diamond(\cP,\Jac(\cP))$ at hand. 


\begin{define}\label{DefLeibnizGraph}
A \emph{Leibniz graph} is a graph whose vertices are either sinks, or the sources for two arrows, 
or the Jacobiator (which is a source for three arrows). 
There must be at least one Jacobiator vertex.
The three arrows originating from a Jacobiator vertex must land on three distinct vertices (and not on the Jacobiator itself).
Each edge falling on a Jacobiator works by the Leibniz rule on the two internal vertices in~it.
\end{define}

An example of a Leibniz graph is given in Fig.~\ref{FigSample}.
Every Leibniz graph can be expanded to a sum of Kontsevich graphs, by expanding both the Leibniz rule(s) and all copies of the Jacobiator. 
In this way (sums of) Leibniz graphs also encode (poly)dif\-fe\-re\-ntial operators $\Diamond(\cP, \Jac(\cP))$, depending on the bi\/-\/vector $\cP$ and the tri\/-\/vector $\Jac(\cP)$.
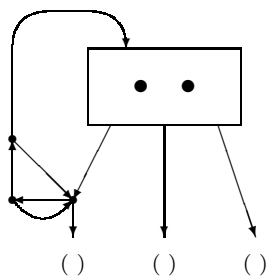
\begin{figure}[htb]
\begin{minipage}{0.3\textwidth}
\begin{align*}
\unitlength=1mm
\special{em:linewidth 0.4pt}
\linethickness{0.4pt}
\begin{picture}(40.67,35.00)
\put(15.00,20.00){\framebox(20.00,10.00)[cc]{$\bullet\quad\bullet$}}
\put(25.00,20.00){\vector(0,-1){15.00}}
\put(18.00,20.00){\vector(-1,-2){5.00}}
\put(32.00,20.00){\vector(1,-3){5.00}}
\put(13.00,10.00){\vector(0,-1){5.00}}
\put(13.00,10.00){\vector(-1,0){8.00}}
\put(13.00,0.00){\makebox(0,0)[cb]{\tiny(\ )}}
\put(25.00,0.00){\makebox(0,0)[cb]{\tiny(\ )}}
\put(5.00,10.00){\vector(0,1){8.00}}
\put(5.00,18.00){\vector(1,-1){8.00}}
\put(37.00,0.00){\makebox(0,0)[cb]{\tiny(\ )}}
\bezier{100}(5.00,10.00)(9.00,5.00)(13.00,10.00)
\put(13.00,10.00){\vector(1,1){0.00}}
\put(5.00,18.00){\line(0,1){12}}
\bezier{80}(5.00,30.00)(5.00,35.00)(10.00,35.00)
\put(10.00,35.00){\line(1,0){5.00}}
\bezier{80}(15.00,35.00)(20.00,35.00)(20.00,30.00)
\put(20.00,30.00){\vector(0,-1){0.00}}
\put(13.00,10.00){\circle*{1}}
\put(5.00,10.00){\circle*{1}}
\put(5.00,18.00){\circle*{1}}
\end{picture}
\end{align*}
\end{minipage}
\small
\begin{minipage}{0.4\textwidth}
\begin{itemize}
\item There is a cycle, 
\item there is a loop, 
\item there are no tadpoles in this graph,
\item an arrow falls back on $\Jac(\cP)$, 
\item and $\Jac(\cP)$ does not stand on all of the three sinks.
\end{itemize}
\end{minipage}
\normalsize
\caption{A nontrivial example of Leibniz graph.}
\label{FigSample}
\end{figure}

\noindent
By design, we have

\begin{proposition}\label{PropLeibnizGraphZero}
For every Poisson bi\/-\/vector~$\cP$ the value --\,at the Jacobiator $\Jac(\cP)$\,-- of every (poly)\/dif\-fe\-ren\-ti\-al operator encoded by the Leibniz graph(s) is zero.
\end{proposition}

\begin{proof}
By induction on the number of arrows falling on the Jacobiator.
In case of zero arrows, the operator is a multiple of a Jacobiator and hence zero.
In general, the operator associated to a Leibniz graph is of the form \[ (\partial_L \Jac(\cP))(A,B,C) \cdot D, \]
where $\partial_L$ are the incoming arrows on the Jacobiator.
Now, $\Jac(\cP)(A,B,C) = 0$ implies
\begin{align*}
0 &= \partial_L(\Jac(\cP)(A,B,C)) \cdot D \\
&= (\partial_L \Jac(\cP))(A,B,C) \cdot D + \sum_{\substack{H+I+J+K=L\\H \neq L}} (\partial_H \Jac(\cP))(\partial_I A, \partial_J B, \partial_K C) \cdot D
\end{align*}
where the terms in the sum on the right are Leibniz graphs with \emph{fewer} arrows falling on the Jacobiator, hence they are zero by induction.
The same proof works for a Leibniz graph with more than one Jacobiator (the extraneous ones -- in $D$ -- are irrelevant).
\end{proof}

Hence, to show that a sum of Kontsevich graphs vanishes at every Poisson structure, it suffices to write it as a sum of Leibniz graphs.

In particular, the mechanism of factorization of the associator for the Kontsevich $\star$-product is known from \cite{MK97}; it has been discussed in \cite{Kiev18}.
Namely, by \cite{MK97} the Jacobi identity is the only obstruction to the Kontsevich $\star$-product associativity.
This is because an expression of the $\star$-product associator as a (possibly, non-unique) sum of Leibniz graphs can be predicted in advance, based on the graphs in the Formality morphism (see \cite{Kiev18} for more details).

\begin{example} 
Consider the associator $\Assoc_\star(f,g,h) \mod \bar{o}(\hbar^3)$ for the $\star$-\/product which is fully known up to order~$3$. 
The assembly 
of factorizing operator~$\Diamond(\cP,{\cdot})$ acting on~$\Jac(\cP)$ is explained in~\cite{sqs15}; linear in its argument, the operator~$\Diamond$ has differential order one with respect to the Jacobiator.
\end{example}

\begin{rem}
The same technique, showing the vanishing of a sum of Kontsevich graphs by writing it as a sum of Leibniz graphs, has been used in \cite{f16}.
The underlying mechanism from \cite{Ascona96} is analyzed in detail in \cite{OrMorphism}.
\end{rem}

\begin{implement}[Encoding of Leibniz graphs]
For a Leibniz graph 
with $\ell$~Jacobiators 
and $n-2\ell$ remaining bi\/-\/vector vertices,
an encoding is defined in terms of the encoding of a Kontsevich graph in its expansion, plus the data which tells where the Jacobiators are.
The full encoding is the integer $\ell$, followed by the Kontsevich graph encoding with $n$ internal vertices, followed by the $\ell$ pairs of Jacobiator vertices $(j_1, j_2)$, where the internal Jacobiator edge is $j_1 \leftarrow j_2$.
Each target in the Kontsevich graph encoding which is a Jacobiator vertex $j_i$ from a Jacobiator $(j_1, j_2)$ (except for the target of the internal Jacobiator edge $j_1 \leftarrow j_2$) should be interpreted as a placeholder for a Leibniz rule acting on both $j_1$ and $j_2$.
\end{implement}

\begin{example}
The Leibniz graph from Fig.~\ref{FigSample} (with $n = 5$ and $\ell = 1$) has the encoding
\begin{verbatim}
    1   3 5 1   0 5 3 6 3 4 3 1 6 2   6 7
\end{verbatim}
Here the first \texttt{6} should be interpreted as a placeholder for the Jacobiator containing the last two vertices \texttt{6} and \texttt{7}; the three arguments of the Jacobiator are \texttt{3}, \texttt{1}, \texttt{2}.
To expand this encoding into Kontsevich graph encodings, cyclically permute the arguments of the Jacobiator and replace the placeholder by \texttt{6} or \texttt{7} (in all possible ways):
\begin{verbatim}
        3 5 1   0 5 3 6 3 4 3 1 6 2
        3 5 1   0 5 3 7 3 4 3 1 6 2
        3 5 1   0 5 3 6 3 4 1 2 6 3
        3 5 1   0 5 3 7 3 4 1 2 6 3
        3 5 1   0 5 3 6 3 4 2 3 6 1
        3 5 1   0 5 3 7 3 4 2 3 6 1
\end{verbatim}
One obtains six terms.
\end{example}

\begin{implement}
\label{ImplReduceModJacobi}
Let the input file {\tt <graph-series-filename>} contain a graph series~$S$ with constant (e.\,g., rational, real or complex) coefficients; here~$S$ is supposed to vanish by virtue of the Jacobi identity and its differential consequences.
Now run the command
\begin{verbatim}
    > reduce_mod_jacobi <graph-series-filename>
\end{verbatim}
The 
program finds a particular solution~$\Diamond$ of the factorization problem
\[
S(f,g,h) = \Diamond(\cP,\Jac(\cP),\ldots,\Jac(\cP))(f,g,h).
\]
In the standard output one obtains the list of encodings of Leibniz graphs in $\Diamond$ that specify differential consequences of the Jacobi identity; every such graph encoding is followed in the output by its sought\/-\/for nonzero coefficient.\footnote{Sample outputs of specified type are contained in Table~\ref{TblReduceModJacobi} in Appendix~\ref{AppAssocEncoding}.}
Two extra options can be set equal to nonnegative integer values, by passing these two numbers as extra command\/-\/line arguments. Namely, 
\begin{itemize}
\item the parameter \texttt{max-jacobiators} restricts the number of Jacobiators in each Leibniz graph, so that by the assignment \texttt{max-jacobiators = 1} the right\/-\/hand side $\Diamond\bigl(\cP,\Jac(\cP)\bigr)$ is linear in the Jacobiator, whereas if \texttt{max-jacobiators = 2}, the right\/-\/hand side $\Diamond\bigl(\cP,\Jac(\cP),\Jac(\cP)\bigr)$ can be quadratic in~$\Jac(\cP)$, and so~on;
\item independently, the parameter \texttt{max-jac-indegree} restricts (from above) the number of arrows falling on the Jacobiator(s) in each of the Leibniz graphs that constitute the factorizing operator~$\Diamond$.
\end{itemize}
Furthermore, if \verb"--solve" is specified as the third extra argument, the input graph series is allowed to contain undetermined coefficients; these are then added as variables to-solve-for in the linear system.
\end{implement}

\begin{theor}\label{ThMainAssocOrd4}
For every component~$S^{(i)}$ of the associator (for $\star$ from Theorem~\ref{ThmBig})
\[
\Assoc_\star(f,g,h) \mod \bar{o}(\hbar^4) \mathrel{{=}{:}} S^{(0)} + p_1S^{(1)} + \ldots + p_{10}S^{(10)},
\]
there exists
a factorizing operator~$\Diamond^{(i)}$ such that
\begin{equation*}
S^{(i)}(f,g,h) = \Diamond^{(i)}\bigl(\cP,\Jac(\cP)\bigr)(f,g,h), 
\qquad 0 \leqslant i \leqslant 10.
\end{equation*}
\noindent%
\noindent%
$\bullet$\quad At no values of the master\/-\/parameters $p_i$
would the solution~$\Diamond = \sum_i \Diamond^{(i)}$ of factorization problem be a \emph{first}\/-\/order differential operator acting on the Jacobiator.
\end{theor}

\begin{proof}[Proof scheme]
Take the associator $\Assoc_\star(f,g,h) \mod \bar{o}(\hbar^4)$ for the $\star$-\/product expansion modulo~$\bar{o}(\hbar^4)$, in the file \verb"assoc4_intermsof10.txt" which was obtained in Theorem~\ref{ThmBig}.
The associator is linear in the ten master\/-\/parameters.
Let us split it into the constant 
term (e.g., at the zero value of every parameter) plus the ten respective components~$S^{(i)}$:
\begin{verbatim}
    $ extract_coefficient assoc4_intermsof10.txt 1 \
          > assoc4_intermsof10_constantpart.txt
    $ extract_coefficient assoc4_intermsof10.txt w_4_100 \
          > assoc4_intermsof10_part100.txt
    $ extract_coefficient assoc4_intermsof10.txt w_4_101 \
          > assoc4_intermsof10_part101.txt
\end{verbatim}
(and so on, for each parameter $p_i$).
In fact, four of the parameters do not show up in the associator (see Remark \ref{RemVanishInAssoc}): the corresponding files do not contain any graphs.
Now run the command {\tt reduce\symbol{"5F}mod\symbol{"5F}jacobi} for each input file with~$S^{(i)}$
, e.g., for $S^{(1)}$:
\begin{verbatim}
    $ reduce_mod_jacobi assoc4_intermsof10_part100.txt
\end{verbatim}
For each $S^{(i)}$ a solution is found: the series vanishes modulo the Jacobi identity.
The output for $S^{(1)}$ is written in Table \ref{TblReduceModJacobi} in Appendix \ref{AppAssocEncoding}.
For the second part of the theorem, we run \verb"reduce_mod_jacobi" with the options \verb"max-jac-indegree = 1" and \verb"--solve":
\begin{verbatim}
    $ reduce_mod_jacobi assoc4_intermsof10.txt 1 1 --solve
\end{verbatim}
(Our setting of \verb"max-jacobiators = 1" here makes no difference.)
No solution is found.
Inspecting the output, we find that the following term in the associator cannot be produced by a first-order differential consequence of the Jacobi identity:
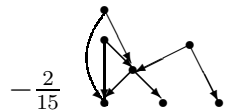
\begin{center}
$-\frac{2}{15}\hbar^4$
\text{\raisebox{-12pt}{
\unitlength=0.70mm
\linethickness{0.4pt}
\begin{picture}(15.00,25.33)
\put(2.00,5.00){\circle*{1.33}}
\put(13.00,5.00){\circle*{1.33}}
\put(23.50,5.00){\circle*{1.33}}
\put(7.33,11.33){\circle*{1.33}}
\put(2.00,17.00){\circle*{1.33}}
\put(2.00,17.00){\vector(0,-1){11.33}}
\put(2.00,17.00){\vector(1,-1){5.33}}
\put(7.33,11.33){\vector(1,-1){5.33}}
\put(7.33,11.33){\vector(-1,-1){5.33}}
\put(2.00,22.67){\circle*{1.33}}
\bezier{64}(2.00,22.00)(-5.00,13.50)(2.33,4.75)
\put(1,6.33){\vector(1,-1){0.67}}
\put(2.00,22.67){\vector(1,-2){5.00}}
\put(18.00,16.00){\circle*{1.33}}
\put(18.00,16.00){\vector(1,-2){5.33}}
\put(18.00,16.00){\vector(-2,-1){10.33}}
\end{picture}
}}
\end{center}
Indeed one can show this graph arises only in a differential consequence of order two.
\end{proof}

\begin{cor}[$\star$-\/product non\/-\/extendability from $\{\cdot,\cdot\}_{\cP}$ to $\{\cdot,\cdot\}_{\bcP}$ at order~$\hbar^4$]\label{CorVariational}
Because there are at least two arrows falling on the object $\Jac(\cP)$ in $\Diamond$ at every value of the ten 
master\/-\/parameters~$p_i$,
the associativity can be broken at order~$\hbar^4$ 
for extensions of the 
$\star$-\/product to infinite\/-\/dimensional set\/-\/up${}^{\text{\ref{FootVariational} on p.~\pageref{FootVariational}}}$ 
of $N^n$-\/valued fields~$\phi\in C^\infty(M^m\to N^n)$ over a given affine manifold~$M^m$, of local functionals~$F,G,H$ taking such fields to numbers, and of variational Poisson brackets $\{\cdot,\cdot\}_{\bcP}$ on the algebra of local functionals.
\end{cor}

Indeed, the Jacobiator $\Jac(\bcP)\cong0$ for a variational Poisson bi\/-\/vector~$\bcP$ is a cohomologically trivial variational tri\/-\/vector on the jet space $J^\infty(M^m\to N^n)$, whence the first variation of~$\Jac(\bcP)$ brought on it by a unique arrow would of course be vanishing identically. Nevertheless, that variational tri\/-\/vector's density is not necessarily equal to zero on~$J^\infty(M^m\to N^n)$ over~$M^m$ for those variational Poisson structures whose coefficients~$\cP^{ij}$ explicitly depend on the fields~$\phi$ or their derivatives along~$M^m$. This is why the second and higher variations of the Jacobiator~$\Jac(\bcP)$ would not always vanish. (Such higher\/-\/order variations of functionals are calculated by using the techniques from~\cite{gvbv,cycle16}.)
We know from~\cite{sqs15} that $\Assoc_\star(F,G,H)\cong 0 \mod \bar{o}(\hbar^3)$, 
i.e.\ the associator is trivial 
up to order~$\hbar^3$ for all variational Poisson brackets~$\{\cdot,\cdot\}_{\bcP}$ but we now see that it can contain cohomologically nontrivial terms proportional to~$\hbar^4$. Consequently, it is the order four at which the associativity of 
$\star$-\/products can start to leak in the course of deformation quantization of Poisson field models.



We now claim that four master\/-\/parameters can simultaneously be gauged out of the star\/-\/product.
\textup{(}That is, either some of the four or all of them at once can be set equal to zero\textup{,} although this may not necessarily be their true value given by formula~\eqref{EqWeight}.\textup{)}%
\footnote{Let us recall that the property of a parameter in a family of star\/-\/products to be removable by some gauge transformation is not the same as setting such parameter to zero (or any other value).
Indeed, other graph coefficients, not depending on the parameter at hand, might get modified by that gauge transformation.
However --\,and similarly to the removal of the loop graph at~$\hbar^2$ in the Kontsevich $\star$-\/product (see Examples~\ref{ExGaugedLoop} and~\ref{ExGaugedLoopImplement})\,-- the trivialization of four parameters at no extra cost is the case which Theorem~\ref{ThGauge} states.}

\begin{theor}\label{ThGauge}
For each $j\in\{2, 3, 9, 10\}$
there exists a gauge transformation $\text{\textup{id}} + \hbar^4 p_j Z_j$
\textup{(}listed in Table \ref{TableGauge} in Appendix \ref{AppGaugeEncoding}\textup{)}
such that the
master\/-\/parameter~$p_j$ is reset to
zero in the deformed star\/-\/product~$\star'$. This is achieved in
such a way that no graph coefficients which initially did not contain
the parameter to gauge out would change at all.

\noindent$\bullet$\quad Moreover\textup{,} the gauge transformation
$\text{\textup{id}}+\hbar^4\cdot\bigl(\sum_j
p_j Z_j \bigr)$ removes at once
all the four master\/-\/parameters\textup{,} still preserving those
coefficients of graphs in~$\star$ which did not depend on any of them.
\end{theor}

\begin{proof}[Proof scheme]
Let the $\star$-product expansion in terms of $10$ parameters (obtained in Theorem \ref{ThmBig}) be contained in \verb"star4_intermsof10.txt".
Construct a gauge transformation of the form ${\rm id} + \hbar^4 G$, where $G$ is the sum over all possible graphs with four internal vertices over one sink which are nonzero, without double edges, without tadpoles, and with positive differential order, taken with undetermined coefficients $g_i$:
\begin{verbatim}
    $ cat > gauge4.txt
    1 0 1           1
    h^4:
    ^D (press Ctrl+D)
    $ generate_graphs 4 1 --normal-forms=yes --zero=no \
          --positive-differential-order=yes \
          --with-coefficients=yes >> gauge4.txt
    $ sed -i 's/w/g/' gauge4.txt # replace coefficient prefix 'w' by 'g'
\end{verbatim}
Obtain gauged star-product expansion $\star'$ by applying the gauge transformation to $\star$:
\begin{verbatim}
    $ gauge star4_intermsof10.txt gauge4.txt \
          > star4_intermsof10_gauged_unreduced.txt
\end{verbatim}
Reduce the graph series for $\star'$ modulo skew-symmetry:
\begin{verbatim}
    $ reduce_mod_skew star4_intermsof10_gauged_unreduced.txt \
          > star4_intermsof10_gauged.txt
\end{verbatim}
Inspect which of the 10 parameters $p_j$ cannot be gauged out, by checking for the existence of graph coefficients containing $p_j$ but not any $g_i$.
For example, for $p_1={}$\verb"w_4_100":
\begin{verbatim}
    $ grep w_4_100 star4_intermsof10_gauged.txt \
           | grep -v g | wc -l
    17
\end{verbatim}
There are 17 graphs with such coefficients, so $p_1={}$\verb"w_4_100" cannot be gauged out.
Following this procedure for all the $10$ parameters, we find that the only candidates to be gauged out are $p_2={}$\verb"w_4_101", $p_3={}$\verb"w_4_102", $p_9={}$\verb"w_4_119", and $p_{10}={}$\verb"w_4_125".
Now inspect the file \verb"star4_intermsof10_gauged.txt" for the lines containing these $p_j$ and (necessarily, some) $g_i$.
For each $p_j$, find a choice of $g_i$ so that $p_j$ is completely removed from the file.
(The $g_i$ will be of the form $g_i = \alpha_{ij} p_j$ for $\alpha_{ij} \in \BBR$.)
It turns out that this is always possible.
Hence this choice of $g_i$ defines the sought\/-\/for gauge transformation ${\rm id} + \hbar^4 p_j Z_j$ which gauges out the parameter $p_j$.
The gauge-transformations which kill the (four) parameters separately may be combined into the gauge-transformation ${\rm id} + \hbar^4(\sum_j p_j Z_j)$ that kills all (four) of them simultaneously.
\end{proof}

\begin{rem}
The master\/-\/parameters which we can gauge out are exactly the ones which do not show up in the associativity equation (see Remark \ref{RemVanishInAssoc}).
\end{rem}

Let us finally address a possible origin of so ample a freedom in the
ten\/-\/parameter family of star\/-\/products (now known up
to~$\bar{o}(\hbar^4)$). We claim that the mechanism of vanishing via
differential consequences of the Jacobi identity, which was recalled
in Lemma~\ref{Lemma} and used in Theorem~\ref{ThMainAssocOrd4}, 
starts working not only for the associator built over~$\star$,
but it may even start working for the $\star$-product expansion itself.

\begin{theor}\label{ThNull}
The ten-parameter family of star-product expansions $\star = ... + \hbar^4\bigl(\star^{(0)} + \sum_{i=1}^{10} p_i \star^{(i)}\bigr) + \bar{o}(\hbar^4)$ does contain, in the ten-dimensional affine subspace parametrized by $p_1, \ldots, p_{10}$ in $\Bbbk[G_{2,4}]$, a unique one-dimensional \textup{(}null or `improper'\textup{)} subspace such that every point $\alpha \cdot (\star^{(9)} - 2\star^{(6)}) = \alpha \cdot \star^{(9|6)}$ in it admits a Leibniz graph factorization \textup{(}via the Jacobiator\textup{)} $\star^{(9|6)} = \nabla(\cP, \Jac(\cP)) \in \Bbbk[G_{2,4}]$.
This null space is the span of the direction \verb"w_4_119"\ {\em : }\verb"w_4_107"\ $ : \ldots = 1 : (-2) : 0 : ... : 0 \in \BBR\BBP^9$, that is, the master-parameters $p_9$ and $p_6$ occur in proportion $1: (-2)$ and all the other $p_i$\!'s are zero.
\end{theor}

In effect, the respective part of the star-product always cancels out for every given Poisson structure $\cP$.
This factorization and uniqueness of the direction $\star^{(9|6)}$ is established by using the same computer\/-\/assisted scheme of reasoning which worked in the proof of Theorem~\ref{ThMainAssocOrd4}.

\section{Discussion}
\label{SecDiscussion}

The coefficients of (sometimes different, sometimes gauge-inequivalent) star-product expansions up to low orders were previously obtained in the papers \cite{Kathotia, PenkavaVanhaecke, AmmarChloupGuttUniversal3, BenAmar, WillwacherObstruction}.
Let us compare the result in this paper with those publications, and let us use the software described in this paper to verify some results about other star-products.

\subsection{Previously known weights}

The values of some (families of) Kontsevich graph weights are given in the literature.
The graphs in the Bernoulli family have scaled Bernoulli numbers as weights (see \cite[Corollary 6.3]{BenAmar} or \cite[Proposition 4.4.1]{Kathotia}), e.g. {\tt w\symbol{"5F}3\symbol{"5F}2}${}=B_3/3!=0$ and {\tt w\symbol{"5F}4\symbol{"5F}12}${}=B_4/4!=-1/720$.
The weights of a family of graphs containing cycles are obtained in \cite[Corollary 6.3]{BenAmar}, e.g. {\tt w\symbol{"5F}3\symbol{"5F}9}${}=\pm B_3/(2\cdot3!) = 0$ and {\tt w\symbol{"5F}4\symbol{"5F}72}${}=-B_4/(2\cdot 4!) = 1/1440$.
Willwacher states in \cite{WillwacherObstruction} the vanishing of three graph weights at the order $3$ 
(they are {\tt w\symbol{"5F}3\symbol{"5F}7}, {\tt w\symbol{"5F}3\symbol{"5F}13}, {\tt w\symbol{"5F}3\symbol{"5F}14} in Figure~\ref{FigBasic3}) and the non-vanishing one other (it is \texttt{w\symbol{"5F}3\symbol{"5F}12} in Figure~\ref{FigBasic3}); this agrees with our calculation in Example~\ref{ExOrder3}.

\subsection{Numerical approximation}

In Tables \ref{TableVerified} and \ref{TableConjectured} in Appendix \ref{AppNumericalWeights} we list numerical approximations of several weights.
These approximations are consistent with the exact weights (and relations) obtained in this paper.
%
%

\subsection{Independent symbolic calculation}

The values for the weights found in this paper agree with a symbolic calculation of the graph weights reported by Pym and Panzer \cite{PymPrivateComm} and reproduced in Table~\ref{Tab10Pym} on p.~\pageref{Tab10Pym} in Appendix~\ref{App10Pym}.

\subsection{The obstruction to the existence of a loopless star product}
\label{SecWillwacherLoopless}

In \cite{WillwacherObstruction}, Willwacher establishes that any universal star-product (defined by the Kontsevich graphs, possibly with different coefficients) which is gauge-equivalent to Kontsevich's $\star$-product \emph{must} contain graphs with $2$-cycles.
To obtain the same result using our software, we proceed as follows.
\begin{example}
We start with Kontsevich's $\star$-product up to $\hbar^3$ in \texttt{star3.txt}.
The unique graph with a loop at order 2 can be removed by extending the gauge transformation from Example~\ref{ExGaugeLoop} which was stored in \texttt{gaugeloop.txt}:
\begin{verbatim}
    $ cp gaugeloop.txt gaugeloop3.txt
    $ echo "h^3:" >> gaugeloop3.txt
    $ gauge star3.txt gaugeloop3.txt > star3_gauge2_unreduced.txt
    $ reduce_mod_skew star3_gauge2_unreduced.txt > star3_gauge2.txt
\end{verbatim}
The gauged $\star$-product is obtained in \texttt{star3\symbol{"5F}gauge2.txt}:
\begin{verbatim}
    h^0:
    2 0 1       1
    h^1:
    2 1 1   0 1    1
    h^2:
    2 2 1   0 1 1 2    -1/3
    2 2 1   0 1 0 2    1/3
    2 2 1   0 1 0 1    1/2
    h^3:
    2 3 1   0 1 1 2 1 2    1/6
    2 3 1   0 1 0 1 1 2    -1/3
    2 3 1   0 1 0 2 0 2    1/6
    2 3 1   0 1 0 1 0 2    1/3
    2 3 1   0 1 0 1 0 1    1/6
    2 3 1   0 3 1 2 2 3    -1/6
    2 3 1   0 1 2 4 2 3    1/12
    2 3 1   0 1 0 2 1 3    -1/6
    2 3 1   0 1 0 4 1 2    -1/6
    2 3 1   0 3 1 4 1 3    -1/6
    2 3 1   0 1 1 2 2 3    -1/6
    2 3 1   0 3 1 2 1 2    1/6
    2 3 1   0 1 1 4 2 3    1/6
    2 3 1   0 3 0 2 1 2    1/6
    2 3 1   0 1 0 2 2 3    -1/6
    2 3 1   0 3 1 2 0 3    -1/6
    2 3 1   0 1 0 4 2 3    1/6
\end{verbatim}
\end{example}
Willwacher denotes this $\star$-product by $a = a_0 + a_1 + a_2 + a_3 + \ldots$, and supposes that our (desirably loopless) $\star$-product reads $b = a_0 + a_1 + a_2 + (a_3 + b_3) + (a_4 +b_4) + \ldots$
The Maurer-Cartan associativity equation $[b,b] = 0$ implies in particular $[a_0,b_3] = 0$ and $[a_1,b_3] + [a_0,b_4] = 0$ (here $[-,-]$ is the Gerstenhaber bracket).
\begin{claimNo}
For any solution of $[b,b]=0$, the sum $a_3 + b_3$ contains graphs with cycles.
\end{claimNo}
We carry out Willwacher's proof, with some minor corrections.
Since $b_3$ is a Hochschild cocycle it suffices to assume $b_3$ is in the image of the (graphical) Hochschild-Kostant-Rosenberg map, so it is a skew-symmetric bi-derivation.\footnote{However, this does not imply that each individual graph in it is a skew-symmetric bi-derivation.
Rather, each graph which is a bi-derivation can be skew-symmetrized, which yields either the original graph or the sum of two graphs which are mirror-reflections of each other.
It is clear that this is what Willwacher intended, e.g. because the mirror-reflection of his graph $D$ is not drawn.}
There are two terms $-\alpha A$ and $-\beta B$ in $a_3$  ($\alpha, \beta \neq 0$) which are skew-symmetric bi-derivations and graphs with cycles.
So for $a_3 + b_3$ to have no cycles, $b_3$ must be the linear combination $b_3 = \alpha A + \beta B$.
The proof proceeds by showing that $[a_1,b_3]+[a_0,b_4] = 0$ cannot hold: indeed, simplifying the $3$-cochain $[a_1,b_3]$ modulo the image of $[a_0, -]$ and the Jacobi identity results in a sum of graphs (called Shoikhet's obstruction) which does not vanish.
Let us illustrate all of this in detail.

\begin{example}
The skew bi-derivation terms in $a_3$ are:
\begin{verbatim}
    2 3 1   0 1 2 4 2 3    1/12
    2 3 1   0 3 1 2 2 3    -1/6
\end{verbatim}
(We have $\alpha = -1/12$ and $\beta = 1/6$.)
We store the correction term $b_3$ in the file \texttt{b3.txt}:
\begin{verbatim}
    2 3 1   0 1 2 4 2 3    -1/12
    2 3 1   0 3 1 2 2 3    1/6
\end{verbatim}
Calculate the Gerstenhaber bracket $[a_1,b_3]$ (using $a_1$ in \texttt{wedge.txt})\footnote{The graphical calculation of $[a_1,b_3]$ in~\cite{WillwacherObstruction} contains errors, e.g. the first graph has a vertex with three outgoing edges and the term with coefficient $\beta$ has arrows in the wrong direction.}:
\begin{verbatim}
    $ echo '2 1 1   0 1    1' > wedge.txt
    $ gerstenhaber_bracket wedge.txt b3.txt > \[wedge,b3\]_unreduced.txt
    $ reduce_mod_skew \[wedge,b3\]_unreduced.txt > \[wedge,b3\].txt
\end{verbatim}
Generate an ansatz for $b_4$:
\begin{verbatim}
    $ generate_graphs 4 --normal-forms=yes --with-coefficients=yes \
          --positive-differential-order=yes --zero=no > b4.txt
\end{verbatim}
Calculate $[a_0,b_4]$ (using $a_0$ in \texttt{dotdot.txt}):
\begin{verbatim}
    $ echo '2 0 1       1' > dotdot.txt
    $ gerstenhaber_bracket dotdot.txt b4.txt > \[dotdot,b4\]_unreduced.txt
    $ reduce_mod_skew \[dotdot,b4\]_unreduced.txt > \[dotdot,b4\].txt
\end{verbatim}
Store Shoikhet's obstruction with undetermined coefficients $A,B$ in \texttt{shoikhet\symbol{"5F}obs.txt}:
\begin{verbatim}
    3 4 1   0 1 2 3 3 4 3 4   A
    3 4 1   2 1 0 3 3 4 3 4   -A
    3 4 1   0 1 2 3 3 4 4 5   B
    3 4 1   2 1 0 3 3 4 4 5   -B
\end{verbatim}
Add $[a_0,b_4]$ and Shoikhet's obstruction to $[a_1,b_3]$:
\begin{verbatim}
    $ cat \[wedge,b3\].txt \[dotdot,b4\].txt shoikhet_obs.txt \
          > \[wedge,b3\]+\[dotdot,b4\]+shoikhet_obs_unreduced.txt
    $ reduce_mod_skew \
          \[wedge,b3\]+\[dotdot,b4\]+shoikhet_obs_unreduced.txt \
          > \[wedge,b3\]+\[dotdot,b4\]+shoikhet_obs.txt
\end{verbatim}
Reduce modulo the Jacobi identity and solve (for the expression to be equal to zero):
\begin{verbatim}
    $ reduce_mod_jacobi \[wedge,b3\]+\[dotdot,b4\]+shoikhet_obs.txt \
           1 10 --solve
\end{verbatim}
Indeed, there is a solution $A = \beta = 1/6$ and $B = -4(\alpha+\beta) = -1/3$.
So, modulo the image of $[a_0,-]$ and the Jacobi identity, $[a_1,b_3]$ is equal to Shoikhet's obstruction with $A = -\beta = -1/6$ and $B = 4(\alpha+\beta) = 1/3$ (the sign changed because we added Shoikhet's obstruction instead of subtracting it).\footnote{The solution reported here differs from Willwacher's, not in sign (which is left ambiguous in~\cite{WillwacherObstruction}) but in proportion: he claims $A = \pm 2\beta$ and $B = \pm 2(\alpha+\beta)$.}

\end{example}


\begin{example}
An example of a Poisson structure for which Shoikhet's obstruction doesn't vanish is given by \texttt{3d-polynomial}:
\begin{verbatim}
    $ poisson_evaluate shoikhet_obs.txt 3d-polynomial
    ...
    # [ x ] [ x ] [ y ]
    -4*A*y^3*z^2*x^2+2*y^3*B*z*x^3-2*y*B*z^3*x^3+y*B*z^4*x^2+...
    ...
\end{verbatim}
For example, the coefficient of the differential monomial $x^2y^3z^2 \partial_x \otimes \partial_x \otimes \partial_y$ is $-4A \neq 0$.
\end{example}

In this section we traced Willwacher's steps.
There is a much simpler proof of the claim when all the coefficients of graphs in $a_3$ are known (which was not the case in~\cite{WillwacherObstruction}):
there are loopful graphs with nonzero coefficients in $a_3$ which cannot be gauged out.

\subsection{Penkava--Vanhaecke deformations}

In \cite{PenkavaVanhaecke} M.~Penkava and P.~Vanhaecke give (among other things) deformations $\pi_\star = \pi + h\pi_1 + h^2\pi_2 + \ldots + h^n\pi_n + \bar{o}(h^n)$ where $\pi$ is the pointwise product, $\pi_1$ is the Poisson bracket, $h$ is the formal parameter, and associativity holds \emph{modulo} $\bar{o}(h^n)$ for arbitrary \emph{polynomial} Poisson algebras.
Note that every $\star$-product (which is associative as a formal power series in $h$) induces such an expansion modulo $\bar{o}(h^n)$ for every $n$, but not every deformation modulo $\bar{o}(h^n)$ can be extended to higher orders.\footnote{Note that this problem is different from the computation of obstructions to Kontsevich's Formality \cite{MK97,Ascona96}. Specifically, in ``Formality Conjecture'' \cite{Ascona96}, Kontsevich reports the absence of obstructions to Formality up to $n \leqslant 6$. Formality is now a theorem: Kontsevich's $\star$-product exists at all orders.}
Indeed, Penkava--Vanhaecke exhibit deformations which can be extended and some which cannot be extended.
(Namely, already at order $3$ there exist formulas which do not extend to higher orders 
--- although such formulas are clearly not the genuine Kontsevich star-product.)
In the following sequence of examples, we verify some of their results and compare them with ours.
\begin{example}\label{ExPV5.1}
Proposition 5.1 in \cite{PenkavaVanhaecke} gives a deformation $\pi + h\pi_1 + h^2\pi_2 + \bar{o}(h^2)$, and in fact it coincides with Kontsevich's $\star$-product modulo $\bar{o}(h^2)$ with the loop graph gauged out; see Example~\ref{ExGaugedLoopImplement} in this text.
\end{example}
\begin{example}\label{ExPV5.6a}
Theorem 5.6 in \cite{PenkavaVanhaecke} provides a deformation $\pi_\star = \pi + h\pi_1 + h^2\pi_2 + h^3\pi_3 + \bar{o}(h^3)$.
The differential polynomials in it can be viewed as Kontsevich graphs; we store their encodings with their numerical coefficients in \texttt{star3pv5.6.txt}:
\begin{verbatim}
    h^0:
    2 0 1                  1
    h^1:
    2 1 1   0 1            1
    h^2:
    2 2 1   0 1 1 2        -1/3
    2 2 1   0 1 0 2        1/3
    2 2 1   0 1 0 1        1/2
    h^3:
    2 3 1   0 1 1 2 2 3    -1/3
    2 3 1   0 1 0 2 2 3    -1/3
    2 3 1   0 1 0 1 0 1    1/6
    2 3 1   0 1 0 2 1 3    -1/6
    2 3 1   0 1 0 4 1 2    -1/6
    2 3 1   0 1 1 2 1 2    1/6
    2 3 1   0 1 0 2 0 2    1/6
    2 3 1   0 1 0 1 0 2    1/3
    2 3 1   0 1 0 1 1 2    -1/3
\end{verbatim}
Calculate the associator in terms of graphs (see Implementation~\ref{ImplAssoc}):
\begin{verbatim}
    $ star_product_associator star3pv5.6.txt > assoc3pv5.6.txt
\end{verbatim}
It vanishes as a consequence of the Jacobi identity (see Implementation~\ref{ImplReduceModJacobi}):
\begin{verbatim}
    $ reduce_mod_jacobi assoc3pv5.6.txt
\end{verbatim}
Hence $\pi_\star$ is associative modulo $\bar{o}(h^3)$ (for arbitrary Poisson structures on $\mathbb{R}^d$).
This deformation is not equal to Kontsevich's $\star$-product modulo $\bar{o}(h^3)$, nor is it Kontsevich's $\star$-product with the loop graph gauged out, but the next example gives the explicit relation between this product and Kontsevich's.
\end{example}
\begin{example}\label{ExPV5.6b}
Theorem 5.6 in \cite{PenkavaVanhaecke} 
further relates arbitrary deformations to $\pi_\star = \pi + h\pi_1 + h^2\pi_2 + h^3\pi_3 + \bar{o}(h^3)$ from Example~\ref{ExPV5.6a}.
Namely, every deformation modulo $\bar{o}(h^3)$ is gauge-equivalent to $\tilde{\pi}_\star = \pi + h\pi_1 + h^2(\pi_2 + \varphi_2) + h^3(\pi_3 + \varphi_3 + \psi_3)$ for some choice of $(\varphi_2,\varphi_3,\psi_3)$, where $\varphi_2$ and $\varphi_3$ are antisymmetric biderivations and $\psi_3$ is a symmetric $2$-cochain satisfying $\partial \psi_3 = [\pi_1, \varphi_2]$.
Let us show that this holds for Kontsevich's $\star$-product expansion $\pi_\star^K \mod \bar{o}(h^3)$.
In Section~\ref{SecWillwacherLoopless} we obtained Kontsevich's $\star$-product with the loop graph gauged out; let us denote it by $\tilde{\pi}_\star^K \mod \bar{o}(h^3)$.
Up to $\bar{o}(h^2)$ the deformations $\tilde{\pi}_\star^K$ and $\pi_\star$ are equal (as we observed in Example~\ref{ExPV5.1}), so we choose $\varphi_2 = 0$.
Subtracting $\pi_\star$ from $\tilde{\pi}_\star^K$ yields the file \texttt{star3\symbol{"5F}gauge2\symbol{"5F}minus\symbol{"5F}pv5.6.txt}:
\begin{verbatim}
    h^3:
    # 1 1 
    2 3 1   0 3 1 2 2 3    -1/6
    2 3 1   0 1 2 4 2 3    1/12
    # 1 2 
    2 3 1   0 3 1 4 1 3    -1/6
    2 3 1   0 1 1 2 2 3    1/6
    2 3 1   0 3 1 2 1 2    1/6
    2 3 1   0 1 1 4 2 3    1/6
    # 2 1 
    2 3 1   0 3 0 2 1 2    1/6
    2 3 1   0 1 0 2 2 3    1/6
    2 3 1   0 3 1 2 0 3    -1/6
    2 3 1   0 1 0 4 2 3    1/6
\end{verbatim}
and it can be seen that it is antisymmetric, so this must be $h^3\varphi_3$ and hence $\psi_3 = 0$.
But the terms are not all of differential order $(1,1)$, so how can $\varphi_3$ be a bi-derivation? 
The answer is that all other terms vanish due to two first-order differential consequences of the Jacobi identity for the Poisson structure.
In other words, there is a Leibniz graph -- with one arrow incoming on the Jacobiator -- which expands to the homogeneous component of order $(1,2)$, and naturally the mirror-reflection of that Leibniz graph expands to the order $(2,1)$ component.
This can be verified using \texttt{reduce\symbol{"5F}mod\symbol{"5F}jacobi}.

\end{example}
\begin{example}
Theorem~6.1 in \cite{PenkavaVanhaecke} gives the obstruction to extending $\pi_\star$ from Example~\ref{ExPV5.6a} to the fourth 
order.
The proof shows that this obstruction is the skew-symmetrization of the degree-$(1,1,1)$ homogeneous component of the associator at $h^4$. 
We reproduce this as follows.
First create a file representing $\pi_\star \mod \bar{o}(\hbar^4)$:
\begin{verbatim}
    $ cp star3pv5.6.txt star4pv6.1.txt
    $ echo "h^4:" >> star4pv6.1.txt
\end{verbatim}
Calculate the associator:
\begin{verbatim}
    $ star_product_associator star4pv6.1.txt > assoc4pv6.1.txt
\end{verbatim}
Skew\/-\/symmetrize (see \S\ref{ApplMultivectorFields} below):
\begin{verbatim}
    $ skew_symmetrize assoc4pv6.1.txt > obs4pv6.1_unreduced.txt
\end{verbatim}
Reduce modulo skew\/-\/symmetry:
\begin{verbatim}
    $ reduce_mod_skew obs4pv6.1_unreduced.txt \ 
          --print-differential-orders > obs4pv6.1.txt
\end{verbatim}
Finally, we see that the degree-$(1,1,1)$ homogeneous component at $\hbar^4$ is 
\begin{verbatim}
    3 4 1   0 1 2 3 3 4 4 5    4/3
    3 4 1   0 2 1 3 3 4 4 5    -4/3
    3 4 1   0 4 1 2 3 4 3 5    -4/3
    3 4 1   0 1 2 3 3 4 3 4    -2/3
    3 4 1   0 2 1 3 3 4 3 4    2/3
    3 4 1   0 4 1 2 3 4 3 4    -2/3
\end{verbatim}
which is (up to an irrelevant constant factor) the sum of six terms written in Theorem 6.1.
We store this component in the file \texttt{obs4pv6.1\symbol{"5F}111.txt} and the others in \texttt{obs4pv6.1\symbol{"5F}rest.txt}.
The latter vanish as a consequence of the Jacobi identity:
\begin{verbatim}
    $ reduce_mod_jacobi obs4pv6.1_rest.txt
\end{verbatim}
as claimed in Theorem~6.1.
The $(1,1,1)$-component does not vanish in general: indeed,
\begin{verbatim}
    $ reduce_mod_jacobi obs4pv6.1_111.txt
\end{verbatim}
does not find any solution.
An explicit Poisson structure for which the obstruction does not vanish is given at the end of \cite[\S 9]{PenkavaVanhaecke}.
We can do the same computation in our software: the respective Poisson structure was added under the name \texttt{4d-pv}, so that
\begin{verbatim}
    $ poisson_evaluate obs4pv6.1_111.txt 4d-pv
\end{verbatim}
shows a multi-vector field which is not identically zero.
\end{example}

\begin{example}
Finally Lemma 8.2 in \cite{PenkavaVanhaecke} gives the correction term $\varphi_3$ to $\pi_3$ for the deformation to extend to the fourth order.
In terms of graph encodings with coefficients, $\varphi_3$ is:
\begin{verbatim}
    2 3 1   0 3 1 2 2 3    -1/6
    2 3 1   0 1 2 4 2 3    1/12
\end{verbatim}
which is exactly the correction term we found in Example~\ref{ExPV5.6b} to make the deformation equal to Kontsevich's $\star$-product modulo $\bar{o}(h^3)$ with the loop graph gauged out.
This proves that the deformation extends to the fourth order.
Alternatively, we can store the full expansion in \texttt{star3pv8.2.txt} and confirm that it extends to the order~$4$ as follows.
First, add graphs with undetermined coefficients at $\hbar^4$:
\begin{verbatim}
    $ cp star3pv8.2.txt star4pv8.2.txt
    $ echo 'h^4:' >> star4pv8.2.txt
    $ generate_graphs 4 --normal-forms=yes --with-coefficients=yes \
          >> star4pv8.2.txt
\end{verbatim}
Calculate the graphical associator:
\begin{verbatim}
    $ star_product_associator star4pv8.2.txt > assoc4pv8.2.txt
\end{verbatim}
Finally, run
\begin{verbatim}
    $ reduce_mod_jacobi assoc4pv8.2.txt 1 2 --solve
\end{verbatim}
and observe that there is a solution.
\end{example}

\subsection{Universal star-products}

We do work on affine Poisson manifolds, so that formulae are coordinate-independent because of the contraction of upper versus lower indices in all tensor objects \emph{and} because all the Jacobians are constant in the course of affine coordinate reparametrizations.
S. Gutt et al in \cite{AmmarChloupGuttUniversal3} provide star-products modulo $\bar{o}(\hbar^3)$ which are universal with respect to all Poisson structures $\cP$ on all smooth finite-dimensional manifolds $\mathcal{M}^d$ equipped with a torsion-free, not necessarily flat, linear connection $\nabla$.
Then the formula of $\star \mod \bar{o}(\hbar^3)$ is expressed in terms of differential polynomials in not only the bi-vector $\cP$ -- clearly, our $\partial_i$ replaced by $\nabla_i$ in every instance -- but also in the \emph{curvature} $R$ of $\nabla$.
\begin{example}
To compare with Kontsevich's formula up to $\bar{o}(\hbar^3)$ which is given in the present paper (also in \cite{sqs15}), we can put $R = 0$ in the formula by Gutt et al.
The terms up to $\bar{o}(\hbar^2)$ clearly match.
A-priori there are $(5-1) \times (5-1) = 16$ terms with coefficient $-1/6$ at $\hbar^3$.
Two pairs of terms double, and they become two terms with coefficients $\pm 1/3$.
One term vanishes identically, because it is the zero graph from Example~\ref{ExZeroGraph}.
The resulting $13$ terms are exactly those in Kontsevich's formula \eqref{EqStarOrd3} at $\hbar^3$.
Hence the formula obtained by Gutt et al. restricted to $R=0$ coincides with Kontsevich's $\star$-product up to $\bar{o}(\hbar^3)$.
\end{example}
It would be interesting to recover such a univeral formula $\star(\cP, R)$ -- depending also on the curvature $R$ -- modulo $\bar{o}(\hbar^4)$ and beyond.

\subsection{Universal flows on spaces of Poisson structures}
\label{ApplMultivectorFields}

The software presented in this paper has been extended to operate on first\/-\/order differential operators which represent (skew-symmetric) multi\/-\/vector fields.
In particular skew\/-\/symmetrization was implemented in \verb"skew_symmetrize" and the graphical Schouten bracket was implemented in \verb"schouten_bracket".
This has been applied by the authors jointly with A.~Bouisaghouane in \cite{f16} to confirm the existence of a universal flow on the spaces of Poisson structures, which was suggested by Kontsevich.
The explicit mechanism that explains why these universal flows exist, based on work by Kontsevich, Willwacher, and Jost, is given by the authors in \cite{OrMorphism}.

\subsection{Open problems}
The following two questions, posed by M.~Kontsevich (private communication) can be approached up to finite orders in $\hbar$ by using the software modules which we have presented:
\begin{itemize}
\item Which quadratic weight relations are determined by the associativity alone?
(We refer to the preprint \cite[p.~61]{BanksPanzerPym} for discussion.)
\item How many degrees of freedom in the graph weights (at a fixed order in $\hbar$) are due to gauge transformations?
\end{itemize}

Independently (Kevin Morand, private communication), an open problem is to describe the action of the graph complex (with suitable cocycles $\gamma \in \ker [\bullet\!\!\!-\!\!\!\bullet, -]$) on the $\star$-product under the infinitesimal symmetries $\hbar\cP \mapsto \hbar\cP + \varepsilon \Or(\gamma)(\hbar\cP) + \bar{o}(\varepsilon)$ of Poisson structures (see \cite{f16,tetra16,OrMorphism} and \cite{MK94,MKZurichICM,Ascona96}).

For a long time, the third and fourth order expansion of Kontsevich's $\star$-product was unknown to the physics community, which may have delayed the implementation of deformation quantization in the study of Nature.
No model of that theory could be tested approximately because it could not be known what any such model actually was.
This is why we present the formula $\star$ mod $\bar{o}(\hbar^4)$ in Eq.~\eqref{EqStarWith10Pym} on pp.~\pageref{EqStarWith10PymStart}--\pageref{EqStarWith10PymEnd} in this paper.

\label{SecDiscussionEnd}

\newpage

\section*{Conclusion}
\label{SecConcl}

\noindent%
The expansion of Kontsevich's star\/-\/product modulo~$\bar{o}(\hbar^4)$
is (here $f,g\in C^\infty(N^n)$)
{\small
\begin{multline*}
f \star g = f \times g 
+\hbar
\cP^{ij} \partial_{i} f \partial_{j} g 
+\hbar^{2}\big(
\tfrac{1}{2} \cP^{ij} \cP^{k\ell} \partial_{k} \partial_{i} f \partial_{\ell} \partial_{j} g 
+\tfrac{1}{3} \partial_{\ell} \cP^{ij} \cP^{k\ell} \partial_{k} \partial_{i} f \partial_{j} g \\
-\tfrac{1}{3} \partial_{\ell} \cP^{ij} \cP^{k\ell} \partial_{i} f \partial_{k} \partial_{j} g 
-\tfrac{1}{6} \partial_{\ell} \cP^{ij} \partial_{j} \cP^{k\ell} \partial_{i} f \partial_{k} g 
\big)
+\hbar^{3}\big(
\tfrac{1}{6} \cP^{ij} \cP^{k\ell} \cP^{mn} \partial_{m} \partial_{k} \partial_{i} f \partial_{n} \partial_{\ell} \partial_{j} g \\
+\tfrac{1}{3} \partial_{n} \cP^{ij} \cP^{k\ell} \cP^{mn} \partial_{m} \partial_{k} \partial_{i} f \partial_{\ell} \partial_{j} g 
-\tfrac{1}{3} \partial_{n} \cP^{ij} \cP^{k\ell} \cP^{mn} \partial_{k} \partial_{i} f \partial_{m} \partial_{\ell} \partial_{j} g \\
-\tfrac{1}{6} \cP^{ij} \partial_{n} \cP^{k\ell} \partial_{\ell} \cP^{mn} \partial_{k} \partial_{i} f \partial_{m} \partial_{j} g 
+\tfrac{1}{6} \partial_{n} \partial_{\ell} \cP^{ij} \cP^{k\ell} \cP^{mn} \partial_{m} \partial_{k} \partial_{i} f \partial_{j} g \\
+\tfrac{1}{6} \partial_{n} \partial_{\ell} \cP^{ij} \cP^{k\ell} \cP^{mn} \partial_{i} f \partial_{m} \partial_{k} \partial_{j} g 
-\tfrac{1}{6} \partial_{m} \partial_{\ell} \cP^{ij} \partial_{n} \cP^{k\ell} \cP^{mn} \partial_{k} \partial_{i} f \partial_{j} g \\
-\tfrac{1}{6} \partial_{m} \partial_{\ell} \cP^{ij} \partial_{n} \cP^{k\ell} \cP^{mn} \partial_{i} f \partial_{k} \partial_{j} g 
-\tfrac{1}{6} \partial_{n} \cP^{ij} \cP^{k\ell} \partial_{\ell} \cP^{mn} \partial_{k} \partial_{i} f \partial_{m} \partial_{j} g \\
-\tfrac{1}{6} \partial_{\ell} \cP^{ij} \partial_{n} \cP^{k\ell} \cP^{mn} \partial_{k} \partial_{i} f \partial_{m} \partial_{j} g 
+\tfrac{1}{6} \partial_{n} \partial_{\ell} \cP^{ij} \partial_{j} \cP^{k\ell} \cP^{mn} \partial_{i} f \partial_{m} \partial_{k} g \\
-\tfrac{1}{6} \partial_{\ell} \cP^{ij} \partial_{n} \partial_{j} \cP^{k\ell} \cP^{mn} \partial_{m} \partial_{i} f \partial_{k} g 
-\tfrac{1}{6} \partial_{m} \partial_{\ell} \cP^{ij} \partial_{n} \partial_{j} \cP^{k\ell} \cP^{mn} \partial_{i} f \partial_{k} g 
\big)\\
+\hbar^{4}\big(
\tfrac{1}{6} \partial_{q} \cP^{ij} \cP^{k\ell} \cP^{mn} \cP^{pq} \partial_{p} \partial_{m} \partial_{k} \partial_{i} f \partial_{n} \partial_{\ell} \partial_{j} g 
-\tfrac{1}{6} \partial_{q} \cP^{ij} \cP^{k\ell} \cP^{mn} \cP^{pq} \partial_{m} \partial_{k} \partial_{i} f \partial_{p} \partial_{n} \partial_{\ell} \partial_{j} g \\
+\tfrac{1}{6} \partial_{q} \partial_{n} \cP^{ij} \cP^{k\ell} \cP^{mn} \cP^{pq} \partial_{p} \partial_{m} \partial_{k} \partial_{i} f \partial_{\ell} \partial_{j} g 
+\tfrac{1}{6} \partial_{q} \partial_{n} \cP^{ij} \cP^{k\ell} \cP^{mn} \cP^{pq} \partial_{k} \partial_{i} f \partial_{p} \partial_{m} \partial_{\ell} \partial_{j} g \\
-\tfrac{1}{6} \partial_{p} \partial_{n} \cP^{ij} \cP^{k\ell} \partial_{q} \cP^{mn} \cP^{pq} \partial_{m} \partial_{k} \partial_{i} f \partial_{\ell} \partial_{j} g 
-\tfrac{1}{6} \partial_{p} \partial_{n} \cP^{ij} \cP^{k\ell} \partial_{q} \cP^{mn} \cP^{pq} \partial_{k} \partial_{i} f \partial_{m} \partial_{\ell} \partial_{j} g \\
-\tfrac{1}{6} \partial_{q} \cP^{ij} \cP^{k\ell} \cP^{mn} \partial_{n} \cP^{pq} \partial_{m} \partial_{k} \partial_{i} f \partial_{p} \partial_{\ell} \partial_{j} g 
-\tfrac{1}{6} \partial_{n} \cP^{ij} \cP^{k\ell} \partial_{q} \cP^{mn} \cP^{pq} \partial_{m} \partial_{k} \partial_{i} f \partial_{p} \partial_{\ell} \partial_{j} g \\
+\tfrac{1}{6} \cP^{ij} \partial_{q} \partial_{n} \cP^{k\ell} \partial_{\ell} \cP^{mn} \cP^{pq} \partial_{k} \partial_{i} f \partial_{p} \partial_{m} \partial_{j} g 
-\tfrac{1}{6} \cP^{ij} \partial_{n} \cP^{k\ell} \partial_{q} \partial_{\ell} \cP^{mn} \cP^{pq} \partial_{p} \partial_{k} \partial_{i} f \partial_{m} \partial_{j} g \\
-\tfrac{1}{6} \cP^{ij} \partial_{p} \partial_{n} \cP^{k\ell} \partial_{q} \partial_{\ell} \cP^{mn} \cP^{pq} \partial_{k} \partial_{i} f \partial_{m} \partial_{j} g 
-\tfrac{1}{9} \partial_{n} \cP^{ij} \partial_{q} \cP^{k\ell} \cP^{mn} \cP^{pq} \partial_{m} \partial_{k} \partial_{i} f \partial_{p} \partial_{\ell} \partial_{j} g \\
-\tfrac{1}{9} \partial_{p} \partial_{n} \cP^{ij} \partial_{q} \cP^{k\ell} \cP^{mn} \cP^{pq} \partial_{m} \partial_{k} \partial_{i} f \partial_{\ell} \partial_{j} g 
-\tfrac{1}{9} \partial_{p} \partial_{n} \cP^{ij} \partial_{q} \cP^{k\ell} \cP^{mn} \cP^{pq} \partial_{k} \partial_{i} f \partial_{m} \partial_{\ell} \partial_{j} g \\
+\tfrac{1}{24} \cP^{ij} \cP^{k\ell} \cP^{mn} \cP^{pq} \partial_{p} \partial_{m} \partial_{k} \partial_{i} f \partial_{q} \partial_{n} \partial_{\ell} \partial_{j} g 
-\tfrac{1}{12} \cP^{ij} \cP^{k\ell} \partial_{q} \cP^{mn} \partial_{n} \cP^{pq} \partial_{m} \partial_{k} \partial_{i} f \partial_{p} \partial_{\ell} \partial_{j} g \\
+\tfrac{1}{18} \partial_{n} \cP^{ij} \partial_{q} \cP^{k\ell} \cP^{mn} \cP^{pq} \partial_{p} \partial_{m} \partial_{k} \partial_{i} f \partial_{\ell} \partial_{j} g 
-\tfrac{1}{18} \partial_{\ell} \cP^{ij} \cP^{k\ell} \partial_{q} \cP^{mn} \partial_{n} \cP^{pq} \partial_{m} \partial_{k} \partial_{i} f \partial_{p} \partial_{j} g \\
+\tfrac{1}{18} \partial_{n} \cP^{ij} \partial_{q} \cP^{k\ell} \cP^{mn} \cP^{pq} \partial_{k} \partial_{i} f \partial_{p} \partial_{m} \partial_{\ell} \partial_{j} g 
+\tfrac{1}{18} \partial_{q} \cP^{ij} \partial_{n} \cP^{k\ell} \partial_{\ell} \cP^{mn} \cP^{pq} \partial_{k} \partial_{i} f \partial_{p} \partial_{m} \partial_{j} g \\
+\tfrac{1}{72} \partial_{\ell} \cP^{ij} \partial_{j} \cP^{k\ell} \partial_{q} \cP^{mn} \partial_{n} \cP^{pq} \partial_{m} \partial_{i} f \partial_{p} \partial_{k} g 
-\tfrac{1}{18} \partial_{n} \cP^{ij} \partial_{p} \cP^{k\ell} \partial_{q} \cP^{mn} \cP^{pq} \partial_{m} \partial_{k} \partial_{i} f \partial_{\ell} \partial_{j} g \\
-\tfrac{1}{18} \partial_{n} \cP^{ij} \partial_{p} \cP^{k\ell} \partial_{q} \cP^{mn} \cP^{pq} \partial_{k} \partial_{i} f \partial_{m} \partial_{\ell} \partial_{j} g 
+\tfrac{1}{30} \partial_{q} \partial_{n} \partial_{\ell} \cP^{ij} \cP^{k\ell} \cP^{mn} \cP^{pq} \partial_{p} \partial_{m} \partial_{k} \partial_{i} f \partial_{j} g \\
-\tfrac{1}{30} \partial_{q} \partial_{n} \partial_{\ell} \cP^{ij} \cP^{k\ell} \cP^{mn} \cP^{pq} \partial_{i} f \partial_{p} \partial_{m} \partial_{k} \partial_{j} g 
+\tfrac{2}{45} \partial_{n} \partial_{\ell} \cP^{ij} \partial_{q} \cP^{k\ell} \cP^{mn} \cP^{pq} \partial_{p} \partial_{m} \partial_{k} \partial_{i} f \partial_{j} g \\
-\tfrac{2}{45} \partial_{n} \partial_{\ell} \cP^{ij} \partial_{q} \cP^{k\ell} \cP^{mn} \cP^{pq} \partial_{i} f \partial_{p} \partial_{m} \partial_{k} \partial_{j} g 
-\tfrac{1}{30} \partial_{q} \partial_{n} \partial_{\ell} \cP^{ij} \cP^{k\ell} \cP^{mn} \cP^{pq} \partial_{m} \partial_{k} \partial_{i} f \partial_{p} \partial_{j} g \\
+\tfrac{1}{30} \partial_{q} \partial_{n} \partial_{\ell} \cP^{ij} \cP^{k\ell} \cP^{mn} \cP^{pq} \partial_{k} \partial_{i} f \partial_{p} \partial_{m} \partial_{j} g 
-\tfrac{7}{90} \partial_{p} \partial_{n} \partial_{\ell} \cP^{ij} \partial_{q} \cP^{k\ell} \cP^{mn} \cP^{pq} \partial_{m} \partial_{k} \partial_{i} f \partial_{j} g \\
+\tfrac{7}{90} \partial_{p} \partial_{n} \partial_{\ell} \cP^{ij} \partial_{q} \cP^{k\ell} \cP^{mn} \cP^{pq} \partial_{i} f \partial_{m} \partial_{k} \partial_{j} g 
+\tfrac{1}{30} \partial_{\ell} \cP^{ij} \partial_{q} \partial_{n} \cP^{k\ell} \cP^{mn} \cP^{pq} \partial_{p} \partial_{m} \partial_{k} \partial_{i} f \partial_{j} g \\
-\tfrac{1}{30} \partial_{\ell} \cP^{ij} \partial_{q} \partial_{n} \cP^{k\ell} \cP^{mn} \cP^{pq} \partial_{i} f \partial_{p} \partial_{m} \partial_{k} \partial_{j} g 
-\tfrac{1}{45} \partial_{\ell} \cP^{ij} \partial_{n} \cP^{k\ell} \partial_{q} \cP^{mn} \cP^{pq} \partial_{p} \partial_{m} \partial_{k} \partial_{i} f \partial_{j} g \\
+\tfrac{1}{45} \partial_{\ell} \cP^{ij} \partial_{n} \cP^{k\ell} \partial_{q} \cP^{mn} \cP^{pq} \partial_{i} f \partial_{p} \partial_{m} \partial_{k} \partial_{j} g 
+\tfrac{1}{45} \partial_{q} \partial_{\ell} \cP^{ij} \partial_{n} \cP^{k\ell} \cP^{mn} \cP^{pq} \partial_{m} \partial_{k} \partial_{i} f \partial_{p} \partial_{j} g \\
-\tfrac{1}{45} \partial_{n} \partial_{\ell} \cP^{ij} \cP^{k\ell} \partial_{q} \cP^{mn} \cP^{pq} \partial_{k} \partial_{i} f \partial_{p} \partial_{m} \partial_{j} g 
+\tfrac{1}{30} \partial_{\ell} \cP^{ij} \partial_{q} \partial_{n} \cP^{k\ell} \cP^{mn} \cP^{pq} \partial_{m} \partial_{k} \partial_{i} f \partial_{p} \partial_{j} g \\
-\tfrac{1}{30} \partial_{n} \cP^{ij} \cP^{k\ell} \partial_{q} \partial_{\ell} \cP^{mn} \cP^{pq} \partial_{k} \partial_{i} f \partial_{p} \partial_{m} \partial_{j} g 
-\tfrac{1}{90} \partial_{\ell} \cP^{ij} \partial_{n} \cP^{k\ell} \partial_{q} \cP^{mn} \cP^{pq} \partial_{m} \partial_{k} \partial_{i} f \partial_{p} \partial_{j} g \\
+\tfrac{1}{90} \partial_{q} \cP^{ij} \cP^{k\ell} \partial_{\ell} \cP^{mn} \partial_{n} \cP^{pq} \partial_{k} \partial_{i} f \partial_{p} \partial_{m} \partial_{j} g 
+\tfrac{1}{90} \partial_{p} \partial_{\ell} \cP^{ij} \partial_{q} \partial_{n} \cP^{k\ell} \cP^{mn} \cP^{pq} \partial_{m} \partial_{k} \partial_{i} f \partial_{j} g \\
-\tfrac{1}{90} \partial_{p} \partial_{\ell} \cP^{ij} \partial_{q} \partial_{n} \cP^{k\ell} \cP^{mn} \cP^{pq} \partial_{i} f \partial_{m} \partial_{k} \partial_{j} g 
-\tfrac{1}{30} \partial_{p} \partial_{\ell} \cP^{ij} \partial_{n} \cP^{k\ell} \partial_{q} \cP^{mn} \cP^{pq} \partial_{m} \partial_{k} \partial_{i} f \partial_{j} g \\
+\tfrac{1}{30} \partial_{p} \partial_{\ell} \cP^{ij} \partial_{n} \cP^{k\ell} \partial_{q} \cP^{mn} \cP^{pq} \partial_{i} f \partial_{m} \partial_{k} \partial_{j} g 
-\tfrac{2}{45} \partial_{\ell} \cP^{ij} \partial_{p} \partial_{n} \cP^{k\ell} \partial_{q} \cP^{mn} \cP^{pq} \partial_{m} \partial_{k} \partial_{i} f \partial_{j} g
\end{multline*}
\begin{multline*}
+\tfrac{2}{45} \partial_{\ell} \cP^{ij} \partial_{p} \partial_{n} \cP^{k\ell} \partial_{q} \cP^{mn} \cP^{pq} \partial_{i} f \partial_{m} \partial_{k} \partial_{j} g 
-\tfrac{1}{30} \partial_{q} \partial_{\ell} \cP^{ij} \cP^{k\ell} \cP^{mn} \partial_{n} \cP^{pq} \partial_{m} \partial_{k} \partial_{i} f \partial_{p} \partial_{j} g \\
+\tfrac{1}{30} \partial_{n} \partial_{\ell} \cP^{ij} \partial_{q} \cP^{k\ell} \cP^{mn} \cP^{pq} \partial_{k} \partial_{i} f \partial_{p} \partial_{m} \partial_{j} g 
-\tfrac{1}{15} \partial_{\ell} \cP^{ij} \partial_{q} \cP^{k\ell} \cP^{mn} \partial_{n} \cP^{pq} \partial_{m} \partial_{k} \partial_{i} f \partial_{p} \partial_{j} g \\
+\tfrac{1}{15} \partial_{n} \cP^{ij} \partial_{q} \cP^{k\ell} \partial_{\ell} \cP^{mn} \cP^{pq} \partial_{k} \partial_{i} f \partial_{p} \partial_{m} \partial_{j} g 
-\tfrac{1}{30} \partial_{p} \partial_{\ell} \cP^{ij} \partial_{q} \cP^{k\ell} \cP^{mn} \partial_{n} \cP^{pq} \partial_{m} \partial_{k} \partial_{i} f \partial_{j} g \\
+\tfrac{1}{30} \partial_{p} \partial_{\ell} \cP^{ij} \partial_{q} \cP^{k\ell} \cP^{mn} \partial_{n} \cP^{pq} \partial_{i} f \partial_{m} \partial_{k} \partial_{j} g 
-\tfrac{1}{45} \partial_{p} \partial_{\ell} \cP^{ij} \cP^{k\ell} \partial_{q} \cP^{mn} \partial_{n} \cP^{pq} \partial_{m} \partial_{k} \partial_{i} f \partial_{j} g \\
+\tfrac{1}{45} \partial_{p} \partial_{\ell} \cP^{ij} \cP^{k\ell} \partial_{q} \cP^{mn} \partial_{n} \cP^{pq} \partial_{i} f \partial_{m} \partial_{k} \partial_{j} g 
+\tfrac{1}{90} \partial_{\ell} \cP^{ij} \partial_{p} \cP^{k\ell} \partial_{q} \cP^{mn} \partial_{n} \cP^{pq} \partial_{m} \partial_{k} \partial_{i} f \partial_{j} g \\
-\tfrac{1}{90} \partial_{\ell} \cP^{ij} \partial_{p} \cP^{k\ell} \partial_{q} \cP^{mn} \partial_{n} \cP^{pq} \partial_{i} f \partial_{m} \partial_{k} \partial_{j} g 
+\tfrac{2}{45} \partial_{p} \partial_{n} \partial_{\ell} \cP^{ij} \partial_{q} \cP^{k\ell} \cP^{mn} \cP^{pq} \partial_{k} \partial_{i} f \partial_{m} \partial_{j} g \\
-\tfrac{2}{45} \partial_{p} \partial_{n} \partial_{\ell} \cP^{ij} \cP^{k\ell} \partial_{q} \cP^{mn} \cP^{pq} \partial_{k} \partial_{i} f \partial_{m} \partial_{j} g 
+\tfrac{1}{15} \partial_{\ell} \cP^{ij} \partial_{q} \partial_{n} \cP^{k\ell} \cP^{mn} \cP^{pq} \partial_{k} \partial_{i} f \partial_{p} \partial_{m} \partial_{j} g \\
-\tfrac{1}{15} \partial_{n} \cP^{ij} \cP^{k\ell} \partial_{q} \partial_{\ell} \cP^{mn} \cP^{pq} \partial_{p} \partial_{k} \partial_{i} f \partial_{m} \partial_{j} g 
+\tfrac{1}{90} \partial_{\ell} \cP^{ij} \partial_{n} \cP^{k\ell} \partial_{q} \cP^{mn} \cP^{pq} \partial_{k} \partial_{i} f \partial_{p} \partial_{m} \partial_{j} g \\
-\tfrac{1}{90} \partial_{q} \cP^{ij} \cP^{k\ell} \partial_{\ell} \cP^{mn} \partial_{n} \cP^{pq} \partial_{m} \partial_{k} \partial_{i} f \partial_{p} \partial_{j} g 
+\tfrac{1}{90} \partial_{p} \partial_{\ell} \cP^{ij} \cP^{k\ell} \partial_{q} \cP^{mn} \partial_{n} \cP^{pq} \partial_{k} \partial_{i} f \partial_{m} \partial_{j} g \\
-\tfrac{1}{90} \partial_{q} \partial_{m} \cP^{ij} \partial_{n} \cP^{k\ell} \partial_{\ell} \cP^{mn} \cP^{pq} \partial_{k} \partial_{i} f \partial_{p} \partial_{j} g 
+\tfrac{1}{90} \partial_{p} \cP^{ij} \partial_{q} \cP^{k\ell} \partial_{\ell} \cP^{mn} \partial_{n} \cP^{pq} \partial_{m} \partial_{k} \partial_{i} f \partial_{j} g \\
-\tfrac{1}{90} \partial_{p} \cP^{ij} \partial_{q} \cP^{k\ell} \partial_{\ell} \cP^{mn} \partial_{n} \cP^{pq} \partial_{i} f \partial_{m} \partial_{k} \partial_{j} g 
+\tfrac{1}{90} \partial_{p} \cP^{ij} \cP^{k\ell} \partial_{q} \partial_{\ell} \cP^{mn} \partial_{n} \cP^{pq} \partial_{m} \partial_{k} \partial_{i} f \partial_{j} g \\
-\tfrac{1}{90} \partial_{p} \cP^{ij} \cP^{k\ell} \partial_{q} \partial_{\ell} \cP^{mn} \partial_{n} \cP^{pq} \partial_{i} f \partial_{m} \partial_{k} \partial_{j} g 
-\tfrac{1}{30} \partial_{m} \cP^{ij} \partial_{n} \cP^{k\ell} \partial_{q} \partial_{\ell} \cP^{mn} \cP^{pq} \partial_{p} \partial_{k} \partial_{i} f \partial_{j} g \\
+\tfrac{1}{30} \partial_{m} \cP^{ij} \partial_{n} \cP^{k\ell} \partial_{q} \partial_{\ell} \cP^{mn} \cP^{pq} \partial_{i} f \partial_{p} \partial_{k} \partial_{j} g 
+\tfrac{1}{90} \partial_{q} \cP^{ij} \partial_{j} \cP^{k\ell} \partial_{\ell} \cP^{mn} \partial_{n} \cP^{pq} \partial_{m} \partial_{k} \partial_{i} f \partial_{p} g \\
+\tfrac{1}{90} \partial_{q} \cP^{ij} \partial_{j} \cP^{k\ell} \partial_{\ell} \cP^{mn} \partial_{n} \cP^{pq} \partial_{i} f \partial_{p} \partial_{m} \partial_{k} g 
+\tfrac{1}{90} \cP^{ij} \partial_{q} \partial_{j} \cP^{k\ell} \partial_{\ell} \cP^{mn} \partial_{n} \cP^{pq} \partial_{m} \partial_{k} \partial_{i} f \partial_{p} g \\
+\tfrac{1}{90} \partial_{n} \cP^{ij} \partial_{q} \partial_{j} \cP^{k\ell} \partial_{\ell} \cP^{mn} \cP^{pq} \partial_{i} f \partial_{p} \partial_{m} \partial_{k} g 
+\tfrac{1}{90} \cP^{ij} \partial_{j} \cP^{k\ell} \partial_{q} \partial_{\ell} \cP^{mn} \partial_{n} \cP^{pq} \partial_{m} \partial_{k} \partial_{i} f \partial_{p} g \\
+\tfrac{1}{90} \partial_{\ell} \cP^{ij} \partial_{n} \partial_{j} \cP^{k\ell} \partial_{q} \cP^{mn} \cP^{pq} \partial_{i} f \partial_{p} \partial_{m} \partial_{k} g 
-\tfrac{1}{30} \partial_{n} \cP^{ij} \partial_{q} \partial_{j} \cP^{k\ell} \partial_{\ell} \cP^{mn} \cP^{pq} \partial_{p} \partial_{k} \partial_{i} f \partial_{m} g \\
-\tfrac{1}{30} \partial_{n} \cP^{ij} \partial_{j} \cP^{k\ell} \partial_{q} \partial_{\ell} \cP^{mn} \cP^{pq} \partial_{i} f \partial_{p} \partial_{m} \partial_{k} g 
-\tfrac{1}{45} \partial_{n} \cP^{ij} \partial_{j} \cP^{k\ell} \partial_{q} \partial_{\ell} \cP^{mn} \cP^{pq} \partial_{p} \partial_{k} \partial_{i} f \partial_{m} g \\
-\tfrac{1}{45} \partial_{q} \partial_{n} \cP^{ij} \partial_{j} \cP^{k\ell} \partial_{\ell} \cP^{mn} \cP^{pq} \partial_{i} f \partial_{p} \partial_{m} \partial_{k} g 
-\tfrac{1}{90} \partial_{n} \cP^{ij} \partial_{j} \cP^{k\ell} \partial_{q} \partial_{\ell} \cP^{mn} \cP^{pq} \partial_{k} \partial_{i} f \partial_{p} \partial_{m} g \\
-\tfrac{1}{90} \cP^{ij} \partial_{q} \partial_{j} \cP^{k\ell} \partial_{\ell} \cP^{mn} \partial_{n} \cP^{pq} \partial_{k} \partial_{i} f \partial_{p} \partial_{m} g 
-\tfrac{1}{60} \cP^{ij} \partial_{q} \partial_{n} \partial_{j} \cP^{k\ell} \partial_{\ell} \cP^{mn} \cP^{pq} \partial_{p} \partial_{k} \partial_{i} f \partial_{m} g \\
-\tfrac{1}{60} \partial_{\ell} \cP^{ij} \partial_{q} \partial_{n} \partial_{j} \cP^{k\ell} \cP^{mn} \cP^{pq} \partial_{i} f \partial_{p} \partial_{m} \partial_{k} g 
-\tfrac{1}{45} \cP^{ij} \partial_{n} \partial_{j} \cP^{k\ell} \partial_{q} \partial_{\ell} \cP^{mn} \cP^{pq} \partial_{p} \partial_{k} \partial_{i} f \partial_{m} g \\
-\tfrac{1}{45} \partial_{n} \partial_{\ell} \cP^{ij} \partial_{q} \partial_{j} \cP^{k\ell} \cP^{mn} \cP^{pq} \partial_{i} f \partial_{p} \partial_{m} \partial_{k} g 
+\tfrac{1}{30} \cP^{ij} \partial_{q} \partial_{n} \partial_{j} \cP^{k\ell} \partial_{\ell} \cP^{mn} \cP^{pq} \partial_{k} \partial_{i} f \partial_{p} \partial_{m} g \\
+\tfrac{1}{30} \partial_{\ell} \cP^{ij} \partial_{q} \partial_{n} \partial_{j} \cP^{k\ell} \cP^{mn} \cP^{pq} \partial_{m} \partial_{i} f \partial_{p} \partial_{k} g 
-\tfrac{1}{90} \cP^{ij} \partial_{n} \partial_{j} \cP^{k\ell} \partial_{q} \partial_{\ell} \cP^{mn} \cP^{pq} \partial_{k} \partial_{i} f \partial_{p} \partial_{m} g \\
-\tfrac{1}{45} \cP^{ij} \partial_{j} \cP^{k\ell} \partial_{q} \partial_{\ell} \cP^{mn} \partial_{n} \cP^{pq} \partial_{p} \partial_{k} \partial_{i} f \partial_{m} g 
-\tfrac{1}{45} \partial_{n} \partial_{\ell} \cP^{ij} \partial_{j} \cP^{k\ell} \partial_{q} \cP^{mn} \cP^{pq} \partial_{i} f \partial_{p} \partial_{m} \partial_{k} g \\
-\tfrac{1}{20} \partial_{\ell} \cP^{ij} \partial_{q} \partial_{n} \partial_{j} \cP^{k\ell} \cP^{mn} \cP^{pq} \partial_{p} \partial_{m} \partial_{i} f \partial_{k} g 
-\tfrac{1}{20} \partial_{q} \partial_{n} \partial_{\ell} \cP^{ij} \partial_{j} \cP^{k\ell} \cP^{mn} \cP^{pq} \partial_{i} f \partial_{p} \partial_{m} \partial_{k} g \\
-\tfrac{13}{90} \partial_{n} \partial_{\ell} \cP^{ij} \partial_{q} \cP^{k\ell} \cP^{mn} \cP^{pq} \partial_{m} \partial_{k} \partial_{i} f \partial_{p} \partial_{j} g 
+\tfrac{13}{90} \partial_{q} \partial_{n} \cP^{ij} \cP^{k\ell} \partial_{\ell} \cP^{mn} \cP^{pq} \partial_{k} \partial_{i} f \partial_{p} \partial_{m} \partial_{j} g \\
+\tfrac{13}{90} \partial_{q} \partial_{\ell} \cP^{ij} \partial_{n} \partial_{j} \cP^{k\ell} \cP^{mn} \cP^{pq} \partial_{m} \partial_{i} f \partial_{p} \partial_{k} g 
-16  p_4 \partial_{p} \partial_{m} \partial_{\ell} \cP^{ij} \partial_{n} \cP^{k\ell} \partial_{q} \cP^{mn} \cP^{pq} \partial_{k} \partial_{i} f \partial_{j} g \\
+16  p_4 \partial_{p} \partial_{m} \partial_{\ell} \cP^{ij} \partial_{n} \cP^{k\ell} \partial_{q} \cP^{mn} \cP^{pq} \partial_{i} f \partial_{k} \partial_{j} g 
-16  p_5 \partial_{p} \partial_{m} \cP^{ij} \partial_{n} \cP^{k\ell} \partial_{q} \partial_{\ell} \cP^{mn} \cP^{pq} \partial_{k} \partial_{i} f \partial_{j} g \\
+16  p_5 \partial_{p} \partial_{m} \cP^{ij} \partial_{n} \cP^{k\ell} \partial_{q} \partial_{\ell} \cP^{mn} \cP^{pq} \partial_{i} f \partial_{k} \partial_{j} g 
-16  p_4 \partial_{p} \partial_{m} \cP^{ij} \partial_{q} \cP^{k\ell} \partial_{\ell} \cP^{mn} \partial_{n} \cP^{pq} \partial_{k} \partial_{i} f \partial_{j} g \\
+16  p_4 \partial_{p} \partial_{m} \cP^{ij} \partial_{q} \cP^{k\ell} \partial_{\ell} \cP^{mn} \partial_{n} \cP^{pq} \partial_{i} f \partial_{k} \partial_{j} g 
+16  p_4 \partial_{m} \cP^{ij} \partial_{p} \cP^{k\ell} \partial_{q} \partial_{\ell} \cP^{mn} \partial_{n} \cP^{pq} \partial_{k} \partial_{i} f \partial_{j} g \\
-16  p_4 \partial_{m} \cP^{ij} \partial_{p} \cP^{k\ell} \partial_{q} \partial_{\ell} \cP^{mn} \partial_{n} \cP^{pq} \partial_{i} f \partial_{k} \partial_{j} g 
+16  p_5 \partial_{p} \cP^{ij} \partial_{m} \cP^{k\ell} \partial_{q} \partial_{\ell} \cP^{mn} \partial_{n} \cP^{pq} \partial_{k} \partial_{i} f \partial_{j} g \\
-16  p_5 \partial_{p} \cP^{ij} \partial_{m} \cP^{k\ell} \partial_{q} \partial_{\ell} \cP^{mn} \partial_{n} \cP^{pq} \partial_{i} f \partial_{k} \partial_{j} g 
+16  p_1 \partial_{m} \cP^{ij} \partial_{q} \partial_{j} \cP^{k\ell} \partial_{\ell} \cP^{mn} \partial_{n} \cP^{pq} \partial_{k} \partial_{i} f \partial_{p} g \\
-16  p_1 \partial_{m} \cP^{ij} \partial_{q} \partial_{j} \cP^{k\ell} \partial_{\ell} \cP^{mn} \partial_{n} \cP^{pq} \partial_{i} f \partial_{p} \partial_{k} g 
+16  p_2 \partial_{m} \cP^{ij} \partial_{j} \cP^{k\ell} \partial_{q} \partial_{\ell} \cP^{mn} \partial_{n} \cP^{pq} \partial_{k} \partial_{i} f \partial_{p} g \\
-16  p_2 \partial_{k} \cP^{ij} \partial_{q} \partial_{j} \cP^{k\ell} \partial_{\ell} \cP^{mn} \partial_{n} \cP^{pq} \partial_{i} f \partial_{p} \partial_{m} g 
+16  p_3 \partial_{q} \cP^{ij} \partial_{m} \partial_{j} \cP^{k\ell} \partial_{\ell} \cP^{mn} \partial_{n} \cP^{pq} \partial_{k} \partial_{i} f \partial_{p} g \\
\end{multline*}
\begin{multline*}
-16  p_3 \partial_{p} \cP^{ij} \partial_{j} \cP^{k\ell} \partial_{q} \partial_{\ell} \cP^{mn} \partial_{n} \cP^{pq} \partial_{i} f \partial_{m} \partial_{k} g 
+16  p_4 \cP^{ij} \partial_{q} \partial_{m} \partial_{j} \cP^{k\ell} \partial_{\ell} \cP^{mn} \partial_{n} \cP^{pq} \partial_{k} \partial_{i} f \partial_{p} g \\
-16  p_4 \partial_{m} \cP^{ij} \partial_{q} \partial_{n} \partial_{j} \cP^{k\ell} \partial_{\ell} \cP^{mn} \cP^{pq} \partial_{i} f \partial_{p} \partial_{k} g 
+16  p_5 \cP^{ij} \partial_{m} \partial_{j} \cP^{k\ell} \partial_{q} \partial_{\ell} \cP^{mn} \partial_{n} \cP^{pq} \partial_{k} \partial_{i} f \partial_{p} g \\
-16  p_5 \partial_{k} \cP^{ij} \partial_{n} \partial_{j} \cP^{k\ell} \partial_{q} \partial_{\ell} \cP^{mn} \cP^{pq} \partial_{i} f \partial_{p} \partial_{m} g 
+16  p_6 \partial_{\ell} \cP^{ij} \partial_{n} \partial_{j} \cP^{k\ell} \partial_{q} \cP^{mn} \cP^{pq} \partial_{m} \partial_{i} f \partial_{p} \partial_{k} g \\
+16  p_6 \partial_{q} \partial_{\ell} \cP^{ij} \partial_{j} \cP^{k\ell} \cP^{mn} \partial_{n} \cP^{pq} \partial_{m} \partial_{i} f \partial_{p} \partial_{k} g 
+16  p_7 \partial_{p} \partial_{\ell} \cP^{ij} \partial_{q} \partial_{n} \partial_{j} \cP^{k\ell} \cP^{mn} \cP^{pq} \partial_{m} \partial_{i} f \partial_{k} g \\
-16  p_7 \partial_{p} \partial_{n} \partial_{\ell} \cP^{ij} \partial_{q} \partial_{j} \cP^{k\ell} \cP^{mn} \cP^{pq} \partial_{i} f \partial_{m} \partial_{k} g 
+16  p_8 \partial_{p} \partial_{\ell} \cP^{ij} \partial_{n} \partial_{j} \cP^{k\ell} \partial_{q} \cP^{mn} \cP^{pq} \partial_{m} \partial_{i} f \partial_{k} g \\
+16  p_8 \partial_{n} \partial_{\ell} \cP^{ij} \partial_{p} \partial_{j} \cP^{k\ell} \partial_{q} \cP^{mn} \cP^{pq} \partial_{i} f \partial_{m} \partial_{k} g 
+16  p_9 \partial_{p} \cP^{ij} \partial_{q} \partial_{j} \cP^{k\ell} \partial_{\ell} \cP^{mn} \partial_{n} \cP^{pq} \partial_{m} \partial_{i} f \partial_{k} g \\
-16  p_9 \partial_{q} \partial_{m} \cP^{ij} \partial_{j} \cP^{k\ell} \partial_{\ell} \cP^{mn} \partial_{n} \cP^{pq} \partial_{i} f \partial_{p} \partial_{k} g 
-32  p_4 \partial_{m} \cP^{ij} \partial_{q} \partial_{n} \partial_{j} \cP^{k\ell} \partial_{\ell} \cP^{mn} \cP^{pq} \partial_{p} \partial_{i} f \partial_{k} g \\
+32  p_4 \partial_{q} \partial_{n} \partial_{k} \cP^{ij} \partial_{j} \cP^{k\ell} \partial_{\ell} \cP^{mn} \cP^{pq} \partial_{i} f \partial_{p} \partial_{m} g 
+16  p_{10} \partial_{p} \partial_{m} \cP^{ij} \partial_{q} \partial_{n} \partial_{j} \cP^{k\ell} \partial_{\ell} \cP^{mn} \cP^{pq} \partial_{i} f \partial_{k} g \\
+16  p_{10} \partial_{p} \partial_{n} \partial_{k} \cP^{ij} \partial_{j} \cP^{k\ell} \partial_{q} \partial_{\ell} \cP^{mn} \cP^{pq} \partial_{i} f \partial_{m} g 
+32  p_5 \cP^{ij} \partial_{m} \partial_{j} \cP^{k\ell} \partial_{q} \partial_{\ell} \cP^{mn} \partial_{n} \cP^{pq} \partial_{p} \partial_{i} f \partial_{k} g \\
+32  p_5 \partial_{q} \partial_{k} \cP^{ij} \partial_{n} \partial_{j} \cP^{k\ell} \partial_{\ell} \cP^{mn} \cP^{pq} \partial_{i} f \partial_{p} \partial_{m} g 
+(\tfrac{1}{12}+8 p_7) \partial_{p} \partial_{m} \partial_{\ell} \cP^{ij} \partial_{q} \partial_{n} \cP^{k\ell} \cP^{mn} \cP^{pq} \partial_{k} \partial_{i} f \partial_{j} g \\
+(-\tfrac{1}{12}-8 p_7) \partial_{p} \partial_{m} \partial_{\ell} \cP^{ij} \partial_{q} \partial_{n} \cP^{k\ell} \cP^{mn} \cP^{pq} \partial_{i} f \partial_{k} \partial_{j} g \\
+(-\tfrac{1}{12}-8 p_7) \partial_{m} \cP^{ij} \partial_{p} \partial_{n} \cP^{k\ell} \partial_{q} \partial_{\ell} \cP^{mn} \cP^{pq} \partial_{k} \partial_{i} f \partial_{j} g \\
+(\tfrac{1}{12}+8 p_7) \partial_{m} \cP^{ij} \partial_{p} \partial_{n} \cP^{k\ell} \partial_{q} \partial_{\ell} \cP^{mn} \cP^{pq} \partial_{i} f \partial_{k} \partial_{j} g \\
+(-\tfrac{1}{12}-8 p_7) \cP^{ij} \partial_{p} \partial_{n} \partial_{j} \cP^{k\ell} \partial_{q} \partial_{\ell} \cP^{mn} \cP^{pq} \partial_{k} \partial_{i} f \partial_{m} g \\
+(\tfrac{1}{12}+8 p_7) \partial_{p} \partial_{\ell} \cP^{ij} \partial_{q} \partial_{n} \partial_{j} \cP^{k\ell} \cP^{mn} \cP^{pq} \partial_{i} f \partial_{m} \partial_{k} g \\
+(\tfrac{1}{45}+8 p_7) \partial_{p} \partial_{m} \partial_{\ell} \cP^{ij} \partial_{q} \partial_{n} \partial_{j} \cP^{k\ell} \cP^{mn} \cP^{pq} \partial_{i} f \partial_{k} g \\
+(-\tfrac{1}{60}-24 p_7) \partial_{p} \partial_{\ell} \cP^{ij} \partial_{q} \partial_{n} \cP^{k\ell} \cP^{mn} \cP^{pq} \partial_{k} \partial_{i} f \partial_{m} \partial_{j} g \\
+(\tfrac{1}{60}+24 p_7) \partial_{p} \partial_{n} \cP^{ij} \cP^{k\ell} \partial_{q} \partial_{\ell} \cP^{mn} \cP^{pq} \partial_{k} \partial_{i} f \partial_{m} \partial_{j} g \\
+(\tfrac{4}{45}-16 p_6) \partial_{q} \cP^{ij} \partial_{n} \partial_{j} \cP^{k\ell} \partial_{\ell} \cP^{mn} \cP^{pq} \partial_{k} \partial_{i} f \partial_{p} \partial_{m} g \\
+(\tfrac{4}{45}-16 p_6) \partial_{\ell} \cP^{ij} \partial_{q} \partial_{j} \cP^{k\ell} \cP^{mn} \partial_{n} \cP^{pq} \partial_{m} \partial_{i} f \partial_{p} \partial_{k} g \\
+(\tfrac{17}{90}+24 p_7) \partial_{p} \partial_{m} \cP^{ij} \partial_{q} \partial_{n} \cP^{k\ell} \cP^{mn} \cP^{pq} \partial_{k} \partial_{i} f \partial_{\ell} \partial_{j} g \\
+(-\tfrac{1}{12}+16 p_6+48 p_5) \partial_{m} \cP^{ij} \partial_{p} \cP^{k\ell} \partial_{q} \cP^{mn} \partial_{n} \cP^{pq} \partial_{k} \partial_{i} f \partial_{\ell} \partial_{j} g \\
+(\tfrac{7}{90}-16 p_6-48 p_5) \partial_{p} \cP^{ij} \cP^{k\ell} \partial_{q} \partial_{\ell} \cP^{mn} \partial_{n} \cP^{pq} \partial_{k} \partial_{i} f \partial_{m} \partial_{j} g \\
+(-\tfrac{7}{90}+16 p_6+48 p_5) \partial_{m} \cP^{ij} \partial_{q} \partial_{n} \cP^{k\ell} \partial_{\ell} \cP^{mn} \cP^{pq} \partial_{k} \partial_{i} f \partial_{p} \partial_{j} g \\
+(-\tfrac{1}{36}+16 p_4-16 p_5-8 p_7) \partial_{p} \cP^{ij} \partial_{q} \partial_{m} \cP^{k\ell} \partial_{\ell} \cP^{mn} \partial_{n} \cP^{pq} \partial_{k} \partial_{i} f \partial_{j} g \\
+(\tfrac{1}{36}-16 p_4+16 p_5+8 p_7) \partial_{p} \cP^{ij} \partial_{q} \partial_{m} \cP^{k\ell} \partial_{\ell} \cP^{mn} \partial_{n} \cP^{pq} \partial_{i} f \partial_{k} \partial_{j} g \\
+(-\tfrac{1}{9}+16 p_6-48 p_4+48 p_5) \partial_{n} \cP^{ij} \partial_{q} \partial_{j} \cP^{k\ell} \partial_{\ell} \cP^{mn} \cP^{pq} \partial_{k} \partial_{i} f \partial_{p} \partial_{m} g \\
+(-\tfrac{1}{9}+16 p_6-48 p_4+48 p_5) \cP^{ij} \partial_{q} \partial_{j} \cP^{k\ell} \partial_{\ell} \cP^{mn} \partial_{n} \cP^{pq} \partial_{m} \partial_{i} f \partial_{p} \partial_{k} g \\
+(-\tfrac{1}{36}+16 p_4-16 p_5-8 p_7) \cP^{ij} \partial_{p} \partial_{j} \cP^{k\ell} \partial_{q} \partial_{\ell} \cP^{mn} \partial_{n} \cP^{pq} \partial_{k} \partial_{i} f \partial_{m} g \\
+(\tfrac{1}{36}-16 p_4+16 p_5+8 p_7) \partial_{n} \partial_{k} \cP^{ij} \partial_{j} \cP^{k\ell} \partial_{q} \partial_{\ell} \cP^{mn} \cP^{pq} \partial_{i} f \partial_{p} \partial_{m} g \\
+(-\tfrac{13}{180}+8 p_6+24 p_5-8 p_1) \partial_{m} \partial_{\ell} \cP^{ij} \partial_{p} \partial_{n} \cP^{k\ell} \partial_{q} \cP^{mn} \cP^{pq} \partial_{k} \partial_{i} f \partial_{j} g \\
+(\tfrac{13}{180}-8 p_6-24 p_5+8 p_1) \partial_{m} \partial_{\ell} \cP^{ij} \partial_{p} \partial_{n} \cP^{k\ell} \partial_{q} \cP^{mn} \cP^{pq} \partial_{i} f \partial_{k} \partial_{j} g \\
+(\tfrac{4}{45}-16 p_6+48 p_4-48 p_5) \partial_{q} \partial_{n} \cP^{ij} \partial_{j} \cP^{k\ell} \partial_{\ell} \cP^{mn} \cP^{pq} \partial_{k} \partial_{i} f \partial_{p} \partial_{m} g \\
\end{multline*}
\begin{multline*}
+(\tfrac{4}{45}-16 p_6+48 p_4-48 p_5) \partial_{n} \cP^{ij} \partial_{q} \partial_{j} \cP^{k\ell} \partial_{\ell} \cP^{mn} \cP^{pq} \partial_{p} \partial_{i} f \partial_{m} \partial_{k} g \\
+(-\tfrac{1}{18}+32 p_4-32 p_5-16 p_7 )\cP^{ij} \partial_{p} \partial_{j} \cP^{k\ell} \partial_{q} \partial_{\ell} \cP^{mn} \partial_{n} \cP^{pq} \partial_{m} \partial_{i} f \partial_{k} g \\
+(-\tfrac{1}{18}+32 p_4-32 p_5-16 p_7) \partial_{q} \partial_{m} \cP^{ij} \partial_{n} \partial_{j} \cP^{k\ell} \partial_{\ell} \cP^{mn} \cP^{pq} \partial_{i} f \partial_{p} \partial_{k} g \\
+(\tfrac{17}{90}-32 p_6+96 p_4-96 p_5) \partial_{q} \cP^{ij} \partial_{j} \cP^{k\ell} \partial_{\ell} \cP^{mn} \partial_{n} \cP^{pq} \partial_{k} \partial_{i} f \partial_{p} \partial_{m} g \\
+(\tfrac{23}{360}+16 p_4-8 p_1+12 p_7) \partial_{p} \partial_{m} \partial_{\ell} \cP^{ij} \partial_{n} \partial_{j} \cP^{k\ell} \partial_{q} \cP^{mn} \cP^{pq} \partial_{i} f \partial_{k} g \\
+(-\tfrac{23}{360}-16 p_4+8 p_1-12 p_7) \partial_{m} \partial_{\ell} \cP^{ij} \partial_{p} \partial_{n} \partial_{j} \cP^{k\ell} \partial_{q} \cP^{mn} \cP^{pq} \partial_{i} f \partial_{k} g \\
+(\tfrac{49}{180}-16 p_6-48 p_5+24 p_7) \partial_{p} \partial_{m} \cP^{ij} \partial_{n} \cP^{k\ell} \partial_{q} \cP^{mn} \cP^{pq} \partial_{k} \partial_{i} f \partial_{\ell} \partial_{j} g \\
+(-\tfrac{19}{180}-32 p_4+16 p_1-16 p_7) \partial_{m} \partial_{\ell} \cP^{ij} \partial_{p} \partial_{j} \cP^{k\ell} \partial_{q} \cP^{mn} \partial_{n} \cP^{pq} \partial_{i} f \partial_{k} g \\
+(-\tfrac{31}{180}+32 p_6-96 p_4+96 p_5) \partial_{q} \cP^{ij} \partial_{j} \cP^{k\ell} \partial_{\ell} \cP^{mn} \partial_{n} \cP^{pq} \partial_{m} \partial_{i} f \partial_{p} \partial_{k} g \\
+(-\tfrac{1}{90}-8 p_6-24 p_5+8 p_1-8 p_7) \partial_{p} \partial_{m} \cP^{ij} \cP^{k\ell} \partial_{q} \partial_{\ell} \cP^{mn} \partial_{n} \cP^{pq} \partial_{k} \partial_{i} f \partial_{j} g \\
+(\tfrac{1}{90}+8 p_6+24 p_5-8 p_1+8 p_7) \partial_{p} \partial_{m} \cP^{ij} \cP^{k\ell} \partial_{q} \partial_{\ell} \cP^{mn} \partial_{n} \cP^{pq} \partial_{i} f \partial_{k} \partial_{j} g \\
+(\tfrac{2}{9}-48 p_8+96 p_4-96 p_5+48 p_7) \partial_{\ell} \cP^{ij} \partial_{p} \partial_{n} \cP^{k\ell} \partial_{q} \cP^{mn} \cP^{pq} \partial_{k} \partial_{i} f \partial_{m} \partial_{j} g \\
+(\tfrac{2}{9}-48 p_8+96 p_4-96 p_5+48 p_7) \partial_{n} \cP^{ij} \partial_{p} \cP^{k\ell} \partial_{q} \partial_{\ell} \cP^{mn} \cP^{pq} \partial_{k} \partial_{i} f \partial_{m} \partial_{j} g \\
+(\tfrac{1}{6}-32 p_8+64 p_4-64 p_5+32 p_7) \partial_{p} \cP^{ij} \partial_{q} \partial_{n} \partial_{j} \cP^{k\ell} \partial_{\ell} \cP^{mn} \cP^{pq} \partial_{k} \partial_{i} f \partial_{m} g \\
+(-\tfrac{1}{6}+32 p_8-64 p_4+64 p_5-32 p_7) \partial_{\ell} \cP^{ij} \partial_{p} \partial_{n} \partial_{j} \cP^{k\ell} \partial_{q} \cP^{mn} \cP^{pq} \partial_{i} f \partial_{m} \partial_{k} g \\
+(-\tfrac{1}{9}+16 p_8-32 p_4+32 p_5-16 p_7) \partial_{k} \cP^{ij} \partial_{q} \partial_{n} \partial_{j} \cP^{k\ell} \partial_{\ell} \cP^{mn} \cP^{pq} \partial_{p} \partial_{i} f \partial_{m} g \\
+(\tfrac{1}{9}-16 p_8+32 p_4-32 p_5+16 p_7) \partial_{k} \cP^{ij} \partial_{q} \partial_{n} \partial_{j} \cP^{k\ell} \partial_{\ell} \cP^{mn} \cP^{pq} \partial_{i} f \partial_{p} \partial_{m} g \\
+(\tfrac{1}{120}-8 p_8+16 p_4-24 p_5+4 p_7) \partial_{p} \partial_{k} \cP^{ij} \partial_{q} \partial_{n} \partial_{j} \cP^{k\ell} \partial_{\ell} \cP^{mn} \cP^{pq} \partial_{i} f \partial_{m} g \\
+(\tfrac{1}{120}-8 p_8+16 p_4-24 p_5+4 p_7) \partial_{k} \cP^{ij} \partial_{p} \partial_{n} \partial_{j} \cP^{k\ell} \partial_{q} \partial_{\ell} \cP^{mn} \cP^{pq} \partial_{i} f \partial_{m} g \\
+(\tfrac{5}{18}-48 p_8+96 p_4-96 p_5+48 p_7) \partial_{\ell} \cP^{ij} \partial_{p} \cP^{k\ell} \partial_{q} \cP^{mn} \partial_{n} \cP^{pq} \partial_{k} \partial_{i} f \partial_{m} \partial_{j} g \\
+(\tfrac{5}{18}-48 p_8+96 p_4-96 p_5+48 p_7) \partial_{q} \cP^{ij} \partial_{m} \cP^{k\ell} \partial_{\ell} \cP^{mn} \partial_{n} \cP^{pq} \partial_{k} \partial_{i} f \partial_{p} \partial_{j} g \\
+(\tfrac{4}{15}-48 p_8+96 p_4-96 p_5+48 p_7) \partial_{m} \cP^{ij} \partial_{n} \cP^{k\ell} \partial_{q} \partial_{\ell} \cP^{mn} \cP^{pq} \partial_{k} \partial_{i} f \partial_{p} \partial_{j} g \\
+(-\tfrac{4}{15}+48 p_8-96 p_4+96 p_5-48 p_7) \partial_{m} \cP^{ij} \cP^{k\ell} \partial_{q} \partial_{\ell} \cP^{mn} \partial_{n} \cP^{pq} \partial_{k} \partial_{i} f \partial_{p} \partial_{j} g \\
+(-\tfrac{1}{18}+16 p_8-32 p_4+32 p_5-16 p_7) \partial_{\ell} \cP^{ij} \partial_{p} \partial_{n} \partial_{j} \cP^{k\ell} \partial_{q} \cP^{mn} \cP^{pq} \partial_{m} \partial_{i} f \partial_{k} g \\
+(-\tfrac{1}{18}+16 p_8-32 p_4+32 p_5-16 p_7) \partial_{p} \partial_{n} \partial_{\ell} \cP^{ij} \partial_{j} \cP^{k\ell} \partial_{q} \cP^{mn} \cP^{pq} \partial_{i} f \partial_{m} \partial_{k} g \\
+(-\tfrac{7}{36}+32 p_8-64 p_4+64 p_5-32 p_7) \partial_{p} \partial_{\ell} \cP^{ij} \partial_{j} \cP^{k\ell} \partial_{q} \cP^{mn} \partial_{n} \cP^{pq} \partial_{m} \partial_{i} f \partial_{k} g \\
+(-\tfrac{7}{36}+32 p_8-64 p_4+64 p_5-32 p_7) \partial_{\ell} \cP^{ij} \partial_{p} \partial_{j} \cP^{k\ell} \partial_{q} \cP^{mn} \partial_{n} \cP^{pq} \partial_{i} f \partial_{m} \partial_{k} g \\
+(-\tfrac{1}{12}+16 p_8-32 p_4+32 p_5-16 p_7) \partial_{\ell} \cP^{ij} \partial_{p} \partial_{j} \cP^{k\ell} \partial_{q} \cP^{mn} \partial_{n} \cP^{pq} \partial_{m} \partial_{i} f \partial_{k} g \\
+(-\tfrac{1}{12}+16 p_8-32 p_4+32 p_5-16 p_7) \partial_{p} \partial_{\ell} \cP^{ij} \partial_{j} \cP^{k\ell} \partial_{q} \cP^{mn} \partial_{n} \cP^{pq} \partial_{i} f \partial_{m} \partial_{k} g \\
+(-\tfrac{1}{90}-16 p_8+32 p_4-80 p_5+16 p_1) \partial_{p} \partial_{k} \cP^{ij} \partial_{n} \partial_{j} \cP^{k\ell} \partial_{q} \partial_{\ell} \cP^{mn} \cP^{pq} \partial_{i} f \partial_{m} g \\
+(\tfrac{1}{360}-16 p_8-16 p_3+32 p_4-48 p_5) \partial_{p} \partial_{k} \cP^{ij} \partial_{j} \cP^{k\ell} \partial_{q} \partial_{\ell} \cP^{mn} \partial_{n} \cP^{pq} \partial_{i} f \partial_{m} g \\
+(\tfrac{11}{120}-16 p_8+32 p_4-16 p_5+16 p_7) \partial_{k} \cP^{ij} \partial_{m} \partial_{j} \cP^{k\ell} \partial_{q} \partial_{\ell} \cP^{mn} \partial_{n} \cP^{pq} \partial_{i} f \partial_{p} g \\
+(\tfrac{1}{90}+8 p_6+16 p_4+24 p_5-8 p_1+8 p_7) \partial_{m} \partial_{\ell} \cP^{ij} \partial_{p} \cP^{k\ell} \partial_{q} \cP^{mn} \partial_{n} \cP^{pq} \partial_{k} \partial_{i} f \partial_{j} g \\
\end{multline*}
\begin{multline*}
+(-\tfrac{1}{90}-8 p_6-16 p_4-24 p_5+8 p_1-8 p_7) \partial_{m} \partial_{\ell} \cP^{ij} \partial_{p} \cP^{k\ell} \partial_{q} \cP^{mn} \partial_{n} \cP^{pq} \partial_{i} f \partial_{k} \partial_{j} g \\
+(-\tfrac{1}{15}+8 p_8+8 p_4+8 p_2+16 p_{10}-16 p_7) \partial_{m} \cP^{ij} \partial_{p} \partial_{j} \cP^{k\ell} \partial_{q} \partial_{\ell} \cP^{mn} \partial_{n} \cP^{pq} \partial_{i} f \partial_{k} g \\
+(-\tfrac{1}{15}+8 p_8+8 p_4+8 p_2+16 p_{10}-16 p_7) \partial_{p} \partial_{n} \cP^{ij} \partial_{q} \partial_{j} \cP^{k\ell} \partial_{k} \cP^{mn} \partial_{\ell} \cP^{pq} \partial_{i} f \partial_{m} g \\
+(\tfrac{1}{20}-8 p_8+24 p_4-16 p_5-8 p_2+8 p_7) \partial_{k} \cP^{ij} \partial_{q} \partial_{m} \partial_{j} \cP^{k\ell} \partial_{\ell} \cP^{mn} \partial_{n} \cP^{pq} \partial_{i} f \partial_{p} g \\
+(-\tfrac{1}{20}+8 p_8-24 p_4+16 p_5+8 p_2-8 p_7) \partial_{p} \cP^{ij} \partial_{q} \partial_{n} \partial_{j} \cP^{k\ell} \partial_{k} \cP^{mn} \partial_{\ell} \cP^{pq} \partial_{i} f \partial_{m} g \\
+(-\tfrac{1}{40}+8 p_8+16 p_4+8 p_5+16 p_{10}-12 p_7) \partial_{m} \cP^{ij} \partial_{p} \partial_{n} \partial_{j} \cP^{k\ell} \partial_{q} \partial_{\ell} \cP^{mn} \cP^{pq} \partial_{i} f \partial_{k} g \\
+(\tfrac{1}{40}-8 p_8-16 p_4-8 p_5-16 p_{10}+12 p_7) \partial_{p} \partial_{n} \partial_{k} \cP^{ij} \partial_{q} \partial_{j} \cP^{k\ell} \partial_{\ell} \cP^{mn} \cP^{pq} \partial_{i} f \partial_{m} g \\
+(\tfrac{11}{90}+8 p_6-16 p_4+40 p_5-8 p_1+24 p_7) \partial_{p} \partial_{m} \cP^{ij} \partial_{q} \partial_{n} \cP^{k\ell} \partial_{\ell} \cP^{mn} \cP^{pq} \partial_{k} \partial_{i} f \partial_{j} g \\
+(-\tfrac{11}{90}-8 p_6+16 p_4-40 p_5+8 p_1-24 p_7) \partial_{p} \partial_{m} \cP^{ij} \partial_{q} \partial_{n} \cP^{k\ell} \partial_{\ell} \cP^{mn} \cP^{pq} \partial_{i} f \partial_{k} \partial_{j} g \\
+(\tfrac{1}{5}-32 p_8-48 p_5-32 p_{10}+16 p_1+48 p_7) \partial_{p} \partial_{n} \cP^{ij} \partial_{q} \partial_{j} \cP^{k\ell} \partial_{\ell} \cP^{mn} \cP^{pq} \partial_{k} \partial_{i} f \partial_{m} g \\
+(\tfrac{1}{5}-32 p_8-48 p_5-32 p_{10}+16 p_1+48 p_7) \partial_{n} \cP^{ij} \partial_{p} \partial_{j} \cP^{k\ell} \partial_{q} \partial_{\ell} \cP^{mn} \cP^{pq} \partial_{i} f \partial_{m} \partial_{k} g \\
+(-\tfrac{1}{6}+16 p_8-16 p_3+32 p_4-16 p_1-32 p_7) \partial_{p} \cP^{ij} \partial_{j} \cP^{k\ell} \partial_{q} \partial_{\ell} \cP^{mn} \partial_{n} \cP^{pq} \partial_{m} \partial_{i} f \partial_{k} g \\
+(\tfrac{1}{6}-16 p_8+16 p_3-32 p_4+16 p_1+32 p_7) \partial_{q} \cP^{ij} \partial_{m} \partial_{j} \cP^{k\ell} \partial_{\ell} \cP^{mn} \partial_{n} \cP^{pq} \partial_{i} f \partial_{p} \partial_{k} g \\
+(\tfrac{1}{9}-16 p_8+16 p_4-32 p_5+16 p_1+16 p_7) \partial_{m} \cP^{ij} \partial_{n} \partial_{j} \cP^{k\ell} \partial_{q} \partial_{\ell} \cP^{mn} \cP^{pq} \partial_{i} f \partial_{p} \partial_{k} g \\
+(-\tfrac{1}{9}+16 p_8-16 p_4+32 p_5-16 p_1-16 p_7) \partial_{n} \partial_{k} \cP^{ij} \partial_{q} \partial_{j} \cP^{k\ell} \partial_{\ell} \cP^{mn} \cP^{pq} \partial_{p} \partial_{i} f \partial_{m} g \\
+(-\tfrac{1}{9}+16 p_8-32 p_4+48 p_5+16 p_2-16 p_7) \partial_{m} \cP^{ij} \partial_{j} \cP^{k\ell} \partial_{q} \partial_{\ell} \cP^{mn} \partial_{n} \cP^{pq} \partial_{p} \partial_{i} f \partial_{k} g \\
+(-\tfrac{1}{9}+16 p_8-32 p_4+48 p_5+16 p_2-16 p_7) \partial_{n} \cP^{ij} \partial_{q} \partial_{j} \cP^{k\ell} \partial_{k} \cP^{mn} \partial_{\ell} \cP^{pq} \partial_{i} f \partial_{p} \partial_{m} g \\
+(\tfrac{1}{9}-16 p_8+32 p_4-48 p_5+16 p_2+16 p_7) \partial_{m} \cP^{ij} \partial_{j} \cP^{k\ell} \partial_{q} \partial_{\ell} \cP^{mn} \partial_{n} \cP^{pq} \partial_{i} f \partial_{p} \partial_{k} g \\
+(-\tfrac{1}{9}+16 p_8-32 p_4+48 p_5-16 p_2-16 p_7) \partial_{k} \cP^{ij} \partial_{q} \partial_{j} \cP^{k\ell} \partial_{\ell} \cP^{mn} \partial_{n} \cP^{pq} \partial_{p} \partial_{i} f \partial_{m} g \\
+(\tfrac{7}{90}-16 p_8+40 p_4-40 p_5+8 p_2+12 p_7) \partial_{k} \cP^{ij} \partial_{p} \partial_{j} \cP^{k\ell} \partial_{q} \partial_{\ell} \cP^{mn} \partial_{n} \cP^{pq} \partial_{i} f \partial_{m} g \\
+(\tfrac{7}{90}-16 p_8+40 p_4-40 p_5+8 p_2+12 p_7) \partial_{m} \partial_{k} \cP^{ij} \partial_{j} \cP^{k\ell} \partial_{q} \partial_{\ell} \cP^{mn} \partial_{n} \cP^{pq} \partial_{i} f \partial_{p} g \\
+(\tfrac{1}{180}-16 p_8-16 p_4-16 p_5-32 p_{10}+8 p_7) \partial_{n} \cP^{ij} \partial_{p} \partial_{j} \cP^{k\ell} \partial_{q} \partial_{\ell} \cP^{mn} \cP^{pq} \partial_{k} \partial_{i} f \partial_{m} g \\
+(-\tfrac{1}{180}+16 p_8+16 p_4+16 p_5+32 p_{10}-8 p_7) \partial_{p} \partial_{n} \cP^{ij} \partial_{j} \cP^{k\ell} \partial_{q} \partial_{\ell} \cP^{mn} \cP^{pq} \partial_{i} f \partial_{m} \partial_{k} g \\
+(\tfrac{37}{90}-48 p_8-16 p_6+96 p_4-96 p_5+48 p_7) \partial_{p} \cP^{ij} \partial_{q} \partial_{n} \cP^{k\ell} \partial_{\ell} \cP^{mn} \cP^{pq} \partial_{k} \partial_{i} f \partial_{m} \partial_{j} g \\
+(-\tfrac{37}{90}+48 p_8+16 p_6-96 p_4+96 p_5-48 p_7) \partial_{p} \cP^{ij} \partial_{n} \cP^{k\ell} \partial_{q} \partial_{\ell} \cP^{mn} \cP^{pq} \partial_{k} \partial_{i} f \partial_{m} \partial_{j} g \\
+(\tfrac{29}{360}-16 p_8-16 p_4-16 p_{10}+8 p_1+20 p_7) \partial_{p} \partial_{m} \cP^{ij} \partial_{n} \partial_{j} \cP^{k\ell} \partial_{q} \partial_{\ell} \cP^{mn} \cP^{pq} \partial_{i} f \partial_{k} g \\
+(\tfrac{29}{360}-16 p_8-16 p_4-16 p_{10}+8 p_1+20 p_7) \partial_{n} \partial_{k} \cP^{ij} \partial_{p} \partial_{j} \cP^{k\ell} \partial_{q} \partial_{\ell} \cP^{mn} \cP^{pq} \partial_{i} f \partial_{m} g \\
+(\tfrac{34}{45}-96 p_8-32 p_6+240 p_4-288 p_5+96 p_7) \partial_{p} \cP^{ij} \partial_{q} \cP^{k\ell} \partial_{\ell} \cP^{mn} \partial_{n} \cP^{pq} \partial_{k} \partial_{i} f \partial_{m} \partial_{j} g \\
+(-\tfrac{34}{45}+96 p_8+32 p_6-240 p_4+288 p_5-96 p_7) \partial_{m} \cP^{ij} \partial_{q} \cP^{k\ell} \partial_{\ell} \cP^{mn} \partial_{n} \cP^{pq} \partial_{k} \partial_{i} f \partial_{p} \partial_{j} g \\
+(-\tfrac{2}{45}+8 p_9+4 p_6-8 p_3+4 p_5-4 p_1-4 p_7) \partial_{p} \cP^{ij} \partial_{q} \partial_{m} \partial_{j} \cP^{k\ell} \partial_{\ell} \cP^{mn} \partial_{n} \cP^{pq} \partial_{i} f \partial_{k} g \\
+(\tfrac{2}{45}-8 p_9-4 p_6+8 p_3-4 p_5+4 p_1+4 p_7) \partial_{q} \partial_{m} \partial_{k} \cP^{ij} \partial_{j} \cP^{k\ell} \partial_{\ell} \cP^{mn} \partial_{n} \cP^{pq} \partial_{i} f \partial_{p} g \\
+(\tfrac{1}{30}+8 p_6+32 p_4+8 p_5+32 p_{10}-8 p_1+8 p_7) \partial_{p} \partial_{n} \cP^{ij} \partial_{j} \cP^{k\ell} \partial_{q} \partial_{\ell} \cP^{mn} \cP^{pq} \partial_{k} \partial_{i} f \partial_{m} g \\
+(-\tfrac{1}{30}-8 p_6-32 p_4-8 p_5-32 p_{10}+8 p_1-8 p_7) \partial_{p} \partial_{n} \cP^{ij} \partial_{q} \partial_{j} \cP^{k\ell} \partial_{\ell} \cP^{mn} \cP^{pq} \partial_{i} f \partial_{m} \partial_{k} g \\
\end{multline*}
\begin{multline}
+(-\tfrac{7}{90}+8 p_8-16 p_3+16 p_4-8 p_5-8 p_1-16 p_7) \partial_{p} \cP^{ij} \partial_{m} \partial_{j} \cP^{k\ell} \partial_{q} \partial_{\ell} \cP^{mn} \partial_{n} \cP^{pq} \partial_{i} f \partial_{k} g \\
+(\tfrac{7}{90}-8 p_8+16 p_3-16 p_4+8 p_5+8 p_1+16 p_7) \partial_{q} \partial_{k} \cP^{ij} \partial_{m} \partial_{j} \cP^{k\ell} \partial_{\ell} \cP^{mn} \partial_{n} \cP^{pq} \partial_{i} f \partial_{p} g \\
+(-\tfrac{7}{180}+16 p_8-8 p_6-16 p_4+8 p_5+8 p_1-16 p_7) \partial_{m} \cP^{ij} \partial_{n} \partial_{j} \cP^{k\ell} \partial_{q} \partial_{\ell} \cP^{mn} \cP^{pq} \partial_{p} \partial_{i} f \partial_{k} g \\
+(\tfrac{7}{180}-16 p_8+8 p_6+16 p_4-8 p_5-8 p_1+16 p_7) \partial_{n} \partial_{k} \cP^{ij} \partial_{q} \partial_{j} \cP^{k\ell} \partial_{\ell} \cP^{mn} \cP^{pq} \partial_{i} f \partial_{p} \partial_{m} g \\
+(\tfrac{13}{360}-8 p_8+24 p_4-32 p_5-8 p_2+8 p_1+4 p_7) \partial_{m} \partial_{k} \cP^{ij} \partial_{q} \partial_{j} \cP^{k\ell} \partial_{\ell} \cP^{mn} \partial_{n} \cP^{pq} \partial_{i} f \partial_{p} g \\
+(-\tfrac{13}{360}+8 p_8-24 p_4+32 p_5+8 p_2-8 p_1-4 p_7) \partial_{p} \cP^{ij} \partial_{n} \partial_{j} \cP^{k\ell} \partial_{q} \partial_{k} \cP^{mn} \partial_{\ell} \cP^{pq} \partial_{i} f \partial_{m} g \\
+(\tfrac{1}{15}-32 p_8+8 p_6+48 p_4-72 p_5+24 p_1+24 p_7) \partial_{p} \cP^{ij} \partial_{n} \partial_{j} \cP^{k\ell} \partial_{q} \partial_{\ell} \cP^{mn} \cP^{pq} \partial_{k} \partial_{i} f \partial_{m} g \\
+(-\tfrac{1}{15}+32 p_8-8 p_6-48 p_4+72 p_5-24 p_1-24 p_7) \partial_{p} \partial_{\ell} \cP^{ij} \partial_{n} \partial_{j} \cP^{k\ell} \partial_{q} \cP^{mn} \cP^{pq} \partial_{i} f \partial_{m} \partial_{k} g \\
+(-\tfrac{11}{180}-16 p_9+16 p_8-8 p_6+8 p_5+8 p_1-16 p_7) \partial_{p} \cP^{ij} \partial_{q} \partial_{j} \cP^{k\ell} \partial_{k} \cP^{mn} \partial_{n} \partial_{\ell} \cP^{pq} \partial_{i} f \partial_{m} g \\
+(-\tfrac{17}{180}+16 p_8-8 p_6-32 p_4+40 p_5-8 p_1-16 p_7) \partial_{p} \partial_{\ell} \cP^{ij} \partial_{q} \partial_{j} \cP^{k\ell} \cP^{mn} \partial_{n} \cP^{pq} \partial_{m} \partial_{i} f \partial_{k} g \\
+(\tfrac{17}{180}-16 p_8+8 p_6+32 p_4-40 p_5+8 p_1+16 p_7) \partial_{p} \partial_{\ell} \cP^{ij} \partial_{q} \partial_{j} \cP^{k\ell} \cP^{mn} \partial_{n} \cP^{pq} \partial_{i} f \partial_{m} \partial_{k} g \\
+(\tfrac{61}{180}-48 p_8-8 p_6+96 p_4-120 p_5+24 p_1+48 p_7) \partial_{p} \partial_{\ell} \cP^{ij} \partial_{q} \cP^{k\ell} \cP^{mn} \partial_{n} \cP^{pq} \partial_{k} \partial_{i} f \partial_{m} \partial_{j} g \\
+(-\tfrac{61}{180}+48 p_8+8 p_6-96 p_4+120 p_5-24 p_1-48 p_7) \partial_{q} \partial_{m} \cP^{ij} \cP^{k\ell} \partial_{\ell} \cP^{mn} \partial_{n} \cP^{pq} \partial_{k} \partial_{i} f \partial_{p} \partial_{j} g \\
+(\tfrac{53}{90}-96 p_8-16 p_6+192 p_4-240 p_5+48 p_1+96 p_7) \partial_{n} \partial_{\ell} \cP^{ij} \partial_{p} \cP^{k\ell} \partial_{q} \cP^{mn} \cP^{pq} \partial_{k} \partial_{i} f \partial_{m} \partial_{j} g \\
+(-\tfrac{49}{90}+48 p_8+24 p_6-144 p_4+168 p_5-24 p_1-72 p_7) \partial_{p} \partial_{\ell} \cP^{ij} \partial_{n} \cP^{k\ell} \partial_{q} \cP^{mn} \cP^{pq} \partial_{k} \partial_{i} f \partial_{m} \partial_{j} g \\
+(\tfrac{49}{90}-48 p_8-24 p_6+144 p_4-168 p_5+24 p_1+72 p_7) \partial_{p} \partial_{n} \cP^{ij} \partial_{q} \cP^{k\ell} \partial_{\ell} \cP^{mn} \cP^{pq} \partial_{k} \partial_{i} f \partial_{m} \partial_{j} g \\
+(\tfrac{1}{90}-16 p_8+8 p_6-16 p_3+16 p_4-24 p_5-8 p_1+8 p_7) \partial_{p} \cP^{ij} \partial_{j} \cP^{k\ell} \partial_{q} \partial_{\ell} \cP^{mn} \partial_{n} \cP^{pq} \partial_{k} \partial_{i} f \partial_{m} g \\
+(-\tfrac{1}{90}+16 p_8-8 p_6+16 p_3-16 p_4+24 p_5+8 p_1-8 p_7) \partial_{q} \partial_{k} \cP^{ij} \partial_{j} \cP^{k\ell} \partial_{\ell} \cP^{mn} \partial_{n} \cP^{pq} \partial_{i} f \partial_{p} \partial_{m} g \\
+(\tfrac{3}{20}+16 p_9-32 p_8+8 p_6+16 p_4-40 p_5-8 p_1+32 p_7) \partial_{p} \cP^{ij} \partial_{q} \partial_{j} \cP^{k\ell} \partial_{\ell} \cP^{mn} \partial_{n} \cP^{pq} \partial_{k} \partial_{i} f \partial_{m} g \\
+(-\tfrac{3}{20}-16 p_9+32 p_8-8 p_6-16 p_4+40 p_5+8 p_1-32 p_7) \partial_{n} \cP^{ij} \partial_{j} \cP^{k\ell} \partial_{q} \partial_{k} \cP^{mn} \partial_{\ell} \cP^{pq} \partial_{i} f \partial_{p} \partial_{m} g \\
+(-\tfrac{7}{180}-16 p_9+16 p_8-8 p_6+16 p_4+8 p_5-8 p_1-16 p_7) \partial_{q} \partial_{m} \cP^{ij} \partial_{j} \cP^{k\ell} \partial_{\ell} \cP^{mn} \partial_{n} \cP^{pq} \partial_{k} \partial_{i} f \partial_{p} g \\
+(\tfrac{7}{180}+16 p_9-16 p_8+8 p_6-16 p_4-8 p_5+8 p_1+16 p_7) \partial_{p} \cP^{ij} \partial_{q} \partial_{j} \cP^{k\ell} \partial_{\ell} \cP^{mn} \partial_{n} \cP^{pq} \partial_{i} f \partial_{m} \partial_{k} g \\
+(\tfrac{7}{120}+16 p_9-8 p_8+8 p_6+16 p_4+16 p_{10}-8 p_1+12 p_7) \partial_{p} \partial_{m} \cP^{ij} \partial_{q} \partial_{j} \cP^{k\ell} \partial_{\ell} \cP^{mn} \partial_{n} \cP^{pq} \partial_{i} f \partial_{k} g \\
+(-\tfrac{7}{120}-16 p_9+8 p_8-8 p_6-16 p_4-16 p_{10}+8 p_1-12 p_7) \partial_{p} \partial_{n} \cP^{ij} \partial_{j} \cP^{k\ell} \partial_{q} \partial_{k} \cP^{mn} \partial_{\ell} \cP^{pq} \partial_{i} f \partial_{m} g \\
+(8 p_9-8 p_8+4 p_6-8 p_3-8 p_4+4 p_5-8 p_2-4 p_1+4 p_7) \partial_{p} \partial_{k} \cP^{ij} \partial_{q} \partial_{j} \cP^{k\ell} \partial_{\ell} \cP^{mn} \partial_{n} \cP^{pq} \partial_{i} f \partial_{m} g \\
+(-8 p_9+8 p_8-4 p_6+8 p_3+8 p_4-4 p_5+8 p_2+4 p_1-4 p_7) \partial_{p} \cP^{ij} \partial_{j} \cP^{k\ell} \partial_{q} \partial_{k} \cP^{mn} \partial_{n} \partial_{\ell} \cP^{pq} \partial_{i} f \partial_{m} g \\
+(\tfrac{23}{360}+8 p_9-16 p_8+4 p_6-8 p_3+8 p_4-20 p_5+8 p_2-16 p_{10}-4 p_1+16 p_7) \\ \partial_{p} \partial_{m} \cP^{ij} \partial_{j} \cP^{k\ell} \partial_{q} \partial_{\ell} \cP^{mn} \partial_{n} \cP^{pq} \partial_{i} f \partial_{k} g \\
+(-\tfrac{23}{360}-8 p_9+16 p_8-4 p_6+8 p_3-8 p_4+20 p_5-8 p_2+16 p_{10}+4 p_1-16 p_7) \\ \partial_{n} \cP^{ij} \partial_{p} \partial_{j} \cP^{k\ell} \partial_{q} \partial_{k} \cP^{mn} \partial_{\ell} \cP^{pq} \partial_{i} f \partial_{m} g
\big)
+ \bar{o}(\hbar^4)
.\label{EqStar4}
\end{multline}
}

\noindent%
The ten master\/-\/parameters in~\eqref{EqStar4} are the still unknown
weights of the prime graphs which are portrayed in Fig.~\ref{10Graphs}
on p.~\pageref{10Graphs}. The four underlined parameters can be gauged
out (\emph{without} modifying the coefficients of any other
Kontsevich graphs with four internal vertices), see
Theorem~\ref{ThGauge} on p.~\pageref{ThGauge}. At all values of the
ten master\/-\/parameters, that is, irrespective of their true values
given by formula~\eqref{EqWeight}, the $\star$-\/product is proven in
Theorem~\ref{ThMainAssocOrd4} to be associative
modulo~$\bar{o}(\hbar^4)$.

\subsubsection*{Acknowledgements}
The authors are grateful to the anonymous referees for their critical comments and suggestions which helped us improve this text, to prof. S. Tabachnikov (Editor-in-Chief) for persistence and constructive criticism, and B. Pym and E. Panzer for communicating the values of ten master-parameters obtained via a different technique \cite{PymPrivateComm}.
We thank prof. M. Gerstenhaber and M. Kontsevich for their attention to our work.

This research was supported in part by JBI~RUG project~106552 (Groningen, The Netherlands) and IM JGU project 5020 (Mainz, Germany). 
The authors also thank the Center for Information Technology of the University of Groningen 
for providing access to 
\textsf{Peregrine} high performance computing cluster.
A~part of this research was done while the authors were visiting at the $\smash{\text{IH\'ES}}$ in Bures\/-\/sur-\/Yvette, France and AVK was visiting at the MPIM Bonn, Germany; warm hospitality and partial financial support by these institutions are gratefully acknowledged.

\vspace{3em}
\rightline{
\small \textit{This text was submitted in its original form on 20 December 2017.}}

\newpage
\appendix


\section{
Approximations and conjectured values of weight integrals}

\noindent
The material presented here is an expanded version of section 3 of the note \cite{Decin} by the authors. 

\subsection{The weight integral in Cartesian coordinates}
\label{AppNumericalWeights}

Recall the integral formula for the weight of a graph $\Gamma \in \tilde{G}_{2,k}$ (see section \ref{SecStar}):
\begin{equation}\tag{\ref{EqWeight}}
w(\Gamma) = \frac{1}{(2\pi)^{2k}}\int_{C_k(\mathbb{H})} \bigwedge_{j=1}^k {\rm d}\varphi(p_j,p_{\text{Left}(j)}) \wedge {\rm d}\varphi(p_j,p_{\text{Right}(j)}),
\end{equation}
such that the integral is taken over the configuration space of $k$~points in the upper half\/-\/plane $\mathbb{H} \subset \mathbb{C}$,
\[
C_k(\mathbb{H}) = \{(p_1,\ldots,p_k) \in \mathbb{H}^k : p_i \textrm{ pairwise distinct}\},
\]
and where $\varphi\colon C_2(\mathbb{H}) \to [0, 2\pi)$ was defined by $\varphi(p,q) = \Arg\big(\frac{q-p}{q-\bar{p}}\big)$.

For nonzero $z = x+iy$ in $\mathbb{H}$ we have $\operatorname{Arg}(x+iy) \cong \arctan(y/x)$, where $\cong$ denotes equality of functions up to a constant.
Since $\frac{d}{dt} \arctan(t) = 1/(1+t^2)$, the weight integrand is a rational function of the Cartesian coordinates: for $p = a + ib$ and $q = x + iy$,
\begin{equation}
\label{eq:cartesianphi}
\varphi(p,q) \cong \arctan\bigg(\frac{2b(a-x)}{(a-x)^2 + (y+b)(y-b)}\bigg).
\end{equation}
In Cartesian coordinates $(x_1,y_1,\ldots,x_k,y_k)$, the weight integrand can now be written as the Jacobian determinant of the map $\Phi_\Gamma\colon C_k(\mathbb{H}) \to [0,2\pi)^{2k}$ defined by\footnote{Called a Gauss map by M. Polyak \cite{Polyak}.}
\[
\Phi_\Gamma(p_1,\ldots,p_k) = (\varphi(p_1,p_{\text{Left}(1)}), \varphi(p_1, p_{\text{Right}(1)}), \ldots, \varphi(p_k, p_{\text{Left}(k)}), \varphi(p_k, p_{\text{Right}(k)}))
\]
considered as a function of the $(x_j, y_j)$ through $p_j = x_j + iy_j$.

\begin{implement}\label{ImplWeightIntegrands}
The command
\begin{verbatim}
    > weight_integrands <graph-series-file>
\end{verbatim}
takes as input a list of graphs $\Gamma \in \tilde{G}_{2,k}$ with (possibly undetermined) coefficients, and sends to the standard output lines of the following form:
\begin{verbatim}
    (* <graph encoding>    <coefficient> *)
    <weight integrand of the graph above>
\end{verbatim}
where the weight integrands are written in {\sf Mathematica} format, as {\tt Det[...]}.
\end{implement}

We can take integration domain to be $\mathbb{H}^k$, since for any $i \neq j$ the set $\{(p_1,\ldots,p_k) \in \mathbb{C}^k : p_i = p_j\}$ is a strict linear subspace of $\mathbb{C}^k$, which has measure zero.
The weight integral is absolutely convergent \cite{MK97}, so by the Fubini--Tonelli theorem we may evaluate it as an iterated integral in any order.
We can use the residue theorem\footnote{G. Dito used the residue method for one graph \cite{Dito} at $k = 2$, and remarked that that it becomes unpractical for $k \geqslant 3$.} to integrate out the Cartesian coordinates corresponding to the $k$ real parts, halving the dimension.
It then remains to integrate the result (a function of the $k$ imaginary parts) over $\mathbb{R}^k$.

\begin{example}
For the wedge graph $\Lambda$
we have the Cartesian coordinates $x + iy$ in the upper half-plane and the integrand (obtained using Implementation \ref{ImplWeightIntegrands})
\[
f(x,y) = \frac{4y}{((x-1)^2 + y^2)(x^2 + y^2)}.
\]
To apply the residue theorem we interpret $f(x,y)$ as a rational function in a single {\em complex} variable $x$.
Its poles are then $\pm iy$ and $1 \pm iy$, so the poles in the upper half-plane are $iy$ and $1 + iy$ (since $y > 0$).
The residues at these poles are $r_1 = 2/(i+2y)$ and $r_2 = -2/(2y-i)$ respectively.
Hence the residue theorem yields that the integral of $f(x,y)$ with respect to $x$ over the real line is $2\pi i(r_1 + r_2) = 8\pi/(1+4y^2)$.
When we integrate this over $y > 0$ and divide by $(2\pi)^2$ we obtain $1/2$, as desired.
\end{example}

This is of course a toy example.
For higher $k$ the expressions become larger, but also one has to consider more carefully which poles are in the upper half-plane.
From the expression \eqref{eq:cartesianphi} for $\varphi$ one can see that this issue depends on the relative position of the coordinates on the imaginary axis ($y$ and $b$ in that formula).

For $k = 3$ with coordinates on $\mathbb{H}^3$ given by $$a+bi, \quad c+di, \quad e+fi,$$ let us agree to call $a,c,e$ the real coordinates and $b, d, f$ the imaginary coordinates.
We now split the integral into a sum of integrals over $3! = 6$ regions, one for each possible ordering of the imaginary coordinates:
\[
b<d<f;\ \ \ b<f<d;\ \ \ d<b<f;\ \ \ d<f<b;\ \ \ f<b<d;\ \ \ f<d<b.
\]
In each such region it is known for every (complexified) real coordinate which poles are in the upper half-plane, so we can apply the residue theorem three times.
The result can be numerically integrated more effectively than the original expression, for one because we have halved the dimension of the integration domain.

\begin{rem}
To integrate over the region of $\mathbb{H}^3$ defined by $b<d<f$, one can choose integration bounds as follows: $\int_0^\infty {\rm d}b \int_b^\infty {\rm d}f \int_b^f {\rm d}d$ (and similarly for the other permutations).
For the region of $\mathbb{H}^4$ defined by $b<d<f<h$ one can choose the integration bounds $\int_0^\infty {\rm d}b \int_b^\infty {\rm d}h \int_b^h {\rm d}d \int_d^h {\rm d}f$, and so on.
\end{rem}


\begin{implement}
The strategy above is implemented by the following {\sf Mathematica} code (for the order $4$, but it can be adapted for others), where $W$ is the weight integrand.
\vskip 1em
\noindent
{\tt W = an integrand, e.g. from list \cite{OnlineList};}
\begin{verbatim}

integrationvariables = {a, b, c, d, e, f, g, h};
imaginaryvariables = 
 integrationvariables[[2 #1]] & /@ 
  Range[1, Length[integrationvariables]/2];
realvariables = 
 integrationvariables[[2 #1 - 1]] & /@ 
  Range[1, Length[integrationvariables]/2];

basicAssumptions =
  Element[a, Reals] && Element[c, Reals] && Element[e, Reals] &&
   Element[g, Reals] && b > 0 && d > 0 && f > 0 && h > 0;

ContourIntegrate[function_, variable_, assumptions_] :=
 2*Pi*I*Total[
   Map[
    Function[{p}, (Numerator[Together[function]]/
        D[Denominator[Together[function]], variable]) /. {variable ->
        p}],
    Select[
     ReplaceList[variable,
      Assuming[assumptions,
       Flatten[FullSimplify[
         Solve[Denominator[Together[function]] == 0, variable,
          Complexes]]]]],
     Function[{r}, Simplify[ComplexExpand[Im[r]] > 0, assumptions]]]]]

IteratedContourIntegrate[function_, variables_, assumptions_] :=
 Fold[ContourIntegrate[Together[#1], #2, assumptions] &, function,
  variables]

integrals = Map[
  NIntegrate[
     Simplify[
      IteratedContourIntegrate[W, realvariables, 
       basicAssumptions && #1[[1]] < #1[[2]] < #1[[3]] < #1[[4]]]
     TimeConstraint -> Infinity],
     Evaluate[
      Sequence @@
       {{#1[[1]], 0, Infinity}, {#1[[3]], #1[[1]], 
         Infinity}, {#1[[2]], #1[[1]], #1[[3]]}, {#1[[4]], #1[[3]], 
         Infinity}}
      ],
     Method -> {GlobalAdaptive, MaxErrorIncreases -> 10^4}
     ] &, Permutations[imaginaryvariables]]

Print[integrals]
Print[Total[integrals]]
Print[Total[integrals]/N[(2 Pi)^8]]
\end{verbatim}
\end{implement}

\begin{rem}
This strategy allows effective numerical integration of all weights up to order $3$.
At the order $4$, it works for some weights but not others: see Tables \ref{TableVerified} and \ref{TableConjectured}.
The call(s) to {\tt Map} may be replaced by {\tt ParallelMap} to parallelize the computation.
\end{rem}

%
\begin{example}
The second Bernoulli graph
\text{\raisebox{-16pt}{
\unitlength=0.70mm
\linethickness{0.4pt}
\begin{picture}(15.00,23.33)
\put(5.00,5.00){\circle*{1.33}}
\put(16.00,5.00){\circle*{1.33}}
\put(-2.00,17.00){\circle*{1.33}}
\put(-2.00,17.00){\vector(1,0){6.33}}
\put(-2.00,17.00){\vector(1,-2){6.23}}
\put(5.00,17.00){\circle*{1.33}}
\put(5.00,17.00){\vector(0,-1){11.33}}
\put(5.00,17.00){\vector(1,-1){5.33}}
\put(10.33,11.33){\circle*{1.33}}
\put(10.33,11.33){\vector(1,-1){5.33}}
\put(10.33,11.33){\vector(-1,-1){5.33}}
\end{picture}
}}
\cite{WillwacherFelderIrrationality}
has the weight integrand
{\tiny \[
\frac{64 b f d \left(c \left((a-c)^2+b^2\right)+d^2 (c-2 a)\right) \left(f^2 (e-2 c)+e \left((e-c)^2+d^2\right)\right)}{\left(a^2+b^2\right) \left(f^2+(e-1)^2\right) \left(f^2+e^2\right) \left(c^2+d^2\right) \left((a-c)^2+(b-d)^2\right) \left((a-c)^2+(b+d)^2\right) \left((f+d)^2+(e-c)^2\right)}
\]}
\noindent
The residue calculation followed by the numerical integration leads to the estimate $5.71871 \times 10^{-9} - 5.92495 \times 10^{-21} i$ of the weight; this leads to the guess that it is zero and in fact it is true.
\end{example}

\begin{table}[htb]
\caption{Verified values}
\label{TableVerified}
\begin{tabular}{l | r | r c r}
Weight & Approximation & \multicolumn{3}{c}{True value} \\
\hline
{\tt w\symbol{"5F}4\symbol{"5F}1} & $-0.0069444401170 \pm 0.000000906189$ & $-1/144$ & $\approx$ & $-0.00694444$
\end{tabular}
\end{table}

\begin{table}[htb]
\caption{Conjectured values}
\label{TableConjectured}
\begin{tabular}{l | r | r c r}
Weight & Approximation & \multicolumn{3}{c}{Conjectured true value} \\
\hline
{\tt w\symbol{"5F}4\symbol{"5F}103} & $-0.000086894703 \pm 0.000000681076$ & $-1/11520$ & $\approx$ & $0.000086805$ \\
{\tt w\symbol{"5F}4\symbol{"5F}104} &  $0.000347214860 \pm 0.000000371598$ & $1/2880$ & $\approx$ & $0.000347222$ \\
{\tt w\symbol{"5F}4\symbol{"5F}112} & $-0.000347219933 \pm 0.000000042901$ & $-1/2880$ & $\approx$ & $-0.000347222$ \\
{\tt w\symbol{"5F}4\symbol{"5F}113} &  $0.000694441623 \pm 0.000000093136$ & $1/1440$ & $\approx$ & $0.000694444$ \\
{\tt w\symbol{"5F}4\symbol{"5F}133} &  $0.000694443060 \pm 0.000000078774$ & $1/1440$ & $\approx$ & $0.000694444$ \\
{\tt w\symbol{"5F}4\symbol{"5F}138} & $-0.001041664533 \pm 0.000000095465$ & $-1/960$ & $\approx$ & $-0.001041666$ \\
{\tt w\symbol{"5F}4\symbol{"5F}147} & $-0.000043376821 \pm 0.000000095465$ & $-1/23040$ & $\approx$ & $-0.000043402$ \\
{\tt w\symbol{"5F}4\symbol{"5F}148} &  $0.000173611294 \pm 0.000000015063$ & $1/5760$ & $\approx$ & $0.000173611$
\end{tabular}
\end{table}

In particular, this table lists the approximate value of the master\/-\/parameters $p_4={}${\tt w\symbol{"5F}4\symbol{"5F}103} and $p_5={}${\tt w\symbol{"5F}4\symbol{"5F}104}.
The relation {\tt w\symbol{"5F}4\symbol{"5F}133}${}=2\cdot{}${\tt w\symbol{"5F}4\symbol{"5F}104} which was found in Theorem \ref{ThmBig} and listed in Table \ref{Table149via10} of Appendix \ref{AppStarEncoding} is satisfied approximately.
Furthermore, the relation \verb"w_4_103"${}=2\cdot{}$\verb"w_4_147" seems to hold approximately.

\newpage

\subsection{Claimed values of the 10 master-parameters}
\label{App10Pym}

By using a different technique 
B. Pym and E. Panzer have obtained the exact values of the ten master-parameters.

\begin{claimNo}[\cite{PymPrivateComm}]
The values of ten master-parameters \textup{(}which are the weights of ten graphs in Figure~\ref{10Graphs} on p.~\pageref{10Graphs}\textup{)} are given in Table~\ref{Tab10Pym} below.
\end{claimNo}

\begin{table}[htb]
\caption{Recently suggested values of the master-parameters \cite{PymPrivateComm}.}
\label{Tab10Pym}
\begin{tabular}{l l}
Master-parameter & Value \\
\hline
$p_1 = {}$\verb"w_4_100" & $1/1440$ \\
$p_2 = {}$\verb"w_4_101" & $1/2880$ \\
$p_3 = {}$\verb"w_4_102" & $1/5760$ \\
$p_4 = {}$\verb"w_4_103" & $-1/11520$ \\
$p_5 = {}$\verb"w_4_104" & $1/2880$ \\
$p_6 = {}$\verb"w_4_107" & $13/2880$ \\
$p_7 = {}$\verb"w_4_108" & $-17/2880$ \\
$p_8 = {}$\verb"w_4_109" & $-1/1152$ \\
$p_9 = {}$\verb"w_4_119" & $-1/1280$ \\
$p_{10} = {}$\verb"w_4_125" & $-1/960$
\end{tabular}
\end{table}

Let it be emphasized that these ten values are conjectured via a use of software which is currently under development. 

\begin{rem}
The exact values of two master-parameters \verb"w_4_103" and \verb"w_4_104" reproduce the values which had been conjectured in Table~\ref{TableConjectured}.
We also note that all the weights of graphs in $\star$ mod $\bar{o}(\hbar^4)$ are rational numbers.
Thirdly, the values of non-master parameters (namely, \verb"w_4_112", \verb"w_4_113", \verb"w_4_133", \verb"w_4_138", \verb"w_4_147", and \verb"w_4_148") in Table~\ref{TableConjectured}, whenever recalculated on the basis of conjectured values from Table~\ref{Tab10Pym}, do all match the numerical approximations in Table~\ref{TableConjectured}, reproducing our conjectured rational values in its rightmost column.
\end{rem}

In conclusion, provided that all the ten values in Table~\ref{Tab10Pym} are true, this is the authentic Kontsevich star-product up to $\bar{o}(\hbar^4)$:
\label{EqStarWith10PymStart}

\begin{multline*}
f \star g = f \times g
+\hbar\cP^{ij} \partial_{i} f \partial_{j} g 
+\hbar^{2}\big(
-\tfrac{1}{6} \partial_{\ell} \cP^{ij} \partial_{j} \cP^{k\ell} \partial_{i} f \partial_{k} g 
-\tfrac{1}{3} \partial_{\ell} \cP^{ij} \cP^{k\ell} \partial_{i} f \partial_{k} \partial_{j} g \\
+\tfrac{1}{3} \partial_{\ell} \cP^{ij} \cP^{k\ell} \partial_{k} \partial_{i} f \partial_{j} g 
+\tfrac{1}{2} \cP^{ij} \cP^{k\ell} \partial_{k} \partial_{i} f \partial_{\ell} \partial_{j} g 
\big)
+\hbar^{3}\big(
-\tfrac{1}{6} \partial_{m} \partial_{\ell} \cP^{ij} \partial_{n} \partial_{j} \cP^{k\ell} \cP^{mn} \partial_{i} f \partial_{k} g \\
+\tfrac{1}{6} \partial_{n} \partial_{\ell} \cP^{ij} \cP^{k\ell} \cP^{mn} \partial_{i} f \partial_{m} \partial_{k} \partial_{j} g 
-\tfrac{1}{3} \partial_{n} \cP^{ij} \cP^{k\ell} \cP^{mn} \partial_{k} \partial_{i} f \partial_{m} \partial_{\ell} \partial_{j} g \\
+\tfrac{1}{6} \partial_{n} \partial_{\ell} \cP^{ij} \cP^{k\ell} \cP^{mn} \partial_{m} \partial_{k} \partial_{i} f \partial_{j} g 
+\tfrac{1}{3} \partial_{n} \cP^{ij} \cP^{k\ell} \cP^{mn} \partial_{m} \partial_{k} \partial_{i} f \partial_{\ell} \partial_{j} g \\
+\tfrac{1}{6} \cP^{ij} \cP^{k\ell} \cP^{mn} \partial_{m} \partial_{k} \partial_{i} f \partial_{n} \partial_{\ell} \partial_{j} g 
-\tfrac{1}{6} \partial_{m} \partial_{\ell} \cP^{ij} \partial_{n} \cP^{k\ell} \cP^{mn} \partial_{i} f \partial_{k} \partial_{j} g \\
+\tfrac{1}{6} \partial_{n} \partial_{\ell} \cP^{ij} \partial_{j} \cP^{k\ell} \cP^{mn} \partial_{i} f \partial_{m} \partial_{k} g 
-\tfrac{1}{6} \partial_{m} \partial_{\ell} \cP^{ij} \partial_{n} \cP^{k\ell} \cP^{mn} \partial_{k} \partial_{i} f \partial_{j} g \\
-\tfrac{1}{6} \partial_{\ell} \cP^{ij} \partial_{n} \partial_{j} \cP^{k\ell} \cP^{mn} \partial_{m} \partial_{i} f \partial_{k} g 
-\tfrac{1}{6} \cP^{ij} \partial_{n} \cP^{k\ell} \partial_{\ell} \cP^{mn} \partial_{k} \partial_{i} f \partial_{m} \partial_{j} g \\
-\tfrac{1}{6} \partial_{n} \cP^{ij} \cP^{k\ell} \partial_{\ell} \cP^{mn} \partial_{k} \partial_{i} f \partial_{m} \partial_{j} g 
-\tfrac{1}{6} \partial_{\ell} \cP^{ij} \partial_{n} \cP^{k\ell} \cP^{mn} \partial_{k} \partial_{i} f \partial_{m} \partial_{j} g 
\big) + {}
\end{multline*}
\begin{multline*}
+\hbar^{4}\big(
-\tfrac{1}{6} \partial_{q} \cP^{ij} \cP^{k\ell} \cP^{mn} \cP^{pq} \partial_{m} \partial_{k} \partial_{i} f \partial_{p} \partial_{n} \partial_{\ell} \partial_{j} g 
+\tfrac{1}{6} \partial_{q} \cP^{ij} \cP^{k\ell} \cP^{mn} \cP^{pq} \partial_{p} \partial_{m} \partial_{k} \partial_{i} f \partial_{n} \partial_{\ell} \partial_{j} g \\
+\tfrac{1}{24} \cP^{ij} \cP^{k\ell} \cP^{mn} \cP^{pq} \partial_{p} \partial_{m} \partial_{k} \partial_{i} f \partial_{q} \partial_{n} \partial_{\ell} \partial_{j} g 
+\tfrac{1}{6} \partial_{q} \partial_{n} \cP^{ij} \cP^{k\ell} \cP^{mn} \cP^{pq} \partial_{k} \partial_{i} f \partial_{p} \partial_{m} \partial_{\ell} \partial_{j} g \\
+\tfrac{1}{18} \partial_{n} \cP^{ij} \partial_{q} \cP^{k\ell} \cP^{mn} \cP^{pq} \partial_{k} \partial_{i} f \partial_{p} \partial_{m} \partial_{\ell} \partial_{j} g 
+\tfrac{1}{6} \partial_{q} \partial_{n} \cP^{ij} \cP^{k\ell} \cP^{mn} \cP^{pq} \partial_{p} \partial_{m} \partial_{k} \partial_{i} f \partial_{\ell} \partial_{j} g \\
+\tfrac{1}{18} \partial_{n} \cP^{ij} \partial_{q} \cP^{k\ell} \cP^{mn} \cP^{pq} \partial_{p} \partial_{m} \partial_{k} \partial_{i} f \partial_{\ell} \partial_{j} g 
-\tfrac{1}{30} \partial_{q} \partial_{n} \partial_{\ell} \cP^{ij} \cP^{k\ell} \cP^{mn} \cP^{pq} \partial_{i} f \partial_{p} \partial_{m} \partial_{k} \partial_{j} g \\
-\tfrac{2}{45} \partial_{n} \partial_{\ell} \cP^{ij} \partial_{q} \cP^{k\ell} \cP^{mn} \cP^{pq} \partial_{i} f \partial_{p} \partial_{m} \partial_{k} \partial_{j} g 
-\tfrac{1}{30} \partial_{\ell} \cP^{ij} \partial_{q} \partial_{n} \cP^{k\ell} \cP^{mn} \cP^{pq} \partial_{i} f \partial_{p} \partial_{m} \partial_{k} \partial_{j} g \\
+\tfrac{1}{45} \partial_{\ell} \cP^{ij} \partial_{n} \cP^{k\ell} \partial_{q} \cP^{mn} \cP^{pq} \partial_{i} f \partial_{p} \partial_{m} \partial_{k} \partial_{j} g 
-\tfrac{1}{12} \cP^{ij} \cP^{k\ell} \partial_{q} \cP^{mn} \partial_{n} \cP^{pq} \partial_{m} \partial_{k} \partial_{i} f \partial_{p} \partial_{\ell} \partial_{j} g \\
-\tfrac{1}{6} \partial_{q} \cP^{ij} \cP^{k\ell} \cP^{mn} \partial_{n} \cP^{pq} \partial_{m} \partial_{k} \partial_{i} f \partial_{p} \partial_{\ell} \partial_{j} g 
-\tfrac{1}{6} \partial_{n} \cP^{ij} \cP^{k\ell} \partial_{q} \cP^{mn} \cP^{pq} \partial_{m} \partial_{k} \partial_{i} f \partial_{p} \partial_{\ell} \partial_{j} g \\
-\tfrac{1}{9} \partial_{n} \cP^{ij} \partial_{q} \cP^{k\ell} \cP^{mn} \cP^{pq} \partial_{m} \partial_{k} \partial_{i} f \partial_{p} \partial_{\ell} \partial_{j} g 
+\tfrac{1}{30} \partial_{q} \partial_{n} \partial_{\ell} \cP^{ij} \cP^{k\ell} \cP^{mn} \cP^{pq} \partial_{p} \partial_{m} \partial_{k} \partial_{i} f \partial_{j} g \\
+\tfrac{2}{45} \partial_{n} \partial_{\ell} \cP^{ij} \partial_{q} \cP^{k\ell} \cP^{mn} \cP^{pq} \partial_{p} \partial_{m} \partial_{k} \partial_{i} f \partial_{j} g 
+\tfrac{1}{30} \partial_{\ell} \cP^{ij} \partial_{q} \partial_{n} \cP^{k\ell} \cP^{mn} \cP^{pq} \partial_{p} \partial_{m} \partial_{k} \partial_{i} f \partial_{j} g \\
-\tfrac{1}{45} \partial_{\ell} \cP^{ij} \partial_{n} \cP^{k\ell} \partial_{q} \cP^{mn} \cP^{pq} \partial_{p} \partial_{m} \partial_{k} \partial_{i} f \partial_{j} g 
-\tfrac{1}{6} \partial_{p} \partial_{n} \cP^{ij} \cP^{k\ell} \partial_{q} \cP^{mn} \cP^{pq} \partial_{k} \partial_{i} f \partial_{m} \partial_{\ell} \partial_{j} g \\
+\tfrac{1}{6} \cP^{ij} \partial_{q} \partial_{n} \cP^{k\ell} \partial_{\ell} \cP^{mn} \cP^{pq} \partial_{k} \partial_{i} f \partial_{p} \partial_{m} \partial_{j} g 
+\tfrac{1}{18} \partial_{q} \cP^{ij} \partial_{n} \cP^{k\ell} \partial_{\ell} \cP^{mn} \cP^{pq} \partial_{k} \partial_{i} f \partial_{p} \partial_{m} \partial_{j} g \\
-\tfrac{1}{9} \partial_{p} \partial_{n} \cP^{ij} \partial_{q} \cP^{k\ell} \cP^{mn} \cP^{pq} \partial_{k} \partial_{i} f \partial_{m} \partial_{\ell} \partial_{j} g 
-\tfrac{1}{18} \partial_{n} \cP^{ij} \partial_{p} \cP^{k\ell} \partial_{q} \cP^{mn} \cP^{pq} \partial_{k} \partial_{i} f \partial_{m} \partial_{\ell} \partial_{j} g \\
+\tfrac{1}{30} \partial_{q} \partial_{n} \partial_{\ell} \cP^{ij} \cP^{k\ell} \cP^{mn} \cP^{pq} \partial_{k} \partial_{i} f \partial_{p} \partial_{m} \partial_{j} g 
+\tfrac{13}{90} \partial_{q} \partial_{n} \cP^{ij} \cP^{k\ell} \partial_{\ell} \cP^{mn} \cP^{pq} \partial_{k} \partial_{i} f \partial_{p} \partial_{m} \partial_{j} g \\
-\tfrac{1}{45} \partial_{n} \partial_{\ell} \cP^{ij} \cP^{k\ell} \partial_{q} \cP^{mn} \cP^{pq} \partial_{k} \partial_{i} f \partial_{p} \partial_{m} \partial_{j} g 
-\tfrac{1}{30} \partial_{n} \cP^{ij} \cP^{k\ell} \partial_{q} \partial_{\ell} \cP^{mn} \cP^{pq} \partial_{k} \partial_{i} f \partial_{p} \partial_{m} \partial_{j} g \\
+\tfrac{1}{90} \partial_{q} \cP^{ij} \cP^{k\ell} \partial_{\ell} \cP^{mn} \partial_{n} \cP^{pq} \partial_{k} \partial_{i} f \partial_{p} \partial_{m} \partial_{j} g 
+\tfrac{1}{30} \partial_{n} \partial_{\ell} \cP^{ij} \partial_{q} \cP^{k\ell} \cP^{mn} \cP^{pq} \partial_{k} \partial_{i} f \partial_{p} \partial_{m} \partial_{j} g \\
+\tfrac{1}{15} \partial_{n} \cP^{ij} \partial_{q} \cP^{k\ell} \partial_{\ell} \cP^{mn} \cP^{pq} \partial_{k} \partial_{i} f \partial_{p} \partial_{m} \partial_{j} g 
+\tfrac{1}{15} \partial_{\ell} \cP^{ij} \partial_{q} \partial_{n} \cP^{k\ell} \cP^{mn} \cP^{pq} \partial_{k} \partial_{i} f \partial_{p} \partial_{m} \partial_{j} g \\
+\tfrac{1}{90} \partial_{\ell} \cP^{ij} \partial_{n} \cP^{k\ell} \partial_{q} \cP^{mn} \cP^{pq} \partial_{k} \partial_{i} f \partial_{p} \partial_{m} \partial_{j} g 
-\tfrac{1}{6} \partial_{p} \partial_{n} \cP^{ij} \cP^{k\ell} \partial_{q} \cP^{mn} \cP^{pq} \partial_{m} \partial_{k} \partial_{i} f \partial_{\ell} \partial_{j} g \\
-\tfrac{1}{6} \cP^{ij} \partial_{n} \cP^{k\ell} \partial_{q} \partial_{\ell} \cP^{mn} \cP^{pq} \partial_{p} \partial_{k} \partial_{i} f \partial_{m} \partial_{j} g 
-\tfrac{1}{18} \partial_{\ell} \cP^{ij} \cP^{k\ell} \partial_{q} \cP^{mn} \partial_{n} \cP^{pq} \partial_{m} \partial_{k} \partial_{i} f \partial_{p} \partial_{j} g \\
-\tfrac{1}{9} \partial_{p} \partial_{n} \cP^{ij} \partial_{q} \cP^{k\ell} \cP^{mn} \cP^{pq} \partial_{m} \partial_{k} \partial_{i} f \partial_{\ell} \partial_{j} g 
-\tfrac{1}{18} \partial_{n} \cP^{ij} \partial_{p} \cP^{k\ell} \partial_{q} \cP^{mn} \cP^{pq} \partial_{m} \partial_{k} \partial_{i} f \partial_{\ell} \partial_{j} g \\
-\tfrac{1}{30} \partial_{q} \partial_{n} \partial_{\ell} \cP^{ij} \cP^{k\ell} \cP^{mn} \cP^{pq} \partial_{m} \partial_{k} \partial_{i} f \partial_{p} \partial_{j} g 
-\tfrac{13}{90} \partial_{n} \partial_{\ell} \cP^{ij} \partial_{q} \cP^{k\ell} \cP^{mn} \cP^{pq} \partial_{m} \partial_{k} \partial_{i} f \partial_{p} \partial_{j} g \\
+\tfrac{1}{45} \partial_{q} \partial_{\ell} \cP^{ij} \partial_{n} \cP^{k\ell} \cP^{mn} \cP^{pq} \partial_{m} \partial_{k} \partial_{i} f \partial_{p} \partial_{j} g 
+\tfrac{1}{30} \partial_{\ell} \cP^{ij} \partial_{q} \partial_{n} \cP^{k\ell} \cP^{mn} \cP^{pq} \partial_{m} \partial_{k} \partial_{i} f \partial_{p} \partial_{j} g \\
-\tfrac{1}{90} \partial_{\ell} \cP^{ij} \partial_{n} \cP^{k\ell} \partial_{q} \cP^{mn} \cP^{pq} \partial_{m} \partial_{k} \partial_{i} f \partial_{p} \partial_{j} g 
-\tfrac{1}{30} \partial_{q} \partial_{\ell} \cP^{ij} \cP^{k\ell} \cP^{mn} \partial_{n} \cP^{pq} \partial_{m} \partial_{k} \partial_{i} f \partial_{p} \partial_{j} g \\
-\tfrac{1}{15} \partial_{\ell} \cP^{ij} \partial_{q} \cP^{k\ell} \cP^{mn} \partial_{n} \cP^{pq} \partial_{m} \partial_{k} \partial_{i} f \partial_{p} \partial_{j} g 
-\tfrac{1}{15} \partial_{n} \cP^{ij} \cP^{k\ell} \partial_{q} \partial_{\ell} \cP^{mn} \cP^{pq} \partial_{p} \partial_{k} \partial_{i} f \partial_{m} \partial_{j} g \\
-\tfrac{1}{90} \partial_{q} \cP^{ij} \cP^{k\ell} \partial_{\ell} \cP^{mn} \partial_{n} \cP^{pq} \partial_{m} \partial_{k} \partial_{i} f \partial_{p} \partial_{j} g 
+\tfrac{7}{90} \partial_{p} \partial_{n} \partial_{\ell} \cP^{ij} \partial_{q} \cP^{k\ell} \cP^{mn} \cP^{pq} \partial_{i} f \partial_{m} \partial_{k} \partial_{j} g \\
-\tfrac{1}{90} \partial_{p} \partial_{\ell} \cP^{ij} \partial_{q} \partial_{n} \cP^{k\ell} \cP^{mn} \cP^{pq} \partial_{i} f \partial_{m} \partial_{k} \partial_{j} g 
+\tfrac{1}{30} \partial_{p} \partial_{\ell} \cP^{ij} \partial_{n} \cP^{k\ell} \partial_{q} \cP^{mn} \cP^{pq} \partial_{i} f \partial_{m} \partial_{k} \partial_{j} g \\
+\tfrac{2}{45} \partial_{\ell} \cP^{ij} \partial_{p} \partial_{n} \cP^{k\ell} \partial_{q} \cP^{mn} \cP^{pq} \partial_{i} f \partial_{m} \partial_{k} \partial_{j} g 
+\tfrac{1}{30} \partial_{p} \partial_{\ell} \cP^{ij} \partial_{q} \cP^{k\ell} \cP^{mn} \partial_{n} \cP^{pq} \partial_{i} f \partial_{m} \partial_{k} \partial_{j} g \\
+\tfrac{1}{45} \partial_{p} \partial_{\ell} \cP^{ij} \cP^{k\ell} \partial_{q} \cP^{mn} \partial_{n} \cP^{pq} \partial_{i} f \partial_{m} \partial_{k} \partial_{j} g 
-\tfrac{1}{90} \partial_{\ell} \cP^{ij} \partial_{p} \cP^{k\ell} \partial_{q} \cP^{mn} \partial_{n} \cP^{pq} \partial_{i} f \partial_{m} \partial_{k} \partial_{j} g \\
-\tfrac{1}{90} \partial_{p} \cP^{ij} \partial_{q} \cP^{k\ell} \partial_{\ell} \cP^{mn} \partial_{n} \cP^{pq} \partial_{i} f \partial_{m} \partial_{k} \partial_{j} g 
-\tfrac{1}{90} \partial_{p} \cP^{ij} \cP^{k\ell} \partial_{q} \partial_{\ell} \cP^{mn} \partial_{n} \cP^{pq} \partial_{i} f \partial_{m} \partial_{k} \partial_{j} g \\
+\tfrac{1}{30} \partial_{m} \cP^{ij} \partial_{n} \cP^{k\ell} \partial_{q} \partial_{\ell} \cP^{mn} \cP^{pq} \partial_{i} f \partial_{p} \partial_{k} \partial_{j} g 
+\tfrac{1}{90} \partial_{q} \cP^{ij} \partial_{j} \cP^{k\ell} \partial_{\ell} \cP^{mn} \partial_{n} \cP^{pq} \partial_{i} f \partial_{p} \partial_{m} \partial_{k} g \\
+\tfrac{1}{90} \partial_{n} \cP^{ij} \partial_{q} \partial_{j} \cP^{k\ell} \partial_{\ell} \cP^{mn} \cP^{pq} \partial_{i} f \partial_{p} \partial_{m} \partial_{k} g 
+\tfrac{1}{90} \partial_{\ell} \cP^{ij} \partial_{n} \partial_{j} \cP^{k\ell} \partial_{q} \cP^{mn} \cP^{pq} \partial_{i} f \partial_{p} \partial_{m} \partial_{k} g \\
-\tfrac{1}{30} \partial_{n} \cP^{ij} \partial_{j} \cP^{k\ell} \partial_{q} \partial_{\ell} \cP^{mn} \cP^{pq} \partial_{i} f \partial_{p} \partial_{m} \partial_{k} g 
-\tfrac{1}{45} \partial_{q} \partial_{n} \cP^{ij} \partial_{j} \cP^{k\ell} \partial_{\ell} \cP^{mn} \cP^{pq} \partial_{i} f \partial_{p} \partial_{m} \partial_{k} g
\end{multline*}
\begin{multline*}
-\tfrac{1}{60} \partial_{\ell} \cP^{ij} \partial_{q} \partial_{n} \partial_{j} \cP^{k\ell} \cP^{mn} \cP^{pq} \partial_{i} f \partial_{p} \partial_{m} \partial_{k} g 
-\tfrac{1}{45} \partial_{n} \partial_{\ell} \cP^{ij} \partial_{q} \partial_{j} \cP^{k\ell} \cP^{mn} \cP^{pq} \partial_{i} f \partial_{p} \partial_{m} \partial_{k} g \\
-\tfrac{1}{45} \partial_{n} \partial_{\ell} \cP^{ij} \partial_{j} \cP^{k\ell} \partial_{q} \cP^{mn} \cP^{pq} \partial_{i} f \partial_{p} \partial_{m} \partial_{k} g 
-\tfrac{1}{20} \partial_{q} \partial_{n} \partial_{\ell} \cP^{ij} \partial_{j} \cP^{k\ell} \cP^{mn} \cP^{pq} \partial_{i} f \partial_{p} \partial_{m} \partial_{k} g \\
-\tfrac{7}{90} \partial_{p} \partial_{n} \partial_{\ell} \cP^{ij} \partial_{q} \cP^{k\ell} \cP^{mn} \cP^{pq} \partial_{m} \partial_{k} \partial_{i} f \partial_{j} g 
+\tfrac{1}{90} \partial_{p} \partial_{\ell} \cP^{ij} \partial_{q} \partial_{n} \cP^{k\ell} \cP^{mn} \cP^{pq} \partial_{m} \partial_{k} \partial_{i} f \partial_{j} g \\
-\tfrac{1}{30} \partial_{p} \partial_{\ell} \cP^{ij} \partial_{n} \cP^{k\ell} \partial_{q} \cP^{mn} \cP^{pq} \partial_{m} \partial_{k} \partial_{i} f \partial_{j} g 
-\tfrac{2}{45} \partial_{\ell} \cP^{ij} \partial_{p} \partial_{n} \cP^{k\ell} \partial_{q} \cP^{mn} \cP^{pq} \partial_{m} \partial_{k} \partial_{i} f \partial_{j} g \\
-\tfrac{1}{30} \partial_{p} \partial_{\ell} \cP^{ij} \partial_{q} \cP^{k\ell} \cP^{mn} \partial_{n} \cP^{pq} \partial_{m} \partial_{k} \partial_{i} f \partial_{j} g 
-\tfrac{1}{45} \partial_{p} \partial_{\ell} \cP^{ij} \cP^{k\ell} \partial_{q} \cP^{mn} \partial_{n} \cP^{pq} \partial_{m} \partial_{k} \partial_{i} f \partial_{j} g \\
+\tfrac{1}{90} \partial_{\ell} \cP^{ij} \partial_{p} \cP^{k\ell} \partial_{q} \cP^{mn} \partial_{n} \cP^{pq} \partial_{m} \partial_{k} \partial_{i} f \partial_{j} g 
+\tfrac{1}{90} \partial_{p} \cP^{ij} \partial_{q} \cP^{k\ell} \partial_{\ell} \cP^{mn} \partial_{n} \cP^{pq} \partial_{m} \partial_{k} \partial_{i} f \partial_{j} g \\
+\tfrac{1}{90} \partial_{p} \cP^{ij} \cP^{k\ell} \partial_{q} \partial_{\ell} \cP^{mn} \partial_{n} \cP^{pq} \partial_{m} \partial_{k} \partial_{i} f \partial_{j} g 
-\tfrac{1}{30} \partial_{m} \cP^{ij} \partial_{n} \cP^{k\ell} \partial_{q} \partial_{\ell} \cP^{mn} \cP^{pq} \partial_{p} \partial_{k} \partial_{i} f \partial_{j} g \\
+\tfrac{1}{90} \partial_{q} \cP^{ij} \partial_{j} \cP^{k\ell} \partial_{\ell} \cP^{mn} \partial_{n} \cP^{pq} \partial_{m} \partial_{k} \partial_{i} f \partial_{p} g 
+\tfrac{1}{90} \cP^{ij} \partial_{q} \partial_{j} \cP^{k\ell} \partial_{\ell} \cP^{mn} \partial_{n} \cP^{pq} \partial_{m} \partial_{k} \partial_{i} f \partial_{p} g \\
+\tfrac{1}{90} \cP^{ij} \partial_{j} \cP^{k\ell} \partial_{q} \partial_{\ell} \cP^{mn} \partial_{n} \cP^{pq} \partial_{m} \partial_{k} \partial_{i} f \partial_{p} g 
-\tfrac{1}{30} \partial_{n} \cP^{ij} \partial_{q} \partial_{j} \cP^{k\ell} \partial_{\ell} \cP^{mn} \cP^{pq} \partial_{p} \partial_{k} \partial_{i} f \partial_{m} g \\
-\tfrac{1}{45} \partial_{n} \cP^{ij} \partial_{j} \cP^{k\ell} \partial_{q} \partial_{\ell} \cP^{mn} \cP^{pq} \partial_{p} \partial_{k} \partial_{i} f \partial_{m} g 
-\tfrac{1}{60} \cP^{ij} \partial_{q} \partial_{n} \partial_{j} \cP^{k\ell} \partial_{\ell} \cP^{mn} \cP^{pq} \partial_{p} \partial_{k} \partial_{i} f \partial_{m} g \\
-\tfrac{1}{45} \cP^{ij} \partial_{n} \partial_{j} \cP^{k\ell} \partial_{q} \partial_{\ell} \cP^{mn} \cP^{pq} \partial_{p} \partial_{k} \partial_{i} f \partial_{m} g 
-\tfrac{1}{45} \cP^{ij} \partial_{j} \cP^{k\ell} \partial_{q} \partial_{\ell} \cP^{mn} \partial_{n} \cP^{pq} \partial_{p} \partial_{k} \partial_{i} f \partial_{m} g \\
-\tfrac{1}{20} \partial_{\ell} \cP^{ij} \partial_{q} \partial_{n} \partial_{j} \cP^{k\ell} \cP^{mn} \cP^{pq} \partial_{p} \partial_{m} \partial_{i} f \partial_{k} g 
-\tfrac{1}{40} \partial_{p} \partial_{m} \partial_{\ell} \cP^{ij} \partial_{q} \partial_{n} \partial_{j} \cP^{k\ell} \cP^{mn} \cP^{pq} \partial_{i} f \partial_{k} g \\
-\tfrac{1}{72} \partial_{p} \partial_{m} \partial_{\ell} \cP^{ij} \partial_{n} \partial_{j} \cP^{k\ell} \partial_{q} \cP^{mn} \cP^{pq} \partial_{i} f \partial_{k} g 
+\tfrac{1}{72} \partial_{m} \partial_{\ell} \cP^{ij} \partial_{p} \partial_{n} \partial_{j} \cP^{k\ell} \partial_{q} \cP^{mn} \cP^{pq} \partial_{i} f \partial_{k} g \\
+\tfrac{1}{360} \partial_{m} \partial_{\ell} \cP^{ij} \partial_{p} \partial_{j} \cP^{k\ell} \partial_{q} \cP^{mn} \partial_{n} \cP^{pq} \partial_{i} f \partial_{k} g 
-\tfrac{1}{60} \partial_{p} \partial_{m} \cP^{ij} \partial_{q} \partial_{n} \partial_{j} \cP^{k\ell} \partial_{\ell} \cP^{mn} \cP^{pq} \partial_{i} f \partial_{k} g \\
-\tfrac{1}{60} \partial_{p} \partial_{n} \partial_{k} \cP^{ij} \partial_{j} \cP^{k\ell} \partial_{q} \partial_{\ell} \cP^{mn} \cP^{pq} \partial_{i} f \partial_{m} g 
+\tfrac{17}{720} \partial_{m} \cP^{ij} \partial_{p} \partial_{n} \partial_{j} \cP^{k\ell} \partial_{q} \partial_{\ell} \cP^{mn} \cP^{pq} \partial_{i} f \partial_{k} g \\
-\tfrac{17}{720} \partial_{p} \partial_{n} \partial_{k} \cP^{ij} \partial_{q} \partial_{j} \cP^{k\ell} \partial_{\ell} \cP^{mn} \cP^{pq} \partial_{i} f \partial_{m} g 
-\tfrac{1}{180} \partial_{p} \partial_{m} \cP^{ij} \partial_{q} \partial_{j} \cP^{k\ell} \partial_{\ell} \cP^{mn} \partial_{n} \cP^{pq} \partial_{i} f \partial_{k} g \\
+\tfrac{1}{180} \partial_{p} \partial_{n} \cP^{ij} \partial_{j} \cP^{k\ell} \partial_{q} \partial_{k} \cP^{mn} \partial_{\ell} \cP^{pq} \partial_{i} f \partial_{m} g 
+\tfrac{1}{360} \partial_{p} \partial_{m} \cP^{ij} \partial_{j} \cP^{k\ell} \partial_{q} \partial_{\ell} \cP^{mn} \partial_{n} \cP^{pq} \partial_{i} f \partial_{k} g \\
-\tfrac{1}{360} \partial_{n} \cP^{ij} \partial_{p} \partial_{j} \cP^{k\ell} \partial_{q} \partial_{k} \cP^{mn} \partial_{\ell} \cP^{pq} \partial_{i} f \partial_{m} g 
+\tfrac{1}{160} \partial_{m} \cP^{ij} \partial_{p} \partial_{j} \cP^{k\ell} \partial_{q} \partial_{\ell} \cP^{mn} \partial_{n} \cP^{pq} \partial_{i} f \partial_{k} g \\
+\tfrac{1}{160} \partial_{p} \partial_{n} \cP^{ij} \partial_{q} \partial_{j} \cP^{k\ell} \partial_{k} \cP^{mn} \partial_{\ell} \cP^{pq} \partial_{i} f \partial_{m} g 
-\tfrac{17}{1440} \partial_{p} \cP^{ij} \partial_{q} \partial_{m} \partial_{j} \cP^{k\ell} \partial_{\ell} \cP^{mn} \partial_{n} \cP^{pq} \partial_{i} f \partial_{k} g \\
+\tfrac{17}{1440} \partial_{q} \partial_{m} \partial_{k} \cP^{ij} \partial_{j} \cP^{k\ell} \partial_{\ell} \cP^{mn} \partial_{n} \cP^{pq} \partial_{i} f \partial_{p} g 
-\tfrac{1}{360} \partial_{p} \cP^{ij} \partial_{m} \partial_{j} \cP^{k\ell} \partial_{q} \partial_{\ell} \cP^{mn} \partial_{n} \cP^{pq} \partial_{i} f \partial_{k} g \\
+\tfrac{1}{360} \partial_{q} \partial_{k} \cP^{ij} \partial_{m} \partial_{j} \cP^{k\ell} \partial_{\ell} \cP^{mn} \partial_{n} \cP^{pq} \partial_{i} f \partial_{p} g 
-\tfrac{13}{720} \partial_{p} \partial_{k} \cP^{ij} \partial_{q} \partial_{n} \partial_{j} \cP^{k\ell} \partial_{\ell} \cP^{mn} \cP^{pq} \partial_{i} f \partial_{m} g \\
-\tfrac{13}{720} \partial_{k} \cP^{ij} \partial_{p} \partial_{n} \partial_{j} \cP^{k\ell} \partial_{q} \partial_{\ell} \cP^{mn} \cP^{pq} \partial_{i} f \partial_{m} g 
-\tfrac{1}{60} \partial_{p} \partial_{k} \cP^{ij} \partial_{n} \partial_{j} \cP^{k\ell} \partial_{q} \partial_{\ell} \cP^{mn} \cP^{pq} \partial_{i} f \partial_{m} g \\
-\tfrac{7}{720} \partial_{p} \partial_{k} \cP^{ij} \partial_{q} \partial_{j} \cP^{k\ell} \partial_{\ell} \cP^{mn} \partial_{n} \cP^{pq} \partial_{i} f \partial_{m} g 
+\tfrac{7}{720} \partial_{p} \cP^{ij} \partial_{j} \cP^{k\ell} \partial_{q} \partial_{k} \cP^{mn} \partial_{n} \partial_{\ell} \cP^{pq} \partial_{i} f \partial_{m} g \\
-\tfrac{1}{180} \partial_{p} \partial_{k} \cP^{ij} \partial_{j} \cP^{k\ell} \partial_{q} \partial_{\ell} \cP^{mn} \partial_{n} \cP^{pq} \partial_{i} f \partial_{m} g 
+\tfrac{1}{160} \partial_{k} \cP^{ij} \partial_{p} \partial_{j} \cP^{k\ell} \partial_{q} \partial_{\ell} \cP^{mn} \partial_{n} \cP^{pq} \partial_{i} f \partial_{m} g \\
+\tfrac{1}{160} \partial_{m} \partial_{k} \cP^{ij} \partial_{j} \cP^{k\ell} \partial_{q} \partial_{\ell} \cP^{mn} \partial_{n} \cP^{pq} \partial_{i} f \partial_{p} g 
+\tfrac{13}{1440} \partial_{m} \partial_{k} \cP^{ij} \partial_{q} \partial_{j} \cP^{k\ell} \partial_{\ell} \cP^{mn} \partial_{n} \cP^{pq} \partial_{i} f \partial_{p} g \\
-\tfrac{13}{1440} \partial_{p} \cP^{ij} \partial_{n} \partial_{j} \cP^{k\ell} \partial_{q} \partial_{k} \cP^{mn} \partial_{\ell} \cP^{pq} \partial_{i} f \partial_{m} g 
-\tfrac{1}{1440} \partial_{k} \cP^{ij} \partial_{q} \partial_{m} \partial_{j} \cP^{k\ell} \partial_{\ell} \cP^{mn} \partial_{n} \cP^{pq} \partial_{i} f \partial_{p} g \\
+\tfrac{1}{1440} \partial_{p} \cP^{ij} \partial_{q} \partial_{n} \partial_{j} \cP^{k\ell} \partial_{k} \cP^{mn} \partial_{\ell} \cP^{pq} \partial_{i} f \partial_{m} g 
+\tfrac{1}{360} \partial_{k} \cP^{ij} \partial_{m} \partial_{j} \cP^{k\ell} \partial_{q} \partial_{\ell} \cP^{mn} \partial_{n} \cP^{pq} \partial_{i} f \partial_{p} g \\
+\tfrac{1}{240} \partial_{p} \cP^{ij} \partial_{q} \partial_{j} \cP^{k\ell} \partial_{k} \cP^{mn} \partial_{n} \partial_{\ell} \cP^{pq} \partial_{i} f \partial_{m} g 
-\tfrac{13}{360} \partial_{p} \partial_{m} \partial_{\ell} \cP^{ij} \partial_{q} \partial_{n} \cP^{k\ell} \cP^{mn} \cP^{pq} \partial_{i} f \partial_{k} \partial_{j} g \\
-\tfrac{1}{720} \partial_{p} \partial_{m} \partial_{\ell} \cP^{ij} \partial_{n} \cP^{k\ell} \partial_{q} \cP^{mn} \cP^{pq} \partial_{i} f \partial_{k} \partial_{j} g 
+\tfrac{1}{30} \partial_{m} \partial_{\ell} \cP^{ij} \partial_{p} \partial_{n} \cP^{k\ell} \partial_{q} \cP^{mn} \cP^{pq} \partial_{i} f \partial_{k} \partial_{j} g \\
-\tfrac{1}{720} \partial_{m} \partial_{\ell} \cP^{ij} \partial_{p} \cP^{k\ell} \partial_{q} \cP^{mn} \partial_{n} \cP^{pq} \partial_{i} f \partial_{k} \partial_{j} g 
-\tfrac{19}{720} \partial_{p} \partial_{m} \cP^{ij} \partial_{q} \partial_{n} \cP^{k\ell} \partial_{\ell} \cP^{mn} \cP^{pq} \partial_{i} f \partial_{k} \partial_{j} g \\
+\tfrac{1}{180} \partial_{p} \partial_{m} \cP^{ij} \partial_{n} \cP^{k\ell} \partial_{q} \partial_{\ell} \cP^{mn} \cP^{pq} \partial_{i} f \partial_{k} \partial_{j} g 
+\tfrac{13}{360} \partial_{m} \cP^{ij} \partial_{p} \partial_{n} \cP^{k\ell} \partial_{q} \partial_{\ell} \cP^{mn} \cP^{pq} \partial_{i} f \partial_{k} \partial_{j} g
\end{multline*}
\begin{multline*}
-\tfrac{1}{720} \partial_{p} \partial_{m} \cP^{ij} \partial_{q} \cP^{k\ell} \partial_{\ell} \cP^{mn} \partial_{n} \cP^{pq} \partial_{i} f \partial_{k} \partial_{j} g 
+\tfrac{1}{360} \partial_{p} \partial_{m} \cP^{ij} \cP^{k\ell} \partial_{q} \partial_{\ell} \cP^{mn} \partial_{n} \cP^{pq} \partial_{i} f \partial_{k} \partial_{j} g \\
+\tfrac{1}{720} \partial_{m} \cP^{ij} \partial_{p} \cP^{k\ell} \partial_{q} \partial_{\ell} \cP^{mn} \partial_{n} \cP^{pq} \partial_{i} f \partial_{k} \partial_{j} g 
-\tfrac{1}{80} \partial_{p} \cP^{ij} \partial_{q} \partial_{m} \cP^{k\ell} \partial_{\ell} \cP^{mn} \partial_{n} \cP^{pq} \partial_{i} f \partial_{k} \partial_{j} g \\
-\tfrac{1}{180} \partial_{p} \cP^{ij} \partial_{m} \cP^{k\ell} \partial_{q} \partial_{\ell} \cP^{mn} \partial_{n} \cP^{pq} \partial_{i} f \partial_{k} \partial_{j} g 
-\tfrac{1}{36} \partial_{n} \cP^{ij} \partial_{p} \partial_{j} \cP^{k\ell} \partial_{q} \partial_{\ell} \cP^{mn} \cP^{pq} \partial_{i} f \partial_{m} \partial_{k} g \\
+\tfrac{1}{60} \partial_{p} \partial_{n} \cP^{ij} \partial_{q} \partial_{j} \cP^{k\ell} \partial_{\ell} \cP^{mn} \cP^{pq} \partial_{i} f \partial_{m} \partial_{k} g 
-\tfrac{1}{720} \partial_{p} \partial_{n} \cP^{ij} \partial_{j} \cP^{k\ell} \partial_{q} \partial_{\ell} \cP^{mn} \cP^{pq} \partial_{i} f \partial_{m} \partial_{k} g \\
+\tfrac{1}{45} \partial_{\ell} \cP^{ij} \partial_{p} \partial_{n} \partial_{j} \cP^{k\ell} \partial_{q} \cP^{mn} \cP^{pq} \partial_{i} f \partial_{m} \partial_{k} g 
+\tfrac{17}{720} \partial_{p} \partial_{\ell} \cP^{ij} \partial_{n} \partial_{j} \cP^{k\ell} \partial_{q} \cP^{mn} \cP^{pq} \partial_{i} f \partial_{m} \partial_{k} g \\
+\tfrac{13}{360} \partial_{p} \partial_{\ell} \cP^{ij} \partial_{q} \partial_{n} \partial_{j} \cP^{k\ell} \cP^{mn} \cP^{pq} \partial_{i} f \partial_{m} \partial_{k} g 
+\tfrac{1}{120} \partial_{n} \cP^{ij} \partial_{j} \cP^{k\ell} \partial_{q} \partial_{k} \cP^{mn} \partial_{\ell} \cP^{pq} \partial_{i} f \partial_{p} \partial_{m} g \\
+\tfrac{1}{240} \partial_{q} \partial_{k} \cP^{ij} \partial_{j} \cP^{k\ell} \partial_{\ell} \cP^{mn} \partial_{n} \cP^{pq} \partial_{i} f \partial_{p} \partial_{m} g 
-\tfrac{1}{80} \partial_{n} \partial_{k} \cP^{ij} \partial_{j} \cP^{k\ell} \partial_{q} \partial_{\ell} \cP^{mn} \cP^{pq} \partial_{i} f \partial_{p} \partial_{m} g \\
-\tfrac{1}{72} \partial_{p} \cP^{ij} \partial_{q} \partial_{j} \cP^{k\ell} \partial_{\ell} \cP^{mn} \partial_{n} \cP^{pq} \partial_{i} f \partial_{m} \partial_{k} g 
-\tfrac{1}{90} \partial_{m} \cP^{ij} \partial_{q} \partial_{j} \cP^{k\ell} \partial_{\ell} \cP^{mn} \partial_{n} \cP^{pq} \partial_{i} f \partial_{p} \partial_{k} g \\
-\tfrac{1}{180} \partial_{k} \cP^{ij} \partial_{q} \partial_{j} \cP^{k\ell} \partial_{\ell} \cP^{mn} \partial_{n} \cP^{pq} \partial_{i} f \partial_{p} \partial_{m} g 
-\tfrac{1}{360} \partial_{p} \cP^{ij} \partial_{j} \cP^{k\ell} \partial_{q} \partial_{\ell} \cP^{mn} \partial_{n} \cP^{pq} \partial_{i} f \partial_{m} \partial_{k} g \\
+\tfrac{1}{720} \partial_{m} \cP^{ij} \partial_{q} \partial_{n} \partial_{j} \cP^{k\ell} \partial_{\ell} \cP^{mn} \cP^{pq} \partial_{i} f \partial_{p} \partial_{k} g 
-\tfrac{1}{180} \partial_{k} \cP^{ij} \partial_{n} \partial_{j} \cP^{k\ell} \partial_{q} \partial_{\ell} \cP^{mn} \cP^{pq} \partial_{i} f \partial_{p} \partial_{m} g \\
+\tfrac{17}{180} \partial_{p} \partial_{n} \partial_{\ell} \cP^{ij} \partial_{q} \partial_{j} \cP^{k\ell} \cP^{mn} \cP^{pq} \partial_{i} f \partial_{m} \partial_{k} g 
-\tfrac{1}{72} \partial_{n} \partial_{\ell} \cP^{ij} \partial_{p} \partial_{j} \cP^{k\ell} \partial_{q} \cP^{mn} \cP^{pq} \partial_{i} f \partial_{m} \partial_{k} g \\
+\tfrac{7}{180} \partial_{p} \partial_{n} \partial_{\ell} \cP^{ij} \partial_{j} \cP^{k\ell} \partial_{q} \cP^{mn} \cP^{pq} \partial_{i} f \partial_{m} \partial_{k} g 
+\tfrac{7}{180} \partial_{p} \partial_{\ell} \cP^{ij} \partial_{q} \partial_{j} \cP^{k\ell} \cP^{mn} \partial_{n} \cP^{pq} \partial_{i} f \partial_{m} \partial_{k} g \\
-\tfrac{1}{180} \partial_{\ell} \cP^{ij} \partial_{p} \partial_{j} \cP^{k\ell} \partial_{q} \cP^{mn} \partial_{n} \cP^{pq} \partial_{i} f \partial_{m} \partial_{k} g 
+\tfrac{1}{90} \partial_{p} \partial_{\ell} \cP^{ij} \partial_{j} \cP^{k\ell} \partial_{q} \cP^{mn} \partial_{n} \cP^{pq} \partial_{i} f \partial_{m} \partial_{k} g \\
+\tfrac{1}{80} \partial_{q} \partial_{m} \cP^{ij} \partial_{j} \cP^{k\ell} \partial_{\ell} \cP^{mn} \partial_{n} \cP^{pq} \partial_{i} f \partial_{p} \partial_{k} g 
+\tfrac{1}{120} \partial_{q} \cP^{ij} \partial_{m} \partial_{j} \cP^{k\ell} \partial_{\ell} \cP^{mn} \partial_{n} \cP^{pq} \partial_{i} f \partial_{p} \partial_{k} g \\
+\tfrac{1}{40} \partial_{q} \partial_{m} \cP^{ij} \partial_{n} \partial_{j} \cP^{k\ell} \partial_{\ell} \cP^{mn} \cP^{pq} \partial_{i} f \partial_{p} \partial_{k} g 
-\tfrac{1}{360} \partial_{q} \partial_{n} \partial_{k} \cP^{ij} \partial_{j} \cP^{k\ell} \partial_{\ell} \cP^{mn} \cP^{pq} \partial_{i} f \partial_{p} \partial_{m} g \\
-\tfrac{11}{720} \partial_{n} \partial_{k} \cP^{ij} \partial_{q} \partial_{j} \cP^{k\ell} \partial_{\ell} \cP^{mn} \cP^{pq} \partial_{i} f \partial_{p} \partial_{m} g 
+\tfrac{7}{240} \partial_{m} \cP^{ij} \partial_{n} \partial_{j} \cP^{k\ell} \partial_{q} \partial_{\ell} \cP^{mn} \cP^{pq} \partial_{i} f \partial_{p} \partial_{k} g \\
-\tfrac{1}{180} \partial_{n} \cP^{ij} \partial_{q} \partial_{j} \cP^{k\ell} \partial_{k} \cP^{mn} \partial_{\ell} \cP^{pq} \partial_{i} f \partial_{p} \partial_{m} g 
+\tfrac{1}{60} \partial_{m} \cP^{ij} \partial_{j} \cP^{k\ell} \partial_{q} \partial_{\ell} \cP^{mn} \partial_{n} \cP^{pq} \partial_{i} f \partial_{p} \partial_{k} g \\
+\tfrac{1}{90} \partial_{q} \partial_{k} \cP^{ij} \partial_{n} \partial_{j} \cP^{k\ell} \partial_{\ell} \cP^{mn} \cP^{pq} \partial_{i} f \partial_{p} \partial_{m} g 
+\tfrac{1}{60} \partial_{k} \cP^{ij} \partial_{q} \partial_{n} \partial_{j} \cP^{k\ell} \partial_{\ell} \cP^{mn} \cP^{pq} \partial_{i} f \partial_{p} \partial_{m} g \\
+\tfrac{13}{360} \partial_{p} \partial_{m} \partial_{\ell} \cP^{ij} \partial_{q} \partial_{n} \cP^{k\ell} \cP^{mn} \cP^{pq} \partial_{k} \partial_{i} f \partial_{j} g 
+\tfrac{1}{720} \partial_{p} \partial_{m} \partial_{\ell} \cP^{ij} \partial_{n} \cP^{k\ell} \partial_{q} \cP^{mn} \cP^{pq} \partial_{k} \partial_{i} f \partial_{j} g \\
-\tfrac{1}{30} \partial_{m} \partial_{\ell} \cP^{ij} \partial_{p} \partial_{n} \cP^{k\ell} \partial_{q} \cP^{mn} \cP^{pq} \partial_{k} \partial_{i} f \partial_{j} g 
+\tfrac{1}{720} \partial_{m} \partial_{\ell} \cP^{ij} \partial_{p} \cP^{k\ell} \partial_{q} \cP^{mn} \partial_{n} \cP^{pq} \partial_{k} \partial_{i} f \partial_{j} g \\
+\tfrac{19}{720} \partial_{p} \partial_{m} \cP^{ij} \partial_{q} \partial_{n} \cP^{k\ell} \partial_{\ell} \cP^{mn} \cP^{pq} \partial_{k} \partial_{i} f \partial_{j} g 
-\tfrac{1}{180} \partial_{p} \partial_{m} \cP^{ij} \partial_{n} \cP^{k\ell} \partial_{q} \partial_{\ell} \cP^{mn} \cP^{pq} \partial_{k} \partial_{i} f \partial_{j} g \\
-\tfrac{13}{360} \partial_{m} \cP^{ij} \partial_{p} \partial_{n} \cP^{k\ell} \partial_{q} \partial_{\ell} \cP^{mn} \cP^{pq} \partial_{k} \partial_{i} f \partial_{j} g 
+\tfrac{1}{720} \partial_{p} \partial_{m} \cP^{ij} \partial_{q} \cP^{k\ell} \partial_{\ell} \cP^{mn} \partial_{n} \cP^{pq} \partial_{k} \partial_{i} f \partial_{j} g \\
-\tfrac{1}{360} \partial_{p} \partial_{m} \cP^{ij} \cP^{k\ell} \partial_{q} \partial_{\ell} \cP^{mn} \partial_{n} \cP^{pq} \partial_{k} \partial_{i} f \partial_{j} g 
-\tfrac{1}{720} \partial_{m} \cP^{ij} \partial_{p} \cP^{k\ell} \partial_{q} \partial_{\ell} \cP^{mn} \partial_{n} \cP^{pq} \partial_{k} \partial_{i} f \partial_{j} g \\
+\tfrac{1}{80} \partial_{p} \cP^{ij} \partial_{q} \partial_{m} \cP^{k\ell} \partial_{\ell} \cP^{mn} \partial_{n} \cP^{pq} \partial_{k} \partial_{i} f \partial_{j} g 
+\tfrac{1}{180} \partial_{p} \cP^{ij} \partial_{m} \cP^{k\ell} \partial_{q} \partial_{\ell} \cP^{mn} \partial_{n} \cP^{pq} \partial_{k} \partial_{i} f \partial_{j} g \\
-\tfrac{1}{36} \partial_{p} \partial_{n} \cP^{ij} \partial_{q} \partial_{j} \cP^{k\ell} \partial_{\ell} \cP^{mn} \cP^{pq} \partial_{k} \partial_{i} f \partial_{m} g 
-\tfrac{1}{60} \partial_{p} \partial_{n} \cP^{ij} \partial_{j} \cP^{k\ell} \partial_{q} \partial_{\ell} \cP^{mn} \cP^{pq} \partial_{k} \partial_{i} f \partial_{m} g \\
+\tfrac{1}{720} \partial_{n} \cP^{ij} \partial_{p} \partial_{j} \cP^{k\ell} \partial_{q} \partial_{\ell} \cP^{mn} \cP^{pq} \partial_{k} \partial_{i} f \partial_{m} g 
-\tfrac{1}{45} \partial_{p} \cP^{ij} \partial_{q} \partial_{n} \partial_{j} \cP^{k\ell} \partial_{\ell} \cP^{mn} \cP^{pq} \partial_{k} \partial_{i} f \partial_{m} g \\
-\tfrac{17}{720} \partial_{p} \cP^{ij} \partial_{n} \partial_{j} \cP^{k\ell} \partial_{q} \partial_{\ell} \cP^{mn} \cP^{pq} \partial_{k} \partial_{i} f \partial_{m} g 
-\tfrac{13}{360} \cP^{ij} \partial_{p} \partial_{n} \partial_{j} \cP^{k\ell} \partial_{q} \partial_{\ell} \cP^{mn} \cP^{pq} \partial_{k} \partial_{i} f \partial_{m} g \\
-\tfrac{1}{120} \partial_{p} \cP^{ij} \partial_{q} \partial_{j} \cP^{k\ell} \partial_{\ell} \cP^{mn} \partial_{n} \cP^{pq} \partial_{k} \partial_{i} f \partial_{m} g 
-\tfrac{1}{240} \partial_{p} \cP^{ij} \partial_{j} \cP^{k\ell} \partial_{q} \partial_{\ell} \cP^{mn} \partial_{n} \cP^{pq} \partial_{k} \partial_{i} f \partial_{m} g \\
+\tfrac{1}{80} \cP^{ij} \partial_{p} \partial_{j} \cP^{k\ell} \partial_{q} \partial_{\ell} \cP^{mn} \partial_{n} \cP^{pq} \partial_{k} \partial_{i} f \partial_{m} g 
+\tfrac{1}{72} \partial_{q} \partial_{m} \cP^{ij} \partial_{j} \cP^{k\ell} \partial_{\ell} \cP^{mn} \partial_{n} \cP^{pq} \partial_{k} \partial_{i} f \partial_{p} g \\
+\tfrac{1}{90} \partial_{m} \cP^{ij} \partial_{q} \partial_{j} \cP^{k\ell} \partial_{\ell} \cP^{mn} \partial_{n} \cP^{pq} \partial_{k} \partial_{i} f \partial_{p} g 
+\tfrac{1}{180} \partial_{m} \cP^{ij} \partial_{j} \cP^{k\ell} \partial_{q} \partial_{\ell} \cP^{mn} \partial_{n} \cP^{pq} \partial_{k} \partial_{i} f \partial_{p} g \\
+\tfrac{1}{360} \partial_{q} \cP^{ij} \partial_{m} \partial_{j} \cP^{k\ell} \partial_{\ell} \cP^{mn} \partial_{n} \cP^{pq} \partial_{k} \partial_{i} f \partial_{p} g 
-\tfrac{1}{720} \cP^{ij} \partial_{q} \partial_{m} \partial_{j} \cP^{k\ell} \partial_{\ell} \cP^{mn} \partial_{n} \cP^{pq} \partial_{k} \partial_{i} f \partial_{p} g 
\end{multline*}
\begin{multline}
\label{EqStarWith10Pym}
+\tfrac{1}{180} \cP^{ij} \partial_{m} \partial_{j} \cP^{k\ell} \partial_{q} \partial_{\ell} \cP^{mn} \partial_{n} \cP^{pq} \partial_{k} \partial_{i} f \partial_{p} g
-\tfrac{17}{180} \partial_{p} \partial_{\ell} \cP^{ij} \partial_{q} \partial_{n} \partial_{j} \cP^{k\ell} \cP^{mn} \cP^{pq} \partial_{m} \partial_{i} f \partial_{k} g \\
-\tfrac{1}{72} \partial_{p} \partial_{\ell} \cP^{ij} \partial_{n} \partial_{j} \cP^{k\ell} \partial_{q} \cP^{mn} \cP^{pq} \partial_{m} \partial_{i} f \partial_{k} g 
+\tfrac{7}{180} \partial_{\ell} \cP^{ij} \partial_{p} \partial_{n} \partial_{j} \cP^{k\ell} \partial_{q} \cP^{mn} \cP^{pq} \partial_{m} \partial_{i} f \partial_{k} g \\
-\tfrac{7}{180} \partial_{p} \partial_{\ell} \cP^{ij} \partial_{q} \partial_{j} \cP^{k\ell} \cP^{mn} \partial_{n} \cP^{pq} \partial_{m} \partial_{i} f \partial_{k} g 
-\tfrac{1}{180} \partial_{p} \partial_{\ell} \cP^{ij} \partial_{j} \cP^{k\ell} \partial_{q} \cP^{mn} \partial_{n} \cP^{pq} \partial_{m} \partial_{i} f \partial_{k} g \\
+\tfrac{1}{90} \partial_{\ell} \cP^{ij} \partial_{p} \partial_{j} \cP^{k\ell} \partial_{q} \cP^{mn} \partial_{n} \cP^{pq} \partial_{m} \partial_{i} f \partial_{k} g 
-\tfrac{1}{80} \partial_{p} \cP^{ij} \partial_{q} \partial_{j} \cP^{k\ell} \partial_{\ell} \cP^{mn} \partial_{n} \cP^{pq} \partial_{m} \partial_{i} f \partial_{k} g \\
-\tfrac{1}{120} \partial_{p} \cP^{ij} \partial_{j} \cP^{k\ell} \partial_{q} \partial_{\ell} \cP^{mn} \partial_{n} \cP^{pq} \partial_{m} \partial_{i} f \partial_{k} g 
+\tfrac{1}{40} \cP^{ij} \partial_{p} \partial_{j} \cP^{k\ell} \partial_{q} \partial_{\ell} \cP^{mn} \partial_{n} \cP^{pq} \partial_{m} \partial_{i} f \partial_{k} g \\
+\tfrac{1}{360} \partial_{m} \cP^{ij} \partial_{q} \partial_{n} \partial_{j} \cP^{k\ell} \partial_{\ell} \cP^{mn} \cP^{pq} \partial_{p} \partial_{i} f \partial_{k} g 
+\tfrac{11}{720} \partial_{m} \cP^{ij} \partial_{n} \partial_{j} \cP^{k\ell} \partial_{q} \partial_{\ell} \cP^{mn} \cP^{pq} \partial_{p} \partial_{i} f \partial_{k} g \\
-\tfrac{7}{240} \partial_{n} \partial_{k} \cP^{ij} \partial_{q} \partial_{j} \cP^{k\ell} \partial_{\ell} \cP^{mn} \cP^{pq} \partial_{p} \partial_{i} f \partial_{m} g 
-\tfrac{1}{180} \partial_{m} \cP^{ij} \partial_{j} \cP^{k\ell} \partial_{q} \partial_{\ell} \cP^{mn} \partial_{n} \cP^{pq} \partial_{p} \partial_{i} f \partial_{k} g \\
-\tfrac{1}{60} \partial_{k} \cP^{ij} \partial_{q} \partial_{j} \cP^{k\ell} \partial_{\ell} \cP^{mn} \partial_{n} \cP^{pq} \partial_{p} \partial_{i} f \partial_{m} g 
+\tfrac{1}{90} \cP^{ij} \partial_{m} \partial_{j} \cP^{k\ell} \partial_{q} \partial_{\ell} \cP^{mn} \partial_{n} \cP^{pq} \partial_{p} \partial_{i} f \partial_{k} g \\
-\tfrac{1}{60} \partial_{k} \cP^{ij} \partial_{q} \partial_{n} \partial_{j} \cP^{k\ell} \partial_{\ell} \cP^{mn} \cP^{pq} \partial_{p} \partial_{i} f \partial_{m} g 
-\tfrac{1}{6} \cP^{ij} \partial_{p} \partial_{n} \cP^{k\ell} \partial_{q} \partial_{\ell} \cP^{mn} \cP^{pq} \partial_{k} \partial_{i} f \partial_{m} \partial_{j} g \\
+\tfrac{1}{72} \partial_{\ell} \cP^{ij} \partial_{j} \cP^{k\ell} \partial_{q} \cP^{mn} \partial_{n} \cP^{pq} \partial_{m} \partial_{i} f \partial_{p} \partial_{k} g 
+\tfrac{17}{360} \partial_{p} \partial_{m} \cP^{ij} \partial_{q} \partial_{n} \cP^{k\ell} \cP^{mn} \cP^{pq} \partial_{k} \partial_{i} f \partial_{\ell} \partial_{j} g \\
+\tfrac{1}{24} \partial_{p} \partial_{m} \cP^{ij} \partial_{n} \cP^{k\ell} \partial_{q} \cP^{mn} \cP^{pq} \partial_{k} \partial_{i} f \partial_{\ell} \partial_{j} g 
+\tfrac{1}{180} \partial_{m} \cP^{ij} \partial_{p} \cP^{k\ell} \partial_{q} \cP^{mn} \partial_{n} \cP^{pq} \partial_{k} \partial_{i} f \partial_{\ell} \partial_{j} g \\
+\tfrac{2}{45} \partial_{p} \partial_{n} \partial_{\ell} \cP^{ij} \partial_{q} \cP^{k\ell} \cP^{mn} \cP^{pq} \partial_{k} \partial_{i} f \partial_{m} \partial_{j} g 
-\tfrac{2}{45} \partial_{p} \partial_{n} \partial_{\ell} \cP^{ij} \cP^{k\ell} \partial_{q} \cP^{mn} \cP^{pq} \partial_{k} \partial_{i} f \partial_{m} \partial_{j} g \\
-\tfrac{1}{30} \partial_{n} \partial_{\ell} \cP^{ij} \partial_{p} \cP^{k\ell} \partial_{q} \cP^{mn} \cP^{pq} \partial_{k} \partial_{i} f \partial_{m} \partial_{j} g 
+\tfrac{1}{8} \partial_{p} \partial_{\ell} \cP^{ij} \partial_{q} \partial_{n} \cP^{k\ell} \cP^{mn} \cP^{pq} \partial_{k} \partial_{i} f \partial_{m} \partial_{j} g \\
-\tfrac{1}{8} \partial_{p} \partial_{n} \cP^{ij} \cP^{k\ell} \partial_{q} \partial_{\ell} \cP^{mn} \cP^{pq} \partial_{k} \partial_{i} f \partial_{m} \partial_{j} g 
+\tfrac{1}{720} \partial_{p} \partial_{\ell} \cP^{ij} \partial_{n} \cP^{k\ell} \partial_{q} \cP^{mn} \cP^{pq} \partial_{k} \partial_{i} f \partial_{m} \partial_{j} g \\
-\tfrac{1}{720} \partial_{p} \partial_{n} \cP^{ij} \partial_{q} \cP^{k\ell} \partial_{\ell} \cP^{mn} \cP^{pq} \partial_{k} \partial_{i} f \partial_{m} \partial_{j} g 
-\tfrac{11}{180} \partial_{\ell} \cP^{ij} \partial_{p} \partial_{n} \cP^{k\ell} \partial_{q} \cP^{mn} \cP^{pq} \partial_{k} \partial_{i} f \partial_{m} \partial_{j} g \\
-\tfrac{11}{180} \partial_{n} \cP^{ij} \partial_{p} \cP^{k\ell} \partial_{q} \partial_{\ell} \cP^{mn} \cP^{pq} \partial_{k} \partial_{i} f \partial_{m} \partial_{j} g 
+\tfrac{1}{36} \partial_{p} \partial_{\ell} \cP^{ij} \partial_{q} \cP^{k\ell} \cP^{mn} \partial_{n} \cP^{pq} \partial_{k} \partial_{i} f \partial_{m} \partial_{j} g \\
-\tfrac{1}{36} \partial_{q} \partial_{m} \cP^{ij} \cP^{k\ell} \partial_{\ell} \cP^{mn} \partial_{n} \cP^{pq} \partial_{k} \partial_{i} f \partial_{p} \partial_{j} g 
+\tfrac{1}{90} \partial_{p} \partial_{\ell} \cP^{ij} \cP^{k\ell} \partial_{q} \cP^{mn} \partial_{n} \cP^{pq} \partial_{k} \partial_{i} f \partial_{m} \partial_{j} g \\
-\tfrac{1}{90} \partial_{q} \partial_{m} \cP^{ij} \partial_{n} \cP^{k\ell} \partial_{\ell} \cP^{mn} \cP^{pq} \partial_{k} \partial_{i} f \partial_{p} \partial_{j} g 
-\tfrac{1}{180} \partial_{\ell} \cP^{ij} \partial_{p} \cP^{k\ell} \partial_{q} \cP^{mn} \partial_{n} \cP^{pq} \partial_{k} \partial_{i} f \partial_{m} \partial_{j} g \\
-\tfrac{1}{180} \partial_{q} \cP^{ij} \partial_{m} \cP^{k\ell} \partial_{\ell} \cP^{mn} \partial_{n} \cP^{pq} \partial_{k} \partial_{i} f \partial_{p} \partial_{j} g 
+\tfrac{1}{18} \partial_{p} \cP^{ij} \partial_{q} \partial_{n} \cP^{k\ell} \partial_{\ell} \cP^{mn} \cP^{pq} \partial_{k} \partial_{i} f \partial_{m} \partial_{j} g \\
-\tfrac{1}{18} \partial_{p} \cP^{ij} \partial_{n} \cP^{k\ell} \partial_{q} \partial_{\ell} \cP^{mn} \cP^{pq} \partial_{k} \partial_{i} f \partial_{m} \partial_{j} g 
+\tfrac{1}{144} \partial_{p} \cP^{ij} \partial_{q} \cP^{k\ell} \partial_{\ell} \cP^{mn} \partial_{n} \cP^{pq} \partial_{k} \partial_{i} f \partial_{m} \partial_{j} g \\
-\tfrac{1}{144} \partial_{m} \cP^{ij} \partial_{q} \cP^{k\ell} \partial_{\ell} \cP^{mn} \partial_{n} \cP^{pq} \partial_{k} \partial_{i} f \partial_{p} \partial_{j} g 
-\tfrac{1}{90} \partial_{p} \cP^{ij} \cP^{k\ell} \partial_{q} \partial_{\ell} \cP^{mn} \partial_{n} \cP^{pq} \partial_{k} \partial_{i} f \partial_{m} \partial_{j} g \\
+\tfrac{1}{90} \partial_{m} \cP^{ij} \partial_{q} \partial_{n} \cP^{k\ell} \partial_{\ell} \cP^{mn} \cP^{pq} \partial_{k} \partial_{i} f \partial_{p} \partial_{j} g 
-\tfrac{1}{60} \partial_{m} \cP^{ij} \partial_{n} \cP^{k\ell} \partial_{q} \partial_{\ell} \cP^{mn} \cP^{pq} \partial_{k} \partial_{i} f \partial_{p} \partial_{j} g \\
+\tfrac{1}{60} \partial_{m} \cP^{ij} \cP^{k\ell} \partial_{q} \partial_{\ell} \cP^{mn} \partial_{n} \cP^{pq} \partial_{k} \partial_{i} f \partial_{p} \partial_{j} g 
-\tfrac{1}{240} \partial_{q} \partial_{n} \cP^{ij} \partial_{j} \cP^{k\ell} \partial_{\ell} \cP^{mn} \cP^{pq} \partial_{k} \partial_{i} f \partial_{p} \partial_{m} g \\
-\tfrac{1}{240} \partial_{n} \cP^{ij} \partial_{q} \partial_{j} \cP^{k\ell} \partial_{\ell} \cP^{mn} \cP^{pq} \partial_{p} \partial_{i} f \partial_{m} \partial_{k} g 
-\tfrac{13}{720} \partial_{n} \cP^{ij} \partial_{q} \partial_{j} \cP^{k\ell} \partial_{\ell} \cP^{mn} \cP^{pq} \partial_{k} \partial_{i} f \partial_{p} \partial_{m} g \\
-\tfrac{13}{720} \cP^{ij} \partial_{q} \partial_{j} \cP^{k\ell} \partial_{\ell} \cP^{mn} \partial_{n} \cP^{pq} \partial_{m} \partial_{i} f \partial_{p} \partial_{k} g 
-\tfrac{1}{90} \partial_{n} \cP^{ij} \partial_{j} \cP^{k\ell} \partial_{q} \partial_{\ell} \cP^{mn} \cP^{pq} \partial_{k} \partial_{i} f \partial_{p} \partial_{m} g \\
-\tfrac{1}{90} \cP^{ij} \partial_{q} \partial_{j} \cP^{k\ell} \partial_{\ell} \cP^{mn} \partial_{n} \cP^{pq} \partial_{k} \partial_{i} f \partial_{p} \partial_{m} g 
+\tfrac{1}{60} \partial_{q} \cP^{ij} \partial_{n} \partial_{j} \cP^{k\ell} \partial_{\ell} \cP^{mn} \cP^{pq} \partial_{k} \partial_{i} f \partial_{p} \partial_{m} g \\
+\tfrac{1}{60} \partial_{\ell} \cP^{ij} \partial_{q} \partial_{j} \cP^{k\ell} \cP^{mn} \partial_{n} \cP^{pq} \partial_{m} \partial_{i} f \partial_{p} \partial_{k} g 
+\tfrac{1}{30} \cP^{ij} \partial_{q} \partial_{n} \partial_{j} \cP^{k\ell} \partial_{\ell} \cP^{mn} \cP^{pq} \partial_{k} \partial_{i} f \partial_{p} \partial_{m} g \\
+\tfrac{1}{30} \partial_{\ell} \cP^{ij} \partial_{q} \partial_{n} \partial_{j} \cP^{k\ell} \cP^{mn} \cP^{pq} \partial_{m} \partial_{i} f \partial_{p} \partial_{k} g 
-\tfrac{1}{90} \cP^{ij} \partial_{n} \partial_{j} \cP^{k\ell} \partial_{q} \partial_{\ell} \cP^{mn} \cP^{pq} \partial_{k} \partial_{i} f \partial_{p} \partial_{m} g \\
+\tfrac{1}{360} \partial_{q} \cP^{ij} \partial_{j} \cP^{k\ell} \partial_{\ell} \cP^{mn} \partial_{n} \cP^{pq} \partial_{k} \partial_{i} f \partial_{p} \partial_{m} g 
+\tfrac{13}{90} \partial_{q} \partial_{\ell} \cP^{ij} \partial_{n} \partial_{j} \cP^{k\ell} \cP^{mn} \cP^{pq} \partial_{m} \partial_{i} f \partial_{p} \partial_{k} g \\
+\tfrac{13}{180} \partial_{\ell} \cP^{ij} \partial_{n} \partial_{j} \cP^{k\ell} \partial_{q} \cP^{mn} \cP^{pq} \partial_{m} \partial_{i} f \partial_{p} \partial_{k} g 
+\tfrac{13}{180} \partial_{q} \partial_{\ell} \cP^{ij} \partial_{j} \cP^{k\ell} \cP^{mn} \partial_{n} \cP^{pq} \partial_{m} \partial_{i} f \partial_{p} \partial_{k} g \\
+\tfrac{1}{72} \partial_{q} \cP^{ij} \partial_{j} \cP^{k\ell} \partial_{\ell} \cP^{mn} \partial_{n} \cP^{pq} \partial_{m} \partial_{i} f \partial_{p} \partial_{k} g 
\big)
+ \bar{o}(\hbar^4).
\end{multline}
\label{EqStarWith10PymEnd}%
Out of $247$ graphs at $\hbar^4$, as many as $138$ contain two-cycles (or ``eyes'', as in Fig.~\ref{FigTadpoleEye} on p.~\pageref{FigTadpoleEye}), cf. expansion~\eqref{EqStarOrd3} up to $\bar{o}(\hbar^3)$ on p.~\pageref{EqStarOrd3}.
\label{LoopsDominant}


\newpage

\setcounter{page}{1}
\renewcommand{\thepage}{\roman{page}}

\section{\texttt{C++} classes and methods}\label{AppCPP}


\centerline{\textsc{Class} {\tt KontsevichGraph}}
\vskip 1em

\noindent
Summary: a (signed) Kontsevich graph.

\vskip 1em

\noindent
Data members (private):
\begin{verbatim}
    size_t d_internal = 0;
    size_t d_external = 0;
    std::vector< std::pair<char, char> > d_targets;
    int d_sign = 1;
\end{verbatim}

\noindent
Public typedefs:
\begin{verbatim}
    typedef char Vertex;
    typedef std::pair<Vertex, Vertex> VertexPair;
\end{verbatim}

\noindent
Constructors:
\begin{verbatim}
    KontsevichGraph() = default;
    KontsevichGraph(size_t internal, size_t external,
                    std::vector<VertexPair> targets,
                    int sign = 1, bool normalized = false);
\end{verbatim}

\noindent
Accessor methods:
\begin{verbatim}
    std::vector<VertexPair> targets() const;
    VertexPair targets(Vertex internal_vertex) const;
    int sign() const;
    int sign(int new_sign);
    size_t internal() const;
    size_t external() const;
\end{verbatim}

\noindent
Methods to obtain numerical information:
\begin{verbatim}
    size_t vertices() const;
    std::vector<Vertex> internal_vertices() const;
    std::pair< size_t, std::vector<VertexPair> > abs() const;
    size_t multiplicity() const;
    size_t in_degree(KontsevichGraph::Vertex vertex) const;
    std::vector<size_t> in_degrees() const;
    std::vector<Vertex> neighbors_in(Vertex vertex) const;
    KontsevichGraph mirror_image() const;
    std::string as_sage_expression() const;
    std::string encoding() const;
    std::vector< std::tuple<KontsevichGraph, int, int> > permutations() const;
\end{verbatim}

\noindent
Methods that modify the graph:
\begin{verbatim}
    void normalize();
    KontsevichGraph& operator*=(const KontsevichGraph& rhs);
\end{verbatim}

\noindent
Methods that test for graph properties:
\begin{verbatim}
    bool operator<(const KontsevichGraph& rhs) const;
    bool is_zero() const;
    bool is_prime() const;
    bool positive_differential_order() const;
    bool has_cycles() const;
    bool has_tadpoles() const;
    bool has_multiple_edges() const;
    bool has_max_internal_indegree(size_t max_indegree) const;
\end{verbatim}

\noindent
Static methods:
\begin{verbatim}
    static std::set<KontsevichGraph> graphs(size_t internal,
        size_t external = 2, bool modulo_signs = false,
        bool modulo_mirror_images = false,
        std::function<void(KontsevichGraph)> const& callback = nullptr,
        std::function<bool(KontsevichGraph)> const& filter = nullptr);
\end{verbatim}

\noindent
Private methods:
\begin{verbatim}
    friend std::ostream& operator<<(std::ostream &os, const KontsevichGraph& g);
    friend std::istream& operator>>(std::istream& is, KontsevichGraph& g);
    friend bool operator==(const KontsevichGraph &lhs, const KontsevichGraph& rhs);
    friend bool operator!=(const KontsevichGraph &lhs, const KontsevichGraph& rhs);
\end{verbatim}

\noindent
Functions defined outside the class:
\begin{verbatim}
    KontsevichGraph operator*(KontsevichGraph lhs, const KontsevichGraph& rhs);
    std::ostream& operator<<(std::ostream &os, const KontsevichGraph::Vertex v);
\end{verbatim}

\vskip 1em

\centerline{\textsc{Class} {\tt KontsevichGraphSum<T>}}
\vskip 1em

$\bullet$ Template parameter {\tt T}: type of the coefficients (e.g. {\tt KontsevichGraphSum<int>}).

$\bullet$ Publically extends: {\tt std::vector< std::pair<T, KontsevichGraph> >}.

\vskip 1em

\noindent
Summary: a sum of Kontsevich graphs, with method to reduce modulo skew-symmetry.

\vskip 1em

\noindent
Data members: inherited.

\noindent
Public typedefs:
\begin{verbatim}
    typedef std::pair<T, KontsevichGraph> Term;
\end{verbatim}

\noindent
Constructors (inherited):
\begin{verbatim}
    using std::vector< std::pair<T, KontsevichGraph> >::vector;
\end{verbatim}

\noindent
Accessor methods:
\begin{verbatim}
    using std::vector< std::pair<T, KontsevichGraph> >::operator[];
    KontsevichGraphSum<T> operator[](std::vector<size_t> indegrees) const;
    T operator[](KontsevichGraph) const;
\end{verbatim}

\noindent
Arithmetic operators:
\begin{verbatim}
    KontsevichGraphSum<T> operator()(std::vector< KontsevichGraphSum<T> >) const;
    KontsevichGraphSum<T>& operator+=(const KontsevichGraphSum<T>& rhs);
    KontsevichGraphSum<T>& operator-=(const KontsevichGraphSum<T>& rhs);
    KontsevichGraphSum<T>& operator=(const KontsevichGraphSum<T>&) = default;
\end{verbatim}

\noindent
Methods:
\begin{verbatim}
    std::vector< std::vector<size_t> > in_degrees(bool ascending = false) const;
    KontsevichGraphSum<T> skew_symmetrization() const;
\end{verbatim}

\noindent
Methods that modify the graph sum:
\begin{verbatim}
    void reduce_mod_skew();
\end{verbatim}

\noindent
Comparison operators:
\begin{verbatim}
    bool operator==(const KontsevichGraphSum<T>& other) const;
    bool operator==(int other) const;
    bool operator!=(const KontsevichGraphSum<T>& other) const;
    bool operator!=(int other) const;
\end{verbatim}

\noindent
Friend operators:
\begin{verbatim}
    friend std::ostream& operator<< <>(std::ostream& os,
                                       const KontsevichGraphSum<T>::Term& term);
    friend std::ostream& operator<< <>(std::ostream& os,
                                       const KontsevichGraphSum<T>& gs);
    friend std::istream& operator>> <>(std::istream& is,
                                       KontsevichGraphSum<T>& sum);
\end{verbatim}

\noindent
Functions defined outside the class:
\begin{verbatim}
    KontsevichGraphSum<T> operator+(KontsevichGraphSum<T> lhs,
                                    const KontsevichGraphSum<T>& rhs);
    KontsevichGraphSum<T> operator-(KontsevichGraphSum<T> lhs,
                                    const KontsevichGraphSum<T>& rhs);
    KontsevichGraphSum<T> operator*(T lhs,
                                    KontsevichGraphSum<T> rhs);
    std::ostream& operator<<(std::ostream&, const std::pair<T, KontsevichGraph>&);
    std::ostream& operator<<(std::ostream&, const KontsevichGraphSum<T>&);
    std::istream& operator>>(std::istream&, KontsevichGraphSum<T>&);
\end{verbatim}

\vskip 1em

\centerline{\textsc{Class} {\tt KontsevichGraphSeries<T>}}

\vskip 1em

$\bullet$ Template parameter {\tt T}: type of the coefficients (e.g. {\tt KontsevichGraphSeries<int>}).

$\bullet$ Publically extends: {\tt std::map< size\symbol{"5F}t, KontsevichGraphSum<T> >}

\vskip 1em

\noindent
Summary: a formal power series expansion; sums of Kontsevich graphs as coefficients.

\vskip 1em

\noindent
Data members: inherited, plus (private):
\begin{verbatim}
    size_t d_precision = std::numeric_limits<std::size_t>::max();
\end{verbatim}

\noindent
Constructors (inherited):
\begin{verbatim}
    using std::map< size_t, KontsevichGraphSum<T> >::map;
\end{verbatim}

\noindent
Accessor methods:
\begin{verbatim}
    size_t precision() const;
    size_t precision(size_t new_precision);
\end{verbatim}

\noindent
Arithmetic operators:
\begin{verbatim}
    KontsevichGraphSeries<T> operator()(std::vector< KontsevichGraphSeries<T> >)
                                                                          const;
    KontsevichGraphSeries<T>& operator+=(const KontsevichGraphSeries<T>& rhs);
    KontsevichGraphSeries<T>& operator-=(const KontsevichGraphSeries<T>& rhs);
\end{verbatim}

\noindent
Methods:
\begin{verbatim}
    KontsevichGraphSeries<T> skew_symmetrization() const;
    KontsevichGraphSeries<T> inverse() const;
    KontsevichGraphSeries<T> gauge_transform(const KontsevichGraphSeries<T>& gauge);
\end{verbatim}

\noindent
Comparison operators:
\begin{verbatim}
    bool operator==(int other) const;
    bool operator!=(int other) const;
\end{verbatim}

\noindent
Methods that modify the graph series:
\begin{verbatim}
    void reduce_mod_skew();
\end{verbatim}

\noindent
Static methods:
\begin{verbatim}
    static KontsevichGraphSeries<T> from_istream(std::istream& is,
        std::function<T(std::string)> const& parser,
        std::function<bool(KontsevichGraph, size_t)> const& filter = nullptr);
\end{verbatim}

\noindent
Friend methods:
\begin{verbatim}
    friend std::ostream& operator<< <>(std::ostream& os,
                                       const KontsevichGraphSeries<T>& series);
\end{verbatim}

\noindent
Functions defined outside the class:
\begin{verbatim}
    KontsevichGraphSeries<T> operator+(KontsevichGraphSeries<T> lhs,
                                       const KontsevichGraphSeries<T>& rhs);
    KontsevichGraphSeries<T> operator-(KontsevichGraphSeries<T> lhs,
                                       const KontsevichGraphSeries<T>& rhs);
    std::ostream& operator<<(std::ostream&, const KontsevichGraphSeries<T>&);
\end{verbatim}

\newpage

\section{Encoding of the entire $\star$-\/product modulo $\bar{o}(\hbar^4)$}\label{AppStarEncoding}

In the following two tables, containing the sets of basic graphs and the $\star$-product expansion respectively, encodings of graphs (see Implementation \ref{DefEncoding} on p. \pageref{DefEncoding}) are followed by their coefficients.

\begin{table}[h]
\caption{Basic sets of Kontsevich graphs, up to order $4$, including zero graphs.}
\label{TblBasic}
\renewcommand{\baselinestretch}{0.86}
\begin{tabular}{p{0.33\textwidth} | p{0.33\textwidth} | p{0.33\textwidth}}
\tiny
\begin{verbatim}
h^0:
2 0 1                      1
h^1:
2 1 1   0 1                1/2
h^2:
2 2 1   0 1 0 2            1/12
2 2 1   0 3 1 2            -1/24
h^3:
2 3 0   0 1 0 1 2 3        0
2 3 1   0 1 0 2 0 2        1/24
2 3 1   0 1 0 2 0 3        0
2 3 1   0 1 0 2 1 2        0
2 3 1   0 1 0 2 1 3        -1/48
2 3 1   0 1 0 2 2 3        -1/48
2 3 1   0 1 0 4 2 3        0
2 3 1   0 1 2 4 2 3        0
2 3 1   0 3 0 2 1 2        0
2 3 1   0 3 0 4 1 2        0
2 3 1   0 3 0 4 1 3        0
2 3 1   0 3 1 2 0 3        -1/48
2 3 1   0 3 1 2 2 3        -1/48
2 3 1   0 3 1 4 2 3        0
2 3 1   0 3 2 4 1 3        0
h^4:
2 4 1   0 1 0 1 0 2 2 3    w_4_1
2 4 1   0 1 0 1 0 2 3 4    w_4_2
2 4 0   0 1 0 1 0 5 2 3    0
2 4 1   0 1 0 1 2 3 2 3    w_4_3
2 4 1   0 1 0 1 2 3 2 4    w_4_4
2 4 1   0 1 0 1 2 5 3 4    w_4_5
2 4 1   0 1 0 2 0 2 0 2    w_4_6
2 4 1   0 1 0 2 0 2 0 3    w_4_7
2 4 1   0 1 0 2 0 2 1 2    w_4_8
2 4 1   0 1 0 2 0 2 1 3    w_4_9
2 4 1   0 1 0 2 0 2 2 3    w_4_10
2 4 0   0 1 0 2 0 2 3 4    0
2 4 1   0 1 0 2 0 3 0 3    w_4_11
2 4 1   0 1 0 2 0 3 0 4    w_4_12
2 4 1   0 1 0 2 0 3 1 2    w_4_13
2 4 1   0 1 0 2 0 3 1 3    w_4_14
2 4 1   0 1 0 2 0 3 1 4    w_4_15
2 4 1   0 1 0 2 0 3 2 3    w_4_16
2 4 1   0 1 0 2 0 3 2 4    w_4_17
2 4 1   0 1 0 2 0 3 3 4    w_4_18
2 4 1   0 1 0 2 0 5 1 2    w_4_19
2 4 1   0 1 0 2 0 5 1 3    w_4_20
2 4 1   0 1 0 2 0 5 2 3    w_4_21
2 4 1   0 1 0 2 0 5 2 4    w_4_22
2 4 1   0 1 0 2 0 5 3 4    w_4_23
2 4 1   0 1 0 2 1 2 2 3    w_4_24
2 4 1   0 1 0 2 1 2 3 4    w_4_25
2 4 1   0 1 0 2 1 3 1 3    w_4_26
2 4 1   0 1 0 2 1 3 1 4    w_4_27
2 4 1   0 1 0 2 1 3 2 3    w_4_28
2 4 1   0 1 0 2 1 3 2 4    w_4_29
2 4 1   0 1 0 2 1 3 3 4    w_4_30
2 4 1   0 1 0 2 1 5 2 3    w_4_31
2 4 1   0 1 0 2 1 5 2 4    w_4_32
2 4 1   0 1 0 2 1 5 3 4    w_4_33
2 4 1   0 1 0 2 2 3 2 3    w_4_34
\end{verbatim}
&
\tiny
\begin{verbatim}
2 4 1   0 1 0 2 2 3 2 4    w_4_35
2 4 1   0 1 0 2 2 3 3 4    w_4_36
2 4 1   0 1 0 2 2 5 2 4    w_4_37
2 4 1   0 1 0 2 2 5 3 4    w_4_38
2 4 1   0 1 0 2 3 5 3 4    w_4_39
2 4 1   0 1 0 4 0 3 2 3    w_4_40
2 4 1   0 1 0 4 0 5 2 3    w_4_41
2 4 1   0 1 0 4 0 5 2 4    w_4_42
2 4 1   0 1 0 4 1 3 2 3    w_4_43
2 4 1   0 1 0 4 1 5 2 3    w_4_44
2 4 1   0 1 0 4 1 5 2 4    w_4_45
2 4 1   0 1 0 4 2 3 0 4    w_4_46
2 4 1   0 1 0 4 2 3 1 4    w_4_47
2 4 1   0 1 0 4 2 3 2 3    w_4_48
2 4 1   0 1 0 4 2 3 2 4    w_4_49
2 4 1   0 1 0 4 2 3 3 4    w_4_50
2 4 1   0 1 0 4 2 5 2 3    w_4_51
2 4 1   0 1 0 4 2 5 2 4    w_4_52
2 4 1   0 1 0 4 2 5 3 4    w_4_53
2 4 1   0 1 0 4 3 5 2 3    w_4_54
2 4 1   0 1 0 4 3 5 2 4    w_4_55
2 4 1   0 1 2 4 2 3 2 3    w_4_56
2 4 0   0 1 2 4 2 3 3 4    0
2 4 1   0 1 2 4 2 5 2 3    w_4_57
2 4 1   0 1 2 4 2 5 3 4    w_4_58
2 4 0   0 1 2 4 3 5 2 4    0
2 4 1   0 1 2 4 3 5 3 4    w_4_59
2 4 1   0 3 0 2 0 2 1 2    w_4_60
2 4 1   0 3 0 2 0 2 1 3    w_4_61
2 4 1   0 3 0 2 0 2 1 4    w_4_62
2 4 1   0 3 0 2 0 5 1 2    w_4_63
2 4 1   0 3 0 2 1 2 1 2    w_4_64
2 4 1   0 3 0 2 1 2 1 3    w_4_65
2 4 1   0 3 0 2 1 2 1 4    w_4_66
2 4 1   0 3 0 2 1 2 2 3    w_4_67
2 4 1   0 3 0 2 1 2 2 4    w_4_68
2 4 1   0 3 0 2 1 2 3 4    w_4_69
2 4 0   0 3 0 2 1 5 2 3    0
2 4 1   0 3 0 2 1 5 2 4    w_4_70
2 4 1   0 3 0 4 0 2 1 2    w_4_71
2 4 1   0 3 0 4 0 5 1 2    w_4_72
2 4 1   0 3 0 4 0 5 1 3    w_4_73
2 4 1   0 3 0 4 0 5 1 4    w_4_74
2 4 1   0 3 0 4 1 2 0 3    w_4_75
2 4 1   0 3 0 4 1 2 0 4    w_4_76
2 4 1   0 3 0 4 1 2 1 2    w_4_77
2 4 1   0 3 0 4 1 2 1 3    w_4_78
2 4 1   0 3 0 4 1 2 1 4    w_4_79
2 4 1   0 3 0 4 1 2 2 3    w_4_80
2 4 1   0 3 0 4 1 2 2 4    w_4_81
2 4 1   0 3 0 4 1 2 3 4    w_4_82
2 4 1   0 3 0 4 1 3 0 3    w_4_83
2 4 1   0 3 0 4 1 3 0 4    w_4_84
2 4 1   0 3 0 4 1 3 1 2    w_4_85
2 4 1   0 3 0 4 1 3 1 3    w_4_86
2 4 1   0 3 0 4 1 3 1 4    w_4_87
2 4 1   0 3 0 4 1 3 2 3    w_4_88
2 4 1   0 3 0 4 1 3 2 4    w_4_89
2 4 1   0 3 0 4 1 3 3 4    w_4_90
2 4 1   0 3 0 4 1 5 0 4    w_4_91
\end{verbatim}
&
\tiny
\begin{verbatim}
2 4 1   0 3 0 4 1 5 1 2    w_4_92
2 4 1   0 3 0 4 1 5 2 3    w_4_93
2 4 1   0 3 0 4 1 5 2 4    w_4_94
2 4 1   0 3 0 4 1 5 3 4    w_4_95
2 4 1   0 3 0 4 2 3 1 2    w_4_96
2 4 1   0 3 0 4 2 3 1 3    w_4_97
2 4 1   0 3 0 4 2 3 1 4    w_4_98
2 4 1   0 3 0 4 2 5 1 2    w_4_99
2 4 1   0 3 0 4 2 5 1 3    w_4_100
2 4 1   0 3 0 4 2 5 1 4    w_4_101
2 4 1   0 3 0 4 3 5 1 2    w_4_102
2 4 1   0 3 0 4 3 5 1 3    w_4_103
2 4 1   0 3 0 4 3 5 1 4    w_4_104
2 4 1   0 3 1 2 0 3 0 3    w_4_105
2 4 1   0 3 1 2 0 3 1 2    w_4_106
2 4 1   0 3 1 2 0 3 1 4    w_4_107
2 4 1   0 3 1 2 0 3 2 3    w_4_108
2 4 1   0 3 1 2 0 3 2 4    w_4_109
2 4 1   0 3 1 2 0 3 3 4    w_4_110
2 4 1   0 3 1 2 0 5 2 3    w_4_111
2 4 1   0 3 1 2 0 5 2 4    w_4_112
2 4 1   0 3 1 2 0 5 3 4    w_4_113
2 4 1   0 3 1 2 2 3 2 3    w_4_114
2 4 1   0 3 1 2 2 3 2 4    w_4_115
2 4 1   0 3 1 2 2 5 2 4    w_4_116
2 4 1   0 3 1 2 2 5 3 4    w_4_117
2 4 1   0 3 1 4 0 5 1 2    w_4_118
2 4 1   0 3 1 4 0 5 2 3    w_4_119
2 4 1   0 3 1 4 0 5 2 4    w_4_120
2 4 1   0 3 1 4 0 5 3 4    w_4_121
2 4 1   0 3 1 4 2 3 0 3    w_4_122
2 4 1   0 3 1 4 2 3 0 4    w_4_123
2 4 1   0 3 1 4 2 3 1 4    w_4_124
2 4 1   0 3 1 4 2 3 2 3    w_4_125
2 4 1   0 3 1 4 2 3 2 4    w_4_126
2 4 1   0 3 1 4 2 3 3 4    w_4_127
2 4 0   0 3 1 4 2 5 0 3    0
2 4 1   0 3 1 4 2 5 0 4    w_4_128
2 4 1   0 3 1 4 2 5 1 4    w_4_129
2 4 1   0 3 1 4 2 5 2 3    w_4_130
2 4 1   0 3 1 4 2 5 2 4    w_4_131
2 4 1   0 3 1 4 2 5 3 4    w_4_132
2 4 1   0 3 1 4 3 5 0 4    w_4_133
2 4 1   0 3 1 4 3 5 1 4    w_4_134
2 4 1   0 3 1 4 3 5 2 3    w_4_135
2 4 1   0 3 1 4 3 5 2 4    w_4_136
2 4 1   0 3 1 4 3 5 3 4    w_4_137
2 4 0   0 3 2 4 0 3 1 3    0
2 4 1   0 3 2 4 1 3 0 3    w_4_138
2 4 1   0 3 2 4 1 3 2 3    w_4_139
2 4 1   0 3 2 4 1 3 2 4    w_4_140
2 4 1   0 3 2 4 1 5 2 3    w_4_141
2 4 1   0 3 2 4 1 5 2 4    w_4_142
2 4 1   0 3 2 4 1 5 3 4    w_4_143
2 4 1   0 3 2 4 2 3 1 3    w_4_144
2 4 1   0 3 2 4 2 3 1 4    w_4_145
2 4 1   0 3 2 4 2 5 1 3    w_4_146
2 4 1   0 3 2 4 3 5 1 3    w_4_147
2 4 1   0 3 2 4 3 5 1 4    w_4_148
2 4 1   0 3 4 5 1 5 2 3    w_4_149
\end{verbatim}
\end{tabular}
\end{table}

\begin{table}[h!]
\caption{Kontsevich's star product up to order $4$.}
\label{TableStar}
\renewcommand{\baselinestretch}{0.7}
\begin{tabular}{p{0.5\textwidth} | p{0.5\textwidth}}
\tiny
\begin{verbatim}
h^0:
2 0 1                      1
h^1:
2 1 1   0 1                1
h^2:
2 2 1   0 1 0 1            1/2
2 2 1   0 1 0 2            1/3
2 2 1   0 1 1 2            -1/3
2 2 1   0 3 1 2            -1/6
h^3:
2 3 1   0 1 0 1 0 1        1/6
2 3 1   0 1 0 1 0 2        1/3
2 3 1   0 1 0 1 1 2        -1/3
2 3 1   0 1 0 4 1 3        -1/6
2 3 1   0 1 0 2 0 2        1/6
2 3 1   0 1 1 2 1 2        1/6
2 3 1   0 1 0 2 2 3        -1/6
2 3 1   0 1 1 2 2 3        -1/6
2 3 1   0 1 0 4 1 2        -1/6
2 3 1   0 1 0 2 1 3        -1/6
2 3 1   0 3 1 2 1 2        1/6
2 3 1   0 3 1 2 0 3        -1/6
2 3 1   0 3 1 2 2 3        -1/6
h^4:
2 4 1   0 1 0 1 0 1 0 1    1/24
2 4 1   0 1 0 1 0 1 0 2    1/6
2 4 1   0 1 0 1 0 1 1 2    -1/6
2 4 1   0 1 0 1 0 5 1 4    -1/12
2 4 1   0 1 0 1 0 2 0 2    1/6
2 4 1   0 1 0 1 1 2 1 2    1/6
2 4 1   0 1 0 1 0 2 2 4    -1/6
2 4 1   0 1 0 1 1 2 2 4    -1/6
2 4 1   0 1 0 1 0 5 1 2    -1/6
2 4 1   0 1 0 1 0 2 1 4    -1/6
2 4 1   0 1 0 4 1 3 1 3    1/6
2 4 1   0 1 0 4 1 3 0 4    -1/6
2 4 1   0 1 0 4 1 3 3 4    -1/6
2 4 1   0 1 0 1 0 2 0 3    1/18
2 4 1   0 1 0 1 0 2 1 3    -1/9
2 4 1   0 1 0 2 0 5 1 4    -1/18
2 4 1   0 1 0 1 1 2 1 3    1/18
2 4 1   0 1 0 4 1 3 1 2    1/18
2 4 1   0 3 1 2 0 5 1 4    1/72
2 4 1   0 1 0 1 0 2 2 3    16*w_4_1
2 4 1   0 1 0 1 1 2 2 3    16*w_4_1
2 4 1   0 1 0 1 0 2 3 4    16*w_4_2
2 4 1   0 1 0 1 1 2 3 4    16*w_4_2
2 4 1   0 1 0 1 2 3 2 3    4*w_4_3
2 4 1   0 1 0 1 2 3 2 4    16*w_4_4
2 4 1   0 1 0 1 2 5 3 4    8*w_4_5
2 4 1   0 1 0 2 0 2 0 2    8/3*w_4_6
2 4 1   0 1 1 2 1 2 1 2    -8/3*w_4_6
2 4 1   0 1 0 2 0 2 0 3    16*w_4_7
2 4 1   0 1 1 2 1 2 1 3    -16*w_4_7
2 4 1   0 1 0 2 0 2 1 2    8*w_4_8
2 4 1   0 1 0 2 1 2 1 2    -8*w_4_8
2 4 1   0 1 0 2 0 2 1 3    16*w_4_9
2 4 1   0 1 0 4 1 2 1 2    -16*w_4_9
2 4 1   0 1 0 2 0 2 2 3    16*w_4_10
2 4 1   0 1 1 2 1 2 2 3    -16*w_4_10
2 4 1   0 1 0 2 0 3 0 3    8*w_4_11
2 4 1   0 1 1 2 1 3 1 3    -8*w_4_11
2 4 1   0 1 0 2 0 3 0 4    16*w_4_12
2 4 1   0 1 1 2 1 3 1 4    -16*w_4_12
2 4 1   0 1 0 2 0 3 1 2    16*w_4_13
2 4 1   0 1 0 2 1 2 1 4    -16*w_4_13
2 4 1   0 1 0 2 0 3 1 3    16*w_4_14
2 4 1   0 1 0 4 1 2 1 4    -16*w_4_14
2 4 1   0 1 0 2 0 3 1 4    16*w_4_15
2 4 1   0 1 0 4 1 5 1 2    -16*w_4_15
2 4 1   0 1 0 2 0 3 2 3    16*w_4_16
2 4 1   0 1 1 2 1 3 2 3    -16*w_4_16
2 4 1   0 1 0 2 0 3 2 4    16*w_4_17
2 4 1   0 1 1 2 1 3 2 4    -16*w_4_17
2 4 1   0 1 0 2 0 3 3 4    16*w_4_18
2 4 1   0 1 1 2 1 3 3 4    -16*w_4_18
2 4 1   0 1 0 2 0 5 1 2    16*w_4_19
2 4 1   0 1 0 2 1 2 1 3    -16*w_4_19
2 4 1   0 1 0 2 0 5 1 3    16*w_4_20
2 4 1   0 1 0 4 1 2 1 3    -16*w_4_20
2 4 1   0 1 0 2 0 5 2 3    16*w_4_21
\end{verbatim}
&
\tiny
\begin{verbatim}
2 4 1   0 1 1 2 1 5 2 3    -16*w_4_21
2 4 1   0 1 0 2 0 5 2 4    16*w_4_22
2 4 1   0 1 1 2 1 5 2 4    -16*w_4_22
2 4 1   0 1 0 2 0 5 3 4    16*w_4_23
2 4 1   0 1 1 2 1 5 3 4    -16*w_4_23
2 4 1   0 1 0 2 1 2 2 3    16*w_4_24
2 4 1   0 1 0 2 1 2 2 4    -16*w_4_24
2 4 1   0 1 0 2 1 2 3 4    16*w_4_25
2 4 1   0 1 0 2 1 3 1 3    8*w_4_26
2 4 1   0 1 0 4 1 2 0 4    -8*w_4_26
2 4 1   0 1 0 2 1 3 1 4    16*w_4_27
2 4 1   0 1 0 4 0 5 1 2    -16*w_4_27
2 4 1   0 1 0 2 1 3 2 3    16*w_4_28
2 4 1   0 1 0 4 1 2 2 4    -16*w_4_28
2 4 1   0 1 0 2 1 3 2 4    16*w_4_29
2 4 1   0 1 0 4 1 2 2 3    -16*w_4_29
2 4 1   0 1 0 2 1 3 3 4    16*w_4_30
2 4 1   0 1 0 4 1 2 3 4    16*w_4_30
2 4 1   0 1 0 2 1 5 2 3    16*w_4_31
2 4 1   0 1 0 4 2 5 1 2    -16*w_4_31
2 4 1   0 1 0 2 1 5 2 4    16*w_4_32
2 4 1   0 1 0 4 2 3 1 2    -16*w_4_32
2 4 1   0 1 0 2 1 5 3 4    16*w_4_33
2 4 1   0 1 0 4 3 5 1 2    16*w_4_33
2 4 1   0 1 0 2 2 3 2 3    8*w_4_34
2 4 1   0 1 1 2 2 3 2 3    -8*w_4_34
2 4 1   0 1 0 2 2 3 2 4    16*w_4_35
2 4 1   0 1 1 2 2 3 2 4    -16*w_4_35
2 4 1   0 1 0 2 2 3 3 4    16*w_4_36
2 4 1   0 1 1 2 2 3 3 4    -16*w_4_36
2 4 1   0 1 0 2 2 5 2 4    8*w_4_37
2 4 1   0 1 1 2 2 5 2 4    -8*w_4_37
2 4 1   0 1 0 2 2 5 3 4    16*w_4_38
2 4 1   0 1 1 2 2 5 3 4    -16*w_4_38
2 4 1   0 1 0 2 3 5 3 4    8*w_4_39
2 4 1   0 1 1 2 3 5 3 4    -8*w_4_39
2 4 1   0 1 0 4 0 3 2 3    16*w_4_40
2 4 1   0 1 1 4 1 3 2 3    -16*w_4_40
2 4 1   0 1 0 4 0 5 2 3    16*w_4_41
2 4 1   0 1 1 4 1 5 2 3    -16*w_4_41
2 4 1   0 1 0 4 0 5 2 4    16*w_4_42
2 4 1   0 1 1 4 1 5 2 4    -16*w_4_42
2 4 1   0 1 0 4 1 3 2 3    16*w_4_43
2 4 1   0 1 0 4 1 3 2 4    -16*w_4_43
2 4 1   0 1 0 4 1 5 2 3    16*w_4_44
2 4 1   0 1 0 4 2 5 1 3    -16*w_4_44
2 4 1   0 1 0 4 1 5 2 4    16*w_4_45
2 4 1   0 1 0 4 2 3 1 3    -16*w_4_45
2 4 1   0 1 0 4 2 3 0 4    16*w_4_46
2 4 1   0 1 1 4 2 3 1 4    -16*w_4_46
2 4 1   0 1 0 4 2 3 1 4    16*w_4_47
2 4 1   0 1 0 4 2 5 1 4    -16*w_4_47
2 4 1   0 1 0 4 2 3 2 3    16*w_4_48
2 4 1   0 1 1 4 2 3 2 3    -16*w_4_48
2 4 1   0 1 0 4 2 3 2 4    16*w_4_49
2 4 1   0 1 1 4 2 3 2 4    -16*w_4_49
2 4 1   0 1 0 4 2 3 3 4    16*w_4_50
2 4 1   0 1 1 4 2 3 3 4    -16*w_4_50
2 4 1   0 1 0 4 2 5 2 3    16*w_4_51
2 4 1   0 1 1 4 2 5 2 3    -16*w_4_51
2 4 1   0 1 0 4 2 5 2 4    16*w_4_52
2 4 1   0 1 1 4 2 5 2 4    -16*w_4_52
2 4 1   0 1 0 4 2 5 3 4    16*w_4_53
2 4 1   0 1 1 4 2 5 3 4    -16*w_4_53
2 4 1   0 1 0 4 3 5 2 3    16*w_4_54
2 4 1   0 1 1 4 3 5 2 3    -16*w_4_54
2 4 1   0 1 0 4 3 5 2 4    16*w_4_55
2 4 1   0 1 1 4 3 5 2 4    -16*w_4_55
2 4 1   0 1 2 4 2 3 2 3    16*w_4_56
2 4 1   0 1 2 4 2 5 2 3    16/3*w_4_57
2 4 1   0 1 2 4 2 5 3 4    16*w_4_58
2 4 1   0 1 2 4 3 5 3 4    16*w_4_59
2 4 1   0 3 0 2 0 2 1 2    16*w_4_60
2 4 1   0 3 1 4 1 3 1 3    16*w_4_60
2 4 1   0 3 0 2 0 2 1 3    16*w_4_61
2 4 1   0 3 1 4 1 3 1 4    16*w_4_61
2 4 1   0 3 0 2 0 2 1 4    16*w_4_62
2 4 1   0 3 1 4 1 5 1 4    16*w_4_62
2 4 1   0 3 0 2 0 5 1 2    16*w_4_63
2 4 1   0 3 1 4 1 3 1 2    16*w_4_63
2 4 1   0 3 0 2 1 2 1 2    8*w_4_64
2 4 1   0 3 1 4 1 3 0 3    8*w_4_64
\end{verbatim}
\end{tabular}
\end{table}
\begin{table}
\renewcommand{\baselinestretch}{0.7}
\centerline{\textsc{Table}~\ref{TableStar} (continued).}\label{TableStar2}
\begin{tabular}{p{0.5\textwidth} | p{0.5\textwidth}}
\tiny
\begin{verbatim}
2 4 1   0 3 0 2 1 2 1 3    8*w_4_65
2 4 1   0 3 1 4 1 3 0 4    8*w_4_65
2 4 1   0 3 0 2 1 2 1 4    16*w_4_66
2 4 1   0 3 0 4 1 5 1 4    16*w_4_66
2 4 1   0 3 0 2 1 2 2 3    16*w_4_67
2 4 1   0 3 1 4 1 3 3 4    16*w_4_67
2 4 1   0 3 0 2 1 2 2 4    16*w_4_68
2 4 1   0 3 1 4 1 3 2 3    -16*w_4_68
2 4 1   0 3 0 2 1 2 3 4    16*w_4_69
2 4 1   0 3 1 4 1 3 2 4    -16*w_4_69
2 4 1   0 3 0 2 1 5 2 4    16*w_4_70
2 4 1   0 3 2 4 1 5 1 4    -16*w_4_70
2 4 1   0 3 0 4 0 2 1 2    16*w_4_71
2 4 1   0 3 1 4 1 5 1 3    16*w_4_71
2 4 1   0 3 0 4 0 5 1 2    16*w_4_72
2 4 1   0 3 1 4 1 5 1 2    16*w_4_72
2 4 1   0 3 0 4 0 5 1 3    16*w_4_73
2 4 1   0 3 1 4 1 2 1 3    16*w_4_73
2 4 1   0 3 0 4 0 5 1 4    16*w_4_74
2 4 1   0 3 1 2 1 3 1 4    16*w_4_74
2 4 1   0 3 0 4 1 2 0 3    16*w_4_75
2 4 1   0 3 1 4 1 2 1 4    16*w_4_75
2 4 1   0 3 0 4 1 2 0 4    16*w_4_76
2 4 1   0 3 1 4 1 2 1 2    16*w_4_76
2 4 1   0 3 0 4 1 2 1 2    16*w_4_77
2 4 1   0 3 1 4 1 2 0 3    16*w_4_77
2 4 1   0 3 0 4 1 2 1 3    16*w_4_78
2 4 1   0 3 1 4 0 5 1 3    16*w_4_78
2 4 1   0 3 0 4 1 2 1 4    16*w_4_79
2 4 1   0 3 0 4 1 5 1 3    16*w_4_79
2 4 1   0 3 0 4 1 2 2 3    16*w_4_80
2 4 1   0 3 1 4 1 2 3 4    16*w_4_80
2 4 1   0 3 0 4 1 2 2 4    16*w_4_81
2 4 1   0 3 1 4 1 2 2 3    -16*w_4_81
2 4 1   0 3 0 4 1 2 3 4    16*w_4_82
2 4 1   0 3 1 4 1 2 2 4    -16*w_4_82
2 4 1   0 3 0 4 1 3 0 3    8*w_4_83
2 4 1   0 3 1 2 1 3 1 3    8*w_4_83
2 4 1   0 3 0 4 1 3 0 4    16*w_4_84
2 4 1   0 3 1 2 1 2 1 3    16*w_4_84
2 4 1   0 3 0 4 1 3 1 2    16*w_4_85
2 4 1   0 3 1 2 0 5 1 3    16*w_4_85
2 4 1   0 3 0 4 1 3 1 3    16*w_4_86
2 4 1   0 3 1 2 0 3 1 3    16*w_4_86
2 4 1   0 3 0 4 1 3 1 4    16*w_4_87
2 4 1   0 3 0 4 1 3 2 3    16*w_4_88
2 4 1   0 3 1 2 1 3 3 4    -16*w_4_88
2 4 1   0 3 0 4 1 3 2 4    16*w_4_89
2 4 1   0 3 1 2 1 3 2 4    -16*w_4_89
2 4 1   0 3 0 4 1 3 3 4    16*w_4_90
2 4 1   0 3 1 2 1 3 2 3    -16*w_4_90
2 4 1   0 3 0 4 1 5 0 4    16*w_4_91
2 4 1   0 3 1 2 1 2 1 4    16*w_4_91
2 4 1   0 3 0 4 1 5 1 2    16*w_4_92
2 4 1   0 3 0 4 1 5 2 3    16*w_4_93
2 4 1   0 3 4 5 1 2 1 4    -16*w_4_93
2 4 1   0 3 0 4 1 5 2 4    16*w_4_94
2 4 1   0 3 2 4 1 5 1 2    -16*w_4_94
2 4 1   0 3 0 4 1 5 3 4    16*w_4_95
2 4 1   0 3 2 4 1 2 1 4    -16*w_4_95
2 4 1   0 3 0 4 2 3 1 2    16*w_4_96
2 4 1   0 3 1 4 1 5 3 4    16*w_4_96
2 4 1   0 3 0 4 2 3 1 3    16*w_4_97
2 4 1   0 3 1 4 3 5 1 3    -16*w_4_97
2 4 1   0 3 0 4 2 3 1 4    16*w_4_98
2 4 1   0 3 4 5 1 3 1 4    -16*w_4_98
2 4 1   0 3 0 4 2 5 1 2    16*w_4_99
2 4 1   0 3 1 4 1 5 2 3    -16*w_4_99
2 4 1   0 3 0 4 2 5 1 3    16*w_4_100
2 4 1   0 3 1 4 2 5 1 3    -16*w_4_100
2 4 1   0 3 0 4 2 5 1 4    16*w_4_101
2 4 1   0 3 2 4 1 5 1 3    -16*w_4_101
2 4 1   0 3 0 4 3 5 1 2    16*w_4_102
2 4 1   0 3 1 4 1 5 2 4    -16*w_4_102
2 4 1   0 3 0 4 3 5 1 3    16*w_4_103
2 4 1   0 3 1 4 2 3 1 3    -16*w_4_103
2 4 1   0 3 0 4 3 5 1 4    16*w_4_104
2 4 1   0 3 2 4 1 3 1 4    -16*w_4_104
2 4 1   0 3 1 2 0 3 0 3    8*w_4_105
2 4 1   0 3 1 2 1 2 1 2    8*w_4_105
2 4 1   0 3 1 2 0 3 1 2    16*w_4_106
2 4 1   0 3 1 2 0 3 1 4    16*w_4_107
\end{verbatim}
&
\tiny
\begin{verbatim}
2 4 1   0 3 1 2 0 5 1 2    16*w_4_107
2 4 1   0 3 1 2 0 3 2 3    16*w_4_108
2 4 1   0 3 1 2 1 2 2 3    -16*w_4_108
2 4 1   0 3 1 2 0 3 2 4    16*w_4_109
2 4 1   0 3 1 2 1 2 3 4    16*w_4_109
2 4 1   0 3 1 2 0 3 3 4    16*w_4_110
2 4 1   0 3 1 2 1 2 2 4    16*w_4_110
2 4 1   0 3 1 2 0 5 2 3    16*w_4_111
2 4 1   0 3 1 2 1 5 2 3    -16*w_4_111
2 4 1   0 3 1 2 0 5 2 4    16*w_4_112
2 4 1   0 3 1 2 1 5 3 4    16*w_4_112
2 4 1   0 3 1 2 0 5 3 4    16*w_4_113
2 4 1   0 3 1 2 1 5 2 4    16*w_4_113
2 4 1   0 3 1 2 2 3 2 3    8*w_4_114
2 4 1   0 3 1 2 2 3 2 4    16*w_4_115
2 4 1   0 3 1 2 2 3 3 4    -16*w_4_115
2 4 1   0 3 1 2 2 5 2 4    8*w_4_116
2 4 1   0 3 1 2 3 5 3 4    8*w_4_116
2 4 1   0 3 1 2 2 5 3 4    16*w_4_117
2 4 1   0 3 1 4 0 5 1 2    8*w_4_118
2 4 1   0 3 1 4 0 5 2 3    16*w_4_119
2 4 1   0 3 1 4 2 5 1 2    -16*w_4_119
2 4 1   0 3 1 4 0 5 2 4    16*w_4_120
2 4 1   0 3 1 4 3 5 1 2    -16*w_4_120
2 4 1   0 3 1 4 0 5 3 4    16*w_4_121
2 4 1   0 3 1 4 2 3 1 2    16*w_4_121
2 4 1   0 3 1 4 2 3 0 3    16*w_4_122
2 4 1   0 3 2 4 1 2 1 2    -16*w_4_122
2 4 1   0 3 1 4 2 3 0 4    16*w_4_123
2 4 1   0 3 2 4 1 2 1 3    -16*w_4_123
2 4 1   0 3 1 4 2 3 1 4    16*w_4_124
2 4 1   0 3 2 4 1 2 0 3    -16*w_4_124
2 4 1   0 3 1 4 2 3 2 3    16*w_4_125
2 4 1   0 3 2 4 1 2 2 4    16*w_4_125
2 4 1   0 3 1 4 2 3 2 4    16*w_4_126
2 4 1   0 3 2 4 1 2 3 4    16*w_4_126
2 4 1   0 3 1 4 2 3 3 4    16*w_4_127
2 4 1   0 3 2 4 1 2 2 3    -16*w_4_127
2 4 1   0 3 1 4 2 5 0 4    16*w_4_128
2 4 1   0 3 4 5 1 2 1 3    16*w_4_128
2 4 1   0 3 1 4 2 5 1 4    16*w_4_129
2 4 1   0 3 2 4 1 5 0 3    -16*w_4_129
2 4 1   0 3 1 4 2 5 2 3    16*w_4_130
2 4 1   0 3 4 5 1 2 2 4    -16*w_4_130
2 4 1   0 3 1 4 2 5 2 4    16*w_4_131
2 4 1   0 3 4 5 1 2 3 4    -16*w_4_131
2 4 1   0 3 1 4 2 5 3 4    16*w_4_132
2 4 1   0 3 4 5 1 2 2 3    16*w_4_132
2 4 1   0 3 1 4 3 5 0 4    16*w_4_133
2 4 1   0 3 2 4 1 3 1 2    16*w_4_133
2 4 1   0 3 1 4 3 5 1 4    16*w_4_134
2 4 1   0 3 2 4 0 3 1 2    16*w_4_134
2 4 1   0 3 1 4 3 5 2 3    16*w_4_135
2 4 1   0 3 2 4 2 5 1 2    -16*w_4_135
2 4 1   0 3 1 4 3 5 2 4    16*w_4_136
2 4 1   0 3 2 4 3 5 1 2    -16*w_4_136
2 4 1   0 3 1 4 3 5 3 4    16*w_4_137
2 4 1   0 3 2 4 2 3 1 2    16*w_4_137
2 4 1   0 3 2 4 1 3 0 3    16*w_4_138
2 4 1   0 3 2 4 1 3 1 3    -16*w_4_138
2 4 1   0 3 2 4 1 3 2 3    16*w_4_139
2 4 1   0 3 2 4 1 3 3 4    16*w_4_139
2 4 1   0 3 2 4 1 3 2 4    16*w_4_140
2 4 1   0 3 2 4 1 5 2 3    16*w_4_141
2 4 1   0 3 4 5 1 5 2 4    -16*w_4_141
2 4 1   0 3 2 4 1 5 2 4    16*w_4_142
2 4 1   0 3 2 4 1 5 3 4    16*w_4_143
2 4 1   0 3 2 4 2 5 1 4    16*w_4_143
2 4 1   0 3 2 4 2 3 1 3    16*w_4_144
2 4 1   0 3 4 5 1 3 3 4    -16*w_4_144
2 4 1   0 3 2 4 2 3 1 4    16*w_4_145
2 4 1   0 3 4 5 1 5 3 4    -16*w_4_145
2 4 1   0 3 2 4 2 5 1 3    16*w_4_146
2 4 1   0 3 4 5 1 3 2 4    -16*w_4_146
2 4 1   0 3 2 4 3 5 1 3    16*w_4_147
2 4 1   0 3 4 5 1 3 2 3    -16*w_4_147
2 4 1   0 3 2 4 3 5 1 4    16*w_4_148
2 4 1   0 3 4 5 1 5 2 3    16*w_4_149
\end{verbatim}
\end{tabular}
\end{table}

\begin{table}[h!]
\caption{Relations between weights of $\hbar^4$-basic graphs: 149 via 10.}
\renewcommand{\baselinestretch}{0.7}
\label{Table149via10}
\tiny\begin{verbatim}
w_4_1==-1/144
w_4_2==-1/288
w_4_3==17/360 + 6*w_4_108
w_4_4==49/2880 - 3*w_4_104 - w_4_107 + (3*w_4_108)/2
w_4_5==-1/96 + 6*w_4_104 + 2*w_4_107
w_4_6==1/80
w_4_7==1/360
w_4_8==-1/240
w_4_9==-13/1440
w_4_10==-7/1440
w_4_11==1/240
w_4_12==-1/720
w_4_13==1/720
w_4_14==1/480
w_4_15==-1/1440
w_4_16==1/1440
w_4_17==-1/480
w_4_18==-1/360
w_4_19==-1/480
w_4_20==-1/240
w_4_21==-1/480
w_4_22==-1/720
w_4_23==1/1440
w_4_24==1/360
w_4_25==53/1440 + 3*w_4_100 + 12*w_4_103 - 15*w_4_104 - w_4_107 + 6*w_4_108 - 6*w_4_109
w_4_26==1/120
w_4_27==1/1440
w_4_28==-1/960 - (3*w_4_108)/2
w_4_29==-49/1440 - (3*w_4_100)/2 - 9*w_4_103 + (21*w_4_104)/2 + (3*w_4_107)/2 - (9*w_4_108)/2 + 3*w_4_109
w_4_30==1/72 + 6*w_4_103 - 6*w_4_104 + 3*w_4_108 - 3*w_4_109
w_4_31==61/2880 + (3*w_4_100)/2 + 6*w_4_103 - (15*w_4_104)/2 - w_4_107/2 + 3*w_4_108 - 3*w_4_109
w_4_32==1/1440
w_4_33==5/288 + 6*w_4_103 - 6*w_4_104 + 3*w_4_108 - 3*w_4_109
w_4_34==1/96 + w_4_108
w_4_35==-w_4_103
w_4_36==-13/2880 - w_4_100/2 + (3*w_4_104)/2 + w_4_107/2
w_4_37==0
w_4_38==1/1440 - w_4_100/2 + w_4_103 + (3*w_4_104)/2 + w_4_107/2 + w_4_108/2
w_4_39==0
w_4_40==0
w_4_41==1/1440
w_4_42==1/1440
w_4_43==37/1440 + 6*w_4_103 - 6*w_4_104 - w_4_107 + 3*w_4_108 - 3*w_4_109
w_4_44==17/360 + 15*w_4_103 - 18*w_4_104 - 2*w_4_107 + 6*w_4_108 - 6*w_4_109
w_4_45==7/1440 - 3*w_4_104 - w_4_107
w_4_46==-1/480
w_4_47==1/60 + 6*w_4_103 - 6*w_4_104 + 3*w_4_108 - 3*w_4_109
w_4_48==11/1440 - w_4_100/2 - w_4_103 + (5*w_4_104)/2 + w_4_107/2 + (3*w_4_108)/2
w_4_49==-w_4_104
w_4_50==-1/192 - w_4_108/2
w_4_51==-w_4_103
w_4_52==-1/1440 + w_4_100/2 - (3*w_4_104)/2 - w_4_107/2 - w_4_108/2
w_4_53==w_4_103
w_4_54==-1/576 + w_4_103 - w_4_104 - w_4_108/2
w_4_55==w_4_104
w_4_56==0
w_4_57==0
w_4_58==0
w_4_59==0
w_4_60==0
w_4_61==0
w_4_62==0
w_4_63==0
w_4_64==0
w_4_65==0
w_4_66==0
w_4_67==0
w_4_68==0
w_4_69==0
w_4_70==0
w_4_71==0
w_4_72==1/1440
w_4_73==1/1440
w_4_74==1/1440
w_4_75==-1/480
w_4_76==-1/720
w_4_77==1/180 + 3*w_4_103 - 3*w_4_104 - w_4_107
w_4_78==-1/144 - 3*w_4_103 + 3*w_4_104 + w_4_107
w_4_79==-1/1440
w_4_80==1/80 + w_4_100 - 3*w_4_104 + 3*w_4_108 - 2*w_4_109 - 2*w_4_125
w_4_81==1/480 - w_4_100/2 + 2*w_4_103 + w_4_104/2 + w_4_107/2 + w_4_108/2 + 2*w_4_125
\end{verbatim}
\end{table}

\begin{table}[h!]
\centerline{\textsc{Table}~\ref{Table149via10} (continued).}\label{Table149via10-part2}
\renewcommand{\baselinestretch}{0.7}
\tiny\begin{verbatim}
w_4_82==1/2880 - w_4_103 - w_4_104 + w_4_108/2 - w_4_109 - 2*w_4_125
w_4_83==-1/480
w_4_84==-1/720
w_4_85==1/180 - w_4_107
w_4_86==1/480
w_4_87==-1/1440
w_4_88==1/96 + 4*w_4_103 - 4*w_4_104 + 2*w_4_108 - 2*w_4_109
w_4_89==1/240 + (3*w_4_100)/2 + 3*w_4_103 - (9*w_4_104)/2 + w_4_107/2 + (3*w_4_108)/2 - 2*w_4_109
w_4_90==-1/192 - w_4_108/2
w_4_91==-1/720
w_4_92==17/1440 + 6*w_4_103 - 6*w_4_104 - 2*w_4_107
w_4_93==3/320 - w_4_100/2 + w_4_103 - (5*w_4_104)/2 + w_4_107/2 + 2*w_4_108 - 2*w_4_109 + w_4_119
w_4_94==1/1440 - w_4_100/2 - w_4_102 + w_4_103 - (3*w_4_104)/2 + w_4_107/2 + w_4_108/2 - w_4_109
w_4_95==-1/576 + w_4_103 - w_4_104 - w_4_108/2
w_4_96==0
w_4_97==0
w_4_98==0
w_4_99==-7/2880 - w_4_100/2 + w_4_103 + w_4_104/2 - w_4_107/2 - w_4_108 + w_4_109 - w_4_119
w_4_105==-1/160
w_4_106==13/1440
w_4_110==-1/288 - 2*w_4_103 + 2*w_4_104 - w_4_108 + w_4_109
w_4_111==-17/2880 - w_4_100/2 - 2*w_4_103 + (5*w_4_104)/2 - w_4_107/2 - w_4_108 + w_4_109
w_4_112==-7/576 - 4*w_4_103 + 4*w_4_104 - 2*w_4_108 + 2*w_4_109
w_4_113==-1/192 - 2*w_4_103 + 2*w_4_104 - w_4_108 + w_4_109
w_4_114==1/360 + w_4_108
w_4_115==23/5760 - w_4_100/2 + w_4_103 + (3*w_4_108)/4
w_4_116==0
w_4_117==-19/2880 + w_4_100 - 2*w_4_103 - w_4_108
w_4_118==-31/1440 - 12*w_4_103 + 12*w_4_104 + 4*w_4_107
w_4_120==-1/96 - w_4_100 - w_4_102 + 2*w_4_103 - 2*w_4_108 + w_4_109
w_4_121==-1/288 + 2*w_4_103 - 2*w_4_104 - w_4_108
w_4_122==-2*w_4_103
w_4_123==-7/2880 + w_4_100/2 - w_4_103 + w_4_104/2 - w_4_107/2 - w_4_108 + w_4_109
w_4_124==1/144 + w_4_100 + w_4_103 - 2*w_4_104 + w_4_108 - w_4_109
w_4_126==29/5760 + w_4_100/2 - w_4_103 + (5*w_4_108)/4 - w_4_109 - w_4_125
w_4_127==-1/640 + w_4_103 + w_4_104/2 - (3*w_4_108)/4 + w_4_109/2 + w_4_125
w_4_128==-1/144 + w_4_101 - 2*w_4_103 + 3*w_4_104 - w_4_108 + w_4_109
w_4_129==1/144 + w_4_101 + 2*w_4_103 - 3*w_4_104 + w_4_108 - w_4_109
w_4_130==7/1920 - w_4_100/2 + w_4_103 + w_4_107/2 + (3*w_4_108)/4 - w_4_109/2 + w_4_119 + w_4_125
w_4_131==23/5760 - w_4_100/4 + w_4_101/2 - w_4_102/2 + w_4_103/2 - (5*w_4_104)/4 + w_4_107/4 + w_4_108 - w_4_109 +
         w_4_119/2 - w_4_125
w_4_132==-1/240 + w_4_101/2 + w_4_103/2 - w_4_108 + w_4_109/2 + w_4_125
w_4_133==2*w_4_104
w_4_134==0
w_4_135==-1/360 - w_4_100/4 - w_4_102/2 + w_4_104/4 + w_4_107/4 - w_4_108/4 + w_4_119/2
w_4_136==-7/1440 - w_4_100/2 - w_4_102 + w_4_103 - w_4_104/2 - w_4_108 + w_4_109/2
w_4_137==0
w_4_138==-1/144 - 2*w_4_103 + 2*w_4_104 - w_4_108 + w_4_109
w_4_139==1/1920 + w_4_103 - (3*w_4_104)/2 + w_4_108/4 - w_4_109/2
w_4_140==-1/1440 + w_4_100 + 2*w_4_103 - 5*w_4_104 - w_4_109
w_4_141==-w_4_100/4 - w_4_101/2 - w_4_102/2 - w_4_103/2 + w_4_104/4 + w_4_107/4 + w_4_108/4 - w_4_109/2 + w_4_119/2
w_4_142==1/5760 - w_4_102 + 2*w_4_103 - 3*w_4_104 - w_4_109
w_4_143==7/1440 + w_4_101/2 + (5*w_4_103)/2 - (5*w_4_104)/2 + (3*w_4_108)/4 - w_4_109
w_4_144==0
w_4_145==0
w_4_146==13/5760 + w_4_100/2 - w_4_101/2 + (3*w_4_103)/2 - 2*w_4_104 + w_4_108/4 - w_4_109/2
w_4_147==1/320 - w_4_101/2 + (3*w_4_103)/2 - w_4_104 + w_4_108/2 - w_4_109/2
w_4_148==11/1920 + 2*w_4_103 - w_4_104 + w_4_108 - w_4_109
w_4_149==-11/2880 + w_4_100/2 + w_4_104/2 - w_4_107/2 - w_4_108 + w_4_109 - w_4_119
\end{verbatim}
\end{table}

\clearpage

\section{Encoding of the associator of the $\star$-product modulo $\bar{o}(\hbar^4)$}
\label{AppAssocEncoding}

Encodings of graphs (see Implementation \ref{DefEncoding} on p. \pageref{DefEncoding}) are followed by their coefficients, in the following table containing the expansion of the associator $(f\star g)\star h - f \star (g\star h)$. 
\begin{table}[h!]
\caption{The associator of $\star$ up to order $4$ in terms of $149$ parameters.}
\label{TableAssoc4}
\begin{tabular}{p{0.33\textwidth} | p{0.33\textwidth} | p{0.33\textwidth}}
\renewcommand{\baselinestretch}{0.7}
\tiny\begin{verbatim}
h^0:
h^1:
h^2:
# 1 1 1 
3 2 1   0 1 2 3    -2/3
3 2 1   0 2 1 3    2/3
3 2 1   0 4 1 2    -2/3
h^3:
# 1 2 2 
3 3 1   0 1 1 2 2 3    -2/3
3 3 1   0 2 1 2 1 3    2/3
3 3 1   0 4 1 2 1 2    -2/3
# 2 1 2 
3 3 1   0 1 0 2 2 3    -2/3
3 3 1   0 2 0 2 1 3    2/3
3 3 1   0 2 0 5 1 2    -2/3
# 2 2 1 
3 3 1   0 1 0 1 2 3    -2/3
3 3 1   0 1 0 2 1 4    2/3
3 3 1   0 1 0 5 1 2    -2/3
# 1 2 1 
3 3 1   0 1 1 3 2 4    1/3
3 3 1   0 1 1 5 2 3    -1/3
3 3 1   0 4 1 5 1 2    -1/3
3 3 1   0 4 1 2 1 3    1/3
# 1 1 1 
3 3 1   0 4 1 3 2 4    1/6
3 3 1   0 1 2 3 3 4    -1/6
3 3 1   0 1 2 5 3 4    -1/6
3 3 1   0 4 1 2 3 4    -1/6
3 3 1   0 4 2 3 1 4    -1/6
3 3 1   0 4 1 5 2 4    1/6
3 3 1   0 4 3 5 1 2    -1/6
3 3 1   0 4 2 3 1 3    -1/6
# 1 1 2 
3 3 1   0 1 2 3 2 3    1/3
3 3 1   0 1 2 3 2 4    1/3
3 3 1   0 2 1 5 2 3    -1/6
3 3 1   0 4 2 5 1 2    1/6
3 3 1   0 2 1 3 2 4    -1/6
3 3 1   0 4 1 2 2 3    1/6
3 3 1   0 4 1 2 2 4    1/3
3 3 1   0 2 1 2 3 4    -1/3
3 3 1   0 2 1 3 2 3    -1/3
# 2 1 1 
3 3 1   0 1 0 3 2 3    -1/3
3 3 1   0 1 0 3 2 4    -1/6
3 3 1   0 1 0 2 3 4    1/3
3 3 1   0 1 0 5 2 3    -1/6
3 3 1   0 2 0 3 1 3    1/3
3 3 1   0 4 1 2 0 4    -1/3
3 3 1   0 4 0 5 1 2    -1/3
3 3 1   0 2 0 5 1 3    1/6
3 3 1   0 2 0 3 1 4    1/6
h^4:
# 3 3 1 
3 4 1   0 1 0 1 0 1 2 3    -1/3
3 4 1   0 1 0 1 0 2 1 5    1/3
3 4 1   0 1 0 1 0 6 1 2    -1/3
# 3 2 2 
3 4 1   0 1 0 1 0 2 2 3    -2/3
3 4 1   0 1 0 2 0 2 1 4    2/3
3 4 1   0 1 0 2 0 6 1 2    -2/3
# 3 1 3 
3 4 1   0 1 0 2 0 2 2 3    -1/3
3 4 1   0 2 0 2 0 2 1 3    1/3
3 4 1   0 2 0 2 0 6 1 2    -1/3
# 2 3 2 
3 4 1   0 1 0 1 1 2 2 3    -2/3
3 4 1   0 1 0 2 1 2 1 4    2/3
3 4 1   0 1 0 5 1 2 1 2    -2/3
# 1 3 3 
3 4 1   0 1 1 2 1 2 2 3    -1/3
\end{verbatim}
&
\renewcommand{\baselinestretch}{0.7}
\tiny\begin{verbatim}
3 4 1   0 2 1 2 1 2 1 3    1/3
3 4 1   0 4 1 2 1 2 1 2    -1/3
# 2 2 3 
3 4 1   0 1 0 2 1 2 2 3    -2/3
3 4 1   0 2 0 2 1 2 1 3    2/3
3 4 1   0 2 0 5 1 2 1 2    -2/3
# 1 3 2 
3 4 1   0 1 1 2 1 3 2 5    1/3
3 4 1   0 1 1 2 1 4 2 3    -2/9
3 4 1   0 1 1 2 1 6 2 3    -1/3
3 4 1   0 2 1 2 1 3 1 4    2/9
3 4 1   0 4 1 2 1 2 1 5    -2/9
3 4 1   0 4 1 5 1 2 1 2    -1/3
3 4 1   0 4 1 2 1 2 1 3    1/3
# 2 3 1 
3 4 1   0 1 0 1 1 3 2 4    2/9
3 4 1   0 1 0 1 1 3 2 5    1/3
3 4 1   0 1 0 1 1 6 2 3    -1/3
3 4 1   0 1 0 2 1 3 1 4    -2/9
3 4 1   0 1 0 5 1 2 1 3    2/9
3 4 1   0 1 0 5 1 6 1 2    -1/3
3 4 1   0 1 0 5 1 2 1 4    1/3
# 2 1 3 
3 4 1   0 1 0 2 2 3 2 3    1/3
3 4 1   0 1 0 2 2 3 2 4    2/9
3 4 1   0 1 0 2 2 3 2 5    1/3
3 4 1   0 2 0 2 1 6 2 3    -1/6
3 4 1   0 2 0 5 2 6 1 2    1/6
3 4 1   0 2 0 2 1 3 2 5    -1/6
3 4 1   0 2 0 5 1 2 2 4    1/6
3 4 1   0 2 0 2 1 3 2 4    -2/9
3 4 1   0 2 0 5 1 2 2 3    2/9
3 4 1   0 2 0 5 1 2 2 5    1/3
3 4 1   0 2 0 2 1 2 3 5    -1/3
3 4 1   0 2 0 2 1 3 2 3    -1/3
# 3 2 1 
3 4 1   0 1 0 1 0 3 2 3    -1/3
3 4 1   0 1 0 1 0 3 2 4    -2/9
3 4 1   0 1 0 1 0 3 2 5    -1/6
3 4 1   0 1 0 1 0 2 3 5    1/3
3 4 1   0 1 0 1 0 6 2 3    -1/6
3 4 1   0 1 0 2 0 3 1 4    2/9
3 4 1   0 1 0 3 0 6 1 2    -2/9
3 4 1   0 1 0 2 0 4 1 4    1/3
3 4 1   0 1 0 5 1 2 0 5    -1/3
3 4 1   0 1 0 5 0 6 1 2    -1/3
3 4 1   0 1 0 2 0 6 1 4    1/6
3 4 1   0 1 0 2 0 4 1 5    1/6
# 3 1 2 
3 4 1   0 1 0 2 0 3 2 3    -1/3
3 4 1   0 1 0 2 0 3 2 5    -1/6
3 4 1   0 1 0 2 0 2 3 4    1/3
3 4 1   0 1 0 2 0 4 2 3    -2/9
3 4 1   0 1 0 2 0 6 2 3    -1/6
3 4 1   0 2 0 2 0 3 1 3    1/3
3 4 1   0 2 0 5 1 2 0 5    -1/3
3 4 1   0 2 0 2 0 3 1 4    2/9
3 4 1   0 2 0 3 0 6 1 2    -2/9
3 4 1   0 2 0 5 0 6 1 2    -1/3
3 4 1   0 2 0 2 0 6 1 3    1/6
3 4 1   0 2 0 2 0 3 1 5    1/6
# 1 2 3 
3 4 1   0 1 1 2 2 3 2 3    1/3
3 4 1   0 1 1 2 2 3 2 4    2/9
3 4 1   0 1 1 2 2 3 2 5    1/3
3 4 1   0 2 1 2 1 6 2 3    -1/6
3 4 1   0 4 2 5 1 2 1 2    1/6
3 4 1   0 2 1 2 1 3 2 5    -1/6
3 4 1   0 4 1 2 1 2 2 3    1/6
3 4 1   0 2 1 2 1 3 2 4    -2/9
3 4 1   0 4 1 2 1 2 2 5    2/9
3 4 1   0 4 1 2 1 2 2 4    1/3
3 4 1   0 2 1 2 1 2 3 4    -1/3
\end{verbatim}
&
\renewcommand{\baselinestretch}{0.7}
\tiny\begin{verbatim}
3 4 1   0 2 1 2 1 3 2 3    -1/3
# 2 2 2 
3 4 1   0 1 0 1 2 3 2 3    1/3
3 4 1   0 1 0 1 2 3 2 4    4/9
3 4 1   0 1 0 3 1 2 2 3    -1/3
3 4 1   0 1 0 3 1 2 2 4    -1/6
3 4 1   0 1 0 2 1 3 2 5    1/3
3 4 1   0 1 0 1 2 3 2 5    1/3
3 4 1   0 1 0 2 1 2 3 4    1/3
3 4 1   0 1 0 2 1 4 2 3    -4/9
3 4 1   0 1 0 2 1 6 2 3    -1/3
3 4 1   0 1 0 5 2 3 1 2    -1/6
3 4 1   0 1 0 2 1 6 2 4    -1/6
3 4 1   0 1 0 5 2 6 1 2    1/6
3 4 1   0 1 0 2 1 4 2 5    -1/6
3 4 1   0 1 0 5 1 2 2 4    1/6
3 4 1   0 2 0 3 1 2 1 3    1/3
3 4 1   0 4 1 2 0 4 1 2    -1/3
3 4 1   0 2 0 5 1 2 1 3    4/9
3 4 1   0 4 1 2 0 6 1 2    -4/9
3 4 1   0 2 0 5 1 6 1 2    -1/3
3 4 1   0 1 0 5 1 2 2 5    1/3
3 4 1   0 4 0 5 1 2 1 2    -1/3
3 4 1   0 1 0 2 1 2 4 5    -1/3
3 4 1   0 2 0 5 1 2 1 4    1/3
3 4 1   0 2 0 5 1 3 1 2    1/6
3 4 1   0 2 0 3 1 2 1 4    1/6
3 4 1   0 1 0 2 1 4 2 4    -1/3
# 1 3 1 
3 4 1   0 1 1 3 1 3 2 3    -1/6+8*w_4_6
3 4 1   0 1 1 3 1 3 2 4    -1/3+16*w_4_7
3 4 1   0 1 1 2 1 3 3 4    2/9
3 4 1   0 1 1 2 1 3 4 5    1/9
3 4 1   0 1 1 3 1 6 2 3    1/9+16*w_4_7
3 4 1   0 1 1 3 1 6 2 4    1/9+16*w_4_12
3 4 1   0 1 1 2 1 4 3 4    2/9
3 4 1   0 1 1 5 2 3 1 5    -1/6+8*w_4_11
3 4 1   0 2 1 3 1 3 1 3    16/3*w_4_6
3 4 1   0 4 1 2 1 4 1 4    -1/6+8*w_4_6
3 4 1   0 2 1 3 1 3 1 4    32*w_4_7
3 4 1   0 4 1 2 1 3 1 4    1/9+16*w_4_7
3 4 1   0 4 1 2 1 4 1 5    16*w_4_7
3 4 1   0 4 1 5 1 2 1 5    -1/3+16*w_4_7
3 4 1   0 2 1 3 1 4 1 4    16*w_4_11
3 4 1   0 4 1 2 1 3 1 3    -1/6+8*w_4_11
3 4 1   0 4 1 5 1 2 1 4    16*w_4_11
3 4 1   0 2 1 3 1 4 1 5    32*w_4_12
3 4 1   0 4 1 2 1 3 1 5    16*w_4_12
3 4 1   0 4 1 5 1 2 1 3    1/9+16*w_4_12
3 4 1   0 4 1 5 1 6 1 2    16*w_4_12
3 4 1   0 1 1 2 1 4 3 5    -1/9
3 4 1   0 1 1 3 1 4 2 3    16*w_4_7
3 4 1   0 1 1 3 1 4 2 4    16*w_4_11
3 4 1   0 1 1 3 1 4 2 5    16*w_4_12
3 4 1   0 1 1 5 1 6 2 3    16*w_4_12
# 3 1 1 
3 4 1   0 1 0 3 0 3 2 3    -1/6-8*w_4_8
3 4 1   0 1 0 3 0 3 2 4    -1/3-16*w_4_9
3 4 1   0 1 0 2 0 3 3 4    1/9-16*w_4_1
3 4 1   0 1 0 2 0 3 4 5    -1/9-16*w_4_2
3 4 1   0 1 0 3 0 6 2 3    -1/9-16*w_4_19
3 4 1   0 1 0 3 0 6 2 4    -1/9-16*w_4_20
3 4 1   0 1 0 2 0 4 3 4    1/3+16*w_4_1
3 4 1   0 1 0 5 2 3 0 5    -1/6+8*w_4_26
3 4 1   0 2 0 3 0 3 1 3    -8*w_4_8+8*w_4_6
3 4 1   0 4 1 2 0 4 0 4    -1/6+8/3*w_4_6
3 4 1   0 2 0 3 0 3 1 4    -16*w_4_9+16*w_4_7
3 4 1   0 2 0 3 0 4 1 3    -16*w_4_13+16*w_4_7
3 4 1   0 2 0 3 0 6 1 3    -16*w_4_19+16*w_4_7
3 4 1   0 4 0 5 1 2 0 5    -1/3+16*w_4_7
3 4 1   0 2 0 3 0 4 1 4    16*w_4_11-16*w_4_14
3 4 1   0 2 0 5 1 3 0 5    8*w_4_11+8*w_4_26
3 4 1   0 4 0 5 1 2 0 4    -1/6+8*w_4_11
\end{verbatim}
\end{tabular}
\end{table}

\begin{table}
\centerline{\textsc{Table}~\ref{TableAssoc4} (part 2).}\label{TableAssoc4-part2}
\vskip 1em
\begin{tabular}{p{0.5\textwidth} | p{0.5\textwidth}}
\renewcommand{\baselinestretch}{0.7}
\tiny\begin{verbatim}
3 4 1   0 2 0 3 0 4 1 5    -16*w_4_15+16*w_4_12
3 4 1   0 2 0 3 0 6 1 4    16*w_4_12-16*w_4_20
3 4 1   0 2 0 5 0 6 1 3    16*w_4_12+16*w_4_27
3 4 1   0 4 0 5 0 6 1 2    16*w_4_12
3 4 1   0 1 0 2 0 4 3 5    -16*w_4_2
3 4 1   0 1 0 3 0 4 2 3    -16*w_4_13
3 4 1   0 1 0 3 0 4 2 4    -16*w_4_14
3 4 1   0 1 0 3 0 4 2 5    -16*w_4_15
3 4 1   0 1 0 5 0 6 2 3    16*w_4_27
# 1 1 3 
3 4 1   0 1 2 3 2 3 2 3    -1/6+8/3*w_4_6
3 4 1   0 1 2 3 2 3 2 4    -1/3+16*w_4_7
3 4 1   0 1 2 3 2 4 2 4    -1/6+8*w_4_11
3 4 1   0 2 1 2 2 3 3 4    1/3+16*w_4_1
3 4 1   0 2 1 2 2 4 3 4    1/9-16*w_4_1
3 4 1   0 2 1 2 2 3 4 5    16*w_4_2
3 4 1   0 2 1 2 2 4 3 5    1/9+16*w_4_2
3 4 1   0 2 1 3 2 3 2 3    -8*w_4_8+8*w_4_6
3 4 1   0 4 1 2 2 4 2 4    -1/6-8*w_4_8
3 4 1   0 2 1 5 2 3 2 3    -16*w_4_9+16*w_4_7
3 4 1   0 4 2 5 1 2 2 5    -1/3-16*w_4_9
3 4 1   0 2 1 3 2 3 2 5    -16*w_4_13+16*w_4_7
3 4 1   0 4 1 2 2 4 2 5    -16*w_4_13
3 4 1   0 2 1 5 2 3 2 5    16*w_4_11-16*w_4_14
3 4 1   0 4 2 5 1 2 2 4    -16*w_4_14
3 4 1   0 2 1 5 2 6 2 3    -16*w_4_15+16*w_4_12
3 4 1   0 4 2 5 2 6 1 2    -16*w_4_15
3 4 1   0 2 1 3 2 3 2 4    -16*w_4_19+16*w_4_7
3 4 1   0 4 1 2 2 3 2 4    -1/9-16*w_4_19
3 4 1   0 2 1 5 2 3 2 4    16*w_4_12-16*w_4_20
3 4 1   0 4 2 5 1 2 2 3    -1/9-16*w_4_20
3 4 1   0 2 1 3 2 4 2 4    8*w_4_11+8*w_4_26
3 4 1   0 4 1 2 2 3 2 3    -1/6+8*w_4_26
3 4 1   0 2 1 3 2 4 2 5    16*w_4_12+16*w_4_27
3 4 1   0 4 1 2 2 3 2 5    16*w_4_27
3 4 1   0 1 2 3 2 4 2 5    16*w_4_12
# 2 2 1 
3 4 1   0 1 0 5 1 4 2 3    1/9
3 4 1   0 1 0 5 1 4 2 5    1/6
3 4 1   0 1 0 5 1 3 2 3    1/6+16*w_4_9
3 4 1   0 1 0 5 1 3 2 4    1/6+16*w_4_20
3 4 1   0 1 0 5 1 3 2 5    1/6+16*w_4_14
3 4 1   0 1 0 3 1 4 2 3    1/6+16*w_4_19
3 4 1   0 1 0 3 1 4 2 4    1/6-16*w_4_26
3 4 1   0 1 0 3 1 4 2 5    1/6-16*w_4_27
3 4 1   0 1 0 1 2 3 3 5    -1/6
3 4 1   0 1 0 1 2 3 4 5    -1/3-16*w_4_2
3 4 1   0 1 0 3 1 2 3 5    1/9
3 4 1   0 1 0 3 1 2 4 5    -1/18
3 4 1   0 1 0 3 1 6 2 3    -1/9+16*w_4_13
3 4 1   0 1 0 3 1 6 2 4    -1/9-16*w_4_27
3 4 1   0 1 0 2 1 3 3 4    -1/9-16*w_4_1
3 4 1   0 1 0 2 1 3 4 5    1/9-16*w_4_2
3 4 1   0 1 0 5 2 3 1 3    1/9+16*w_4_9
3 4 1   0 1 0 5 2 6 1 3    1/9+16*w_4_15
3 4 1   0 2 0 5 1 4 1 3    -1/9
3 4 1   0 4 1 2 0 6 1 5    1/9
3 4 1   0 1 0 1 2 6 3 5    -1/6
3 4 1   0 1 0 2 1 4 3 4    1/3+16*w_4_1
3 4 1   0 1 0 5 2 3 1 5    -1/3+16*w_4_14
3 4 1   0 1 0 5 1 2 4 5    -1/6
3 4 1   0 1 0 5 2 4 1 5    -1/6
3 4 1   0 2 0 3 1 3 1 3    8*w_4_8+8*w_4_6
3 4 1   0 4 1 2 0 4 1 4    -1/6+8*w_4_6
3 4 1   0 2 0 3 1 3 1 4    16*w_4_19+16*w_4_7
3 4 1   0 2 0 3 1 3 1 5    16*w_4_13+16*w_4_7
3 4 1   0 2 0 5 1 3 1 3    16*w_4_9+16*w_4_7
3 4 1   0 4 1 2 0 4 1 3    1/3+16*w_4_7
3 4 1   0 4 0 5 1 2 1 5    -1/9+16*w_4_7
3 4 1   0 4 1 2 0 6 1 4    -1/3+16*w_4_7
3 4 1   0 2 0 3 1 4 1 4    8*w_4_11-8*w_4_26
3 4 1   0 2 0 5 1 3 1 5    16*w_4_11+16*w_4_14
3 4 1   0 4 0 5 1 2 1 4    16*w_4_11
3 4 1   0 4 1 5 1 2 0 4    -1/6+8*w_4_11
3 4 1   0 2 0 3 1 4 1 5    16*w_4_12-16*w_4_27
3 4 1   0 2 0 5 1 3 1 4    16*w_4_12+16*w_4_20
3 4 1   0 2 0 5 1 6 1 3    16*w_4_15+16*w_4_12
3 4 1   0 4 0 5 1 2 1 3    1/6+16*w_4_12
3 4 1   0 4 1 2 0 6 1 3    1/6+16*w_4_12
3 4 1   0 4 0 5 1 6 1 2    -1/9+16*w_4_12
3 4 1   0 1 0 5 1 6 2 5    1/6
3 4 1   0 1 0 5 4 6 1 2    -1/6
3 4 1   0 1 0 5 3 6 1 2    -1/6
3 4 1   0 1 0 5 2 4 1 4    -1/6
\end{verbatim}
&
\renewcommand{\baselinestretch}{0.7}
\tiny\begin{verbatim}
3 4 1   0 1 0 1 2 3 3 4    -16*w_4_1
3 4 1   0 1 0 2 1 4 3 5    -16*w_4_2
3 4 1   0 1 0 3 1 3 2 3    16*w_4_8
3 4 1   0 1 0 3 1 3 2 5    16*w_4_13
3 4 1   0 1 0 5 1 6 2 3    16*w_4_15
3 4 1   0 1 0 3 1 3 2 4    16*w_4_19
3 4 1   0 1 0 5 2 3 1 4    16*w_4_20
# 1 2 2 
3 4 1   0 1 1 3 2 3 2 3    -1/6+8*w_4_6
3 4 1   0 1 1 3 2 3 2 4    -1/3+16*w_4_7
3 4 1   0 1 1 3 2 4 2 4    -1/6+8*w_4_11
3 4 1   0 4 1 3 1 2 2 4    1/6
3 4 1   0 1 1 3 2 3 2 5    -1/9+16*w_4_7
3 4 1   0 1 1 3 2 4 2 5    -1/9+16*w_4_12
3 4 1   0 1 1 2 2 3 3 5    -1/6
3 4 1   0 1 1 5 2 3 2 3    1/3+16*w_4_7
3 4 1   0 1 1 2 2 6 3 5    -1/6
3 4 1   0 1 1 5 2 4 2 3    1/9
3 4 1   0 1 1 2 2 4 3 5    -1/18
3 4 1   0 1 1 5 2 6 2 3    1/6+16*w_4_12
3 4 1   0 1 1 2 2 6 3 4    -1/6
3 4 1   0 1 1 5 2 3 2 4    1/6+16*w_4_12
3 4 1   0 4 1 2 1 2 3 4    -1/6
3 4 1   0 4 2 3 1 2 1 4    -1/6
3 4 1   0 2 1 3 1 6 2 5    -1/9
3 4 1   0 4 1 2 1 6 2 5    1/9
3 4 1   0 2 1 2 1 3 3 4    1/3+16*w_4_1
3 4 1   0 2 1 2 1 4 3 4    -1/9-16*w_4_1
3 4 1   0 4 1 2 1 2 4 5    16*w_4_1
3 4 1   0 2 1 2 1 3 4 5    16*w_4_2
3 4 1   0 2 1 2 1 4 3 5    -1/9+16*w_4_2
3 4 1   0 4 1 2 1 2 3 5    -1/3-16*w_4_2
3 4 1   0 2 1 3 1 3 2 3    8*w_4_8+8*w_4_6
3 4 1   0 4 1 2 1 4 2 4    16*w_4_8
3 4 1   0 2 1 3 1 3 2 4    16*w_4_9+16*w_4_7
3 4 1   0 4 1 2 1 4 2 3    1/9+16*w_4_9
3 4 1   0 4 1 2 1 4 2 5    1/6+16*w_4_9
3 4 1   0 2 1 3 1 4 2 3    16*w_4_13+16*w_4_7
3 4 1   0 4 1 2 1 3 2 4    -1/9+16*w_4_13
3 4 1   0 4 1 5 1 2 2 5    16*w_4_13
3 4 1   0 2 1 3 1 4 2 4    16*w_4_11+16*w_4_14
3 4 1   0 4 1 2 1 3 2 3    -1/3+16*w_4_14
3 4 1   0 4 1 5 1 2 2 4    1/6+16*w_4_14
3 4 1   0 2 1 3 1 4 2 5    16*w_4_15+16*w_4_12
3 4 1   0 4 1 2 1 3 2 5    16*w_4_15
3 4 1   0 4 1 5 1 2 2 3    1/9+16*w_4_15
3 4 1   0 2 1 3 1 6 2 3    16*w_4_19+16*w_4_7
3 4 1   0 4 1 2 1 6 2 4    1/6+16*w_4_19
3 4 1   0 4 2 5 1 2 1 5    16*w_4_19
3 4 1   0 2 1 3 1 6 2 4    16*w_4_12+16*w_4_20
3 4 1   0 4 1 2 1 6 2 3    16*w_4_20
3 4 1   0 4 2 5 1 6 1 2    1/6+16*w_4_20
3 4 1   0 2 1 5 2 3 1 5    8*w_4_11-8*w_4_26
3 4 1   0 4 2 5 1 2 1 4    1/6-16*w_4_26
3 4 1   0 2 1 5 1 6 2 3    16*w_4_12-16*w_4_27
3 4 1   0 4 1 5 2 6 1 2    1/6-16*w_4_27
3 4 1   0 4 2 5 1 2 1 3    -1/9-16*w_4_27
3 4 1   0 4 1 5 2 4 1 2    1/6
3 4 1   0 1 1 2 2 4 3 4    1/9
3 4 1   0 4 3 5 1 2 1 2    -1/6
3 4 1   0 4 2 3 1 2 1 3    -1/6
3 4 1   0 1 1 5 2 3 2 5    16*w_4_11
# 2 1 2 
3 4 1   0 1 0 3 2 3 2 3    1/6+8*w_4_8
3 4 1   0 1 0 3 2 3 2 4    1/3+16*w_4_19
3 4 1   0 1 0 3 2 4 2 4    1/6-8*w_4_26
3 4 1   0 2 0 5 1 4 2 5    1/6
3 4 1   0 1 0 3 2 3 2 5    1/9+16*w_4_13
3 4 1   0 1 0 3 2 4 2 5    1/9-16*w_4_27
3 4 1   0 1 0 2 2 3 3 4    -1/3-16*w_4_1
3 4 1   0 1 0 2 2 3 3 5    -1/6
3 4 1   0 1 0 5 2 3 2 3    1/3+16*w_4_9
3 4 1   0 1 0 2 2 6 3 5    -1/6
3 4 1   0 1 0 5 2 4 2 3    1/9
3 4 1   0 1 0 2 2 4 3 5    -1/6-16*w_4_2
3 4 1   0 1 0 5 2 6 2 3    1/6+16*w_4_15
3 4 1   0 1 0 2 2 6 3 4    -1/6
3 4 1   0 1 0 5 2 3 2 4    1/6+16*w_4_20
3 4 1   0 2 0 5 1 2 4 5    -1/6
3 4 1   0 2 0 5 2 4 1 5    -1/6
3 4 1   0 2 0 5 2 4 1 3    -1/9
3 4 1   0 4 1 2 0 6 2 5    1/9
3 4 1   0 2 0 3 1 2 3 5    16*w_4_1
3 4 1   0 2 0 5 1 2 3 5    -1/3-16*w_4_1
\end{verbatim}
\end{tabular}
\end{table}

\begin{table}
\centerline{\textsc{Table}~\ref{TableAssoc4} (part 3).}\label{TableAssoc4-part3}
\vskip 1em
\begin{tabular}{p{0.5\textwidth} | p{0.5\textwidth}}
\renewcommand{\baselinestretch}{0.7}
\tiny\begin{verbatim}
3 4 1   0 2 0 3 1 2 4 5    -1/6-16*w_4_2
3 4 1   0 2 0 5 1 2 3 4    16*w_4_2
3 4 1   0 2 0 3 1 3 2 3    32*w_4_8
3 4 1   0 4 1 2 0 4 2 4    1/6+8*w_4_8
3 4 1   0 2 0 3 1 3 2 4    16*w_4_9+16*w_4_19
3 4 1   0 2 0 3 1 3 2 5    16*w_4_9+16*w_4_13
3 4 1   0 4 1 2 0 4 2 3    1/3+16*w_4_9
3 4 1   0 2 0 3 1 4 2 3    16*w_4_19+16*w_4_13
3 4 1   0 2 0 5 1 3 2 3    16*w_4_9+16*w_4_13
3 4 1   0 4 0 5 1 2 2 5    1/9+16*w_4_13
3 4 1   0 2 0 3 1 4 2 4    -16*w_4_26+16*w_4_14
3 4 1   0 2 0 5 1 3 2 5    32*w_4_14
3 4 1   0 4 0 5 1 2 2 4    16*w_4_14
3 4 1   0 2 0 3 1 4 2 5    16*w_4_15-16*w_4_27
3 4 1   0 2 0 5 1 3 2 4    16*w_4_15+16*w_4_20
3 4 1   0 4 0 5 1 2 2 3    1/6+16*w_4_15
3 4 1   0 2 0 3 1 6 2 3    16*w_4_19+16*w_4_13
3 4 1   0 2 0 5 2 3 1 3    16*w_4_9+16*w_4_19
3 4 1   0 4 1 2 0 6 2 4    1/3+16*w_4_19
3 4 1   0 2 0 3 1 6 2 4    16*w_4_20-16*w_4_27
3 4 1   0 2 0 5 2 6 1 3    16*w_4_15+16*w_4_20
3 4 1   0 4 1 2 0 6 2 3    1/6+16*w_4_20
3 4 1   0 2 0 5 2 3 1 5    -16*w_4_26+16*w_4_14
3 4 1   0 4 2 5 1 2 0 4    1/6-8*w_4_26
3 4 1   0 2 0 5 1 6 2 3    16*w_4_15-16*w_4_27
3 4 1   0 2 0 5 2 3 1 4    16*w_4_20-16*w_4_27
3 4 1   0 4 0 5 2 6 1 2    1/9-16*w_4_27
3 4 1   0 2 0 5 1 6 2 5    1/6
3 4 1   0 2 0 5 4 6 1 2    -1/6
3 4 1   0 2 0 5 3 6 1 2    -1/6
3 4 1   0 2 0 5 2 4 1 4    -1/6
3 4 1   0 1 0 2 2 4 3 4    16*w_4_1
3 4 1   0 1 0 2 2 3 4 5    -16*w_4_2
3 4 1   0 1 0 5 2 3 2 5    16*w_4_14
# 1 2 1 
3 4 1   0 1 1 3 2 3 3 4    1/6+16*w_4_10
3 4 1   0 1 1 3 2 4 3 4    1/6+16*w_4_16
3 4 1   0 1 1 3 2 6 3 4    1/6+16*w_4_21
3 4 1   0 4 1 3 1 3 2 3    -1/6-16*w_4_105+16*w_4_60
3 4 1   0 4 1 3 1 3 2 4    -1/6+16*w_4_61-16*w_4_84
3 4 1   0 4 1 3 1 3 2 5    -1/6-16*w_4_91+16*w_4_62
3 4 1   0 1 1 3 2 3 3 5    1/9+16*w_4_10
3 4 1   0 1 1 3 2 3 4 5    1/9
3 4 1   0 1 1 3 2 4 3 5    1/9+16*w_4_17
3 4 1   0 1 1 3 2 4 4 5    1/9+16*w_4_18
3 4 1   0 4 1 3 1 2 3 5    1/9-16*w_4_40
3 4 1   0 4 1 3 1 2 4 5    -1/18-16*w_4_40
3 4 1   0 4 1 3 1 6 2 3    1/18+16*w_4_63-16*w_4_91
3 4 1   0 4 1 3 1 6 2 4    1/18+16*w_4_63-16*w_4_74
3 4 1   0 1 1 3 2 6 3 5    1/18+16*w_4_22
3 4 1   0 1 1 3 2 6 4 5    1/18+16*w_4_23
3 4 1   0 1 1 5 2 3 3 5    -1/3+16*w_4_16
3 4 1   0 1 1 2 3 4 4 5    -1/6
3 4 1   0 1 1 5 2 3 4 5    -1/6-16*w_4_18
3 4 1   0 1 1 5 2 4 3 5    -1/9+16*w_4_40
3 4 1   0 1 1 5 2 6 3 5    -1/6+16*w_4_42
3 4 1   0 2 1 3 1 3 3 4    32*w_4_10
3 4 1   0 4 1 2 1 4 3 4    -1/9-16*w_4_10
3 4 1   0 4 1 2 1 4 4 5    1/6+16*w_4_10
3 4 1   0 2 1 3 1 4 3 4    32*w_4_16
3 4 1   0 4 1 2 1 3 3 4    1/3-16*w_4_16
3 4 1   0 4 1 5 1 2 4 5    -1/6-16*w_4_16
3 4 1   0 2 1 3 1 4 3 5    32*w_4_17
3 4 1   0 4 1 2 1 3 4 5    16*w_4_17
3 4 1   0 4 1 5 1 2 3 5    -1/9-16*w_4_17
3 4 1   0 2 1 3 1 4 4 5    32*w_4_18
3 4 1   0 4 1 2 1 3 3 5    1/6+16*w_4_18
3 4 1   0 4 1 5 1 2 3 4    -1/9-16*w_4_18
3 4 1   0 2 1 3 1 6 3 4    32*w_4_21
3 4 1   0 4 1 2 1 6 3 4    -16*w_4_21
3 4 1   0 4 5 6 1 2 1 5    1/6+16*w_4_21
3 4 1   0 2 1 3 1 6 3 5    32*w_4_22
3 4 1   0 4 1 2 1 6 4 5    16*w_4_22
3 4 1   0 4 3 5 1 2 1 5    -1/18-16*w_4_22
3 4 1   0 2 1 3 1 6 4 5    32*w_4_23
3 4 1   0 4 1 2 1 6 3 5    16*w_4_23
3 4 1   0 4 3 5 1 6 1 2    -1/18-16*w_4_23
3 4 1   0 2 1 5 1 4 3 4    32*w_4_40
3 4 1   0 2 1 5 1 6 3 4    32*w_4_41
3 4 1   0 4 1 5 3 6 1 2    -16*w_4_41
3 4 1   0 4 5 6 1 2 1 3    16*w_4_41
3 4 1   0 2 1 5 1 6 3 5    32*w_4_42
3 4 1   0 4 1 5 4 6 1 2    -16*w_4_42
3 4 1   0 4 3 5 1 2 1 3    1/6-16*w_4_42
\end{verbatim}
&
\renewcommand{\baselinestretch}{0.7}
\tiny\begin{verbatim}
3 4 1   0 2 1 5 3 4 1 5    32*w_4_46
3 4 1   0 4 3 5 1 2 1 4    -16*w_4_46
3 4 1   0 4 5 6 1 2 1 4    16*w_4_46
3 4 1   0 4 1 3 1 4 2 4    16*w_4_60-16*w_4_83
3 4 1   0 4 1 3 1 4 2 3    16*w_4_61-16*w_4_84
3 4 1   0 4 1 3 1 4 2 5    -16*w_4_74+16*w_4_62
3 4 1   0 4 1 5 1 4 2 3    -16*w_4_63+16*w_4_62
3 4 1   0 4 2 5 1 6 1 5    16*w_4_63-16*w_4_62
3 4 1   0 4 1 5 1 3 2 3    -16*w_4_76+16*w_4_71
3 4 1   0 4 1 5 1 3 2 5    16*w_4_71-16*w_4_75
3 4 1   0 4 1 5 1 3 2 4    -16*w_4_73+16*w_4_71
3 4 1   0 4 1 5 1 6 2 4    16*w_4_73-16*w_4_71
3 4 1   0 4 1 5 2 3 1 3    -16*w_4_76+16*w_4_73
3 4 1   0 4 2 5 1 3 1 5    16*w_4_73-16*w_4_75
3 4 1   0 4 1 5 1 6 2 5    16*w_4_74-16*w_4_62
3 4 1   0 4 1 5 2 4 1 3    -1/18-16*w_4_63+16*w_4_74
3 4 1   0 4 2 3 1 3 1 5    16*w_4_74-16*w_4_91
3 4 1   0 4 1 5 2 3 1 4    -16*w_4_73+16*w_4_75
3 4 1   0 4 2 5 1 3 1 3    -16*w_4_76+16*w_4_75
3 4 1   0 4 1 5 2 6 1 4    -16*w_4_71+16*w_4_75
3 4 1   0 4 1 5 2 3 1 5    16*w_4_76-16*w_4_75
3 4 1   0 4 2 5 1 3 1 4    16*w_4_76-16*w_4_73
3 4 1   0 4 2 5 1 6 1 4    16*w_4_76-16*w_4_71
3 4 1   0 4 1 5 2 4 1 4    -16*w_4_60+16*w_4_83
3 4 1   0 4 2 3 1 3 1 3    -8*w_4_105+8*w_4_83
3 4 1   0 4 1 5 2 4 1 5    1/6-16*w_4_61+16*w_4_84
3 4 1   0 4 2 5 1 4 1 5    -16*w_4_61+16*w_4_84
3 4 1   0 4 1 5 2 6 1 5    1/6+16*w_4_91-16*w_4_62
3 4 1   0 4 2 5 1 4 1 3    -1/18-16*w_4_63+16*w_4_91
3 4 1   0 4 2 3 1 4 1 5    -16*w_4_74+16*w_4_91
3 4 1   0 4 2 3 1 4 1 4    8*w_4_105-8*w_4_83
3 4 1   0 4 2 5 1 4 1 4    1/6+16*w_4_105-16*w_4_60
3 4 1   0 4 1 2 1 4 3 5    -1/9
3 4 1   0 1 1 5 2 4 3 4    1/18+16*w_4_40
3 4 1   0 1 1 2 3 4 3 5    -1/6
3 4 1   0 1 1 5 2 3 3 4    16*w_4_17
3 4 1   0 1 1 5 3 6 2 3    16*w_4_21
3 4 1   0 1 1 5 3 4 2 3    16*w_4_22
3 4 1   0 1 1 5 4 6 2 3    -16*w_4_23
3 4 1   0 1 1 5 2 6 3 4    16*w_4_41
3 4 1   0 1 1 5 3 6 2 4    16*w_4_41
3 4 1   0 1 1 5 3 4 2 4    16*w_4_42
3 4 1   0 1 1 5 3 4 2 5    16*w_4_46
3 4 1   0 1 1 5 3 6 2 5    16*w_4_46
# 2 1 1 
3 4 1   0 1 0 3 2 3 3 4    1/6-16*w_4_24
3 4 1   0 1 0 3 2 4 3 4    1/6-16*w_4_28
3 4 1   0 1 0 3 2 6 3 4    1/6-16*w_4_31
3 4 1   0 4 1 3 0 4 2 3    1/6-16*w_4_106+16*w_4_61
3 4 1   0 4 1 3 0 4 2 4    1/6+16*w_4_60-16*w_4_86
3 4 1   0 4 1 3 0 4 2 5    1/6-16*w_4_107+16*w_4_62
3 4 1   0 1 0 3 2 3 3 5    -1/9+16*w_4_24
3 4 1   0 1 0 3 2 3 4 5    -1/9-16*w_4_25
3 4 1   0 1 0 3 2 4 3 5    -1/9-16*w_4_29
3 4 1   0 1 0 3 2 4 4 5    -1/9-16*w_4_30
3 4 1   0 2 0 5 1 4 3 4    1/18+16*w_4_40-16*w_4_43
3 4 1   0 2 0 5 1 4 3 5    1/18+16*w_4_40+16*w_4_43
3 4 1   0 4 1 3 0 6 2 3    1/18+16*w_4_63-16*w_4_107
3 4 1   0 4 1 3 0 6 2 4    1/18+16*w_4_63-16*w_4_85
3 4 1   0 1 0 3 2 6 3 5    -1/18-16*w_4_32
3 4 1   0 1 0 3 2 6 4 5    -1/18-16*w_4_33
3 4 1   0 1 0 2 3 4 3 4    1/6-8*w_4_3
3 4 1   0 1 0 5 2 3 3 5    -1/3+16*w_4_28
3 4 1   0 1 0 2 3 4 4 5    -1/6+16*w_4_4
3 4 1   0 1 0 5 2 3 4 5    -1/6-16*w_4_30
3 4 1   0 1 0 5 2 4 3 5    -1/6+16*w_4_43
3 4 1   0 1 0 5 2 6 3 5    -1/6-16*w_4_45
3 4 1   0 2 0 3 1 3 3 4    -16*w_4_24+16*w_4_10
3 4 1   0 2 0 3 1 3 3 5    16*w_4_24+16*w_4_10
3 4 1   0 4 1 2 0 4 3 4    -1/3-16*w_4_10
3 4 1   0 2 0 3 1 4 3 4    -16*w_4_28+16*w_4_16
3 4 1   0 2 0 5 1 3 3 5    16*w_4_28+16*w_4_16
3 4 1   0 4 0 5 1 2 4 5    -16*w_4_16
3 4 1   0 2 0 3 1 4 3 5    -16*w_4_29+16*w_4_17
3 4 1   0 2 0 5 1 3 3 4    16*w_4_29+16*w_4_17
3 4 1   0 4 0 5 1 2 3 5    -1/6-16*w_4_17
3 4 1   0 2 0 3 1 4 4 5    16*w_4_18-16*w_4_30
3 4 1   0 2 0 5 1 3 4 5    -16*w_4_18-16*w_4_30
3 4 1   0 4 0 5 1 2 3 4    -1/6-16*w_4_18
3 4 1   0 2 0 3 1 6 3 4    16*w_4_21-16*w_4_31
3 4 1   0 2 0 5 3 6 1 3    16*w_4_21+16*w_4_31
3 4 1   0 4 1 2 0 6 3 4    -1/6-16*w_4_21
3 4 1   0 2 0 3 1 6 3 5    16*w_4_22-16*w_4_32
3 4 1   0 2 0 5 3 4 1 3    16*w_4_22+16*w_4_32
3 4 1   0 4 1 2 0 6 4 5    1/6+16*w_4_22
\end{verbatim}
\end{tabular}
\end{table}

\begin{table}
\centerline{\textsc{Table}~\ref{TableAssoc4} (part 4).}\label{TableAssoc4-part4}
\vskip 1em
\begin{tabular}{p{0.5\textwidth} | p{0.5\textwidth}}
\renewcommand{\baselinestretch}{0.7}
\tiny\begin{verbatim}
3 4 1   0 2 0 3 1 6 4 5    16*w_4_23-16*w_4_33
3 4 1   0 2 0 5 4 6 1 3    -16*w_4_23-16*w_4_33
3 4 1   0 4 1 2 0 6 3 5    16*w_4_23
3 4 1   0 4 0 3 1 2 3 5    -16*w_4_40
3 4 1   0 2 0 5 1 6 3 4    -16*w_4_44+16*w_4_41
3 4 1   0 2 0 5 3 6 1 4    16*w_4_44+16*w_4_41
3 4 1   0 4 0 5 3 6 1 2    -16*w_4_41
3 4 1   0 2 0 5 1 6 3 5    -16*w_4_45+16*w_4_42
3 4 1   0 2 0 5 3 4 1 4    16*w_4_45+16*w_4_42
3 4 1   0 4 0 5 4 6 1 2    -16*w_4_42
3 4 1   0 2 0 5 3 4 1 5    -16*w_4_47+16*w_4_46
3 4 1   0 2 0 5 3 6 1 5    16*w_4_47+16*w_4_46
3 4 1   0 4 3 5 1 2 0 4    -1/6-16*w_4_46
3 4 1   0 4 0 3 1 3 2 3    16*w_4_60-16*w_4_64
3 4 1   0 4 0 5 1 4 2 4    16*w_4_60-16*w_4_86
3 4 1   0 4 0 3 1 3 2 4    16*w_4_61-16*w_4_65
3 4 1   0 4 0 5 1 4 2 5    16*w_4_61-16*w_4_87
3 4 1   0 4 0 3 1 3 2 5    -16*w_4_66+16*w_4_62
3 4 1   0 4 0 5 1 4 2 3    16*w_4_62-16*w_4_85
3 4 1   0 4 0 3 1 6 2 3    -16*w_4_66+16*w_4_63
3 4 1   0 4 0 5 1 3 2 3    -16*w_4_77+16*w_4_71
3 4 1   0 4 0 5 1 3 2 4    16*w_4_71-16*w_4_78
3 4 1   0 4 0 5 1 3 2 5    -16*w_4_79+16*w_4_71
3 4 1   0 4 0 5 1 6 2 3    -16*w_4_92+16*w_4_72
3 4 1   0 4 1 5 0 6 2 3    -16*w_4_118+16*w_4_72
3 4 1   0 4 0 5 2 6 1 3    -16*w_4_92+16*w_4_72
3 4 1   0 4 0 5 1 6 2 4    -16*w_4_79+16*w_4_73
3 4 1   0 4 1 5 0 6 2 4    16*w_4_73-16*w_4_78
3 4 1   0 4 0 5 2 3 1 3    16*w_4_73-16*w_4_77
3 4 1   0 4 0 5 1 6 2 5    1/18-16*w_4_66+16*w_4_74
3 4 1   0 4 1 5 0 6 2 5    16*w_4_74-16*w_4_107
3 4 1   0 4 0 5 2 4 1 3    16*w_4_74-16*w_4_85
3 4 1   0 4 0 5 2 3 1 4    -16*w_4_78+16*w_4_75
3 4 1   0 4 1 5 2 3 0 4    -16*w_4_77+16*w_4_75
3 4 1   0 4 0 5 2 6 1 4    -16*w_4_79+16*w_4_75
3 4 1   0 4 0 5 2 3 1 5    -16*w_4_79+16*w_4_76
3 4 1   0 4 1 5 2 3 0 5    16*w_4_76-16*w_4_78
3 4 1   0 4 2 5 1 3 0 4    16*w_4_76-16*w_4_77
3 4 1   0 4 0 5 2 4 1 4    -16*w_4_86+16*w_4_83
3 4 1   0 4 1 5 2 4 0 4    1/12+8*w_4_83-8*w_4_64
3 4 1   0 4 0 5 2 4 1 5    -16*w_4_87+16*w_4_84
3 4 1   0 4 1 5 2 4 0 5    1/6+16*w_4_84-16*w_4_65
3 4 1   0 4 2 3 0 4 1 3    -16*w_4_106+16*w_4_84
3 4 1   0 4 0 5 2 6 1 5    1/18-16*w_4_66+16*w_4_91
3 4 1   0 4 1 5 2 6 0 5    16*w_4_91-16*w_4_85
3 4 1   0 4 2 3 0 4 1 5    -16*w_4_107+16*w_4_91
3 4 1   0 4 2 3 0 4 1 4    16*w_4_105-16*w_4_86
3 4 1   0 4 2 5 1 4 0 4    1/12+8*w_4_105-8*w_4_64
3 4 1   0 1 0 2 3 4 3 5    -16*w_4_4
3 4 1   0 1 0 2 3 6 4 5    -16*w_4_5
3 4 1   0 2 0 3 1 3 4 5    -16*w_4_25
3 4 1   0 1 0 5 2 3 3 4    16*w_4_29
3 4 1   0 1 0 5 3 6 2 3    16*w_4_31
3 4 1   0 1 0 5 3 4 2 3    16*w_4_32
3 4 1   0 1 0 5 4 6 2 3    -16*w_4_33
3 4 1   0 1 0 5 2 4 3 4    -16*w_4_43
3 4 1   0 1 0 5 2 6 3 4    -16*w_4_44
3 4 1   0 1 0 5 3 6 2 4    16*w_4_44
3 4 1   0 1 0 5 3 4 2 4    16*w_4_45
3 4 1   0 1 0 5 3 4 2 5    -16*w_4_47
3 4 1   0 1 0 5 3 6 2 5    16*w_4_47
# 1 1 2 
3 4 1   0 4 1 3 2 3 2 3    -1/12-8*w_4_105+8*w_4_64
3 4 1   0 4 1 3 2 3 2 4    -1/6-16*w_4_84+16*w_4_65
3 4 1   0 4 1 3 2 4 2 4    -1/12-8*w_4_83+8*w_4_64
3 4 1   0 4 1 3 2 3 2 5    -1/18+16*w_4_66-16*w_4_91
3 4 1   0 4 1 3 2 4 2 5    -1/18+16*w_4_66-16*w_4_74
3 4 1   0 1 2 3 2 3 3 4    1/3+16*w_4_10
3 4 1   0 1 2 3 2 6 3 5    1/6+16*w_4_22
3 4 1   0 1 2 3 2 4 4 5    1/6+16*w_4_18
3 4 1   0 1 2 3 2 4 3 5    1/6+16*w_4_17
3 4 1   0 1 2 3 2 6 3 4    1/6+16*w_4_21
3 4 1   0 1 2 5 3 4 2 5    1/6+16*w_4_46
3 4 1   0 2 1 2 3 4 3 4    -1/6+8*w_4_3
3 4 1   0 2 1 2 3 4 3 5    -1/6+16*w_4_4
3 4 1   0 2 1 2 3 4 4 5    -16*w_4_4
3 4 1   0 2 1 2 3 6 4 5    16*w_4_5
3 4 1   0 2 1 3 2 3 3 4    16*w_4_24+16*w_4_10
3 4 1   0 4 1 2 2 4 3 4    1/9-16*w_4_24
3 4 1   0 2 1 3 2 3 3 5    -16*w_4_24+16*w_4_10
3 4 1   0 4 1 2 2 4 4 5    1/6-16*w_4_24
3 4 1   0 2 1 3 2 3 4 5    16*w_4_25
3 4 1   0 4 1 2 2 4 3 5    1/9+16*w_4_25
3 4 1   0 2 1 3 2 4 3 4    16*w_4_28+16*w_4_16
3 4 1   0 4 1 2 2 3 3 4    1/3-16*w_4_28
\end{verbatim}
&
\renewcommand{\baselinestretch}{0.7}
\tiny\begin{verbatim}
3 4 1   0 2 1 5 2 3 3 5    -16*w_4_28+16*w_4_16
3 4 1   0 4 2 5 1 2 4 5    -1/6+16*w_4_28
3 4 1   0 2 1 3 2 4 3 5    16*w_4_29+16*w_4_17
3 4 1   0 4 1 2 2 3 4 5    16*w_4_29
3 4 1   0 2 1 5 2 3 3 4    -16*w_4_29+16*w_4_17
3 4 1   0 4 2 5 1 2 3 5    1/9+16*w_4_29
3 4 1   0 2 1 3 2 4 4 5    16*w_4_18+16*w_4_30
3 4 1   0 4 1 2 2 3 3 5    1/6+16*w_4_30
3 4 1   0 2 1 5 2 3 4 5    -16*w_4_18+16*w_4_30
3 4 1   0 4 2 5 1 2 3 4    1/9+16*w_4_30
3 4 1   0 2 1 3 2 6 3 4    16*w_4_21+16*w_4_31
3 4 1   0 4 1 2 2 6 3 4    -16*w_4_31
3 4 1   0 2 1 5 3 6 2 3    16*w_4_21-16*w_4_31
3 4 1   0 4 5 6 1 2 2 5    1/6-16*w_4_31
3 4 1   0 2 1 3 2 6 3 5    16*w_4_22+16*w_4_32
3 4 1   0 4 1 2 2 6 4 5    16*w_4_32
3 4 1   0 2 1 5 3 4 2 3    16*w_4_22-16*w_4_32
3 4 1   0 4 3 5 1 2 2 5    1/18+16*w_4_32
3 4 1   0 2 1 3 2 6 4 5    16*w_4_23+16*w_4_33
3 4 1   0 4 1 2 2 6 3 5    16*w_4_33
3 4 1   0 2 1 5 4 6 2 3    -16*w_4_23+16*w_4_33
3 4 1   0 4 3 5 2 6 1 2    1/18+16*w_4_33
3 4 1   0 2 1 5 2 4 3 4    1/18+16*w_4_40+16*w_4_43
3 4 1   0 4 2 3 1 2 3 5    1/6-16*w_4_43
3 4 1   0 2 1 5 2 4 3 5    1/18+16*w_4_40-16*w_4_43
3 4 1   0 4 2 3 1 2 4 5    16*w_4_43
3 4 1   0 2 1 5 2 6 3 4    16*w_4_44+16*w_4_41
3 4 1   0 4 2 5 3 6 1 2    -16*w_4_44
3 4 1   0 2 1 5 3 6 2 4    -16*w_4_44+16*w_4_41
3 4 1   0 4 5 6 1 2 2 3    -16*w_4_44
3 4 1   0 2 1 5 2 6 3 5    16*w_4_45+16*w_4_42
3 4 1   0 4 2 5 4 6 1 2    -16*w_4_45
3 4 1   0 2 1 5 3 4 2 4    -16*w_4_45+16*w_4_42
3 4 1   0 4 3 5 1 2 2 3    1/6+16*w_4_45
3 4 1   0 2 1 5 3 4 2 5    16*w_4_47+16*w_4_46
3 4 1   0 4 3 5 1 2 2 4    -16*w_4_47
3 4 1   0 2 1 5 3 6 2 5    -16*w_4_47+16*w_4_46
3 4 1   0 4 5 6 1 2 2 4    -16*w_4_47
3 4 1   0 4 2 5 2 4 1 4    -16*w_4_60+16*w_4_64
3 4 1   0 4 2 5 2 4 1 5    -16*w_4_61+16*w_4_65
3 4 1   0 4 1 5 2 6 2 5    16*w_4_66-16*w_4_62
3 4 1   0 4 2 5 2 4 1 3    16*w_4_66-16*w_4_63
3 4 1   0 4 1 5 2 3 2 3    -16*w_4_76+16*w_4_77
3 4 1   0 4 2 5 1 3 2 5    16*w_4_77-16*w_4_75
3 4 1   0 4 2 5 2 3 1 4    -16*w_4_73+16*w_4_77
3 4 1   0 4 2 5 2 6 1 4    16*w_4_77-16*w_4_71
3 4 1   0 4 1 5 2 3 2 4    -16*w_4_73+16*w_4_78
3 4 1   0 4 2 5 1 3 2 3    -16*w_4_76+16*w_4_78
3 4 1   0 4 2 5 1 6 2 4    -16*w_4_71+16*w_4_78
3 4 1   0 4 2 5 2 3 1 5    16*w_4_78-16*w_4_75
3 4 1   0 4 1 5 2 3 2 5    16*w_4_79-16*w_4_75
3 4 1   0 4 2 5 1 3 2 4    16*w_4_79-16*w_4_73
3 4 1   0 4 1 5 2 6 2 4    16*w_4_79-16*w_4_71
3 4 1   0 4 2 5 2 3 1 3    16*w_4_79-16*w_4_76
3 4 1   0 4 1 5 2 4 2 3    -1/18-16*w_4_63+16*w_4_85
3 4 1   0 4 2 3 1 3 2 5    -16*w_4_91+16*w_4_85
3 4 1   0 4 2 3 1 6 2 4    -16*w_4_74+16*w_4_85
3 4 1   0 4 2 5 2 6 1 5    -16*w_4_62+16*w_4_85
3 4 1   0 4 1 5 2 4 2 4    -1/6-16*w_4_60+16*w_4_86
3 4 1   0 4 2 3 1 3 2 3    -16*w_4_105+16*w_4_86
3 4 1   0 4 2 3 1 4 2 4    16*w_4_86-16*w_4_83
3 4 1   0 4 2 5 1 4 2 4    -16*w_4_60+16*w_4_86
3 4 1   0 4 1 5 2 4 2 5    -16*w_4_61+16*w_4_87
3 4 1   0 4 2 3 1 3 2 4    16*w_4_87-16*w_4_84
3 4 1   0 4 1 5 2 6 2 3    16*w_4_92-16*w_4_72
3 4 1   0 4 2 5 2 6 1 3    16*w_4_92-16*w_4_72
3 4 1   0 4 2 3 1 4 2 3    16*w_4_106-16*w_4_84
3 4 1   0 4 2 5 1 4 2 5    -1/6+16*w_4_106-16*w_4_61
3 4 1   0 4 2 3 1 4 2 5    -16*w_4_74+16*w_4_107
3 4 1   0 4 2 5 1 4 2 3    -1/18-16*w_4_63+16*w_4_107
3 4 1   0 4 2 3 1 6 2 3    16*w_4_107-16*w_4_91
3 4 1   0 4 2 5 1 6 2 5    -1/6+16*w_4_107-16*w_4_62
3 4 1   0 4 2 5 1 6 2 3    16*w_4_118-16*w_4_72
3 4 1   0 1 2 3 2 4 3 4    16*w_4_16
3 4 1   0 1 2 3 2 6 4 5    16*w_4_23
3 4 1   0 1 2 5 2 4 3 4    16*w_4_40
3 4 1   0 1 2 5 2 6 3 4    16*w_4_41
3 4 1   0 1 2 5 2 6 3 5    16*w_4_42
# 1 1 1 
3 4 1   0 4 1 3 2 3 3 4    1/6+16*w_4_108+16*w_4_67
3 4 1   0 4 1 3 2 4 3 4    1/6-16*w_4_67+16*w_4_90
3 4 1   0 4 1 3 2 6 3 4    1/6+16*w_4_111
3 4 1   0 4 1 3 2 3 3 5    1/18-16*w_4_110+16*w_4_68
3 4 1   0 4 1 3 2 3 4 5    1/18-16*w_4_109+16*w_4_69
3 4 1   0 4 1 3 2 4 3 5    1/18+16*w_4_89+16*w_4_69
3 4 1   0 4 1 3 2 4 4 5    1/18+16*w_4_68+16*w_4_88
\end{verbatim}
\end{tabular}
\end{table}

\begin{table}
\centerline{\textsc{Table}~\ref{TableAssoc4} (part 5).}\label{TableAssoc4-part5}
\vskip 1em
\begin{tabular}{p{0.5\textwidth} | p{0.5\textwidth}}
\renewcommand{\baselinestretch}{0.7}
\tiny\begin{verbatim}
3 4 1   0 4 1 3 2 6 3 5    1/36+16*w_4_70-16*w_4_113
3 4 1   0 4 1 3 2 6 4 5    1/36-16*w_4_112+16*w_4_70
3 4 1   0 1 2 3 3 4 3 4    -1/6+8*w_4_34
3 4 1   0 1 2 3 3 4 4 5    1/6+16*w_4_36
3 4 1   0 1 2 5 3 4 3 4    -1/6+16*w_4_48
3 4 1   0 1 2 5 3 4 4 5    1/6+16*w_4_50
3 4 1   0 2 1 3 3 4 3 4    16*w_4_34
3 4 1   0 4 1 2 3 4 3 4    -1/6+8*w_4_34
3 4 1   0 2 1 3 3 4 3 5    32*w_4_35
3 4 1   0 4 1 2 3 4 4 5    -16*w_4_35
3 4 1   0 2 1 3 3 4 4 5    32*w_4_36
3 4 1   0 4 1 2 3 4 3 5    -1/6-16*w_4_36
3 4 1   0 2 1 3 3 6 3 5    16*w_4_37
3 4 1   0 4 1 2 4 6 4 5    8*w_4_37
3 4 1   0 2 1 3 3 6 4 5    32*w_4_38
3 4 1   0 4 1 2 3 6 4 5    16*w_4_38
3 4 1   0 2 1 3 4 6 4 5    16*w_4_39
3 4 1   0 4 1 2 3 6 3 5    8*w_4_39
3 4 1   0 2 1 5 3 4 3 4    32*w_4_48
3 4 1   0 4 3 5 1 2 3 5    -1/6+16*w_4_48
3 4 1   0 2 1 5 3 4 3 5    32*w_4_49
3 4 1   0 4 3 5 1 2 4 5    16*w_4_49
3 4 1   0 2 1 5 3 4 4 5    32*w_4_50
3 4 1   0 4 3 5 1 2 3 4    -1/6-16*w_4_50
3 4 1   0 2 1 5 3 6 3 4    32*w_4_51
3 4 1   0 4 5 6 1 2 3 5    -16*w_4_51
3 4 1   0 2 1 5 3 6 3 5    32*w_4_52
3 4 1   0 4 5 6 1 2 4 5    -16*w_4_52
3 4 1   0 2 1 5 3 6 4 5    32*w_4_53
3 4 1   0 4 5 6 1 2 3 4    16*w_4_53
3 4 1   0 2 1 5 4 6 3 4    32*w_4_54
3 4 1   0 4 3 5 3 6 1 2    -16*w_4_54
3 4 1   0 2 1 5 4 6 3 5    32*w_4_55
3 4 1   0 4 3 5 4 6 1 2    -16*w_4_55
3 4 1   0 4 1 5 2 3 3 4    16*w_4_80+16*w_4_81
3 4 1   0 4 2 5 1 3 3 5    -16*w_4_80+16*w_4_82
3 4 1   0 4 1 5 2 3 3 5    16*w_4_81+16*w_4_82
3 4 1   0 4 2 5 1 3 4 5    -16*w_4_80-16*w_4_81
3 4 1   0 4 1 5 2 3 4 5    -16*w_4_80+16*w_4_82
3 4 1   0 4 2 5 1 3 3 4    16*w_4_81+16*w_4_82
3 4 1   0 4 1 5 2 4 3 4    1/18+16*w_4_68+16*w_4_88
3 4 1   0 4 2 3 1 3 3 5    -16*w_4_110-16*w_4_88
3 4 1   0 4 1 5 2 4 3 5    1/18+16*w_4_89+16*w_4_69
3 4 1   0 4 2 3 1 3 4 5    -16*w_4_89-16*w_4_109
3 4 1   0 4 1 5 2 4 4 5    1/6-16*w_4_67+16*w_4_90
3 4 1   0 4 2 3 1 3 3 4    16*w_4_108+16*w_4_90
3 4 1   0 4 1 5 2 6 3 4    16*w_4_99+16*w_4_93
3 4 1   0 4 2 5 3 6 1 3    -16*w_4_93+16*w_4_119
3 4 1   0 4 1 5 2 6 3 5    16*w_4_102+16*w_4_94
3 4 1   0 4 2 5 4 6 1 3    -16*w_4_94+16*w_4_120
3 4 1   0 4 1 5 2 6 4 5    16*w_4_95-16*w_4_96
3 4 1   0 4 2 5 3 4 1 3    16*w_4_95-16*w_4_121
3 4 1   0 4 1 5 3 4 2 3    -16*w_4_121+16*w_4_96
3 4 1   0 4 3 5 1 3 2 5    16*w_4_95-16*w_4_96
3 4 1   0 4 1 5 3 4 2 4    16*w_4_103+16*w_4_97
3 4 1   0 4 3 5 1 3 2 3    16*w_4_122-16*w_4_97
3 4 1   0 4 1 5 3 4 2 5    16*w_4_98-16*w_4_124
3 4 1   0 4 3 5 1 3 2 4    16*w_4_123-16*w_4_98
3 4 1   0 4 1 5 3 6 2 3    16*w_4_99+16*w_4_119
3 4 1   0 4 5 6 1 3 2 5    16*w_4_99+16*w_4_93
3 4 1   0 4 1 5 3 6 2 4    32*w_4_100
3 4 1   0 4 5 6 1 3 2 3    16*w_4_100
3 4 1   0 4 1 5 3 6 2 5    -16*w_4_129+16*w_4_101
3 4 1   0 4 5 6 1 3 2 4    -16*w_4_128+16*w_4_101
3 4 1   0 4 1 5 4 6 2 3    16*w_4_102+16*w_4_120
3 4 1   0 4 3 5 2 6 1 3    16*w_4_102+16*w_4_94
3 4 1   0 4 1 5 4 6 2 4    16*w_4_103+16*w_4_97
3 4 1   0 4 3 5 2 3 1 3    16*w_4_122+16*w_4_103
3 4 1   0 4 1 5 4 6 2 5    -16*w_4_134+16*w_4_104
3 4 1   0 4 3 5 2 4 1 3    -16*w_4_133+16*w_4_104
3 4 1   0 4 2 3 1 4 3 4    16*w_4_108+16*w_4_90
3 4 1   0 4 2 5 1 4 4 5    -1/6-16*w_4_108-16*w_4_67
3 4 1   0 4 2 3 1 4 3 5    16*w_4_89+16*w_4_109
3 4 1   0 4 2 5 1 4 3 5    1/18-16*w_4_109+16*w_4_69
3 4 1   0 4 2 3 1 4 4 5    16*w_4_110+16*w_4_88
3 4 1   0 4 2 5 1 4 3 4    1/18-16*w_4_110+16*w_4_68
3 4 1   0 4 2 3 1 6 3 4    32*w_4_111
3 4 1   0 4 5 6 1 6 2 5    1/6+16*w_4_111
3 4 1   0 4 2 3 1 6 3 5    16*w_4_112-16*w_4_113
3 4 1   0 4 3 5 1 6 2 5    1/36-16*w_4_112+16*w_4_70
3 4 1   0 4 2 3 1 6 4 5    -16*w_4_112+16*w_4_113
3 4 1   0 4 3 5 2 6 1 5    1/36+16*w_4_70-16*w_4_113
3 4 1   0 4 2 5 1 6 3 4    16*w_4_99+16*w_4_119
3 4 1   0 4 5 6 1 6 2 3    -16*w_4_93+16*w_4_119
3 4 1   0 4 2 5 1 6 3 5    16*w_4_102+16*w_4_120
\end{verbatim}
&
\renewcommand{\baselinestretch}{0.7}
\tiny\begin{verbatim}
3 4 1   0 4 3 5 1 6 2 3    16*w_4_94-16*w_4_120
3 4 1   0 4 2 5 1 6 4 5    16*w_4_121-16*w_4_96
3 4 1   0 4 3 5 2 3 1 5    16*w_4_95-16*w_4_121
3 4 1   0 4 2 5 3 4 1 4    16*w_4_122+16*w_4_103
3 4 1   0 4 2 5 4 6 1 4    -16*w_4_122+16*w_4_97
3 4 1   0 4 2 5 3 4 1 5    16*w_4_123-16*w_4_124
3 4 1   0 4 5 6 1 6 2 4    16*w_4_123-16*w_4_98
3 4 1   0 4 3 5 2 3 1 4    16*w_4_123-16*w_4_124
3 4 1   0 4 5 6 1 4 2 5    16*w_4_98-16*w_4_124
3 4 1   0 4 2 5 3 6 1 5    16*w_4_128-16*w_4_129
3 4 1   0 4 3 5 1 6 2 4    -16*w_4_128+16*w_4_101
3 4 1   0 4 3 5 2 6 1 4    -16*w_4_129+16*w_4_101
3 4 1   0 4 5 6 1 4 2 3    16*w_4_128-16*w_4_129
3 4 1   0 4 2 5 4 6 1 5    -16*w_4_134+16*w_4_133
3 4 1   0 4 3 5 2 4 1 5    -16*w_4_133+16*w_4_104
3 4 1   0 4 3 5 1 4 2 3    16*w_4_134-16*w_4_133
3 4 1   0 4 3 5 1 4 2 5    -16*w_4_134+16*w_4_104
3 4 1   0 4 3 5 2 4 1 4    32*w_4_138
3 4 1   0 4 5 6 1 4 2 4    16*w_4_138
3 4 1   0 1 2 3 3 4 3 5    16*w_4_35
3 4 1   0 1 2 3 3 6 3 5    8*w_4_37
3 4 1   0 1 2 3 3 6 4 5    16*w_4_38
3 4 1   0 1 2 3 4 6 4 5    8*w_4_39
3 4 1   0 1 2 5 3 4 3 5    16*w_4_49
3 4 1   0 1 2 5 3 6 3 4    16*w_4_51
3 4 1   0 1 2 5 3 6 3 5    16*w_4_52
3 4 1   0 1 2 5 3 6 4 5    16*w_4_53
3 4 1   0 1 2 5 4 6 3 4    16*w_4_54
3 4 1   0 1 2 5 4 6 3 5    16*w_4_55
3 4 1   0 4 2 5 3 6 1 4    16*w_4_100
3 4 1   0 4 3 5 1 4 2 4    16*w_4_138
\end{verbatim}
\end{tabular}
\end{table}

\noindent
The table below contains the output of the command
\begin{verbatim}
    $ reduce_mod_jacobi assoc4_intermsof10_part100.txt
\end{verbatim}
as described in Implementation \ref{ImplReduceModJacobi}.

\begin{table}[h!]
\caption{Sample output of {\tt reduce\symbol{"5F}mod\symbol{"5F}jacobi}.}
\label{TblReduceModJacobi}
\begin{tabular}{p{0.5\textwidth} | p{0.5\textwidth}}
\renewcommand{\baselinestretch}{0.8}
\tiny\begin{verbatim}
3 4 1   0 1 0 3 2 6 3 4    -24+c_1_1221_211
3 4 1   0 4 1 3 2 6 3 4    -8+c_1_1240_111
3 4 1   0 1 0 3 2 3 4 5    -48+c_1_1221_211 
                              -c_1_513_211
3 4 1   0 1 0 3 2 4 3 5    24-c_1_1221_211
3 4 1   0 4 1 3 2 4 3 5    24-c_1_1240_111 
                              -c_1_540_111
3 4 1   0 1 2 3 3 4 4 5    -8-c_1_1228_111
3 4 1   0 1 2 5 3 4 3 4    -8-c_1_1228_111
3 4 1   0 2 0 3 1 4 3 5    24-c_1_1005_211
3 4 1   0 2 0 5 1 3 3 4    -24-c_1_516_211
3 4 1   0 2 0 3 1 6 3 4    -24+c_1_1005_211
3 4 1   0 2 0 5 3 6 1 3    24+c_1_516_211
3 4 1   0 2 1 3 2 3 4 5    48+c_1_1230_112
                              -c_1_1008_112
3 4 1   0 4 1 2 2 4 3 5    48-c_1_525_112
                              -c_1_1239_112
3 4 1   0 2 1 3 2 4 3 5    -24-c_1_1230_112
3 4 1   0 4 1 2 2 3 4 5    -24+c_1_1239_112
3 4 1   0 2 1 5 2 3 3 4    24-c_1_1008_112
3 4 1   0 4 2 5 1 2 3 5    -24+c_1_525_112
3 4 1   0 2 1 3 2 6 3 4    24+c_1_1230_112
3 4 1   0 4 1 2 2 6 3 4    -24+c_1_1239_112
3 4 1   0 2 1 5 3 6 2 3    -24+c_1_1008_112
3 4 1   0 4 5 6 1 2 2 5    -24+c_1_525_112
3 4 1   0 2 1 3 3 4 4 5    -16-c_1_1012_111
3 4 1   0 4 1 2 3 4 3 5    8+c_1_529_111
3 4 1   0 2 1 3 3 6 4 5    -16-c_1_1012_111
3 4 1   0 4 1 2 3 6 4 5    -8-c_1_529_111
3 4 1   0 2 1 5 3 4 3 4    -16-c_1_1012_111
3 4 1   0 4 3 5 1 2 3 5    -8-c_1_529_111
3 4 1   0 2 1 5 3 6 3 5    16+c_1_1012_111
3 4 1   0 4 5 6 1 2 4 5    -8-c_1_529_111
3 4 1   0 4 1 5 2 3 3 4    8-c_1_1023_111
3 4 1   0 4 2 5 1 3 3 5    -16+c_1_540_111
3 4 1   0 4 1 5 2 3 3 5    -8-c_1_538_111
3 4 1   0 4 2 5 1 3 4 5    -8+c_1_1021_111
3 4 1   0 4 1 5 2 3 4 5    -16+c_1_1242_111
3 4 1   0 4 2 5 1 3 3 4    -8-c_1_1245_111
3 4 1   0 4 1 5 2 4 3 5    24-c_1_536_111
                             -c_1_1242_111
3 4 1   0 4 2 3 1 3 4 5    -24-c_1_1245_111
                              +c_1_1019_111
3 4 1   0 4 1 5 2 6 3 4    -16+c_1_1242_111
3 4 1   0 4 2 5 3 6 1 3    8+c_1_1245_111
3 4 1   0 4 1 5 2 6 3 5    -8-c_1_538_111
3 4 1   0 4 2 5 4 6 1 3    -8+c_1_1021_111
3 4 1   0 4 1 5 3 4 2 5    -16+c_1_1242_111
3 4 1   0 4 3 5 1 3 2 4    8+c_1_1245_111
\end{verbatim}
&
\renewcommand{\baselinestretch}{0.8}
\tiny\begin{verbatim}
3 4 1   0 4 1 5 3 6 2 3    -8+c_1_1023_111
3 4 1   0 4 5 6 1 3 2 5    -16+c_1_540_111
3 4 1   0 4 1 5 3 6 2 4    32-c_1_1242_111
                             -c_1_540_111
3 4 1   0 4 5 6 1 3 2 3    16-c_1_1023_111
                             +c_1_1245_111
3 4 1   0 4 1 5 4 6 2 3    -16+c_1_1242_111
3 4 1   0 4 3 5 2 6 1 3    -8-c_1_1245_111
3 4 1   0 4 2 3 1 4 3 5    24+c_1_538_111
                            -c_1_1019_111
3 4 1   0 4 2 3 1 6 3 4    -16+c_1_1019_111
3 4 1   0 4 5 6 1 6 2 5    -8+c_1_536_111
3 4 1   0 4 2 5 1 6 3 4    -8+c_1_1021_111
3 4 1   0 4 5 6 1 6 2 3    8+c_1_538_111
3 4 1   0 4 2 5 1 6 3 5    -16+c_1_540_111
3 4 1   0 4 3 5 1 6 2 3    8-c_1_1023_111
3 4 1   0 4 2 5 3 4 1 5    -8+c_1_1021_111
3 4 1   0 4 5 6 1 6 2 4    8+c_1_538_111
3 4 1   0 4 3 5 2 3 1 4    -8+c_1_1023_111
3 4 1   0 4 5 6 1 4 2 5    -16+c_1_540_111
3 4 1   0 2 0 3 1 3 4 5    -48+c_1_1005_211-c_1_516_211
3 4 1   0 1 0 5 2 3 3 4    -24-c_1_513_211
3 4 1   0 1 0 5 3 6 2 3    24+c_1_513_211
3 4 1   0 1 2 3 3 6 4 5    -8-c_1_1228_111
3 4 1   0 1 2 5 3 6 3 5    8+c_1_1228_111
3 4 1   0 4 2 5 3 6 1 4    16+c_1_538_111-c_1_1021_111
3 4 1   0 4 1 3 2 3 4 5    -c_1_1023_111+c_1_1240_111
3 4 1   0 4 2 5 1 4 3 5    c_1_536_111-c_1_1021_111

3 4 1   0 1 2 3 0 3 5 4    c_1_513_211==-24
3 4 1   0 2 1 3 0 3 5 4    c_1_516_211==-24
3 4 1   1 2 2 3 0 3 5 4    c_1_525_112==24
3 4 1   1 2 3 5 0 3 5 4    c_1_529_111==-8
3 4 1   1 4 2 3 0 3 5 4    c_1_536_111==8
3 4 1   1 4 2 5 0 3 5 4    c_1_538_111==-8
3 4 1   1 5 2 3 0 3 5 4    c_1_540_111==16
3 4 1   0 2 0 3 1 3 5 4    c_1_1005_211==24
3 4 1   0 2 2 3 1 3 5 4    c_1_1008_112==24
3 4 1   0 2 3 5 1 3 5 4    c_1_1012_111==-16
3 4 1   0 4 2 3 1 3 5 4    c_1_1019_111==16
3 4 1   0 4 2 5 1 3 5 4    c_1_1021_111==8
3 4 1   0 5 2 3 1 3 5 4    c_1_1023_111==8
3 4 1   0 1 0 3 2 3 5 4    c_1_1221_211==24
3 4 1   0 1 3 5 2 3 5 4    c_1_1228_111==-8
3 4 1   0 2 1 3 2 3 5 4    c_1_1230_112==-24
3 4 1   0 4 1 2 2 3 5 4    c_1_1239_112==24
3 4 1   0 4 1 3 2 3 5 4    c_1_1240_111==8
3 4 1   0 4 1 5 2 3 5 4    c_1_1242_111==16
3 4 1   0 5 1 3 2 3 5 4    c_1_1245_111==-8
\end{verbatim}
\end{tabular}
\end{table}

\noindent
The first part of the output lists the graph series $S^{(1)}-\Diamond$, reduced modulo skew\/-\/symmetry, wherein the coefficients of $\Diamond$ are still undetermined.
The second part of the ouput (after the blank line) specifies the coefficients such that $S^{(1)}=\Diamond$.
Every coefficient in the second part is preceded by the encoding of the Leibniz graph that specifies a differential operator acting on the Jacobi identity.
Such a differential operator expands into a sum of graphs that can be read in the first part of the output.

\clearpage

\section{Gauge~transformation~that~removes $4$~master\/-\/parameters~out~of~$10$}
\label{AppGaugeEncoding}

Encodings of graphs (see Implementation \ref{DefEncoding} on p. \pageref{DefEncoding}) built over one sink vertex are followed by their coefficients, in the following table containing the gauge transformation which was claimed to exist in Theorem \ref{ThGauge}.

\begin{table}[h]
\caption{Gauge~transformation~that~removes~$4$~master\/-\/parameters~out~of~$10$.}
\label{TableGauge}
\renewcommand{\baselinestretch}{0.8}
\tiny\begin{verbatim}
h^0:
1 0 1                      1
h^4:
1 4 1   0 2 0 3 1 4 0 3    16*w_4_101
1 4 1   0 2 0 3 1 4 1 3    8*w_4_101
1 4 1   0 2 0 3 1 4 2 3    8*w_4_101
1 4 1   0 2 1 3 0 4 1 2    -8*w_4_101
1 4 1   0 2 1 3 0 4 2 3    8*w_4_101
1 4 1   0 2 1 3 1 4 0 2    -8*w_4_101
1 4 1   0 2 1 3 2 4 0 2    -8*w_4_101
1 4 1   0 2 0 3 0 4 1 3    -16*w_4_102
1 4 1   0 2 0 3 1 4 1 3    -8*w_4_102
1 4 1   0 2 0 3 2 4 1 2    -8*w_4_102
1 4 1   0 2 0 3 2 4 1 3    -16*w_4_102
1 4 1   0 2 1 3 0 4 1 2    -8*w_4_102
1 4 1   0 2 1 3 0 4 1 3    -8*w_4_102
1 4 1   0 2 0 3 0 4 1 2    16*w_4_119
1 4 1   0 2 0 3 1 4 1 2    16*w_4_119
1 4 1   0 2 0 3 1 4 1 3    8*w_4_119
1 4 1   0 2 0 3 2 4 1 2    8*w_4_119
1 4 1   0 2 1 3 0 4 1 2    8*w_4_119
1 4 1   0 2 3 4 0 4 1 2    -8*w_4_119
1 4 1   0 2 0 3 0 1 1 2    -32*w_4_125
1 4 1   0 2 0 3 1 2 1 2    16*w_4_125
1 4 1   0 2 0 3 1 2 1 3    -16*w_4_125
1 4 1   0 2 0 3 1 2 2 3    16*w_4_125
1 4 1   0 2 0 3 1 4 1 2    16*w_4_125
1 4 1   0 2 0 3 1 4 1 3    -16*w_4_125
1 4 1   0 2 0 3 1 4 2 3    16*w_4_125
\end{verbatim}
\end{table}

\end{document}